\documentclass[aap]{imsart}

\RequirePackage{amsthm,amsmath,amsfonts,amssymb}
\RequirePackage[numbers,sort&compress]{natbib}
\RequirePackage[colorlinks,citecolor=blue,urlcolor=blue]{hyperref}
\RequirePackage{graphicx}

\usepackage{pdflscape}
\usepackage{amsmath}
\usepackage{fix-cm}
\usepackage{mathrsfs}
\usepackage{graphicx}
\usepackage{comment}
\usepackage{dsfont}
\usepackage{bm}
\usepackage{xcolor}
\usepackage{bigints}
\allowdisplaybreaks[4]
\usepackage{booktabs}
\usepackage{tabularx}
\usepackage{array}

\usepackage{enumerate}
\newenvironment{enumeratei}{\begin{enumerate}[\upshape (i)]\setlength{\itemsep}{0pt}}{\end{enumerate}}
\newenvironment{enumeratea}{\begin{enumerate}[\upshape (a)]\setlength{\itemsep}{0pt}}{\end{enumerate}}

\startlocaldefs
\theoremstyle{plain}
\newtheorem{theorem}{Theorem}[section]
\newtheorem{proposition}[theorem]{Proposition}
\newtheorem{lemma}[theorem]{Lemma}
\newtheorem{corollary}[theorem]{Corollary}
\newtheorem{remark}[theorem]{Remark}
\newtheorem{assumption}{Assumption}

\theoremstyle{definition}
\newtheorem{definition}[theorem]{Definition}

\newcommand{\cB}{\mathcal{B}}\newcommand{\cC}{\mathcal{C}}
\newcommand{\cD}{\mathcal{D}}\newcommand{\cF}{\mathcal{F}}
\newcommand{\cG}{\mathcal{G}}\newcommand{\cH}{\mathcal{H}}
\newcommand{\cK}{\mathcal{K}}
\newcommand{\cN}{\mathcal{N}}\newcommand{\cO}{\mathcal{O}}
\newcommand{\cP}{\mathcal{P}}
\newcommand{\cS}{\mathcal{S}}
\newcommand{\cX}{\mathcal{X}}
\newcommand{\cY}{\mathcal{Y}}  
\newcommand{\vS}{\mathbf{S}}

\newcommand{\vX}{\mathbf{X}}

\newcommand{\vf}{\mathbf{f}}
\newcommand{\vg}{\mathbf{g}}

\newcommand{\mvT}{\boldsymbol{T}}
\newcommand{\mvX}{\boldsymbol{X}}

\newcommand{\mvf}{\boldsymbol{f}}
\newcommand{\mvh}{\boldsymbol{h}}
\newcommand{\mvl}{\boldsymbol{l}}
\newcommand{\mvp}{\boldsymbol{p}}
\newcommand{\mvr}{\boldsymbol{r}}

\newcommand{\mvx}{\boldsymbol{x}}\newcommand{\mvy}{\boldsymbol{y}}

\newcommand{\mvGamma}{\boldsymbol{\Gamma}}
\newcommand{\mvphi}{\boldsymbol{\phi}}

\newcommand{\mvnu}{\boldsymbol{\nu}}

\newcommand{\f}[1]{\mathfrak{#1}}
\newcommand{\fC}{\mathfrak{C}}
\newcommand{\fD}{\mathfrak{D}}\newcommand{\fF}{\mathfrak{F}}

\newcommand{\fL}{\mathfrak{L}}

\newcommand{\bA}{\mathbb{A}}\newcommand{\bC}{\mathbb{C}}
\newcommand{\bD}{\mathbb{D}}\newcommand{\bE}{\mathbb{E}}\newcommand{\bF}{\mathbb{F}}
\newcommand{\bH}{\mathbb{H}}

\newcommand{\bN}{\mathbb{N}}
\newcommand{\bP}{\mathbb{P}}\newcommand{\bQ}{\mathbb{Q}}\newcommand{\bR}{\mathbb{R}}

\newcommand{\rX}{\mathrm{X}}

\DeclareMathOperator{\E}{\mathds{E}}

\DeclareMathOperator{\argmin}{argmin}

\newcommand{\CalibrationPanel}[2]{%
  \par\noindent
  \begin{minipage}{\linewidth}
    \centering
    {\small\bfseries #1\par}
    \vspace{1mm}
    \includegraphics[
      width=\linewidth,
      height=0.19\textheight,
      keepaspectratio
    ]{#2}
  \end{minipage}
  \par
}

\newcommand{\CalibrationFour}[6]{%
  \clearpage
  \begin{figure}[p]
    \centering

    \CalibrationPanel
      {(a) Unsquared Euclidean cost: $c_1$, $\varepsilon=0.05$}
      {#1}

    \vspace{1mm}

    \CalibrationPanel
      {(b) Unsquared Euclidean cost: $c_1$, $\varepsilon=0.2$}
      {#2}

    \vspace{1mm}

    \CalibrationPanel
      {(c) Squared Euclidean cost: $c_2$, $\varepsilon=0.05$}
      {#3}

    \vspace{1mm}

    \CalibrationPanel
      {(d) Squared Euclidean cost: $c_2$, $\varepsilon=0.2$}
      {#4}

    \caption{%
    Empirical coverage of bootstrap confidence intervals for the displayed statistics. The input marginals are $\mu_j=\mathcal{L}(2Y_j-1)$, $j=1,2,3$, where $#5$. Samples are drawn independently within and across the three marginal inputs.
    Each sample size and condition uses $R=5000$ original
    datasets and $B=4999$ bootstrap resamples per dataset.
    Curves correspond to
    $n\in\{50,100,250,500,1000\}$.
    The horizontal and vertical axes show nominal and
    empirical coverage, respectively; dashed line $y=x$ serves as reference.
    }
    \label{#6}
  \end{figure}
  \clearpage
}

\newcommand{\MarginalA}{%
  Y_1\sim\mathrm{Beta}(2,8),\;
  Y_2\sim\mathrm{Beta}(5,5),\;
  Y_3\sim\mathrm{Beta}(8,2)}

\newcommand{\MarginalC}{%
  Y_1\sim\mathrm{Beta}(2,6),\;
  Y_2\sim\mathrm{Beta}(4,4),\;
  Y_3\sim\mathrm{Beta}(6,2)}

\newcommand{\MarginalD}{%
  Y_1\sim\mathrm{Beta}(2,2),\;
  Y_2\sim\mathrm{Beta}(5,5),\;
  Y_3\sim\mathrm{Beta}(9,9)}

\newcommand{\MarginalE}{%
  Y_1\sim\mathrm{Beta}(2,4),\;
  Y_2\sim\mathrm{Beta}(3,5),\;
  Y_3\sim\mathrm{Beta}(6,3)}

\DeclareRobustCommand{\CostDefinition}[1]{%
  \ifnum#1=1\relax
    c_1(x)=\frac14\sum_{j<k}|x_j-x_k|%
  \else
    c_2(x)=\frac18\sum_{j<k}(x_j-x_k)^2%
  \fi
}

\endlocaldefs

\newcommand{\xc}[1]{{{\textcolor{purple}{#1}}}}
\newcommand{\pl}[1]{{{\textcolor{orange}{#1}}}}

\begin{document}

\begin{frontmatter}
\title{Sample complexity and weak limits of nonsmooth multimarginal Schr\"{o}dinger system with application to optimal transport barycenter}
\runtitle{nonsmooth multimarginal Schr\"{o}dinger system}
%\title{A sample article title with some additional note\thanksref{t1}}
%\runtitle{Multimarginal Schr\"{o}dinger Barycenter}
%\thankstext{T1}{A sample additional note to the title.}

\begin{aug}
%%%%%%%%%%%%%%%%%%%%%%%%%%%%%%%%%%%%%%%%%%%%%%%
%% Only one address is permitted per author. %%
%% Only division, organization and e-mail is %%
%% included in the address.                  %%
%% Additional information can be included in %%
\author[A]{\fnms{Pengtao}~\snm{Li}\ead[label=e1]{pengtaol@usc.edu}}
\and
\author[A]{\fnms{Xiaohui}~\snm{Chen}\ead[label=e2]{xiaohuic@usc.edu}\orcid{0009-0009-0328-6159}}
\runauthor{Li and Chen}
%%%%%%%%%%%%%%%%%%%%%%%%%%%%%%%%%%%%%%%%%%%%%%
%% Addresses                                %%
%%%%%%%%%%%%%%%%%%%%%%%%%%%%%%%%%%%%%%%%%%%%%%
\address[A]{University of Southern California\\
%\printead[presep={ ,\ }]{e1,e2}}
\printead{e1,e2}}
\end{aug}

\begin{abstract}
Multimarginal optimal transport (MOT) has emerged as a useful framework for many applied problems. However, compared to the well-studied classical two-marginal optimal transport theory, analysis of MOT is far more challenging and remains much less developed. In this paper, we study the statistical estimation and inference problems for the entropic MOT (EMOT), whose optimal solution is characterized by the multimarginal Schr\"{o}dinger system. Assuming only boundedness of the cost function, we derive sharp sample complexity for estimating several key quantities pertaining to EMOT (cost functional and Schr\"{o}dinger coupling) from point clouds that are randomly sampled from the input marginal distributions. Moreover, with substantially weaker smoothness assumption on the cost function than the existing literature, we derive distributional limits and bootstrap validity of various key EMOT objects. As an application, we propose the multimarginal Schr\"{o}dinger barycenter as a new and natural way to regularize the exact Wasserstein barycenter and demonstrate its statistical optimality.

\end{abstract}



\end{frontmatter}

\section{Introduction}
Multimarginal optimal transport (MOT) problems arise naturally in a broad range of scientific domains, including economics \cite{CarlierIvarEcon10, Pierre_Robert_Lars10}, finance \cite{Mathias_Labordere_Penkner13, Hamza_Quentin_Luca_Brendan24}, 
machine learning \cite{multimarginal_Wasserstein_Gan19, ZoeMichalJamesCuturi24, Michael_Nicholas_Michael_Eric23}, tomography \cite{Application_Tomographic17}, quantum physics \cite{GiuseppeButtazzoLuigiDePascalePaolaGori-Giorgi12PRA, MOT_application_physics}, and the study of Wasserstein barycenters \cite{AguehCarlier2011}. Given $m$ Borel probability measures $\nu_j$ on 
%open \xc{do we need open $\cX_j$?} \pl{I was using open bounded set following arXiv:1210.7372. To formulate MOT, bounded set would be sufficient arXiv:2212.12492 . }
bounded subsets $\cX_j\subset \bR^d$ for $j\in[m] := \{1, \dots, m\}$ and a real-valued cost function $c : \cX_1 \times \cdots \times \cX_m \to \bR$, the MOT seeks to solve the following optimization problem
\begin{align}
    \label{eqn:MOT_formulation}
    \inf_{\pi\in\Pi(\nu_1,\dots,\nu_m)} \int_{\cX_1 \times \cdots \times \cX_m} c(x_1,\dots,x_m) d\pi(x_1,\dots,x_m),
\end{align}
where $\Pi(\nu_1,\dots,\nu_m)$ denotes the set of all possible couplings with marginal distributions $\nu_1, \dots, \nu_m$. The special case of $m = 2$ coincides with the classical theory of optimal transport (OT)~\cite{Vil09,villani2021topics,sabtanbrogio2015_OT}. Recent years have witnessed substantial advances in the computational and statistical aspects of OT~\cite{peyre2019computational,CheNilRig25OT}, owing in part to the Monge structure and regularity properties of the optimal coupling in the two-marginal setting~\cite{Brenier_87,Caffarelli_1991_regularity,Gangbo_McCann_Acta_96}.

In contrast, the general scenario of MOT with $m \geq 3$ marginals is far less-understood than the two-marginal OT setting beyond the existence of the optimal coupling as a solution of~\eqref{eqn:MOT_formulation}, and many theoretical questions currently remain largely open~\cite{Brendan_Pass_MOT15}. One fundamental obstacle is that, unlike the unified theory for the classical OT problem, existing results of MOT rely heavily on the special forms of the cost function $c$. The first result on the optimal coupling structure dates back to \cite{GangboSwiech_MOT} for the quadratic cost function $c(x_1,\dots,x_m) = \sum_{i \neq j}|x_i - x_j|^2$, with a series of follow-up work under cost functions of the same type \cite{HenriMOT02,Brendan_Pass11SIAM}, and other types of cost functions such as repulsive cost \cite{Colombo_Pascale_Marino_repulsive15} and cyclic cost
\cite{BredanPassAdolfoVargasJimenez22}. More generally, as demonstrated in \cite{Brendan_Pass_MOT15}, multimarginal twist conditions on the cost function that imply structural results on the resulting optimal coupling can be proposed but are much more restrictive than their two-marginal counterparts. All the subtlety hints the fundamental divergence from the classical case $m=2$ to the multimarginal setting $m \geq 3$; cf. \cite{Brendan_Pass_MOT15, KimPassMOTRiemannian15, BredanPassAdolfoVargasJimenez22, BredanPassAdolfoVargasJimenez24} for more details and state-of-art results. On the computational side, the MOT problem is computationally intractable to solve or even approximate with a general cost function~\cite{Jason_Enric_NP_MOT21}, highlighting the root difference from the two-marginal OT problems.

%\xc{This paragraph needs revision.} Under mild conditions, a minimizer $\pi^*$ of MOT will exist. However, the uniqueness and structure of solution $\pi^*$ of such problem are subtle and far from being well-understood in general, contrary to the well-developed $m=2$ optimal transport (OT) theory (see \cite{Vil09, villani2021topics, sabtanbrogio2015_OT}). While in the two marginal case, a regularity condition on one marginal distribution and some twist condition on the cost function $c$ \cite{Brenier_87, Gangbo_McCann_Acta_96} ensures the uniqueness and Monge structure of the optimal $\pi^*$ (namely, $\pi^*$ is concentrated on the graph of a function over $x_1$), its analogue for $m\geq3$ becomes delicate and still remain largely open currently. 

Many applications of MOT require costs that can capture the intrinsic structure, rather than squared Euclidean costs.
In quantum physics, electron--electron
repulsion is modeled by the Coulomb interaction cost
$c(x_1,\dots,x_m)=\sum_{1\leq i<j\leq m}\|x_i-x_j\|^{-1}$
\cite{lelotte2024external, nenna2025ode, MOT_repulsive_cost2017}.
Bounded, continuous truncations, such as
$c(x_1,\dots,x_m)=\sum_{1\leq i<j\leq m}
\min\{1,\|x_i-x_j\|^{-1}\}$,
are also studied to remove the collision singularities
\cite{buttazzo2018continuity}.
In mathematical finance, the cost
$c(x_1,\dots,x_m)=-\mathbf{1}_{\{x_1+\cdots+x_m\geq s\}}$
arises in robust risk aggregation, where one seeks the worst-case
probability that aggregate losses exceed a given threshold, subject to
prescribed marginal distributions
\cite{eckstein2021computation}.
A geometrically different example is
$c(x_1,\dots,x_m)=-|\det(x_1,\dots,x_m)|$, with
$x_1,\dots,x_m\in\mathbb{R}^m$, which corresponds to maximizing the
expected volume of the simplex with vertices $0,x_1,\dots,x_m$
over all couplings with the prescribed marginals
\cite{Carlier_determinantOT_2008}.

In this work, we study the statistical estimation and inference for a closely related MOT problem from data. Motivated by the prevalent entropy regularization of the OT problem~\cite{Cuturi_Sinkhorn2013,nutz2022entropic,rigollet2022samplecomplexityentropicoptimal}, we consider the following \emph{entropic multimarginal optimal transport} (EMOT) problem with a regularization parameter $\varepsilon>0$,
\begin{align}
    \label{eqn:EMOT_formulation}
    S_{\varepsilon}(\nu_1,\dots,\nu_m) := \inf_{\pi\in\Pi(\nu_1,\dots,\nu_m)} \int_\cX c(x_1,\dots,x_m) d\pi + \varepsilon \text{KL}( \pi\| \otimes_{k=1}^m \nu_k),
\end{align}
where $\otimes_{k=1}^m \nu_k$ denotes the product measure of $\nu_1, \dots, \nu_m$. Since the unique solution (i.e., optimal coupling) of EMOT is characterized via a Schr\"odinger system of equations~\cite{CarlierLaborde_SS,MarinoGerolin_SchrodingerBridge}, we also refer problem~\eqref{eqn:EMOT_formulation} as the \emph{multimarginal Schr\"odinger system} and its optimal coupling as the \emph{multimarginal Schr\"odinger coupling}. In practice, we do not directly have access to the marginal distributions $\nu_1, \dots, \nu_m$. Instead, one often only has a cloud of data points $X^{(j)}_1,\dots, X^{(j)}_N$ drawn from $\nu_j$ for $j \in [m]$. To study statistical properties of EMOT, we work with the standard sampling model throughout this paper; that is, for each $j\in[m]$, the cloud of data points $X^{(j)}_1,\dots, X^{(j)}_N\sim \nu_j$ are i.i.d., and moreover, the $m$ samples drawn from $\nu_1,\dots,\nu_m$ are independent. %Denote $\hat{\nu}_j^N=\frac{1}{N}\sum_{k=1}^N\delta_{X^{(j)}_k}$ as the empirical distribution of $\nu_j$.

\begin{comment}
The discussion above underscores the importance of reliable numerical and approximation methods, among which the entropic approximation method holds a prominent place due to its simplicity and efficiency; see \cite{Cuturi_Sinkhorn2013, LinHoCuturiJordan_MOT-complexity}.  Specifically, for non-smooth cost functions $c$, we consider the following entropic multimarginal optimal transport problem (EMOT) given a regularization parameter $\varepsilon>0$,
\begin{align*}
    \inf_{\pi\in\Pi(\nu_1,\dots,\nu_m)} \int_\cX c(x_1,\dots,x_m) d\pi + \varepsilon \text{KL}( \pi\| \otimes_{k=1}^m \nu_k).
\end{align*}
\end{comment}

Statistical estimation and inference of EMOT from data are scarce and somewhat fractured in literature, with papers addressing cost functions with particular structures such as smoothness; see Section~\ref{subsec:EMOT_related_work} for the related work. \emph{One of the aims of this paper is to establish a new pipeline of statistical tools from sharp sample complexity to principled uncertainty quantification for general nonsmooth EMOT problems.} The difficulty for understanding and analyzing EMOT problem under nonsmooth cost function $c$ is two-fold. On the one hand, as mentioned above, the multimarginal setting introduces substantially more delicate and intricate nonlinear interactions and coupling among the marginal distributions than in the classical two-marginal case. On the other hand, progress in the two-marginal OT regime and especially the $m\geq3$ regime has traditionally relied on strong structural assumptions on the cost function, with smoothness playing an almost indispensable role. Abandoning smoothness eliminates many standard techniques and makes the analysis significantly more challenging. Consequently, moving from the two-marginal case to the multimarginal setting ($m \ge 3$), or from regular smooth costs to general nonsmooth costs, each on its own introduces significant challenges that are not merely technical but are intrinsically tied to the fundamental structural complexity of the problem.

As a leading application of our general EMOT results, we consider statistical analysis of the Wasserstein barycenter, a notion of averaging for measure-valued data under the optimal transport metric~\cite{AguehCarlier2011}. Barycenters have found numerous applications and advantages in data science and machine learning~\citep{srivastava2018scalable,ZhuangChenYang2022}, computer graphics and image processing ~\citep{rabin2012wasserstein,Solomon_2015} and scalable Bayesian inference~\citep{srivastava2018scalable}. Despite widespread interests, computing the Wasserstein barycenter remains extremely challenging in high (or even moderate) dimensions, partially due to the curse-of-dimensionality barrier~\citep{AltschulerBoix-Adsera2021_barycenter-NPhard}. This exponential dependence in dimension $d$  %\pl{exponential dependence in dimension  $d$ ?} \xc{yes, NP-hardness in dimension scaling.} \pl{All good. I was just about to stress the place to put $d$} 
explains the popularity of various regularized modifications for solving problems of large scale~\citep{NEURIPS2020_cdf1035c,CarlierEichingerKroshnin_entropic-barycenter,BigotCazellesPapadakis_barycenter-penalization,janati2020debiased,chizat2023doublyregularizedentropicwasserstein}. While many existing works focus on the approximating and algorithmic issues, considerably less is known on the statistical performance for estimation and performing inference on Wasserstein barycenters when we only have random sample access to the input marginal measures.

In Section~\ref{Multimarginal_Schrodinger_barycenter_section}, we propose a new approach, termed as the \emph{multimarginal Schr\"odinger barycenter} (MSB) to regularizing the barycenter functional based on the EMOT formulation. In contrast to existing regularized barycenter methods, the MSB applies regularization directly to the MOT problem. This yields a more controlled form of regularization that preserves the fidelity of the unregularized barycenter while offering both statistical and computational benefits.

\subsection{Contributions}
%In this work, we introduce the \emph{multimarginal Schr\"{o}dinger barycenter} (MSB) based on the entropic MOT (EMOT) formulation, which naturally extends the EOT problem in the two-marginal case and admits fast linear time complexity off-the-shelf algorithms to compute. At the same time, we establish non-asymptotic rates of convergence for estimating the cost functional, coupling and barycenter associated with the MSB. Furthermore, the Central Limit Theorem (CLT) and bootstrap consistency for the cost functional, coupling and barycenter pertaining to the MSB are also derived. 

Below we summarize our main contributions. 

\begin{itemize}
    \item Under the standard statistical sampling model where only random point clouds are accessible to the input marginal distributions, we establish \emph{dimension-free} and \emph{parametric} $\sqrt{N}$-rates of convergence for estimating the cost functional (Theorem \ref{thm:rate_cost}) and multimarginal Schr\"{o}dinger coupling on any bounded test function (Theorem \ref{boundtest}).

    \item We further derive the nonparametric rate of convergence of empirical multimarginal Schr\"{o}dinger coupling on the H\"older smooth function class (Theorem \ref{supremp_coupling}). When specializing this general result to the $p$-Wasserstein distance for $p = 1,2$, we obtain a sharp convergence rate that substantially improves the state-of-the-art rates~\citep{BayraktarEcksteinZhang_stability}. To the best of our knowledge, the current work is the first to obtain the optimal convergence rates under $W_p \,(p=1,2)$ distance of entropy regularized couplings in the Wasserstein space given the known lower bounds on empirical Wasserstein distances~\citep{FournierGuillin,WeedBach}. %\xc{Should this paragraph be revised accordingly too?}.

    \item We prove the central limit theorem (CLT) for the quantities of interest, including the cost functional (Theorem \ref{CLT_for_cost_functional}) and optimal Schr\"{o}dinger coupling on any bounded test function (Theorem \ref{CLT_test_function_very_complicated_computation}), and demonstrate the associated bootstrap consistency (Theorem \ref{bootstrap_cost} and Theorem \ref{bootstrap_coupling}), where the latter enables downstream data-driven statistical inferences.

    \item For the MSB, we also derive parallel rates of convergence, weak limits and bootstrap consistency (Corollary \ref{cor: Concentration_empirical_MSB}, Theorem \ref{supremp}, Corollary \ref{CLT_test_function_very_complicated_computation_MSB}, Corollary \ref{cor: bootstrap_MSB}),  as an application of general multimarginal Schr\"{o}dinger system theory established for EMOT.
        
    \item From a technical standpoint, we design a new decoupling mechanism (see Appendix \ref{sec_Concentration_bound_on_the_empirical_gradient_norm}) to handle multimarginal Schr\"{o}dinger system based on induction and the iterative use of its marginal feasibility. This new probability tool, inspired by the idea of higher-order degenerate $U$-statistics \cite{ChenKato2019}, can be of independent interest to other MOT and barycentric inference problems. Building on this, along with other new technical tools, we show in Lemma \ref{diff_empirical_population_f_i_bounded_by_other_potentials} that under the assumption that the cost function $c\in\cC^2$, optimal population potential $\mvf^*$ and optimal empirical potential $\hat{\mvf}^*$ satisfy that
\begin{align*}
\| \mvf^* - \hat{\mvf}^* \|_{\cC^1}
 = \cO_\bP \left( \sqrt{\frac{\log N}{N}} \right),
\end{align*}
achieving a genuine improvement over the previously known results~\cite{EustasioGozalezLoubesNiles23,CalierChizatLaborde_DisplacementSmoothnessofEOT}.
\end{itemize}
The following Table~\ref{tab:regularity_assumptions} summarizes our main results and the corresponding regularity assumptions on the cost function \(c\).

\begin{table}[htbp]
\centering
\caption{Cost regularity assumptions}
\label{tab:regularity_assumptions}
\smallskip
\small
\setlength{\tabcolsep}{5pt}
\renewcommand{\arraystretch}{1.15}

\begin{tabularx}{\linewidth}{
    @{}
    >{\raggedright\arraybackslash}X
    l
    >{\raggedright\arraybackslash}p{0.23\linewidth}
    l
    @{}
}
\toprule
Result
& Reference
& Cost regularity
& MSB counterpart \\
\midrule

Mean-square error bound for the cost functional
& Theorem~\ref{thm:rate_cost}
& Bounded \(c\)
& --- \\
\addlinespace[2pt]

Concentration for coupling evaluations
& Theorem~\ref{boundtest}
& Bounded \(c\)
& Corollary~\ref{cor: Concentration_empirical_MSB} \\
\addlinespace[2pt]

Uniform bounds of coupling over H\"older classes and $W_p$ bounds
& Theorem~\ref{supremp_coupling}
& Bounded \(c\)
& Theorem~\ref{supremp} \\
\addlinespace[2pt]

CLT for the cost functional
& Theorem~\ref{CLT_for_cost_functional}
& Bounded \(c\)
& --- \\

\addlinespace[7pt]

CLT for coupling evaluations
& Theorem~\ref{CLT_test_function_very_complicated_computation}
& \(c\in \cC^2(\cX)\)
& Corollary~\ref{CLT_test_function_very_complicated_computation_MSB} \\
\addlinespace[2pt]

Potential convergence in the quotient \(\cC^1\) norm
& Lemma~\ref{diff_empirical_population_f_i_bounded_by_other_potentials}
& \(c\in \cC^2(\cX)\)
& --- \\

\addlinespace[7pt]

Bootstrap validity for the cost functional
& Theorem~\ref{bootstrap_cost}
& \(c\in \cC^{2k}(\cX)\),  
\newline \(k\ge d/2+1\)
& --- \\
\addlinespace[2pt]

Bootstrap validity for coupling evaluations
& Theorem~\ref{bootstrap_coupling}
& \(c\in \cC^{2k}(\cX)\),
\newline \(k\ge d/2+1\)
& Corollary~\ref{cor: bootstrap_MSB} \\

\bottomrule
\end{tabularx}
%\smallskip\begin{minipage}{\linewidth}\footnotesize\end{minipage}
\end{table}

\subsection{Related work}\label{subsec:EMOT_related_work}
\begin{comment}
\subsubsection{EOT}
In the two marginal entropy regularization OT case with smooth cost function, the parametric behavior is well-known. The rate of convergence under smooth cost function was shown to be parametric in \cite{Genevay_Chizat_Bach_Cuturi_peyer19, MenaNilesWeed19, ChizatRoussillonLegerVialardPeyer20} for compactly supported distributions and sub-Gaussian distributions. The result for the two marginal entropy regularization OT case with general bounded cost function was first obtained by 
\cite{rigollet2022samplecomplexityentropicoptimal}. 

Also, a lot of endeavor was to obtain the weak limit of the cost functional. See \cite{EustasioGozalezLoubesNiles23, Kengo24IMAregularizedOT} for results under smooth cost functions and \cite{Alberto_Sanz_Gonzalez_Shayan_EOT_non_smooth_CLT23} for CLT under non-smooth cost.
\end{comment}

Because of the aforementioned challenges, EMOT under nonsmooth cost has not been well-investigated in the exisinting literature. The well-posedness (existence, uniqueness and smooth
dependence with respect to the data) for the multimarginal Schr\"{o}dinger
system was first established in \cite{CarlierLaborde_SS} for bounded cost function $c$. For smooth cost function, \cite{CalierChizatLaborde_DisplacementSmoothnessofEOT} obtained the Lipschitz continuity of the map from marginal distributions to the optimal Schr\"{o}dinger coupling. \cite{Multimarginal_Unbalanced_OT_CLT_Xu21} investigated the weak limits of unbalanced multimarginal optimal transport under smooth cost function but could not specify the limiting variance. The Hadamard differentiability for EMOT was established in \cite{ZivGoldfeldKengoKatoGabrielRiouxRitwikSadhu} for smooth cost function. All the results above considered the empirical plug-in estimation of EMOT system while \cite{DorZivKristjanHaim2024Neural} employed neural networks and obtained parametric estimation rate for the cost functional with special graph structure.

The Wasserstein barycenter for probability measures on the Euclidean domain was first introduced in~\citep{AguehCarlier2011} and later extended to Riemannian manifolds in~\citep{KIM2017640}, where existence and uniqueness were proved when at least one of the input measures vanishes on small sets. The extension to the setting
of nonsmooth extended metric measure spaces was recently done by \cite{BangxianHan24}. Statistical consistency for empirical Wasserstein barycenters on a general geodesic space was derived in~\citep{LeGouicJeanMichel_barycenter-existunique}, and rates of convergence were established under additional curvature assumptions~\citep{AhidarLeGouicParis_barycenter-rate,LeGouic2019FastCO}. Practical algorithms for computing unregularized Wasserstein barycenter can be found in~\citep{pmlr-v32-cuturi14,KimYaoZhuChen2025_barycenter-nonconvex-concave,kim2025sobolevgradientascentoptimal,AltschulerBoix-Adsera2021_fixed-dim-barycenter}. There are several existing regularization models for the barycenter functional~\citep{pmlr-v32-cuturi14,NEURIPS2020_cdf1035c,CarlierEichingerKroshnin_entropic-barycenter,BigotCazellesPapadakis_barycenter-penalization,janati2020debiased,chizat2023doublyregularizedentropicwasserstein}. However, those works lack either of thorough statistical analysis for estimation or inference of the barycenter, or of efficient algorithm to compute its barycenter. Different from our MSB proposal based on the EMOT formulation, most existing regularized barycenters begin by introducing regularization into the pairwise formulation of the barycenter problem (\ref{eqn:unregbary}).

%Most existing regularized barycenter methods begin by introducing regularization into the pairwise formulation of the barycenter problem (\ref{eqn:unregbary}). These approaches can be categorized—depending on how the regularization is applied—into the following, for example, inner regularized barycenter \cite{janati2020debiased}, Sinkhorn barycenter \cite{janati2020debiased}, outer regularized barycenter \cite{BigotCazellesPapadakis_barycenter-penalization, CarlierEichingerKroshnin_entropic-barycenter}, doubly regularized barycenter \cite{chizat2023doublyregularizedentropicwasserstein}. None of these categories encompass our notion of regularized barycenter. 

\subsection{Notations}
%For any integer $m$, $[m] := \{1, \dots, m\}$. 
For a nonempty compact set $\Omega\subset\bR^d$ that agrees with the closure of its interior, $\cC(\Omega)$ denotes all the continuous functions on $\Omega$. We occasionally omit the dependence on the space $\Omega$ when it could be inferred from the context. Given $(f_1,\dots,f_m)\in\prod_{j=1}^m \cC(\cX_j)$, $\bigoplus (f_1,\dots,f_m) := \sum_{j=1}^m f_j$. For every multi-index $k=(k_1,\dots,k_d)\in \bN_0^d$ with $|k|= \sum_{j=1}^d k_j$ (where $\bN_0 = \bN\cup\{0\}$),
define the differential operator $D^k$ by
\[
  D^k f \;=\; \frac{\partial^{|k|} f}{\partial x_1^{k_1}\cdots \partial x_d^{k_d}},
  \qquad\text{with } D^0 f = f .
\]
For every $s\in\mathbb N_0$, denote by $\cC^s(\Omega)$ the set of functions $f$ on $\Omega$ such that $f$ has
continuous derivatives of all orders $\le s$ on $\mathrm{int}(\Omega)$ and these derivatives have continuous
extensions to $\Omega$ (so $\cC^0(\Omega) = \cC(\Omega)$). Define the norm
\[
  \lVert f\rVert_{\cC^s(\Omega)}
  \;=\;
  \sum_{j=0}^s \; \max_{|k|=j}\;
  \lVert D^k f \rVert_{\infty,\,\mathrm{int}(\Omega)} .
\]
Then $\big(\cC^s(\Omega), \lVert\cdot\rVert_{\cC^s(\Omega)}\big)$ is a separable Banach space. For $s\in\mathbb N_0$, let $\cC^{s}(\Omega;\mathbb R^{m})$
denote the space of vector-valued functions
$\vf= (f_1,\dots,f_m):\Omega\to\mathbb R^{m}$
whose coordinate functions belong to $\cC^{s}(\Omega)$,
equipped with the norm
\[
  \|f\|_{\cC^{s}(\Omega;\mathbb R^{m})}
  = \max_{1\leq j\leq m} \|f_j\|_{\cC^{s}(\Omega)} .
\]
\begin{comment}
    Specifically, we equip, $\cC^1(\Omega)$ the space of real-valued continuously differentiable functions on $\Omega$, with the norm
\[
  \|f\|_{\cC^1} := \|f\|_\infty + \|\nabla f\|_\infty, \qquad
  \|f\|_\infty := \sup_{x\in\Omega} |f(x)|,\quad
 \|\nabla f\|_\infty :=  \max_{1\leq i \leq d} \sup_{x\in\Omega} \left|\frac{\partial f}{\partial x_i}(x)\right|.
\]
Let $\cC^{1}(\Omega;\mathbb R^{m})$
denote the space of vector-valued functions
$\vf=(f_1,\dots,f_m):\Omega \to\mathbb R^{m}$
whose coordinate functions belong to $\cC^{1}(\Omega)$,
equipped with the norm
\[
  \|f\|_{\cC^{1}(\Omega;\mathbb R^{m})}
  = \max_{1\leq j\leq m} \|f_j\|_{\cC^{1}(\Omega)} .
\]
\end{comment}

\begin{comment}
More generally, unless stated otherwise, we adopt the following convention throughout the paper. For  Banach spaces $(E_j,\|\cdot\|_{E_j}),\,j\in[m]$, we define the product space as
\[
E := \prod_{i=1}^m E_i,
\qquad
\|(x_1,\dots,x_m)\|_{E} := \max_{1\le j\le m} \|x_j\|_{E_j}.
\]
Then $(E,\|\cdot\|_{E})$ is a Banach space.
\end{comment}

\noindent The Kullback-Leibler (KL) divergence or relative entropy between probability $P$ and $Q$ is defined as %$\text{KL}(P \parallel Q) = \int \log \left( \frac{dP}{dQ}(x) \right) dP(x)$ if $P \ll Q$, and $\text{KL}(P \parallel Q) = +\infty$ otherwise,
\[
\text{KL}(P \parallel Q) := 
\begin{cases}
	\int \log \left( \frac{dP}{dQ}(x) \right) dP(x), & \text{if } P \ll Q, \\
	+\infty, & \text{otherwise,}
\end{cases}
\]
where $ P \ll Q$ means that $P$ is absolutely continuous with respect to $Q$ and $\frac{dP}{dQ}(x)$ denotes the Radon-Nikodym derivative. We write $\xi_n \xrightarrow{\mathbb{P}} \xi$ (resp. $\xi_n \xrightarrow{w} \xi$) for convergence in probability (resp. weak convergence). For a probability measure $\mu$ and a measurable function $f$, we often write $\mu(f)$ to represent $\int fd\mu$. We write $\|f\|_{L^{\infty}(\nu)}$ as the essential supremum for $\|f(X)\|$ with $X \sim \nu$, and $L^\infty(\nu)$ denotes the Banach space of functions such that $\|f\|_{L^{\infty}(\nu)} < \infty$. Given probability measures $\nu_1,\dots,\nu_m$,
$\Pi(\nu_1, \dots, \nu_m)$ denotes the set of all probability measures $\pi$ with $\nu_i$ as the $i$-th marginal, i.e., $\ (e_i)_{\sharp} \pi = \nu_i \ \text{where} \ e_i : (x_1, \dots, x_m) \mapsto x_i$ is the $i$-th projection map for each $i \in [m]$. 

We are given $m\geq2$ probability space $(\cX_i,\cF_i,\nu_i)$, $i=1,\dots,m$ and set 
\begin{align*}
    \cX := \prod_{i=1}^m \cX_i; \quad \cF :=\bigotimes_{i=1}^m \cF_i; \quad \nu:=\bigotimes_{i=1}^m\nu_i.
\end{align*}
Given $i\in\ [m]$, we will denote by 
$\cX_{-i} : = \prod_{j \neq i} \cX_j$,
$\nu_{-i}: = \bigotimes_{j\neq i}\nu_j$,
and will always identify $\cX$ with $\cX_i\times\cX_{-i}$, namely, will denote $\mvx = (x_1,\dots,x_m)\in\cX$ as $x = (x_i,x_{-i})$. The empirical distribution given $N$ independent and identically distributed (i.i.d.) sample of $\nu_k$ is denoted as $\hat{\nu}_k^N$ for $k\in[m]$. Similarly as before, we write for $i\in[m]$, $\hat{\nu}^N_{-i} := \bigotimes_{j\neq i}\hat{\nu}^N_j$.

Throughout the rest of the paper, we work with a compact domain $\cX_i \subset$ $\mathbb{R}^d$ with a non-empty interior. %\xc{Do we need compact or bounded?} \pl{I would propose compact set}
Without loss of generality, by rescaling we may further assume that $\cX_i \subset [-1, 1]^d$, $i\in[m]$. For clarity and brevity, more definitions of the commonly used notations are provided in Appendix \ref{appendix_section_notations}.

\subsection{Organization}
The rest of the paper is organized as follows. In Section \ref{Background_and_Mathematical_Preliminaries}, we collect background material for the EMOT problems. Optimization geometry of the dual functional of EMOT is introduced in Section \ref{optimization_geometry_section}. The main results on sample complexity, weak limits and bootstrap validity for the empirical EMOT problem are collected in Section \ref{main_results}. %Section \ref{sec:statistical_rates} provides the sample complexity for estimating the cost functional and Schr\"{o}dinger coupling on any bounded test function, along with the rate of convergence under H\"older function class.  The weak limits of all these quantities of interest are provided in Section \ref{section_weak_limits} and the bootstrap consistency is presented in Section \ref{bootstrap_consistency_section}. 
%In Section \ref{Multimarginal_Schrodinger_barycenter_section}, we study the Wasserstein barycenter problem as an application. Specifically, we formulate the Multimarginal Schr\"{o}dinger barycenter, a new  and natural way to apply regularization to barycenter, and establish statistical performance guarantees on this notion of regularized barycenter as an application of main results.  
Statistical guarantees of MSB are established Section \ref{Multimarginal_Schrodinger_barycenter_section}. The Appendix contains proofs that are omitted from the main text, technical tools used in the proofs, and other auxiliary results.

\section{Background and mathematical preliminaries} 
\label{Background_and_Mathematical_Preliminaries}

\subsection{Background on OT}
In this section, we present some background on the optimal transport theory.
%\subsection{Background on optimal transport and entropy regularization}
% $\mathcal{M}_1^+(\mathcal{X})$ be the set of Radon probability measures supported on $\mathcal{X}$ 
Given two probability measures $\mu,\,\nu\in\mathcal{P}_2(\Omega)$, the 2-Wasserstein distance between $\mu$ and $\nu$ is defined as the value of the Kantorovich problem:
\begin{equation}\label{W2}
W_2(\mu, \nu) := \inf_{\pi \in \Pi(\mu,\nu)} \left\{ \int_{\Omega \times \Omega} \|x - y\|^2 d\pi(x,y) \right\}^{1/2}.
\end{equation}
%where $\Pi(\mu, \nu)$ is the set of admissible transport plans (i.e., couplings) of joint probability measures on $\mathcal{X} \times \mathcal{X}$ with marginal distributions $\mu$ and $\nu$, i.e., $\pi(B \times \mathcal{X}) = \mu(B)$ and $\pi(\mathcal{X} \times B) = \nu(B)$ for any measurable $B \subset \mathcal{X}$.
Solving the linear program in (\ref{W2}) on discretized points in $\Omega$ imposes tremendous computational challenges. Entropy regularized optimal transport computed via Sinkhorn's algorithm has been widely used with an efficient approximation at guaranteed low-computational cost even for high-dimensional probability measures~\citep{Cuturi_Sinkhorn2013}. The entropic OT (EOT) problem is defined as 
%Given a reference measure $\xi \in \mathcal{M}_1^+(\mathcal{X})$, the $\varepsilon$-regularized optimal transport problem is defined as
\begin{equation}\label{EW2}
S_{\varepsilon}(\mu,\nu):=\inf_{\pi \in \Pi(\mu,\nu)} \left\{ \int_{\Omega \times \Omega} \|x - y\|^2 d\pi(x,y) + \varepsilon \text{KL}(\pi || \mu \otimes \nu) \right\}, % no need to subtract \varepsilon 
\end{equation}
where $\varepsilon > 0$ is the regularization parameter. It is known that problem (\ref{EW2}) has a unique solution $\pi^* \in  \Pi(\mu,\nu)$ and admits a strong duality form~\citep{leonard_SB}, i.e., zero duality gap (up to an additive constant depending only on $\varepsilon$), given by
\[
\sup_{\substack{f \in L^\infty(\mu)\\g\in L^\infty(\nu)}}  \left\{\int_{\Omega} f d\mu +  \int_{\Omega} g d\nu -  \varepsilon  \int_{\Omega \times \Omega}  
\exp{\left(  
	\frac {f(x)+g(y)-||x-y||^2}{\varepsilon}
	\right)} d(\mu \otimes \nu)(x,y) % +\varepsilon
\right\}.  % no need to add back +\varepsilon
\]
The above supremum is achieved at a unique pair of the optimal dual potentials 
\((f^*, g^*) \in L^\infty(\mu) \times L^\infty(\nu)\) up to translation 
\((f^* + c, g^* - c)\) for \(c \in \mathbb{R}\). In addition, the optimal coupling 
\(\pi^*\) can be recovered from the optimal dual potential pair \((f^*, g^*)\) via
\[
\frac{\mathrm{d} \pi^*}{\mathrm{d} (\mu \otimes \nu)}(x, y) = \exp \left( \frac{f^*(x) + g^*(y) - \|x - y\|^2}{\varepsilon} \right).
\]

%\subsection{Wasserstein barycenter}\label{WBSec}
%\noindent The OT problem~\eqref{W2} can be generalized to more than two marginal distributions~\citep{GangboSwiech_MOT}. Given a \emph{bounded} cost function $c :\cX\to\bR$, we consider the following MOT problem
%\begin{equation}\label{udual}
%\inf_{\pi\in\Pi(\nu_1,\dots,\nu_m)}\int_{\cX}	c(x_1,\dots,x_m)d\pi(x_1,\dots,x_m).
%\end{equation}

\subsection{EMOT and strong duality}
\label{subsec:duality}
%For the MOT problem (\ref{eqn:EMOT_formulation}) with a bounded cost function $c :\cX\to\bR$, one could consider the corresponding entropy regularized EMOT problem
 %\begin{equation}\label{primal0}
	%S_{\varepsilon}(\nu_1,\dots,\nu_m):=\inf_{\pi\in\Pi(\nu_1,\dots,\nu_m)} \Big\{ \int_{\cX} c d\pi+\varepsilon \ \text{KL}(\pi||\nu_1\otimes\cdots\otimes \nu_m) \Big\}.
%\end{equation}

In the literature, the EMOT problem (\ref{eqn:EMOT_formulation}) is also referred as the multimarginal Schr\"{o}dinger problem because it is equivalent to minimize the relative entropy with respect to a Gibbs kernel associated with the transport cost~\citep{MarinoGerolin_SchrodingerBridge,CarlierLaborde_SS}, i.e., $\min_{\pi\in\Pi(\nu_1,\dots,\nu_m)} \ \text{KL}(\pi||\kappa),$
%\begin{equation}\label{mschr}
%	\inf_{\pi\in\Pi(\nu_1,\dots,\nu_m)}\varepsilon \ \text{KL}(\pi||\kappa),
%\end{equation}
where $\kappa$ is the Gibbs kernel 
\[
d\kappa(x_1,\dots,x_m)\propto\exp\left(-c(x_1,\dots,x_m) / \varepsilon\right) \; d(\otimes_{k=1}^m\nu_k).
\]

\noindent Since we work in a bounded domain $\cX \subset ([-1, 1]^{d})^m$, problem (\ref{eqn:EMOT_formulation}) admits a strong duality in the $L^\infty$ setting~\citep{MarinoGerolin_SchrodingerBridge,CarlierLaborde_SS},
\begin{equation}\label{dual1}
	\sup_{f_j\in L^{\infty}(\nu_j)} \Big\{ \Phi_{\varepsilon}(\vf):=\sum_{j=1}^m\int f_jd \nu_j-\varepsilon\int\exp\left(\frac{\sum_{j=1}^mf_j-c}{\varepsilon}\right)d(\otimes_{k=1}^m \nu_k) \Big\}.  %+\varepsilon (again, here we ignore additive constant in \varepsilon
\end{equation}
%where $$\mathcal{L}^{\infty}_{\varepsilon}(\nu_j):=\{u:\mathcal{X}\to\mathbb{R},u\,\text{is a measurable function in} \,\,(\mathcal{X},\nu_j)\, \text{and} \int_{\mathcal{X}}e^{\frac{u}{\varepsilon}}d\nu_j<\infty\}.$$
%\xc{Here, we will work with $L^2$ function space. $\mathcal{L}^{exp}_{\varepsilon}(\nu_j), L^2, L^\infty, {\mathcal C}_b({\mathcal{X}})$ are all equivalent for bounded $\cal X$.}
Supremum in~\eqref{dual1} is achieved at a unique pair of the optimal (Schr\"{o}dinger) dual potentials 
%\((f_1^*,\dots, f_m^*) \in \mathcal{C}(\mathcal{X})^m  \) up to translation 
$\vf^* := (f_1^*,\dots, f_m^*) \in \prod_{i=1}^m L^{\infty}(\nu_i)$ up to translation 
\((f_1^*+c_1,\dots, f_m^*+c_m)\) for \((c_1,\dots,c_m) \in \mathbb{R}^m\) satisfying $\sum_{j=1}^mc_j=0$. We denote by $\mvf^*$ the unique optimal potential vector satisfying $\nu_k(f^*_k)=0$ for $k\in[m-1]$. In addition, the optimal (Schr\"{o}dinger) coupling 
\(\pi_\varepsilon^*\) can be recovered from the optimal dual potentials $\vf^*$ and marginal distribution $\nu_1,\dots,\nu_m$ via
\begin{equation}\label{eqn:optimal_marginal_coupling}
    \frac{\mathrm{d} \pi_\varepsilon^*}{\mathrm{d} (\otimes_{j=1}^m\nu_j)}(x_1,\dots, x_m) = \exp \left( \frac{\sum_{j=1}^m f_j^*(x_j)-c(x_1,\dots, x_m)}{\varepsilon} \right) =: p^*_{\varepsilon}(x_1,\dots,x_m).
\end{equation}
Note that~\eqref{eqn:optimal_marginal_coupling} is equivalent to the multimarginal Schr\"{o}dinger system: we have $\nu_j$-almost surely for each $j \in [m]$,
\begin{equation}\label{multiconst_equiv}
f_j^*(x_j) = - \varepsilon \log \left\{ \int \exp\left(\frac{\sum_{i \neq j} f_i^*(x_i)-c(x_1, \dots, x_m)}{\varepsilon}\right)d\nu_{-j}(x_{-j}) \right\}.
\end{equation}
This allows the optimal potential $\vf^*$ to be extended as continuous functions when the cost function is continuous, and provides a functional analytic perspective that could be beneficial for our statistical analysis. Particularly, we could equivalently consider the extension map
\begin{align}\label{Schrodinger_map_C}
\mvT : \prod_{j=1}^m \cC(\mathcal{X}_j) &\to \prod_{j=1}^m \cC(\mathcal{X}_j)\\
\vf&\mapsto\mvT(\vf)=\left(T_1(\vf),\dots,T_m(\vf)\right)\nonumber
\end{align}
defined for $j \in [m]$ and $x_j \in \mathcal{X}_j$ as
\begin{equation}
\label{Schrodinger_map_extension}
T_j(\vf)(x_j) := \varepsilon\log \left( \int_{\cX_{-j}} e^{ \frac{\sum_{i=1}^m f_i(x_i) - c(x_j, x_{-j})}{\varepsilon}} \, d\nu_{-j}(x_{-j}) \right).
\end{equation}
A function $\vf = (f_1, \dots, f_m)$ is called optimal (Schrödinger) potential associated to $\mvnu := (\nu_1,\dots,\nu_m)$ if it solves the Schrödinger system
\begin{equation}\label{Schrodinger_system_equation_0}
\mvT(\vf) = 0.
\end{equation}
As noted above, if $\vf = (f_1,\dots,f_m)$ solves (\ref{Schrodinger_system_equation_0}) for some fixed $\mvnu$, then so does every family of potentials of the form $(f_1 + c_1,\dots, f_m + c_m)$ where the $c_i\in\bR$ satisfying $\sum_{i=1}^m c_i =0$. This defines an equivalence relation $\sim$ and we define, for every $k\in \bN_0$, the quotient space
\[
\widetilde{\cC}^k:= \left( \prod_{j=1}^m \cC^k(\mathcal{X}_j) \right) / \sim .
\]
Note that $\widetilde{\cC}^k$ is a Banach space when endowed with the quotient norm $\|\cdot\|_{\widetilde{\cC}^k}$ (the infimum of the norm over all representatives in the equivalence class). We use $[\mvf^*]$ to denote the equivalence class of the optimal potential $\mvf^*$.

In the following, we present some basic properties of EMOT and its optimal dual potentials and show that the the optimizer $\pi^*_\varepsilon$ of EMOT in (\ref{eqn:EMOT_formulation}) converges to the optimal coupling $\pi^*_0$ of the unregularized MOT in (\ref{eqn:MOT_formulation}), whenever unique, as $\varepsilon \to 0^+$.

\begin{proposition}[Approximation]
\label{approgamma_coupling}
Assume the cost function $c$ is continuous. Any cluster point of Schr\"{o}dinger couplings $(\pi^*_{\varepsilon})_{\varepsilon>0}$ is an optimal coupling $\pi^*_0$ of MOT problem (\ref{eqn:MOT_formulation}). If futher $\pi^*_0$ is unique, %then the whole sequence $(\bar{\nu}_{\varepsilon})_{\varepsilon>0}$ converges \xc{weakly?} towards $\bar{\nu}$, and
then we have $\lim_{\varepsilon\to0^+}W_2(\pi^*_{\varepsilon}, \pi^*_{0})=0.$
\end{proposition}

\begin{proposition}[Bounded dual potentials]\label{BDP}
The optimal dual potential $\mvf^*$ satisfies that 
\[\max_{j \in [m]} \|f^*_j\|_{\infty} \leq 3 \|c\|_{\infty} \quad \text{and} \quad  \|\sum_{j=1}^mf^*_j\|_{\infty} \leq (2m+1)\|c\|_{\infty},
\]
where $\|f^*_j\|_{\infty} := \|f^*_j\|_{L^{\infty}(\nu_j)}$ and $\|c\|_{\infty} := \|c\|_{L^{\infty}(\otimes_{i=1}^m\nu_i)}$.
\end{proposition}

\section{Optimization Geometry}\label{optimization_geometry_section}

In this section, we recognize an appropriate optimization geometry for the EMOT problem, which will serve as a foundation for subsequent developments. Specifically, a key structure for proving sharp statistical rates of convergence and later weak limits for EMOT quantities (such as cost functional, coupling, potentials and resulting barycenter) in Sections~\ref{main_results} and \ref{Multimarginal_Schrodinger_barycenter_section} is to recognize an optimization geometry of $\Phi_{\varepsilon}(\vf)$ that respects the optimal multimarginal coupling structure in~\eqref{eqn:optimal_marginal_coupling}. For this purpose, we shall consider the gradient of $\Phi_{\varepsilon}(\vf)$ as elements of the dual space of $\mathscr{L}_m:=\left(L^2(\nu_1),\dots, L^2(\nu_m)\right)$. This space could be identified with $(\cC(\cX_1),\dots,\cC(\cX_m))$ or $(\cC^1(\cX_1),\dots,\cC^1(\cX_m))$ through the extension in (\ref{Schrodinger_map_extension}) depending on the regularity of the cost function $c$; however, since the distinction is immaterial for our purposes, we shall not differentiate between them.

%\xc{I think we only need $L^2$. Do we really need a stronger exp class?} 

%\pl{Li: $L^2$ is enough. I used exp class to follow the convention of notation in Marino/Gerolin's paper. Since we are in a compact/bounded set, $L^2$ is enough.}

%\xc{OK. I will revise this as a comment later.}

Motivated by the geometry for two-marginal entropic OT problem~\cite{rigollet2022samplecomplexityentropicoptimal,stromme2024minimum}, we first generalize the gradient to multi-marginal setting.

\begin{definition}[Gradient of the dual objective]
\label{normdef}
The gradient $\nabla\Phi_{\varepsilon}:\mathscr{L}_m\to\mathscr{L}'_m$ is defined as
\begin{align*}
\left\langle \nabla \Phi_{\varepsilon} (\vf), \vg \right\rangle_{\mathscr{L}_{m}} 
% &:=\sum_{j=1}^{m}\int g_j d\nu_j-\int\left(\sum_{j=1}^m g_j\right)\exp\left({\frac{\sum_{j=1}^mf_j-%c}{\varepsilon}}\right)d\left(\otimes_{k=1}^m\nu_k\right)\\
 =\int\Big[\Big(\sum_{j=1}^m g_j\Big)\Big(1-\exp\Big({\frac{\sum_{i=1}^mf_i-c}{\varepsilon}}\Big)\Big)\Big]d\left(\otimes_{k=1}^m\nu_k\right),
\end{align*}
where $\mathscr{L}'_m$ denotes the dual space of $\mathscr{L}_m$.
\end{definition}

\noindent The optimization geometry w.r.t. $\mathscr{L}_m$ induces a norm of the gradient. For $\vg=(g_1,\dots, g_m)\in\mathscr{L}_m$, we define $\|\vg\|_{\mathscr{L}_m}:=( \sum_{j=1}^m\int g_j^2d\nu_j )^{1/2}.$

\begin{lemma}[Norm of dual objective gradient]\label{norm}
We have
	$$\left\|\nabla\Phi_{\varepsilon}(\vf)\right\|_{\mathscr{L}_m}^2=\sum_{j=1}^m\int\Big[\int1-\exp\Big({\frac{\sum_{i=1}^mf_i-c}{\varepsilon}}\Big)d\nu_{-j}(x_{-j})\Big]^2d\nu_j.$$
\end{lemma}

\noindent For $L\geq0$, we define the convex subset $\mathcal{S}_L$ of dual potentials
$$\mathcal{S}_L:=\Big\{\vf\in\prod_{j=1}^m L^{\infty}(\nu_j) : 
\|\sum_{j=1}^mf_j\|_{L^{\infty}(\otimes_{k=1}^m\nu_k)}\leq L,\,\,\nu_i(f_i)=0\,\,\text{for}\,\,i\in[m-1]\Big\}.$$

\begin{proposition}[Strong concavity of dual objective]\label{sconc}
The dual objective $\Phi_{\varepsilon}(\cdot)$ is concave with respect to the norm $\|\cdot\|_{\mathscr{L}_m}$ on $\mathscr{L}_m$, namely $\Phi_{\varepsilon}(\vf)-\Phi_{\varepsilon}(\vg)\geq\left\langle \nabla \Phi_{\varepsilon}(\vf),(\vf-\vg)\right\rangle_{\mathscr{L}_{m}}$ for all $\vf,\vg\in\mathscr{L}_m$. In addition, $\Phi_{\varepsilon}(\cdot)$ is $\beta$-concave with respect to the norm $\|\cdot\|_{\mathscr{L}_m}$ on $\mathcal{S}_{L}$ with $\beta=\frac{1}{\varepsilon}\exp\left(-\frac{L+\|c\|_{\infty}}{\varepsilon}\right)$, i.e., for all $\vf,\vg\in\mathcal{S}_L$,
\begin{equation}\label{sconci}
\Phi_{\varepsilon}(\vf)-\Phi_{\varepsilon}(\vg)\geq\left\langle \nabla \Phi_{\varepsilon} (\vf),(\vf-\vg)\right\rangle_{\mathscr{L}_{m}}+\frac{\beta}{2}\|\vf-\vg\|_{\mathscr{L}_m}^2.
\end{equation}
\end{proposition}

Throughout the rest of this paper, we assume that the regularization parameter $\varepsilon>0$, number of marginals $m$ and cost function $c$ are \emph{fixed}.

\section{Main Results}\label{main_results}
This section summarizes the main results of the paper. Section \ref{sec:statistical_rates} establishes rate of convergence of several EMOT related objects (cost functional, Schr\"{o}dinger coupling) under different discrepancies. Section \ref{section_weak_limits} derives the weak limits of these quantities while Section \ref{bootstrap_consistency_section} proves the bootstrap consistency. 
%We remark that these sections operate on slightly different levels of smoothness of the cost function $c$, with \pl{Sufficient smoothness of the cost function $c$ is only required in Section \ref{bootstrap_consistency_section}.}

Let $\hat{\nu}_j^N=\frac{1}{N}\sum_{k=1}^N\delta_{X^{(j)}_k}$ be the empirical distribution of $\nu_j$. The empirical EMOT problem corresponding to~\eqref{eqn:EMOT_formulation} is defined as
\begin{equation}\label{primal2}
S_{\varepsilon}(\hat{\nu}^N_1,\dots,\hat{\nu}^N_m):=\inf_{\pi\in\Pi(\hat{\nu}^N_1,\dots,\hat{\nu}^N_m)}\int_{\cX} c d\pi+ \varepsilon\text{KL}(\pi||\hat{\nu}^N_1\otimes\cdots\otimes \hat{\nu}^N_m),
\end{equation}
and the empirical dual objective functional is defined as
\begin{equation}\label{dual2}
	\widehat{\Phi}_{\varepsilon}(\vf):=\sum_{j=1}^m\int f_jd\hat{\nu}^N_j-\varepsilon\int\exp\left(\frac{\sum_{j=1}^mf_j- c}{\varepsilon}\right)d (\otimes_{k=1}^m \hat{\nu}^N_k).  % +\varepsilon.
\end{equation}

\noindent As in the population case (cf. Section~\ref{subsec:duality}), we denote $\hat{\mvf}^* = (\hat{f}^*_1,\dots,\hat{f}^*_m)$ as the unique (up to translation) maximizer of $\widehat{\Phi}_{\varepsilon}(\vf)$,
with $\hat{\nu}^N_k(\hat{f}^*_k)=0$ for $k\in[m-1]$.
To lighten the notation, we denote $\hat{p}_{\varepsilon}(x_1,\cdots,x_m) := \exp\left(\frac{\sum_{k=1}^m\hat{f}^*_k(x_k)-c(x_1,\cdots,x_m)}{\varepsilon}\right)$. Similar as the Schr\"odinger map extension in the population version~\eqref{Schrodinger_map_extension}, the marginal feasibility constraints enable us to do the following canonical extension: for $x_j\in\cX_j, \, j\in[m]$,
\begin{equation*}
\hat{f}_j^*(x_j) = - \varepsilon \log \left\{ \int \exp\left(\frac{\sum_{i \neq j} \hat{f}_i^*(x_i)-c(x_1, \dots, x_m)}{\varepsilon}\right)d\hat{\nu}^N_{-j}(x_{-j}) \right\}.
\end{equation*}
From this point on, will employ this extended empirical optimal potential and use $[\hat{\mvf}^*]$ to denote the equivalence class of the optimal potential $\hat{\mvf}^*$. The empirical dual objective function also enjoys similar bounded potential property and strong concavity w.r.t. to the geometry $\widehat{\mathscr{L}}_m:=\left(L^2(\hat{\nu}_1^N),\dots,L^2(\hat{\nu}_m^N)\right)$; see Appendix \ref{sec_concentration_of_empirical_potentials_and_joint_optimal_coupling_density} for precise formulations.

\subsection{Sample complexity of quantities of interest}
\label{sec:statistical_rates}

Our first main result is the convergence rate under the mean squared error (MSE) for the empirical cost to its population cost. The proof is deferred to Appendix \ref{sec_sample_complexity_proof}.

\begin{theorem}[Cost functional]
\label{thm:rate_cost}
For a given bounded cost function $c:\cX\to\bR$, there exists a constant $C=C(m,\varepsilon,\|c\|_\infty) > 0$ depending only on $m,\varepsilon, \|c\|_\infty$ such that
	\begin{equation}
    \label{eqn:rate_cost}
	\mathbb{E}\left[S_{\varepsilon}(\hat{\nu}^N_1,\dots,\hat{\nu}^N_m) - S_{\varepsilon}(\nu_1,\dots,\nu_m)\right]^2 \leq \frac{C}{N}.
	\end{equation}
\end{theorem}

%We did not explicitly put the dependence on $\varepsilon$ and $m$ in our statistical rates as they should be easily read off from the proofs. For example, 
We highlight that the sample complexity result in Theorem~\ref{thm:rate_cost} does not require any smoothness assumption on the cost function $c$, as long as it is bounded on the domain $\mathcal{X}$. Moreover, one may choose $C_{m,\varepsilon,\|c\|_\infty} = m^3 \varepsilon^2 e^{(10m+10)\|c\|_\infty/\varepsilon}$ (up to a numeric constant) in~\eqref{eqn:rate_cost}. Dependence on $\varepsilon$ is well-known to be exponentially poor even for $m = 2$ without further leveraging the additional smoothness of the cost function, and existing statistical convergence results do not allow $\varepsilon \to 0^+$; see, for example, \cite{AltschulerJonathanAustin22, BayraktarEcksteinZhang_stability, rigollet2022samplecomplexityentropicoptimal}.

\begin{remark}[Comparison with existing sample complexity bounds]
First, our result~\eqref{eqn:rate_cost} is \emph{dimension-free} and stronger than Theorem 5.3 in \citep{BayraktarEcksteinZhang_stability} where an intrinsic-dimension dependent sample complexity bound on the cost functional was obtained under regularity conditions on the marginal distributions. More specifically, they showed that: for $d_\nu > 2s$,
\begin{equation}\label{eqn:rate_cost_BEZ}
    \mathbb{E}\left|S_{\varepsilon}(\hat{\nu}^N_1,\dots,\hat{\nu}^N_m) - S_{\varepsilon}(\nu_1,\dots,\nu_m)\right|\lesssim_{m,\varepsilon ,\|c\|_\infty} N^{-s/d_{\nu}},
\end{equation}%\xc{What is the dependence of $\lesssim$ here?}\pl{$\lesssim_{m,\varepsilon}$}\xc{OK.}
where $d_\nu$ is an intrinsic dimension parameter of marginal distributions and $s$ is the order of differentiability of the entropy regularized MOT problem. Clearly, our convergence rate is much faster than~\eqref{eqn:rate_cost_BEZ} which suffers from curse-of-dimensionality, and we do not need any regularity assumption on the marginals $\nu_1, \dots, \nu_m$ and the smoothness of the cost function $c$. In addition, we mention that the constant in Theorem 5.3 for entropy regularization case in~\eqref{eqn:rate_cost_BEZ} is also exponentially poor with respect to $\varepsilon$. Next, for smooth cost with special graphical representation structure, Proposition 2 in \cite{DorZivKristjanHaim2024Neural} indicated a parametric rate for a neural network estimator when the neural network width $w=O(N)$. In contrast, our Theorem~\ref{thm:rate_cost} demonstrates that simple plug-in sample estimator for the EMOT can achieve the optimal rate of convergence with minimal regularity and smoothness. 

\begin{comment}
    provided that $\nu_1,\dots,\nu_m$ satisfies $I(d_\nu, s)$ for some $s\in\mathbb{N}, d_\nu\in(2s,\infty)$. Here, they said that $\nu\in \mathcal{P}(\mathbb{R}^d)$ satisfies $I(d_\nu, s)$, for some $s \in \mathbb{N}$, $d_\nu \in (0, \infty)$, if there exists $K > 0$ such that for every $\varepsilon > 0$, there exists $\Omega_\varepsilon \subset \mathbb{R}^d$ satisfying 
\[
\nu(\Omega_\varepsilon) \leq \mathbb{I}_{\{d_\nu > 2s\}} \varepsilon^{s d_\nu / (d_\nu - 2s)}
\]
and a partition $\mathcal{A}_\varepsilon$ of $\Omega_\varepsilon$ such that
\[
D_{\mathcal{A}_\varepsilon} := \sup_{A \in \mathcal{A}_\varepsilon} \sup_{x, y \in A} \|x - y\| \leq \varepsilon 
\quad \text{and} \quad 
|\mathcal{A}_\varepsilon| \leq K \varepsilon^{-d_\nu},
\]
where $|\mathcal{A}_\varepsilon|$ denotes the number of sets in the partition $\mathcal{A}_\varepsilon$.
\end{comment}
    
\end{remark}

%\begin{remark}
    %For cost with special graphical representation structure, \pl{Although I think the result in \cite{DorZivKristjanHaim2024Neural} works for smooth cost only} using neural network estimator, Proposition 2 in \cite{DorZivKristjanHaim2024Neural} also indicated a parametric rate when the neural network width $w=O(N)$.
%\end{remark}

Our next task is to present a dimension-free concentration bound for the empirical Schr\"{o}dinger coupling to their population analogs when acting on any bounded test function. The following Theorem~\ref{boundtest} not only proves such a parametric rate guarantee, but also serves as a stepping stone for deriving sharp rates of convergence under more general discrepancy measures (including $W_1$ and $W_2$ distances) between $\pi^*_\varepsilon$ and $\hat{\pi}^N_\varepsilon$ (cf. Theorem~\ref{supremp_coupling} below).

\begin{theorem}[Schr\"{o}dinger coupling]\label{boundtest}
Suppose that $\pi^*_\varepsilon$ (resp. $\hat{\pi}^N_\varepsilon$) is the Schr\"{o}dinger coupling of Problem (\ref{eqn:EMOT_formulation}) (resp. Problem (\ref{primal2})) with bounded cost function $c:\cX\to\bR$. Then there exists a constant  $C := C(m,\varepsilon,\|c\|_\infty) > 0$ such that for any $g\in L^\infty(\otimes_{j=1}^m\nu_j)$, and for all $t > 0$, we have with probability at least $1-(2m^2-2m+2)e^{-t}$,
\begin{equation}
    |(\pi_\varepsilon^*-\hat{\pi}^N_\varepsilon)(g)| \leq C \|g\|_\infty \sqrt{\frac{t}{N}}.
\end{equation}
\end{theorem}

Proof of Theorem~\ref{boundtest} can be found in Appendix \ref{sec_concentration_of_empirical_potentials_and_joint_optimal_coupling_density}. Here, we want to emphasize that the crux to establish Theorem \ref{boundtest} is embedded in Lemma \ref{gradnormp} on a novel concentration bound for the empirical gradient norm. Theorem \ref{boundtest} specialized to two marginal case was first obtained by \cite{rigollet2022samplecomplexityentropicoptimal}. Lemma \ref{gradnormp} marks the major technical divergence from the two-marginal case $m = 2$ in \cite{rigollet2022samplecomplexityentropicoptimal}, and its challenge deserves particular clarification. 

When the number of marginals $m \geq 3$, there are substantial obstacles in proving the concentration of the empirical gradient norm of dual objective functional (Lemma \ref{gradnormp}). The argument  in \cite{rigollet2022samplecomplexityentropicoptimal} is based on the standard Hoeffding inequality by simply stacking the independent samples, which {\it only} works for two marginals. For $m\geq3$, the nonlinear dependency structure in the empirical gradient creates highly nontrivial statistical challenges. Specifically, one needs to {\it simultaneously} handle all marginal constraints of the multimarginal Schr\"{o}dinger system by fully harnessing the independence among samples from different marginals. In words, there is no such issue in the 2-marginal case and simple stacking works. Thus, our technique and results are {\it not} extensions of \cite{rigollet2022samplecomplexityentropicoptimal} from two- to $m$-marginal, since their results heavily rely on the specific structure of two-marginal Schr\"{o}dinger system. To resolve the intricate dependencies inside the multimarginal system, one has to find the appropriate approach to decouple its structure in order to exploit the characteristics of the system. Our key Lemma \ref{gradnormp} (concentration of empirical gradient norm) and its supporting induction argument for decoupling are highly nontrivial and require new probability tools inspired from high-order $U$-statistics~\cite{ChenKato2019} to capture the nonlinearity in the dual objective gradient tailored for $m\geq3$.

Next, we can quantify the convergence rate of Schr\"{o}dinger coupling over a rich class of test functions. Specifically, we consider the standard $\beta$-smooth Hölder function class $\widetilde{\mathcal{H} }:= \cH([-1,1]^{md};\beta, L)$, where $\cH(\cX;\beta, L)$ contains the set of functions $f:\cX\to\mathbb{R}$ such that
\[\|f\|_\beta := \max_{|k| \leq\lfloor\beta\rfloor} \sup_{x\in \cX} |D^k f(x)| + \max_{|k|=\lfloor\beta\rfloor} \sup_{ \substack{ x\neq y \\ x,y \in \cX}
} \frac{|D^{k} f(x) - D^{k} f(y)|}{\|x-y\|^{\beta-\lfloor\beta\rfloor}} \leq L\]
for some parameter $L > 0$ and $\lfloor\beta\rfloor$ being the largest integer that is strictly smaller than $\beta$. Note that $\cH(\cX;1, L)$ is the class of Lipschitz continuous functions on $\cX$ with constant $L$. 

\begin{comment}
Before delving into the theorem, we  have to point out that the function class $\mathcal{H}$ is carefully chosen here. It is well known that $W_1(\mu,\nu)=\sup\{|\int f d(\mu-\nu)|, f\in\text{Lip}(\mathcal{X})\}\leq\sup\{|\int f d(\mu-\nu)|, f \text{ is bounded on}\,\mathcal{X}\}$. The following toy example implies that the $W_1$ metric is not an appropriate here due to the inevitable curse-of-dimensionality. Just consider the case where $m=2$ and $\nu_j=\nu_0$ for $j\in[m]$. We have $\nu_\varepsilon=\nu_0$
\end{comment}

%We will present some quantitative results on this statistical error, with one gauged via bounded test functions (Theorem \ref{boundtest}) and furthermore the other measured by some Hölder function class $\mathcal{H}$ (Theorem \ref{supremp}).

\begin{comment}
\pl{We need to furthermore assume $\mathcal{X}$ to be convex with nonempty interior as the requirement of Theorem 2.7.1 in \cite{vanderVaartWellner1996} to control the covering number.}
\end{comment}

%\xc{Write this as a general result in terms of Dudley's entropy integral bound and specialize $W_1$ and $W_2$ as special cases?}

%\pl{I put this general result as a remark here. $W_p$ distance case has been assembled the Corollary after Theorem \ref{supremp}.}

\begin{theorem}[Schr\"{o}dinger coupling convergence rate on H\"older class]\label{supremp_coupling}
There exists a constant $C > 0$ depending only on $m,\varepsilon,\beta, \|c\|_\infty$ and $d$ such that %\pl{For $d=2\beta$ case, should we write the bound as $N^{-1/2}\log(1+N)$ rather than $N^{-1/2}\log(1+N)$ to agree with (\ref{eqn:wasserstein_bound_Bar})} \xc{No need.}
  \begin{equation}\label{eqn:holder_bound_coupling}
    \mathbb{E} \left[ \sup_{g\in\widetilde{\mathcal{H}}} |(\pi_\varepsilon^*-\hat{\pi}^N_\varepsilon)(g)| \right] \le \begin{cases}
        C LN^{-1/2} & \text{if}\quad d<2\beta,\\
        C LN^{-1/2}\log{N} & \text{if}\quad d=2\beta,\\
        C LN^{-\beta/d} &\text{if}\quad d>2\beta.
    \end{cases}
  \end{equation}
  In particular, we have the following expectation bound on the $p$-Wasserstein distances for $p=1,2$:
  \begin{equation}\label{eqn:wasserstein_bound_coupling}
    \mathbb{E} [W_p^p(\pi_\varepsilon^* , \hat{\pi}^N_\varepsilon)] \leq \begin{cases}
        CN^{-1/2} & \text{if}\quad d<2p,\\
        CN^{-1/2}\log{N} & \text{if}\quad d=2p,\\
        CN^{-p/d} &\text{if}\quad d>2p,
    \end{cases}
  \end{equation}
for some constant $C > 0$ depending only on $m,\varepsilon,d, \|c\|_\infty$. 
\end{theorem}

The proof of Theorem~\ref{supremp_coupling} is deferred to Appendix \ref{sec_concentration_of_empirical_potentials_and_joint_optimal_coupling_density}. When specializing this general result to the $p$-Wasserstein distance for $p = 1,2$, we obtain a sharp convergence rate for entropic multimarginal Schr\"{o}dinger coupling that substantially improves the state-of-the-art rates~\citep{BayraktarEcksteinZhang_stability}. More discussions on the optimality in are presented after Theorem \ref{supremp} in Section \ref{Multimarginal_Schrodinger_barycenter_section}.

It is now worth re-iterating that for all sample complexity results in Section~\ref{sec:statistical_rates}, we only assume that the cost function $c$ is bounded, without imposing any smoothness conditions. This advantage enables us to approximate a broader class of aggregation notions of probability measures, even going beyond the Wasserstein barycenter in Section \ref{Multimarginal_Schrodinger_barycenter_section}. For instance, one could approximate Wasserstein median with entropy regularization by the multimarginal formulation of Wasserstein median, cf. Theorem 5.1 in \cite{Wasserstein_Median_Carlier_chenchene_Eichinger24}. %The theorem above could be attained if $c$ is sufficiently smooth as well, see, \cite{CalierChizatLaborde_DisplacementSmoothnessofEOT}.
However, we do not further pursue the direction of statistical analysis for nonsmooth robust functionals in this paper and leave it as a future work.

\subsection{Weak limits of quantities of interest}\label{section_weak_limits}
Now, we consider the weak limits of several EMOT related objects. We start with the CLT for the cost functional whose proof is deferred to Appendix \ref{sec_sample_complexity_proof}.

\begin{theorem}[CLT for cost functional]
\label{CLT_for_cost_functional}
For a given bounded cost function $c:\cX\to\bR$,
    \begin{equation*}
\sqrt{N}\left( S_{\varepsilon}(\hat{\nu}^N_1,\dots,\hat{\nu}^N_m) - S_{\varepsilon}(\nu_1,\dots,\nu_m) \right)\overset{w}{\longrightarrow} \cN\left(0,\sigma^2_\varepsilon\right),
    \end{equation*}
as $N \to \infty$, provided $\sigma^2_\varepsilon = \sum_{j=1}^m\text{Var}_{X_j\sim\nu_j}(f^*_j(X_j))>0$.
\end{theorem}

To the best of our knowledge, Theorem~\ref{CLT_for_cost_functional} provides the first CLT for the EMOT cost functional under general and nonsmooth $c$. Previous statistical results have focused exclusively on the two-marginal EOT setting. Using quite different techniques, relying on the smoothness of the cost function $c$, \cite{EustasioGozalezLoubesNiles23} obtained the CLT for two marginal entropic OT cost functional, by carefully controlling the derivatives of every order for the optimal potential $\mvf^*$. More recently, under significantly weaker (non-smooth) assumptions,
\cite{Alberto_Sanz_Gonzalez_Shayan_EOT_non_smooth_CLT23} got an analogous CLT for the two-marginal EOT problem.

Next, we provide the weak limit of the expectation of a bounded test function with respect to the Schr\"{o}dinger coupling. Let 
\begin{equation}
\label{eqn:term_K}
\begin{gathered}
    \cK(y^{(j)}_{\beta}, \beta = i_{j,1},\dots,i_{j,m-1}; j=1,\dots,m) \qquad \qquad \qquad \qquad \qquad \qquad \qquad \qquad \\
     = \frac{1}{\left((m-1)!\right)^m} \sum_{\substack{ \pi_j\in S_{m-1}\\1\leq j\leq m}}
    \Psi(y^{(j)}_{\beta}, \beta = i_{j,\pi_j(1)},\dots,i_{j,\pi_j(m-1)}; j=1,\dots,m)\\
     + \frac{1}{m-1}\sum_{j=1}^m\sum_{t=1}^{m-1}    \left[(\nu_{-j})(g\,p^*_{\varepsilon})(y_{i_{j,t}}^{(j)}) - (\otimes_{k=1}^m\nu_k)(g\,p^*_{\varepsilon})\right], 
\end{gathered}
\end{equation}
where $S_{m-1}$ denotes all the permutation on $\{1,\dots,m-1\}$ and 
\begin{align*}
    &\Psi(y^{(j)}_{\alpha}, \alpha = i_{j,1},\dots,i_{j,m-1}; j=1,\dots,m) \\
    &: =  -\bigint 
    \bigoplus\left( 
    \mvGamma^{-1}
    \begin{pmatrix}
         p^*_{\varepsilon}(x_1,y^{(2)}_{i_{2,1}},\dots,y^{(m)}_{i_{m,1}})-1\\ 
         p^*_{\varepsilon}(y^{(1)}_{i_{1,1}},x_2,\dots,y^{(m)}_{i_{m,2}})-1\\ 
        \vdots\\
        p^*_{\varepsilon}(y^{(1)}_{i_{1,m-2}},y^{(2)}_{i_{2,m-2}},\dots,x_{m-1},y^{(m)}_{i_{m,m-1}})-1\\ 
        p^*_{\varepsilon}(y^{(1)}_{i_{1,m-1}},y^{(2)}_{i_{2,m-1}},\dots,y^{(m-1)}_{i_{m-1,m-1}},x_m)-1
    \end{pmatrix}  \right) g (\mvx) p^*_{\varepsilon} (\mvx)
    d(\otimes_{k=1}^m\nu_k) (\mvx).
\end{align*}
Here, $\mvGamma^{-1}$ is the inverse of (linear, bounded) operator
\begin{align*}
    \mvGamma :\qquad
    \widetilde{\cC}^1 
    &\;\longrightarrow\; 
    \widetilde{\cC}^1 ,\\[4pt]
    \mvh=(h_1,\ldots,h_m)
    &\;\longmapsto\;
    (\mvGamma_1\mvh,\ldots,\mvGamma_m\mvh).
\end{align*}
with $\mvGamma_k\mvh = h_k + \sum_{l\neq k}\int h_l(x_l) p^*_\varepsilon(x_k, x_{-k})d \nu_{-k}(x_{-k})$ for $k\in[m]$. Such an inverse $\mvGamma^{-1}$ exists by Lemma 3.2 in \cite{CalierChizatLaborde_DisplacementSmoothnessofEOT}; see also more details in the proof in Appendix \ref{Sec_Weak_limit_of_expectation_under_entropic_optimal_transport_coupling}.

%Condition $\bar{\sigma}_\varepsilon^2(g)>0$ (This makes sure the limiting variance is valid. However, existing results \cite{AlbertoGonzalez-SanzEcksteinStephanNutzMarcel2025Lpweaklimit} usually neglected this. Not sure whether to mention this condition or not)
\begin{theorem}[CLT for Schr\"{o}dinger coupling]\label{CLT_test_function_very_complicated_computation}
If the cost function $c\in\cC^2$ and $g\in L^\infty(\otimes_{j=1}^m\nu_j)$ such that $\bar{\sigma}_\varepsilon^2(g)>0$, then we have that as $N \to \infty$,
    \begin{equation*}
        \sqrt{N}\left(\hat{\pi}^N_\varepsilon (g) - {\pi}^*_\varepsilon (g)\right)\overset{w}{\longrightarrow} \cN(0,\bar{\sigma}_\varepsilon^2(g)),
    \end{equation*}
where
\begin{align*}
     \bar{\sigma}_\varepsilon^2(g) = (m-1)^2\sum_{j=1}^m \text{Var}(\bar{\phi}^{(j)}(X_1^{(j)})),
\end{align*}
with
\begin{align*}
    \bar{\phi}^{(j^*)}(x) = \bE \left[  \cK (X^{(j)}_{\beta},  \beta = 1,\dots,m-1, j\in[m]\backslash\{j^*\}; x,X^{(j^*)}_{2},\dots,X^{(j^*)}_{m-1})  \right], 
\end{align*}
for $m$ independent sequences $X_1^{(j)},\dots, X_{m-1}^{(j)}\overset{\text{i.i.d.}}{\sim}\nu_j$ , $j\in[m]$ of independent random variables.
\end{theorem}

\begin{remark}[Non-degeneracy of the limiting variance]\label{remark_non_degenrate_variance} The asymptotic variance $\bar{\sigma}_\varepsilon^2(g) = 0$
if and only if
\[
\bE_{(X_1,\dots,X_m)\sim\pi_\varepsilon^*}[g(X_1,\dots,X_m)|X_j] = \pi_\varepsilon^*(g) \quad \text{for all} \quad j\in[m].
\]
The verification details can be found in Appendix \ref{Sec_Weak_limit_of_expectation_under_entropic_optimal_transport_coupling}.
\end{remark}

The proof of Theorem~\ref{CLT_test_function_very_complicated_computation} is quite involved and consists of seven steps we sketch below. To circumvent the missing regularity, we begin by linearizing the empirical EMOT system, relying on the concentration estimates established in Appendix \ref{sec_Concentration_bound_on_the_empirical_gradient_norm}. We then introduce several auxiliary operators and derive precise operator-norm bounds to control the resulting approximation errors. This analysis ultimately reveals that the objective can be reduces to a multi-sample $V$-statistic up to asymptotically negligible terms. The asymptotic variance $\bar{\sigma}_\varepsilon^2(g)$ in Theorem~\ref{CLT_test_function_very_complicated_computation} depends on linear (a.k.a. Hoeffding) projection of the kernel function $\cK$ in~\eqref{eqn:term_K}. The first term of $\cK$ quantifies the symmetrized fluctuation of the empirical Schr\"odinger potentials and the second term captures the marginal fluctuation of the centered empirical sum of empirical distributions, both weighted by the Gibbs kernel in the dual geometry~\eqref{dual1}. The proof details are given in the Appendix \ref{Sec_Weak_limit_of_expectation_under_entropic_optimal_transport_coupling}.

\begin{comment}
\begin{remark}
    One could easily see from the proof that we have essentially obtained the pointwise weak limit of $\sqrt{N}\left(\hat{\pi}^N_\varepsilon - {\pi}^*_\varepsilon \right)$ in the multimarginal Schr\"{o}dinger, namely, for  $g\in L^\infty(\otimes_{j=1}^m\nu_j)$,
    \begin{equation*}
        \sqrt{N}\left(\hat{\pi}^N_\varepsilon (g) - {\pi}^*_\varepsilon (g)\right)\overset{w}{\longrightarrow} \cN(0,\bar{\sigma}_\varepsilon^2(g)),
    \end{equation*}
with
\begin{align*}
     \bar{\sigma}_\varepsilon^2(g) =(m-1)^2\sum_{j=1}^m \text{Var}(\bar{\phi}^{(j)}(X_1^{(j)})),
\end{align*}
with 
\begin{align*}
    \bar{\phi}^{(j^*)}(x) = \bE \left[  \cK_0 (X^{(j)}_{\beta},  \beta = 1,\dots,m-1, j\in[m]\backslash\{j^*\}; x,X^{(j^*)}_{2},\dots,X^{(j^*)}_{m-1})  \right], 
\end{align*}
for $m$ independent sequences $X_1^{(j)},\dots, X_{m-1}^{(j)}\overset{\text{i.i.d.}}{\sim}\nu_j$ , $j\in[m]$ of independent random variables. Here, 
\begin{align*}
    &\quad \cK_0(y^{(j)}_{\beta}, \beta = i_{j,1},\dots,i_{j,m-1}; j=1,\dots,m)\\
    & = \frac{1}{\left((m-1)!\right)^m} \sum_{\substack{ \pi_j\in S_{m-1}\\1\leq j\leq m}}
    \Psi(y^{(j)}_{\beta}, \beta = i_{j,\pi_j(1)},\dots,i_{j,\pi_j(m-1)}; j=1,\dots,m)\\
    & \quad + \frac{1}{m-1}\sum_{j=1}^m\sum_{t=1}^{m-1}    \left[(\nu_{-j})(g\,p^*_{\varepsilon})(y_t^{(j)}) - (\otimes_{k=1}^m\nu_k)(g\,p^*_{\varepsilon})\right], 
\end{align*}
with $\Psi(\cdot)$ defined as in the preceding theorem.
\end{remark}
\end{comment}

%\begin{remark}[Comparison with existing CLT results]
    We remark that the type of weak limit $\sqrt{N}\left(\hat{\pi}^N_\varepsilon - {\pi}^*_\varepsilon \right)$ was firstly proposed with a modified empirical estimator for  Schr\"{o}dinger’s lazy gas experiment and also conjectured to be valid in EOT setting (i.e., $m = 2$) in \cite{Zaid_Harchaoui_Lang_Liu_Soumik_Pal_conjecture}. In \cite{AlbertoGonzalez-SanzJean-MichelLoubesJonathanNiles-Weed2024weaklimit}, the conjectured was proven in the entropic OT setting with smooth cost function and \cite{Alberto_Sanz_Gonzalez_Shayan_EOT_non_smooth_CLT23} provided corresponding results for nonsmooth case. Recently, \cite{AlbertoGonzalez-SanzEcksteinStephanNutzMarcel2025Lpweaklimit} obtained this pointwise weak limit in the sparse regularized OT problem. %One could easily check that our result, recover the weak limit derived in \cite{AlbertoGonzalez-SanzEcksteinStephanNutzMarcel2025Lpweaklimit, Alberto_Sanz_Gonzalez_Shayan_EOT_non_smooth_CLT23} as a special case when $m=2$.
    Moreover, using quite a different approach, \cite{ZivGoldfeldKengoKatoGabrielRiouxRitwikSadhu} investigated the weak limit by establishing Hadamard differentiability. Additionally, \cite{Kengo24IMAregularizedOT} developed a unified framework for deriving limit distributions of empirical proxies of various regularized OT distances, semiparametric efficiency of the plug-in empirical estimator and bootstrap consistency.
For comparison, the two-marginal CLT for Schr\"{o}dinger coupling established in \cite{Alberto_Sanz_Gonzalez_Shayan_EOT_non_smooth_CLT23}, under bounded cost function assumption, relies on the explicit two-marginal resolvent representation
of the linearized Schrödinger system. After normalizing the second potential, the linearized operator acts on
$L^2(P)\times L^2_0(Q)$ and takes the form
\[
\mvGamma =
\begin{pmatrix}
I & \mathcal A_Q \\
\overline{\mathcal A}_P & I
\end{pmatrix},
\qquad
\overline{\mathcal A}_P h
:=
\mathcal A_P h-\int \mathcal A_P h\,dQ.
\]
Its inverse admits the representation
\[
\mvGamma^{-1} =
\begin{pmatrix}
(I-\mathcal A_Q\overline{\mathcal A}_P)^{-1}
&
-(I-\mathcal A_Q\overline{\mathcal A}_P)^{-1}\mathcal A_Q
\\[6pt]
-(I-\overline{\mathcal A}_P\mathcal A_Q)^{-1}
  \overline{\mathcal A}_P
&
(I-\overline{\mathcal A}_P\mathcal A_Q)^{-1}
\end{pmatrix}.
\]
The corresponding empirical representation permits approximation
of the inverse by truncated auxiliary Neumann expansions, whose
alternating operator products admit explicit representations as an infinite-order \(V\)-statistic.

This construction does not extend directly to the general
multimarginal system: operator products involve multiple coordinate
paths, and each empirical pair kernel depends on the remaining
empirical marginals. An analogous auxiliary expansion with the
required approximation and stochastic error bounds would therefore
require additional analysis. We instead take a different route and compare the empirical and
population linearization operators in the $\mathcal C^1$ topology.
The assumption $c\in\mathcal C^2$ provides the uniform regularity
bounds needed for the estimates.

In the multimarginal setting,  \cite{ZivGoldfeldKengoKatoGabrielRiouxRitwikSadhu} established the Hadamard differentiability. The only work addressing the pointwise weak limit, apart from ours is \cite{Multimarginal_Unbalanced_OT_CLT_Xu21}, where they worked under the unbalanced multimarginal optimal transport setting but the exact limit remained unclear there. Our result provides a precise characterization of the pointwise weak limit in the multimarginal setting, thereby extending the results in the two marginal case \cite{AlbertoGonzalez-SanzJean-MichelLoubesJonathanNiles-Weed2024weaklimit, Zaid_Harchaoui_Lang_Liu_Soumik_Pal_conjecture}. For more comprehensive review on recent progress in the distributional limit of OT related quantities, see \cite{Alberto_Sanz_Gonzalez_del_Barrio_CLT_survey25}. 
%\end{remark}

\begin{comment}
\begin{remark}
    In the two marginal EOT setting, \cite{AlbertoGonzalez-SanzJean-MichelLoubesJonathanNiles-Weed2024weaklimit, Alberto_Sanz_Gonzalez_Shayan_EOT_non_smooth_CLT23} derived the limit distribution with a more explicit form. For example, \cite{AlbertoGonzalez-SanzJean-MichelLoubesJonathanNiles-Weed2024weaklimit} employed a matrix-like manipulation approach to find the inverse of the Fr\'{e}chet derivative in a closed form, thereby obtaining a correspondingly more explicit weak limit. 
    However, such closed-form expression cannot be generally expected for $m\geq3$. Also, for two marginal $f$-divergence regularized OT problem, an explicit form of the inverse of the Fr\'{e}chet derivative was not expected either \cite{AlbertoGonzalez-SanzEcksteinStephanNutzMarcel2025Lpweaklimit}.
\end{remark}
\end{comment}

We also would like to mention the following key lemma, which is of independent interest, in the proof of Theorem~\ref{CLT_test_function_very_complicated_computation}. To the best of our knowledge, Lemma~\ref{diff_empirical_population_f_i_bounded_by_other_potentials} provides the first convergence result for the optimal dual potentials in the EMOT problem under \emph{nonsmooth cost function} in the sense of $\cC^2$.

\begin{lemma}[Convergence rate for optimal dual potentials of EMOT]
\label{diff_empirical_population_f_i_bounded_by_other_potentials}
For the given cost function $c\in\cC^2(\cX)$ with $\cX$ compact support,
\begin{align*}
\| [\mvf^*] - [\hat{\mvf}^*] \|_{\widetilde{\cC}^1}
 = \cO_\bP \left( \sqrt{\frac{d \log N}{N}} \right).
\end{align*}
\end{lemma}

The proof is deferred to Appendix \ref{sec_concentration_of_empirical_potentials_and_joint_optimal_coupling_density}.
We shall emphasize that most of the existing estimates for empirical Schr\"odinger potentials
in smooth function norms, either in the EOT or EMOT setting, relies heavily on higher-order smoothness assumptions on the cost function than $\cC^2$, to mitigate the curse-of-dimensionality. Under the smooth cost function assumption $c\in\cC^\infty$, in the two marginal EOT problem, \cite{EustasioGozalezLoubesNiles23} showed that for $s\geq d/2+1$,
\begin{align*}
\| [\mvf^*] - [\hat{\mvf}^*] \|_{\widetilde{\cC}^s}
 = \cO_\bP \left( N^{-1/2} \right).
\end{align*}
Suppose $c\in\cC^{s+p}$ with $p>d/2$,  \cite{CalierChizatLaborde_DisplacementSmoothnessofEOT} established that for EMOT potential, 
\begin{align*}
\| [\mvf^*] - [\hat{\mvf}^*] \|_{\widetilde{\cC}^s}
 = \cO_\bP \left( N^{-1/2} \right).
\end{align*}
Our Lemma ~\ref{diff_empirical_population_f_i_bounded_by_other_potentials} indicates that the dimension-independent regularity assumption
$c\in\cC^2$ suffices for a multimarginal estimate in the
$\widetilde{\cC}^1$ norm, with rate
$\cO_{\bP}(\sqrt{d\log N/N})$.

Another line of work established convergence in \(L^2\) norms for normalized two-marginal potentials.
For fixed $\varepsilon>0$,
\cite{rigollet2022samplecomplexityentropicoptimal}
established an $\cO_{\bP}(N^{-1/2})$ rate under bounded costs assumption,
with constants depending exponentially on $1/\varepsilon$.
Refining this approach, \cite{stromme2024minimum} obtained
the same rate in $N$, 
with polynomial dependence on $1/\varepsilon$ governed by
intrinsic dimension, under bounded uniformly Lipschitz costs
and additional manifold and density assumptions.

\subsection{Bootstrap inference}
\label{bootstrap_consistency_section}
Despite we have obtained the limiting distribution for EMOT related objects in Theorem \ref{CLT_for_cost_functional} and Theorem \ref{CLT_test_function_very_complicated_computation}, one limitation is that the desired weak limit is non-pivotal, in the sense that it depends on the unknown population distribution $\nu_j$, $j\in[m]$ in a highly nonlinear manner. In this section, we consider the classical empirical bootstrap procedure described below, serving as a remedy for this obstacle.

For $j\in[m]$, given data $X_1^{(j)},\dots,X_N^{(j)}\overset{\text{i.i.d.}}{\sim}\nu_j$, let $X_1^{(j),B},\dots,X_N^{(j),B}$ be an independent sample from $\hat{\nu}_j := \frac{1}{N}\sum_{k=1}^N \delta_{X_k^{(j)}}$ and set $\hat{\nu}^B_j := \frac{1}{N}\sum_{k=1}^N \delta_{X_k^{(j),B}}$ as the bootstrap empirical distribution. We denote $\hat{\pi}^B_\varepsilon$ as the Schr\"{o}dinger coupling computed from the empirical bootstrap distributions $\hat{\nu}^B_1,\dots,\hat{\nu}^B_m$. Let $\bP^B$ denote the conditional probability given the data. The following two theorems establish the bootstrap validity for cost functional and Schr\"{o}dinger coupling under the Kolmogorov metric.

%In this section,  we require the sufficient smoothness of cost function $c$ (namely, $c\in C^k$ for $k\geq \frac{d}{2}+1$) for bootstrap to hold. %\pl{Without enough smoothness assumption}, naive bootstrap could fail when CLT holds, even in the classical setting \cite{Bickel_Peter_David_Freedman_bootstrap81}. So we leave the bootstrap consistency under $\cC^2$ as an open problem.

\begin{theorem}[Bootstrap validity for cost functional]
\label{bootstrap_cost}
    Suppose the cost function $c\in \cC^{2k}$ with $k\geq d/2 + 1$ and $\sigma^2_\varepsilon>0$. Then we have
    \begin{align*}
        \sup_{t\in\bR} &
        \Big| \bP^B \left( \sqrt{N}\left( S_{\varepsilon}(\hat{\nu}^B_1,\dots,\hat{\nu}^B_m) - S_{\varepsilon}(\hat{\nu}^N_1,\dots,\hat{\nu}^N_m)\right) \leq t \right) 
        \\
        & \qquad - \bP\left( \sqrt{N}\left( S_{\varepsilon}(\hat{\nu}^N_1,\dots,\hat{\nu}^N_m) - S_{\varepsilon}(\nu_1,\dots,\nu_m) \right) \leq t\right)\Big| \overset{\bP}{\longrightarrow} 0.
    \end{align*}
\end{theorem}

\begin{theorem}[Bootstrap validity for Schr\"odinger coupling]
\label{bootstrap_coupling}
    Suppose the cost function $c\in \cC^{2k}$ for $k\geq \frac{d}{2}+1$. If $\bar{\sigma}^2_\varepsilon(g)>0$ for $g\in \cC^k$ , then we have
    \begin{align*}
        \sup_{t\in\bR} 
        \left| \bP^B \left( \sqrt{N}\left(\hat{\pi}^B_\varepsilon (g) - \hat{\pi}^N_\varepsilon (g) \right) \leq t \right) 
        - \bP\left( \sqrt{N}\left(\hat{\pi}^N_\varepsilon (g) - \pi^*_\varepsilon (g) \right) \leq t \right)  \right| 
        \overset{\bP}{\longrightarrow} 0.
    \end{align*}
\end{theorem}

Note that in Theorems~\ref{bootstrap_cost} and~\ref{bootstrap_coupling}, we require a slightly stronger smoothness condition on cost function than the weak limit results in Section~\ref{section_weak_limits}. This is partially due to the steps of the Hadamard differentiability and functional delta method are intensely used in the proof. A quick review of them along with the proof of the theorems above can be found in Appendix \ref{Section_Hadamard_differentiability_Intro}. See \cite{vanderVaartWellner1996, Van_der_Vaart_Asymptotic_statistics00, Romisch_Werner_Delta_Method_Intro, ZivGoldfeldKengoKatoGabrielRiouxRitwikSadhu, Kengo24IMAregularizedOT} for more detailed introduction and applications. We shall leave the bootstrap consistency under $\cC^2$ as an open problem. 
We provide numerical verification for bootstrap consistency in Appendix \ref{sec: bootstrap_verification}.

\section{Multimarginal Schr\"{o}dinger barycenter} \label{Multimarginal_Schrodinger_barycenter_section}

In this section, we demonstrate how results derived in Section~\ref{main_results} provide statistical guarantees for the barycenter inference. The Wasserstein barycenter, introduced by~\cite{AguehCarlier2011}, has been widely considered as a natural model under the optimal transport metric for averaging measure-valued data. Given \( m \) probability measures \( \nu_1, \dots, \nu_m \) with support $\mathcal{X}_j \subset \mathbb{R}^d$, $j\in[m]$ with finite second moments and a barycentric coordinate vector \( \alpha := (\alpha_1, \dots, \alpha_m) \) such that $\alpha_j \geq 0$ and $\alpha_1 + \cdots + \alpha_m = 1$, the Wasserstein barycenter $\bar{\nu}$ is defined as a solution of minimizing the weighted variance functional
\begin{equation}\label{eqn:unregbary}
	\inf_{\nu} \sum_{i=1}^{m} \alpha_i W_2^2(\nu_i, \nu),
\end{equation}
where $W_2(\mu, \nu)$ denotes the 2-Wasserstein distance between $\mu$ and $\nu$.

%We first briefly review the background information and related literature about Wasserstein barycenter. After that, we define our version of regularized barycenter along with its statistical performance guarantees.

%\subsection{Introduction on Wasserstein barycenter}

For (discrete) probability measures with finite support of size $N$ (e.g., empirical measures on $N$-point clouds), their Wasserstein barycenter can be computed in poly$(m, N)$-time when the underlying domain dimension $d$ is fixed~\citep{AltschulerBoix-Adsera2021_fixed-dim-barycenter}, while computing (or even constant-accuracy approximating) the Wasserstein barycenter is NP-hard for general dimension~\citep{AltschulerBoix-Adsera2021_barycenter-NPhard}. On the other hand, quantitative statistical guarantees for estimating Wasserstein barycenters when we only have access to $\nu_1,\dots,\nu_m$ via their empirical analogs as random point clouds are far from being well-understood with a few exceptions by making strong curvature assumptions~\citep{AhidarLeGouicParis_barycenter-rate,LeGouic2019FastCO}.

It is known that the MOT and Wasserstein barycenter problems are closely connected. Specifically,~\citep{AguehCarlier2011} shows that the barycenter $\overline{\nu}$ can be recovered from the MOT solution $\pi_0^*$ of~\eqref{eqn:MOT_formulation} via 
$$\bar{\nu}:=\bar{\nu}_0={(T_{\alpha})}_{\sharp}\pi_0^*,$$ for $c(x_1,\dots,x_m) = c_\alpha(x_1,\dots,x_m) := \sum_{1\leq i<j\leq m}|\alpha_ix_i-\alpha_jx_j|^2$, where $T_{\alpha}(\mvx)=\sum_{j=1}^m\alpha_jx_j$ for $\mvx=(x_1,\dots, x_m)$. Here, the pushforward measure $T_\sharp \mu$ of a source probability measure $\mu$ on $\bR^{d_1}$ for a measurable map $T : \mathbb{R}^{d_1} \to \mathbb{R}^{d_2}$, is defined by $(T_\sharp \mu) (B) = \mu(T^{-1}(B))$ for measurable subset $B \subset \mathbb{R}^{d_2}$.

Given the fundamental limitation of computing,  estimating and performing inference on the exact barycenter solution in~\eqref{eqn:unregbary}, we propose a new notion of entropic barycenter for averaging probability measures based on the EMOT formulation~\eqref{eqn:EMOT_formulation}. Our primary goal is to study its statistical properties, including the sample complexity for estimating several key quantities (cost functional, coupling, barycenter) based on point clouds data sampled from the marginal distributions, as well as weak limits of these quantities of interest. 

%These characteristics across multiple datasets reveals critical structures of the underlying OT problem and is fueled with many applications in statistics and machine learning~\citep{srivastava2018scalable,ZhuangChenYang2022}, as well as their adjacent fields such as computer graphics and image processing ~\citep{rabin2012wasserstein,Solomon_2015} \pl{repeat with intro}. In the special two-marginal setting $m = 2$, our formulation recovers the entropic OT (EOT) problem which has a dynamical formulation known as the Schr\"{o}dinger bridge problem~\citep{nutz2022entropic,rigollet2022samplecomplexityentropicoptimal}. In either case, the optimal coupling can be efficiently computed by existing methods such as Sinkhorn's algorithm for $m=2$~\citep{Cuturi_Sinkhorn2013} and a similar multidimensional matrix scaling algorithm for the general setting $m \geq 2$~\citep{Carlier_MOT-Sinkhorn-linear}, thus leading to fast computation of our proposed barycenter.
%This is a copy of the original introduction from MSB section 1.1.

%Our notion of the multimarginal Schr\"odinger barycenter is based on the idea for approximating MOT~\citep{GangboSwiech_MOT} via entropy regularization, where the latter can be solved by (near) linear time complexity algorithms~\citep{Carlier_MOT-Sinkhorn-linear,LinHoCuturiJordan_MOT-complexity}. Our notion of barycenter achieves a careful regularization that ensures statistical and computational advantages without significantly distorting the resulting barycenter from the unregularized barycenter. 

Below, we introduce the \emph{multimarginal Schr\"{o}dinger barycenter} (MSB) based on the EMOT formulation, which admits fast linear time complexity off-the-shelf algorithms to compute~\citep{Carlier_MOT-Sinkhorn-linear,LinHoCuturiJordan_MOT-complexity}. This motivates the following definition.

\begin{definition}
	Given a collection of probability measures $\nu_1, \dots, \nu_m$, the {\bf multimarginal Schr\"{o}dinger barycenter} (MSB) $\bar{\nu}_{\varepsilon}$ is defined as
	\begin{equation}
    \label{eqn:MER-WB}
        \bar{\nu}_{\varepsilon}={(T_{\alpha})}_{\sharp}\pi_{\varepsilon}^*,
	\end{equation}
	where $\pi_{\varepsilon}^*$ is the unique minimizer of the EMOT problem
    \begin{equation}\label{primal1}
	%S_{\varepsilon}(\nu_1,\dots,\nu_m):=
    S_{\alpha,\varepsilon} (\nu_1,\dots,\nu_m) 
    :=
    \inf_{\pi\in\Pi(\nu_1,\dots,\nu_m)} \Big\{ \int_{\cX} c_\alpha d\pi+\varepsilon \ \text{KL}(\pi||\nu_1\otimes\cdots\otimes \nu_m) \Big\},
    \end{equation}
    with the cost function $c_\alpha(x_1,\dots,x_m) = \sum_{1\leq i<j\leq m}|\alpha_ix_i-\alpha_jx_j|^2$.
\end{definition}

\begin{comment}
\begin{remark}
    We comment that, throughout the paper, the cost function $c(\cdot)$ is only assumed to be continuous unless otherwise specified.
\end{remark}
\end{comment}

%Recall that we use $\pi^*_\varepsilon$, (resp. $\pi^*_0$) to denote the optimizer $S_{\varepsilon}(\nu_1,\cdots,\nu_m)$ (resp. $S_{\alpha,0}(\nu_1,\cdots,\nu_m)$) and that $\bar{\nu}_{\varepsilon}={T_{\alpha}}_{\sharp}\pi^*_{\varepsilon}$ is the multimarginal Schr\"{o}dinger barycenter and $\bar{\nu}:=\bar{\nu}_0={T_{\alpha}}_{\sharp}\pi^*_0$ is the unregularized Wasserstein barycenter.

Compared to existing regularized barycenters, MSB strikes a careful balance offering enough regularization for statistical and computational gains without significantly distorting the original barycenter. One could examine this by checking the different notions of regularized barycenters of $m$ Dirac masses. This advantage of manipulating regularization on multimarginal formulation over pairwise regularization has also been empirically observed in a different problem \cite{HasslerRinghChenKarlsson21}. The following basic properties of MSB could be easily derived and their proofs are deferred to Appendix \ref{section_approximation_error} and Appendix \ref{omitted_proofs}.

\begin{proposition}[Approximation]
\label{approgamma}
Any cluster point of $(\bar{\nu}_{\varepsilon})_{\varepsilon>0}$ is a Wasserstein barycenter $\bar{\nu}$. If further $\bar{\nu}$ is unique, %then the whole sequence $(\bar{\nu}_{\varepsilon})_{\varepsilon>0}$ converges \xc{weakly?} towards $\bar{\nu}$, and
then we have $\lim_{\varepsilon\to0^+}W_2(\bar{\nu}_{\varepsilon}, \bar{\nu})=0.$
\end{proposition}

\begin{proposition}[Lipschitz continuity]\label{LipcontiMSB}
Let $\bar{\nu}_\varepsilon$ (resp. $\tilde{\nu}_\varepsilon$) be the MSB for $\bm{\nu}:=(\nu_1,\dots,
\nu_m)\in\Pi_{j=1}^m\mathcal{P}_2(\mathcal{X}_j)$ (resp. $\bm{\tilde{\nu}}:=(\tilde{\nu}_1,\dots,
\tilde{\nu}_m)\in\Pi_{j=1}^m\mathcal{P}_2(\mathcal{X}_j)$). We have 
\[W_1(\bar{\nu}_\varepsilon,\tilde{\nu}_\varepsilon)\leq \sqrt{m}W_2(\bm{\nu},\bm{\tilde{\nu}})+C W_2(\bm{\nu},\bm{\tilde{\nu}})^{1/2},\]
where $W_2(\bm{\nu},\bm{\tilde{\nu}}):=(\sum_{j=1}^mW_2(\nu_j,\tilde{\nu}_j)^2)^{1/2}$ and $C > 0$ is a constant depending on $\varepsilon$, and  the second moment of $\nu_j$ and $\tilde{\nu}_j$ for $j \in [m]$.
\end{proposition}

\subsection{Statistical guarantees}
Given the samples $\{X^{(j)}_i\}_{j \in [m], i \in [N]}$ from the standard sampling model, the empirical MSB $\hat{\nu}^N_{\varepsilon}$ is defined as $\hat{\nu}^N_{\varepsilon}= ({T_{\alpha}})_{\sharp}\hat{\pi}^N_{\varepsilon}$ where $\hat{\pi}^N_{\varepsilon}$ is the unique minimizer of the following problem

\begin{equation}
S_{\alpha,\varepsilon}(\hat{\nu}^N_1,\dots,\hat{\nu}^N_m):=\inf_{\pi\in\Pi(\hat{\nu}^N_1,\dots,\hat{\nu}^N_m)}\int_{\cX} c_\alpha d\pi+ \varepsilon\text{KL}(\pi||\hat{\nu}^N_1\otimes\cdots\otimes \hat{\nu}^N_m).
\end{equation}

\begin{corollary}[Concentration for empirical MSB]\label{cor: Concentration_empirical_MSB}
    For $\bar{\nu}_\varepsilon$ (resp. $\hat{\nu}^N_\varepsilon$) the (resp. empirical) multimarginal Schr\"{o}dinger barycenter, there exists a constant $C := C(m,\varepsilon) > 0$ such that for any $h\in L^{\infty}(\bar{\nu}_\varepsilon)$, and for all $t > 0$, we have with probability at least $1-(2m^2-2m+2)e^{-t}$,
     \begin{equation}
     \label{eqn:rate_bounded_test}
        |(\bar{\nu}_\varepsilon-\hat{\nu}^N_\varepsilon)(h)| \leq C \|h\|_\infty \sqrt{\frac{t}{N}}.
    \end{equation}
    %whenever $t\gtrsim1$ \xc{Explain this requirement}. \pl{$t>0$. Please see the updated remark at page 23.}
\end{corollary}

\begin{proof}
    It comes obvious by setting $g=h\circ T_\alpha$ in Theorem \ref{boundtest}.
\end{proof}

%\pl{ The following $W_p$ bound only makes sense for MSB rather than for optimal coupling. For the coupling, $d\to md$ cuz $\cH$ for coupling would be in $[-1,1]^{md}$ rather than $[-1,1]^d$. }

Furthermore, we quantify the convergence rate of MSB over the $\beta$-smooth Hölder function class $\mathcal{H} := \cH([-1,1]^d;\beta, L)$ as considered in Theorem~\ref{supremp_coupling}. %Note that $\cH([-1,1]^d;1, L)$ is the class of Lipschitz continuous functions on $[-1,1]^d$ with constant $L$. 

\begin{theorem}[Barycenter convergence rate on H\"older class]\label{supremp}
There exists a constant $C > 0$ depending only on $m,\varepsilon,\beta,\|c\|_\infty$ and $d$ such that %\pl{For $d=2\beta$ case, should we write the bound as $N^{-1/2}\log(1+N)$ rather than $N^{-1/2}\log(1+N)$ to agree with (\ref{eqn:wasserstein_bound_Bar})} \xc{No need.}
  \begin{equation}\label{eqn:holder_bound}
    \mathbb{E}\sup_{h\in\mathcal{H}} |(\bar{\nu}_\varepsilon-\hat{\nu}^N_\varepsilon)(h)|\le \begin{cases}
        C LN^{-1/2} & \text{if}\quad d<2\beta,\\
        C LN^{-1/2}\log{N} & \text{if}\quad d=2\beta,\\
        C LN^{-\beta/d} &\text{if}\quad d>2\beta.
    \end{cases}
  \end{equation}
  In particular, we have the following expectation bound on the $p$-Wasserstein distances for $p=1,2$:
    \begin{equation}\label{eqn:wasserstein_bound}
    \mathbb{E} [W_p^p(\bar{\nu}_\varepsilon,\hat{\nu}^N_\varepsilon)] \leq \begin{cases}
        CN^{-1/2} & \text{if}\quad d<2p,\\
        CN^{-1/2}\log{N} & \text{if}\quad d=2p,\\
        CN^{-p/d} &\text{if}\quad d>2p,
    \end{cases}
  \end{equation}
for some constant $C > 0$ depending only on $m,\varepsilon,d, \|c\|_\infty$. 
\end{theorem}

The proof of Theorem~\ref{supremp} is relegated to Appendix \ref{sec_concentration_of_empirical_potentials_and_joint_optimal_coupling_density}.

\begin{remark}
    We highlight that our convergence rate~\eqref{eqn:wasserstein_bound} is sharp up to poly-log terms and in general cannot be improved without extra structural assumptions on $\nu_1,\dots,\nu_m$. To see this, consider the case $m = 1$, or equivalently $m = 2$ and $\alpha = (1, 0)$. Then the MSB degenerates to the empirical measure on the point cloud $X^{(1)}_1, \dots, X^{(1)}_N$, which is known to have the optimal convergence rate under the Wasserstein distance $W_p$ for all $p \geq 1$~\citep{WeedBach,FournierGuillin}. 
    
    In addition, we make comparisons with existing literature. Using the concept of shadow, Theorem 3.3 in \cite{BayraktarEcksteinZhang_stability} implies that $W_1(\bar{\nu}_\varepsilon, \hat{\nu}^N_{\varepsilon}) \leq m^{1/2} \Delta_{N}+ C\Delta_{N}^{1/2}$, where $\Delta_{N}=\sum_{j=1}^mW_1(\nu_j,\hat{\nu}^N_j)$. Substituting the (optimal) $W_1$ rate of convergence of the empirical measure (e.g., Theorem 1 in \cite{FournierGuillin}) into the last inequality, one gets
\begin{align}\label{eqn:wasserstein_bound_Bar}
    \mathbb{E}W_1(\bar{\nu}_\varepsilon, \hat{\nu}^N_{\varepsilon}) \leq \begin{cases}
        CN^{-1/4}, &\quad\text{if} \quad d=1,\\
        CN^{-1/4}(\log{N})^{1/2}, &\quad\text{if} \quad d=2,\\
        CN^{-1/2d}, &\quad\text{if} \quad d>3,\\
    \end{cases}
\end{align}
for some constant $C > 0$ depending only on $m,\varepsilon,d$. 
\noindent In comparison, our rate~\eqref{eqn:wasserstein_bound} specialized to $p=1$ substantially improves~\eqref{eqn:wasserstein_bound_Bar} for every dimension $d \geq 1$. 
\end{remark}

Next, thanks to Theorem \ref{CLT_test_function_very_complicated_computation}, we are ready for the weak limit of the expectation of a bounded test function with respect to MSB, further indicating the advantage of our proposed notion of barycenter.

\begin{corollary}[CLT for MSB]
\label{CLT_test_function_very_complicated_computation_MSB}
For $h\in L^{\infty}(\bar{\nu}_\varepsilon)$ with $\sigma_\varepsilon^2(h)>0$, it holds true that 
    \begin{equation}
    \sqrt{N}\left(\hat{\nu}^N_\varepsilon (h) - \bar{\nu}_\varepsilon (h)\right)\overset{w}{\longrightarrow} \cN(0,\sigma_\varepsilon^2(h)),
    \end{equation}
where 
\begin{align*}
     \sigma_\varepsilon^2(h)=(m-1)^2\sum_{j=1}^m \text{Var}(\phi^{(j)}(X_1^{(j)})),
\end{align*}
with 
\begin{align*}
    \phi^{(j^*)}(x) = \bE \left[  \cK_0 (X^{(j)}_{\beta},  \beta = 1,\dots,m-1, j\in[m]\backslash\{j^*\}; x,X^{(j^*)}_{2},\dots,X^{(j^*)}_{m-1})  \right], 
\end{align*}
for $m$ independent sequences $X_1^{(j)},\dots, X_{m-1}^{(j)}\overset{\text{i.i.d.}}{\sim}\nu_j$ , $j\in[m]$ of independent random variables. Here, 
\begin{align*}
    &\quad \cK_0(y^{(j)}_{\beta}, \beta = i_{j,1},\dots,i_{j,m-1}; j=1,\dots,m)\\
    & = \frac{1}{\left((m-1)!\right)^m} \sum_{\substack{ \pi_j\in S_{m-1}\\1\leq j\leq m}}
    \Psi_0(y^{(j)}_{\beta}, \beta = i_{j,\pi_j(1)},\dots,i_{j,\pi_j(m-1)}; j=1,\dots,m)\\
    & \quad + \frac{1}{m-1}\sum_{j=1}^m\sum_{t=1}^{m-1} \left[(\nu_{-j})(\tilde{h}_\alpha\,p^*_{\varepsilon})(y_{i_{j,t}}^{(j)}) - (\otimes_{k=1}^m\nu_k)(\tilde{h}_\alpha\,p^*_{\varepsilon})\right],
\end{align*}
where $S_{m-1}$ denotes all the permutation on $\{1,\dots,m-1\}$, $\tilde{h}_\alpha = h\circ T_\alpha$, and
\begin{align*}
    &\Psi_0(y^{(j)}_{\alpha}, \alpha = i_{j,1},\dots,i_{j,m-1}; j=1,\dots,m) \\
    &: =  -\bigint 
    \bigoplus\left(
    \mvGamma^{-1}\begin{pmatrix}
         p^*_{\varepsilon}(x_1,y^{(2)}_{i_{2,1}},\dots,y^{(m)}_{i_{m,1}})-1\\ 
         p^*_{\varepsilon}(y^{(1)}_{i_{1,1}},x_2,\dots,y^{(m)}_{i_{m,2}})-1\\ 
        \vdots\\
        p^*_{\varepsilon}(y^{(1)}_{i_{1,m-2}},y^{(2)}_{i_{2,m-2}},\dots,x_{m-1},y^{(m)}_{i_{m,m-1}})-1\\ 
        p^*_{\varepsilon}(y^{(1)}_{i_{1,m-1}},y^{(2)}_{i_{2,m-1}},\dots,y^{(m-1)}_{i_{m-1,m-1}},x_m)-1
    \end{pmatrix}  \right) \tilde{h}_\alpha (\mvx) p^*_{\varepsilon} (\mvx)
    d(\otimes_{k=1}^m\nu_k) (\mvx).
\end{align*}

\begin{comment}
with $\sigma_\varepsilon^2(h) = \text{Var} \left( \sum_{j=1}^m U_j\right)$
where 
\begin{align*}
    U_j = \bE \left[\Xi(X^{(1)},\dots,X^{(m)})|X^{(j)} \right],
\end{align*}
for $(X^{(1)},\dots,X^{(m)})\sim \otimes_{k=1}^m\nu_k$ with

The following formulation is the same as the present statement. We use current one just to have similar form as Alberto Gonzalez-Sanz's psi regularization paper.

  with $\sigma_\varepsilon^2(h) = \sum_{j=1}^m U_j$
where 
\begin{align*}
    U_j &= \text{Var}\left(\bE \left[\Xi(X^{(1)},\dots,X^{(m)})|X^{(j)} \right] \right) \\
    &= \bE \left( \bE \left[\Xi(X^{(1)},\dots,X^{(m)})|X^{(j)} \right] \right)^2 - \left(\bE\Xi(X^{(1)},\dots,X^{(m)})\right)^2 ,
\end{align*}

\begin{align*}
    \Xi (X^{(1)},\dots,X^{(m)}) 
    &= \int (\tilde{h}_\alpha \, p^*_{\varepsilon}) (x_1,\dots,x_m) U(x_1,\dots,x_m; X^{(1)},\dots,X^{(m)})d(\otimes_{k=1}^m\nu_k)(x_1,\dots,x_m) \\
    &+ \sum_{j=1}^m \tilde{h}_\alpha^{(j)}(X^{(j)}), %- \bE\tilde{h}_\alpha^{(j)}(X^{(j)})
\end{align*} 
and $\tilde{h}_\alpha = h\circ T_\alpha$, as well as
\begin{align*}
    U(x_1,\dots,x_m,X^{(1)},\dots,X^{(m)}) := \oplus\left((D_{f^*}\mvT)^{-1}\left[\begin{pmatrix}
        p^*_{\varepsilon}(x_1,X^{(2)},\dots,X^{(m)})\\p^*_{\varepsilon}(X^{(1)}, x_2,\dots,X^{(m)})\\ \dots\\  p^*_{\varepsilon}(X^{(1)},X^{(2)},\dots,x_m)
    \end{pmatrix}\right]\right).
\end{align*}
\end{comment}
\end{corollary}

\begin{proof}
    It comes obvious by setting $g=h\circ T_\alpha$ in Theorem \ref{CLT_test_function_very_complicated_computation}.
\end{proof}

Finally, we provide the bootstrap consistency for MSB. Suppose $\hat{\nu}^B_\varepsilon$ is the MSB computed from the empirical bootstrap distributions $\hat{\nu}^B_1,\dots,\hat{\nu}^B_m$. 

\begin{corollary}[Bootstrap validity for MSB]\label{cor: bootstrap_MSB}
     Provided that $\sigma^2_\varepsilon(h)>0$ for $h\in \cC^{s}([-1,1]^d)$ with $s\geq 1+d/2$, we have
    \begin{align*}
        \sup_{t\in\bR} 
        \left| \bP^B \left( \sqrt{N}\left(\hat{\nu}^B_\varepsilon (h) - \hat{\nu}^N_\varepsilon (h) \right) \leq t \right) 
        - \bP\left( \sqrt{N}\left(\hat{\nu}^N_\varepsilon (h) - \bar{\nu}_\varepsilon (h) \right) \leq t \right)  \right| 
        \overset{\bP}{\longrightarrow} 0.
    \end{align*}
\end{corollary}

\begin{proof}
    It comes obvious by setting $g=h\circ T_\alpha$ in Theorem \ref{bootstrap_coupling}.
\end{proof}

\bibliographystyle{imsart-number} % Style BST file (imsart-number.bst or imsart-nameyear.bst)
\bibliography{entropic_barycenter}       % Bibliography file (usually '*.bib')

\newpage
\appendix

\begin{center}
    \textbf{\huge Appendix}
\end{center}

\noindent The appendix is organized as follows.
\begin{itemize}
    \item Appendix \ref{appendix_section_notations} introduces additional notations used throughout the appendix section.
    
    \item Appendix \ref{section_approximation_error} presents proof of approximation error of EMOT as the regularization $\varepsilon\to 0$.

    \item Appendix \ref{sec_Concentration_bound_on_the_empirical_gradient_norm} focuses on the high probability bound on the empirical gradient norm $\|\nabla\widehat{\Phi}_{\varepsilon}(\vf^*)\|_{\widehat{\mathscr{L}}_m}$, one of the central quantity for the analysis.
    
    \item Appendix \ref{sec_concentration_of_empirical_potentials_and_joint_optimal_coupling_density} develops detailed properties about population and empirical optimal Schr\"{o}dinger potentials $f^*_j, \hat{f}^*_j,\,j\in[m]$ along with proof of Theorem \ref{boundtest}, Lemma \ref{diff_empirical_population_f_i_bounded_by_other_potentials} and Theorem \ref{supremp}.

    \item Appendix \ref{sec_sample_complexity_proof} contains proofs of results related to the cost functional, namely proof of Theorem \ref{thm:rate_cost} (sample complexity) and proof of Theorem \ref{CLT_for_cost_functional} (weak limit).

    \item Appendix \ref{Sec_Weak_limit_of_expectation_under_entropic_optimal_transport_coupling} outlines the proof structure of Theorem \ref{CLT_test_function_very_complicated_computation}, giving weak limit of expectation under entropic optimal transport coupling.
    
   \item Appendices~\ref{appendix_lineratization_empirical_EMOT} and \ref{appendix_technical_details_for_operator_norms_estimates} 
   complete the detailed technical steps deferred from Appendix~\ref{Sec_Weak_limit_of_expectation_under_entropic_optimal_transport_coupling}.
    
    \item Appendix~\ref{omitted_proofs} collects the remaining proofs of the main theorems.

    \item Appendix \ref{appendix_technical_lemmas} gathers auxiliary technical lemmas.
    
    \item Appendix \ref{Section_Hadamard_differentiability_Intro} concerns the Hadamard differentiability used for establishing bootstrap consistency.
    \item Appendix \ref{sec: bootstrap_verification} reports the numerical verification of bootstrap consistency.
\end{itemize}

\section{List of notations}\label{appendix_section_notations}

\noindent For any parameter $\varepsilon$, we use  $x \lesssim_\varepsilon y$ (resp. $x \gtrsim_\varepsilon y$) to denote  $x \leq C_\varepsilon y$ (resp.  $x \geq C_\varepsilon y$) for some constant  $C_\varepsilon > 0$  depending on $\varepsilon$. If $x\lesssim_\varepsilon y$ and $x \gtrsim_\varepsilon y$, we write $x\asymp_\varepsilon y$. For $f_j : \cal X \to \bR$, we shall denote a vector-valued function using the boldface $\vf = (f_1,\dots,f_m)$. We shall also denote a vector $\mvx=(x_1,\dots,x_m)\in\mathbb{R}^m$, $\|\mvx\|_2=\sqrt{\sum_{j=1}^mx_j^2}$, $\|\mvx\|_\infty=\max_{1\leq j \leq m} |x_j|$,
and $\mvx^T$ denotes the transpose of a vector. For a vector space \(V\)  over some field $\bF$, the dual space of \(V\), denoted by \(V'\), is the set of all linear functionals from \(V\) to \(\bF\).

\noindent Given a set $\cD$, $|\cD|$ denotes the cardinality of $\cD$. We define the diameter of a set $S$ in a metric space with distance $d$ as $\text{diam}(S) = \sup_{x,y\in S}d(x,y)$. $a\vee b:=\max\{a,b\}$. For $f:\cX_2\to\cX_3$ and $g:\cX_1\to\cX_2$, the compound function $f\circ g:\cX_1\to\cX_3$ is defined as $f\circ g(x):=f(g(x))$. For a function class $\cF$, $N(\cF,\varepsilon,\|\cdot\|)$ denotes the $\varepsilon$-covering number of class $\cF$ under the metric induced by $\|\cdot\|$. Given $\vf=(f_1,\dots,f_m)\in\prod_{j=1}^m \cC(\cX_j)$, $\bigoplus (f_1,\dots,f_m) := \sum_{j=1}^m f_j$. Often $\sum_{i=1}^m f_i(x_i)$ is written as $f_1\oplus\cdots\oplus f_m$. For $f_j:\cX_j\to\bR$, $\nabla f_j :\cX_j\to\bR^d$ denotes its gradient and $\nabla \vf = (\nabla f_1,\dots, \nabla f_m)$.

\noindent We sometimes consider the $\beta$-smooth Hölder function class $\cH(\cX;\beta, L)$ containing the set of functions $f:\cX\to\mathbb{R}$ such that
\[\|f\|_\beta := \max_{|k| \leq\lfloor\beta\rfloor} \sup_{x\in \cX} |D^k f(x)| + \max_{|k|=\lfloor\beta\rfloor} \sup_{ \substack{ x\neq y \\ x,y \in \cX}
} \frac{|D^{k} f(x) - D^{k} f(y)|}{\|x-y\|^{\beta-\lfloor\beta\rfloor}} \leq L\]
for some parameter $L > 0$, where $\lfloor\beta\rfloor$ is the largest integer that is strictly smaller than $\beta$. Here, for a multi-index $k=(k_1,\dots,k_m)$ of $m$ integers, $|k|=\sum_{j=1}^mk_j$ and the differential operator 
$D^{k}:=\frac{\partial^{|k|}}{\partial x_1^{k_1}\cdots\partial x_m^{k_m}}.$ Note that $\cH(\cX;1, L)$ is the class of Lipschitz continuous functions on $\cX$ with constant $L$.

\noindent We say $a_n=\cO(b_n)$ if $\limsup_{n\to\infty}|\frac{a_n}{b_n}|<\infty$.
For a sequence of random variables $(\xi_n)_{n\geq 1}$ and a sequence of scalars $a_n > 0$, we write $\xi_n = o_\bP(a_n)$
if for every $\varepsilon > 0$,
$\lim_{n \to \infty} \bP\left( \frac{|\xi_n|}{a_n} > \varepsilon \right) = 0.$ We write $\xi_n = \cO_\bP(a_n)$
if for every $\varepsilon > 0$ there exists $M > 0$ such that $\sup_{n} \bP\left( \frac{|\xi_n|}{a_n} > M \right) < \varepsilon.$ 
We use $\cP(\Omega)$ (resp. $\cP_2(\Omega)$) to denote all the probability measures on $\Omega$ (resp. all the probability measures on $\Omega$ with finite second moment). For $\mu\in\bR^d$ and a positive definite matrix $\Sigma$, $\cN(\mu,\Sigma)$ represents the Gaussian distribution with mean vector $\mu$ and covariance matrix $\Sigma$. Let $(\xi_n)_{n\geq 1}$ be a sequence of random variables taking values in a Polish space $\Xi$, and let $\xi$ be another random variable on $\Xi$. We say that $\xi_n$ converges weakly to $\xi$, and write $\xi_n \overset{w}{\longrightarrow} \xi$
if the sequence of laws $(\mathcal{L}(\xi_n))_{n\geq 1}$ converges weakly to the law $\mathcal{L}(\xi)$. Equivalently, for every bounded continuous function $f:\Xi\to\mathbb{R}$, $\lim_{n\to\infty} \mathbb{E}[f(\xi_n)] \;=\; \mathbb{E}[f(\xi)].$

\section{Approximation error}\label{section_approximation_error}
Recall that we suppose that $\nu_j\in\mathcal{P}_2(\mathbb{R}^d)$ is compactly supported on $\mathcal{X}_j\subset\mathbb{R}^d$ for $j\in[m]$. Let $d_0$ be the Euclidean distance in $\mathbb{R}^d$. We endow $\mathcal{P}_2(\mathcal{X}_j)$ with the topology induced by the Wasserstein metric $W_2$. It is known (see \cite{villani2021topics}) that in this compact support setting $W_2$ metrizes the weak topology (namely, $Q_n\to Q$ if and only if $\int fdQ_n\to\int fdQ$ for all $f\in \cC_b(\mathcal{X}_j)$).
Assume that $c$ is continuous.
Define the functionals
$C_\varepsilon : \mathcal{P}\left(\prod_{j=1}^m\cX_j\right)\to\mathbb{R}\cup\{\infty\}$,
\begin{equation*}
C_\varepsilon(\pi) = 
\begin{cases}
    \int c \, d\pi + \varepsilon\text{KL}(\pi|| \nu_1 \otimes \cdots \otimes \nu_m) & \text{if } \pi \in \Pi(\nu_1, \dots, \nu_m),\\
    \infty & \text{otherwise},
\end{cases}
\end{equation*}
and functional $C_{0}:\mathcal{P}\left(\prod_{j=1}^m\cX_j\right)\to\mathbb{R}\cup\{\infty\}$,
\begin{equation*}
C_0(\pi) = 
\begin{cases}
    \int c \, d\pi  & \text{if } \pi \in \Pi(\nu_1, \dots, \nu_m),\\
    \infty & \text{otherwise}.
\end{cases}
\end{equation*}
Obviously, $S_{\varepsilon}(\nu_1,\dots,\nu_m)=\inf_{\pi\in\Pi(\nu_1,\dots,\nu_m)}C_{\varepsilon}(\pi)$ and $S_{0}(\nu_1,\dots,\nu_m)=\inf_{\pi\in\Pi(\nu_1,\dots,\nu_m)}C_{0}(\pi)$.

\begin{definition}[$\Gamma-$convergence]
A sequence of functional $\mathcal{F}_\varepsilon$ is said to \emph{$\Gamma$-converge} to $\mathcal{F}$ if the following two conditions hold for every $x$,
\begin{enumerate}
    \item for any sequence $x_{\varepsilon}$ converging to $x$, $\mathcal{F}(x)\leq\liminf_{\varepsilon\to0}\mathcal{F}_{\varepsilon}(x_\varepsilon)$,
    \item there exists a sequence $x_{\varepsilon}$ converging to $x$, $\mathcal{F}(x)\geq\limsup_{\varepsilon\to0}\mathcal{F}_{\varepsilon}(x_\varepsilon)$.
\end{enumerate}
\end{definition}
\noindent For equivalent definitions and other more background information on $\Gamma$-convergence, we refer the reader to \cite{Brai}.

\begin{proposition}\label{Gamma1}
The sequence $(C_\varepsilon)_{\varepsilon>0}$ $\Gamma$-converges to $C_0$ w.r.t. the weak  topology. 
%\[
%\lim_{\varepsilon \to 0} S_{\varepsilon}(\nu_1,\cdots,\nu_m) = %S_{\alpha,0}(\nu_1,\cdots,\nu_m). 
%\]
\end{proposition}

\begin{proof}[Proof of Proposition~\ref{Gamma1}]
    The proof follows from Lemma \ref{liminf} and Lemma \ref{limsup}. 
\end{proof}

\begin{lemma}[liminf inequality]\label{liminf}
    Let $\pi_\varepsilon$ be an arbitrary sequence in $\mathcal{P}(\Pi_{j=1}^m\mathcal{X}_j)$ that converges to $\pi_0\in\mathcal{P}(\Pi_{j=1}^m\mathcal{X}_j)$. Then
    $$C_{0}(\pi_{0})\leq\liminf_{\varepsilon\to0^+}C_{\varepsilon}(\pi_\varepsilon).$$
\end{lemma}

\begin{lemma}[limsup inequality]\label{limsup}
    For every $\pi_0\in\mathcal{P}(\Pi_{j=1}^m\mathcal{X}_j)$,
    there exists a sequence $\pi_\varepsilon\in$ $\mathcal{P}(\Pi_{j=1}^m\mathcal{X}_j)$ converging to $\pi_0$, such that
    $$C_{0}(\pi_{0})\geq\limsup_{\varepsilon\to0^+}C_{\varepsilon}(\pi_\varepsilon).$$
\end{lemma}

\begin{proof}[Proof of Proposition \ref{approgamma_coupling} and Proposition ~\ref{approgamma}]
Since $\Pi(\nu_1,\dots,\nu_m)$ is tight via Lemma \ref{tight} and weakly closed, we have the compactness via Prohorov's theorem. Hence, $C_{\varepsilon}: \Pi(\nu_1,\dots,\nu_m)\to\mathbb{R}$ is equi-coercive: for all $t\in\mathbb{R}$, the set $\{\pi\in\Pi(\nu_1,\dots,\nu_m),C_\varepsilon(\pi)\leq t\}$ is included in some compact set $K_t$. Indeed, the set above is closed in the topology induced by the Wasserstein metric. 

Combining with the $\Gamma-$convergence in Proposition~\ref{Gamma1}, based on Theorem 2.10 in \cite{Brai}, we have that any cluster point of $\pi^*_\varepsilon$, the optimizer of $C_{\varepsilon}$, is a minimizer of $C_{0}$. All the results follow.
\end{proof}

\begin{proof}[Proof of Lemma \ref{liminf}]
Since $\Pi(\nu_1, \dots, \nu_m)$ is weakly closed, if $\pi_0\notin\Pi(\nu_1, \dots, \nu_m)$, we will eventually have $\pi_\varepsilon\notin\Pi(\nu_1, \dots, \nu_m)$ for $\varepsilon > 0$ sufficiently small. Thus $C_{\varepsilon}(\pi_\varepsilon)=C_{0}(\pi_0)=\infty$. So we assume $\pi_0\in\Pi(\nu_1, \dots, \nu_m)$ from now on.
By the definition of the topology induced by the Wasserstein metric here and the non-negativity of the entropy term, we obtain that
\begin{align*}
\liminf_{\varepsilon\to0^+}C_{\varepsilon}(\pi_\varepsilon)&=\liminf_{\varepsilon\to0^+}\int cd\pi_\varepsilon+\varepsilon\text{KL}(\pi_\varepsilon||\otimes_{j=1}^m\nu_j)\geq\lim_{\varepsilon\to0^+}\int cd\pi_\varepsilon=C_{0}(\pi_0).
\end{align*}
\end{proof}

%To prove Lemma~\ref{limsup}, we essentially follow the block approximation technique introduced in \cite{CDPS17}. This technique was first proposed to show the $\Gamma$-convergence of two marginal entropic optimal transport problems to two marginal optimal transport problems. Later, \cite{CPT23} used it to show the convergence rate of two marginal entropic optimal transport costs. Recently, \cite{NP23} adapted this method to multimarginal problems and gained the convergence rate of the corresponding entropic optimal transport costs. As mentioned in \cite{NP23}, the $\Gamma$-convergence of multimarginal entropic optimal transport costs is clear by the block approximation method. Here for completeness, we write down our proof details of Lemma~\ref{limsup}.

\noindent To prove Lemma~\ref{limsup}, we essentially follow the block approximation technique that was used to show the convergence rate of two marginal entropic optimal transport costs~\citep{CPT23}. Recently, \cite{NP23} adapted this method to multimarginal problems and gained the convergence rate of the corresponding entropic optimal transport costs. As mentioned in \citep{NP23}, the $\Gamma$-convergence of multimarginal entropic optimal transport costs is clear by the block approximation method. Here for completeness, we write down our proof details of Lemma~\ref{limsup}.

\begin{proof}[Proof of Lemma \ref{limsup}]
Take $\pi_0\in \cP (\prod_{j=1}^m\cX_j)$. If $\pi_0\notin \Pi(\nu_1,\dots,\nu_m)$, take $\pi_\varepsilon = \pi_0$ and the conclusion follows. Henceforth, assume $\pi_0\in \Pi(\nu_1,\dots,\nu_m)$.
For every $\varepsilon \in (0,1)$ and $i \in [m]$, consider a partition 
\[ \mathcal{X}_i = \bigsqcup_{1\leq n\leq L_i} A_i^n \]
of Borel sets such that $\operatorname{diam}(A_i^n) \leq \varepsilon $ for every $1\leq n \leq L_i$, with $L_i=L_i(\varepsilon)\lesssim(\frac{1}{\varepsilon})^d$ due to the compactness of $\mathcal{X}_i$. Also, set
\[
\nu_i^n := 
\begin{cases} 
\frac{\nu_i\lfloor_{A_i^n}}{\nu_i(A_i^n)}, & \text{if } \nu_i(A_i^n) > 0, \\ 
0, & \text{otherwise},
\end{cases}
\]
where $\mu \lfloor A$ means the restriction of the Borel measure $\mu$ to the Borel set $A$ defined by $\mu\lfloor A(E):=\mu(A\cap E)$ for every $E$. Then for every $m$-tuple $n = (n_1, \dots, n_m) \in \prod_{j=1}^m[L_j]$, define
\[
(\pi_0)^n := \pi_0(A_1^{n_1} \times \dots \times A_m^{n_m})\otimes_{k=1}^m\nu^{n_k}_{k}.
\]
Finally define 
\[
\pi_\varepsilon  := \sum_{n \in \prod_{j=1}^m[L_j]} (\pi_0)^n.
\]

\noindent By definition, $\pi_\varepsilon  \ll \otimes_{i=1}^m \nu_i$ and it is easily checked that ${e_i}_{\sharp}\pi_\varepsilon=\nu_i$, for $i\in[m]$. Based on the construction above, we also know that 
$$W_2(\pi_0,\pi_\varepsilon )\leq m\varepsilon,$$
equipping the $\Pi_{j=1}^m\mathcal{X}_j$ with the metric $D((z_1,\dots,z_m),(z'_1,\dots,z'_m)):=\sum_{j=1}d_0(z_j,z'_j)$. As a consequence, $\pi_\varepsilon $ converges to $\pi_0$ in the weak topology as $\varepsilon\to0^+$. 

\noindent Besides, $\pi_\varepsilon (A) = \pi_0(A)$ for every $A = \Pi_{i=1}^m A_i^{n_i}$ where $n \in \mathbb{N}^m$, and for $\otimes_{i=1}^m \nu_i$-almost every $(x_1, \dots, x_m) \in \otimes_{i=1}^m A_i^{n_i}$,
\[
\frac{d\pi_\varepsilon }{d\otimes_{i=1}^m\nu_i}(x_1, \dots, x_m) := 
\begin{cases}
\frac{\pi_0(A_1^{n_1} \times \dots \times A_m^{n_m})}{\nu_1(A_1^{n_1})\cdots\nu_m(A_m^{n_m})}, &  \text{if }  \nu_1(A_1^{n_1})  \cdots  \nu_m(A_m^{n_m}) > 0,\\
0, &  \text{otherwise}.
\end{cases}
\]
For the entropy term, we have

\begin{align*}
&\text{KL}(\pi_\varepsilon  | \otimes_{i=1}^{m} \nu_i) = \sum_{n \in \prod_{j=1}^m[L_j]} \int_{\prod_{i=1}^{m} A^{n_i}_i} \log \left( \frac{\pi_0(A^{n_1}_1 \times \dots \times A^{n_m}_m)}{\nu_1(A^{n_1}_1) \cdots \nu_m(A^{n_m}_m)} \right) d\pi_\varepsilon\\
= & \sum_{n \in \prod_{j=1}^m[L_j]} \pi_0(A^{n_1}_1 \times \dots \times A^{n_m}_m) \log \left( \frac{\pi_0(A^{n_1}_1 \times \dots \times A^{n_m}_m)}{\nu_1(A^{n_1}_1) \cdots \nu_m(A^{n_m}_m)} \right)\\
= & \sum_{n \in \prod_{j=1}^m[L_j]} \pi_0(A^{n_1}_1 \times \dots \times A^{n_m}_m) \log \left( \frac{\pi_0(A^{n_1}_1 \times \dots \times A^{n_m}_m)}{\nu_m(A^{n_m}_m)} \right)\\
& + \sum_{j=1}^{m-1} \sum_{n \in \prod_{j=1}^m[L_j]} 	\pi_0(A^{n_1}_1\times\dots\times A^{n_m}_m)\log( 1/\nu_j(A^{n_j}_j))\\
\leq & \sum_{n \in \prod_{j=1}^m[L_j]} \pi_0(A_1^{n_1} \times \dots \times A_m^{n_m}) 
\log\left(\frac{\pi_0(A_1^{n_1} \times \dots \times A_m^{n_m})}{\nu_m(A_m^{n_m})}\right)\\
& + \sum_{j=1}^{m-1} \sum_{n_j \in [L_j]} 
\pi_0\left(\prod_{i=1}^{j-1} \mathcal{X}_i \times A_j^{n_j}\times\prod_{i=j+1}^m \mathcal{X}_i\right) 
 \log(1/\nu_j(A_j^{n_j}))\\
\leq & \sum_{j=1}^{m-1} \sum_{n_j \in [L_j]} \nu_j(A_j^{n_j}) \log(1/\nu_j(A_j^{n_j})).
\end{align*}
where the last inequality comes from the inequality 
$\pi_0(A_1^{n_1} \times \dots \times A_m^{n_m}) 
\leq \nu_m(A_m^{n_m})$. 
By the concavity of $t\mapsto t\log(1/t)$, Jensen's inequality gives
$$\sum_{1\leq n_j\leq L_j}\nu_j(A^{n_j}_j)\log(1/\nu_j(A^{n_j}_j))\leq\log(L_j).$$
Thus, we get

\[
\text{KL}(\pi_\varepsilon  | \otimes_{i=1}^{m} \nu_i) \leq \sum_{j=1}^{m-1} \log L_j.
\]
All the analysis above yields that
\begin{align*}
    C_{\varepsilon}(\pi_\varepsilon )\leq\int c d\pi_\varepsilon +\varepsilon\sum_{j=1}^{m-1} \log L_j.
\end{align*}
By the compactness of $\mathcal{X}_j\subset\mathbb{R}^d$, we know that $\varepsilon\sum_{j=1}^{m-1}\log L_j\lesssim md\varepsilon\log(\frac{1}{\varepsilon})\to0$ as $\varepsilon\to0^+$. As a consequence, we obtain that
\begin{align*}
    \limsup_{\varepsilon\to0^+} C_{\varepsilon}(\pi_\varepsilon)\leq C_{0}(\pi_0).
\end{align*}
\end{proof}

\section{Concentration bound on the empirical gradient norm}\label{sec_Concentration_bound_on_the_empirical_gradient_norm}

This section is devoted to a high probability bound on the empirical gradient norm $\|\nabla\widehat{\Phi}_{\varepsilon}(\vf^*)\|_{\widehat{\mathscr{L}}_m}$, which is fundamental to our main results. Specifically, our goal is to prove the following Lemma \ref{gradnormp}. We present our results in a more general setting as in Lemma \ref{gradnormp} under the following assumptions, which might be of broader usage for multimarginal system analysis. Lemma \ref{gradnormp} offers new technical insights and improvements on the existing literature concerning the multimarginal regularized optimal transport problem.

\begin{assumption}\label{assumption_independent}
We have access to $m$ independent sequence $X^{(j)}_1,\dots,X^{(j)}_N\overset{\text{i.i.d.}}{\sim}\nu_j$ for $\nu_j\in\cP(\bR^d),\,j\in[m]$ of independent random variables.
\end{assumption}

\begin{assumption}\label{assumption_bounded_potential}
    Suppose $c:\prod_{i=1}^m \cX_i \to\bR$ is bounded.
\end{assumption}
    
\begin{assumption}\label{assumption_schrodinger_system}
    Suppose $p_\varepsilon^* = \exp \left( \frac{\sum_{j=1}^mf_j^* - c}{\varepsilon} \right):\prod_{i=1}^m \cX_i \to\bR\,$ for $\,f^*_j:\cX_j\to\bR, j\in[m]$
    satisfies for $j\in[m]$,  
    \begin{align}\label{appendix_system_condition}
        \int p_\varepsilon^*(x_1,\dots,x_m)d\nu_{-j}(x_1,\dots,x_m) = 1.
    \end{align}
\end{assumption}

\begin{comment}
\begin{lemma}[Concentration of empirical gradient norm: general case]\label{gradnormp}
Under Assumption (\ref{assumption_independent})-(\ref{assumption_schrodinger_system}), for all $t>0 $, we have with probability at least $1-2m(m-1)e^{-t}$,
\begin{align*}
\|\nabla\widehat{\Phi}(p^*)\|^2
& := \sum_{j=1}^m \frac{1}{N}\sum_{\ell_j=1}^N \left[\frac{1}{N^{m-1}}\sum_{1\leq\ell_{1},\dots,\ell_{j-1},\ell_{j+1},\dots,\ell_{m}\leq N} \left(1-p^*(X^{(1)}_{\ell_{1}},\dots,X^{(j)}_{\ell_{j}},\dots,X_{\ell_{m}}^{(m)})\right)  \right]^ 2\\
&\lesssim_{m} \frac{t}{N}.
\end{align*}
\end{lemma}
\end{comment}

\begin{lemma}[Concentration of empirical gradient norm]\label{gradnormp}
Under Assumption (\ref{assumption_independent})-(\ref{assumption_schrodinger_system}), for all $t>0 $, we have with probability at least $1-2m(m-1)e^{-t}$,
\[
\|\nabla\widehat{\Phi}_{\varepsilon}(\vf^*)\|^2_{\widehat{\mathscr{L}}_m} \lesssim_{m,\varepsilon} \frac{t}{N}.
\]
\end{lemma}
\noindent The proof of Lemma~\ref{gradnormp} is quite involved. So we divide it into several subsections as follows.

\subsection{Decomposition of the empirical gradient norm}
Recall that
\[
\|\nabla\widehat{\Phi}_{\varepsilon}(\vf^*)\|^2 = \sum_{j=1}^m \frac{1}{N}\sum_{\ell_j=1}^N \left[\frac{1}{N^{m-1}}\sum_{1\leq\ell_{1},\dots,\ell_{j-1},\ell_{j+1},\dots,\ell_{m}\leq N} \left(1-p_{\varepsilon}^*(X^{(1)}_{\ell_{1}},\dots,X^{(j)}_{\ell_{j}},\dots,X_{\ell_{m}}^{(m)})\right)  \right]^ 2.
\]
Denote $Z_{\ell_1\dots\ell_m}=1-p_{\varepsilon}^*(X^{(1)}_{\ell_{1}},\dots,X^{(j)}_{\ell_{j}},\dots,X_{\ell_{m}}^{(m)})$ and
$$K_j=\frac{1}{N}\sum_{\ell_j=1}^N \left[\frac{1}{N^{m-1}}\sum_{1\leq\ell_{1},\dots,\ell_{j-1},\ell_{j+1},\dots,\ell_{m}\leq N} Z_{\ell_1\dots\ell_m}  \right]^ 2.$$
Then, we can write
\begin{align}\label{sumK_j}
\|\nabla\widehat{\Phi}_{\varepsilon}(\vf^*)\|^2 = \sum_{j=1}^m\frac{1}{N}\sum_{\ell_j=1}^N\left[\frac{1}{N^{m-1}}\sum_{1\leq\ell_{1},\dots,\ell_{j-1},\ell_{j+1},\dots,\ell_{m}\leq N}Z_{\ell_1\dots\ell_m}  \right]^ 2 = \sum_{j=1}^mK_j.
\end{align}
In the following, we only detail the analysis of term $K_1$ for simplicity and readability. The terms $K_j$, $j=2,\dots,m$ can be dealt with in the same way. The next lemma is a crucial observation.

\begin{lemma}\label{claim}
Under Assumption (\ref{assumption_independent})-(\ref{assumption_schrodinger_system}), we can bound
\begin{align}\label{claim_equation}
    K_1 \leq & \frac{2}{N^{m+1}}\sum_{1\leq\ell_1,\dots,\ell_{m-1}\leq N}\left(\sum_{1\leq\ell_m\leq N}Z_{\ell_1\dots\ell_m}-\mathbb{E}_{m}Z_{\ell_1\dots\ell_m}\right)^2 \nonumber\\
    &+\sum_{r=3}^m\frac{2^{m-r+2}}{N^r}\sum_{1\leq\ell_1,\dots,\ell_{r-2}\leq N}\left(\sum_{1\leq \ell_{r-1}\leq N}\E_{r,\dots,m}Z_{\ell_1\dots\ell_m}-\mathbb{E}_{r-1,\dots,m}Z_{\ell_1\dots\ell_m}\right)^2,
\end{align}
where $\mathbb{E}_{q,...,m}$ means the expectation taken with respect to $\{X^{(j)}_i\}_{j\in\{q,\dots,m\},i\in[N]}$.
\end{lemma}

Lemma \ref{induction} is needed to prove the Lemma \ref{claim}. Specifically, the second term in (\ref{claim_equation}) comes from recursively applying Lemma \ref{induction} by induction.

\begin{lemma}[Induction lemma]\label{induction} 
Under Assumption (\ref{assumption_bounded_potential}), for any $s\in\{3,\dots,m\}$, we can bound
\begin{align}\label{induction_equation}
&\frac{1}{N}\sum_{1\leq\ell_1\leq N}\left(\frac{1}{N^{s-2}}\sum_{1\leq\ell_2,\dots,\ell_{s-1}\leq N}\E_{s,\dots,m}Z_{\ell_1\dots\ell_m}\right)^2 \nonumber\\
\leq& \frac{2}{N^s}\sum_{1\leq\ell_1,\dots,\ell_{s-2}\leq N}\left(\sum_{1\leq\ell_{s-1}\leq N} \E_{s,\dots,m}Z_{\ell_1\dots\ell_m}-\E_{s-1,\dots,m}Z_{\ell_1\dots\ell_m}\right)^2\\
&\quad +\frac{2}{N}\sum_{1\leq\ell_1\leq N}\left(\frac{1}{N^{s-3}}\sum_{1\leq\ell_2,\dots,\ell_{s-2}\leq N}\E_{s-1,\dots,m}Z_{\ell_1\dots\ell_m}\right)^2. \nonumber
\end{align}
\end{lemma}

\begin{proof}[Proof of Lemma \ref{induction}]
\begin{align*}
&\frac{1}{N}\sum_{\ell_1=1}^N\Big(\frac{1}{N^{s-2}}\sum_{1\leq\ell_2,\dots,\ell_{s-1}\leq N}\E_{s,\dots,m}Z_{\ell_1\dots\ell_m}\Big)^2 \\
=&\frac{1}{N} \sum_{\ell_1=1}^N \Big(\frac{1}{N^{s-2}}\sum_{1\leq\ell_2,\dots,\ell_{s-1}\leq N}(\E_{s,\dots,m}Z_{\ell_1\dots\ell_m}-\E_{s-1,\dots,m}Z_{\ell_1\dots\ell_m}) \\
&\qquad \qquad +\frac{1}{N^{s-3}}\sum_{1\leq\ell_2,\dots,\ell_{s-2}\leq N}\E_{s-1,\dots,m}Z_{\ell_1\dots\ell_m}\Big)^2\\
\leq&\frac{2}{N} \sum_{\ell_1=1}^N \Big(\frac{1}{N^{s-2}}\sum_{1\leq\ell_2,\dots,\ell_{s-1}\leq N}(\E_{s,\dots,m}Z_{\ell_1\dots\ell_m}-\E_{s-1,\dots,m}Z_{\ell_1\dots\ell_m})\Big)^2\\
&\qquad \qquad +\frac{2}{N} \sum_{\ell_1=1}^N\Big(\frac{1}{N^{s-3}}\sum_{1\leq\ell_2,\dots,\ell_{s-2}\leq N}\E_{s-1,\dots,m}Z_{\ell_1\dots\ell_m}\Big)^2\\
=&\frac{2}{N^3} \sum_{\ell_1=1}^N \Big(\frac{1}{N^{s-3}}\sum_{1\leq\ell_2,\dots,\ell_{s-2}\leq N}\sum_{1\leq\ell_{s-1}\leq N}(\E_{s,\dots,m}Z_{\ell_1\dots\ell_m}-\E_{s-1,\dots,m}Z_{\ell_1\dots\ell_m})\Big)^2\\
&\qquad \qquad +\frac{2}{N} \sum_{\ell_1=1}^N\Big(\frac{1}{N^{s-3}}\sum_{1\leq\ell_2,\dots,\ell_{s-2}\leq N}\E_{s-1,\dots,m}Z_{\ell_1\dots\ell_m}\Big)^2\\
\leq&\frac{2}{N^s}\sum_{1\leq\ell_1,\dots,\ell_{s-2}\leq N}\Big(\sum_{1\leq\ell_{s-1}\leq N} \E_{s,\dots,m}Z_{\ell_1\dots\ell_m}-\E_{s-1,\dots,m}Z_{\ell_1\dots\ell_m}\Big)^2\\
&\qquad \qquad +\frac{2}{N}\sum_{\ell_1=1}^N\Big(\frac{1}{N^{s-3}}\sum_{1\leq\ell_2,\dots,\ell_{s-2}\leq N}\E_{s-1,\dots,m}Z_{\ell_1\dots\ell_m}\Big)^2, \nonumber
\end{align*}
where we use basic inequality for the first inequality and Jensen's inequality for the second inequality.
\end{proof}

Now we are ready to prove Lemma \ref{claim}.
\begin{proof}[Proof of Lemma \ref{claim}]
Observe that 
\begin{align}\label{claimpf1}
K_1=\frac{1}{N}\sum_{\ell_1=1}^N&\left[\frac{1}{N^{m-1}}\sum_{1\leq\ell_{2},\dots,\ell_{m}\leq N} \left(Z_{\ell_1\dots\ell_m}-\mathbb{E}_{m}Z_{\ell_1\dots\ell_m}\right) + 
\frac{1}{N^{m-1}}\sum_{1\leq\ell_{2},\dots,\ell_{m}\leq N} \mathbb{E}_{m}Z_{\ell_1\dots\ell_m} \right]^ 2,\nonumber \\
\end{align}
so we get
\begin{align}\label{claimpf2}
K_1\leq&\frac{2}{N}\sum_{\ell_1=1}^N \left[\frac{1}{N^{m-1}}\sum_{1\leq\ell_{2},\dots,\ell_{m-1}\leq N} \sum_{1\leq\ell_{m}\leq N} \left(Z_{\ell_1\dots\ell_m}-\mathbb{E}_{m}Z_{\ell_1\dots\ell_m}\right)
\right]^2    \nonumber \\
&+\frac{2}{N}\sum_{\ell_1=1}^N \left[\frac{1}{N^{m-1}}\sum_{1\leq\ell_{2},\dots,\ell_{m-1}\leq N} \sum_{1\leq\ell_{m}\leq N} \mathbb{E}_{m}Z_{\ell_1\dots\ell_m}
\right]^2     \nonumber\\
=&\frac{2}{N^{3}}\sum_{\ell_1=1}^N \left[\frac{1}{N^{m-2}}\sum_{1\leq\ell_{2},\dots,\ell_{m-1}\leq N} \sum_{1\leq\ell_{m}\leq N} \left(Z_{\ell_1\dots\ell_m}-\mathbb{E}_{m}Z_{\ell_1\dots\ell_m}\right)
\right]^2    \nonumber\\
&+\frac{2}{N}\sum_{\ell_1=1}^N \left[\frac{1}{N^{m-1}}\sum_{1\leq\ell_{2},\dots,\ell_{m-1}\leq N} \sum_{1\leq\ell_{m}\leq N} \mathbb{E}_{m}Z_{\ell_1\dots\ell_m}
\right]^2   \nonumber\\
\leq&\frac{2}{N^{m+1}} \sum_{1\leq\ell_1,\ell_{2},\dots,\ell_{m-1}\leq N} \left[\sum_{1\leq\ell_{m}\leq N} \left(Z_{\ell_1\dots\ell_m}-\mathbb{E}_{m}Z_{\ell_1\dots\ell_m}\right)
\right]^2   \nonumber \\
&+\frac{2}{N}\sum_{\ell_1=1}^N \left[\frac{1}{N^{m-2}}\sum_{1\leq\ell_{2},\dots,\ell_{m-1}\leq N} \mathbb{E}_{m}Z_{\ell_1\dots\ell_m}
\right]^2,   \nonumber \\
\end{align}
where we use Jensen's inequality for the second inequality. Moreover, the marginal feasibility constrain (\ref{appendix_system_condition}) indicates that 
\begin{equation}\label{E2m=0}
\mathbb{E}_{2,\dots,m}Z_{\ell_1,\dots,\ell_m}=0.
\end{equation}
%as well as\begin{equation}\label{E1m=0}\mathbb{E}_{1,\dots,m}Z_{\ell_1,\dots,\ell_m}=0.\end{equation}
As a consequence of equation (\ref{E2m=0}) and using Lemma \ref{induction} recursively, we can deduce that 
\begin{align}\label{claimpf3}
    &\frac{1}{N}\sum_{\ell_1=1}^N \left[\frac{1}{N^{m-2}}\sum_{1\leq\ell_{2},\dots,\ell_{m-1}\leq N}  \mathbb{E}_{m}Z_{\ell_1\dots\ell_m}\right]^2
    \nonumber \\ 
    &\leq\sum_{s=3}^m\frac{2^{m-s+1}}{N^s}\sum_{1\leq\ell_1,\dots,\ell_{s-2}\leq N}\left(\sum_{1\leq\ell_{s-1}\leq N} E_{s,\dots,m}Z_{\ell_1\dots\ell_m}-E_{s-1,\dots,m}Z_{\ell_1\dots\ell_m}\right)^2.
\end{align}
Combining (\ref{claimpf1}), (\ref{claimpf2}) and (\ref{claimpf3}), the proof of the lemma is done.
\end{proof}

\subsection{Concentration of the empirical gradient norm }

Recall that as proved in Lemma \ref{claim},
\begin{align*}
    K_1\leq & \frac{2}{N^{m+1}}\sum_{1\leq\ell_1,\dots,\ell_{m-1}\leq N}\left(\sum_{1\leq\ell_m\leq N}Z_{\ell_1\dots\ell_m}-\mathbb{E}_{m}Z_{\ell_1\dots\ell_m}\right)^2 \nonumber\\
    &+\sum_{r=3}^m\frac{2^{m-r+2}}{N^r}\sum_{1\leq\ell_1,\dots,\ell_{r-2}\leq N}\left(\sum_{1\leq \ell_{r-1}\leq N}E_{r,\dots,m}Z_{\ell_1\dots\ell_m}-\mathbb{E}_{r-1,\dots,m}Z_{\ell_1\dots\ell_m}\right)^2 \\
    =:&K_{1,m+1}+\sum_{r=3}^{m}K_{1,r}.
\end{align*}
Motivated by this fact, we are going to derive high probability bounds for $K_{1,s}$ separately in the following lemmas, for $s\in\{3,\dots,m+1\}$. These bounds together will establish the concentration for $K_1$ and consequently the empirical gradient norm $\|\nabla\widehat{\Phi}_{\varepsilon}(\vf^*)\|_{\widehat{\mathscr{L}}_m}^2$. To lighten the notation, we denote $\mathscr{X}_k=\{X^{(k)}_1,\dots,X^{(k)}_N\}$ for any $k\in[m]$ in the following proofs

\begin{lemma}[Concentration of $K_{1,m+1}$]\label{K1m+1}
Under Assumption (\ref{assumption_independent})-(\ref{assumption_schrodinger_system}), for any $t>0$, with probability more than $1-2e^{-t}$,
\begin{equation}
    K_{1,m+1}\lesssim_\varepsilon\frac{t}{N}.
\end{equation}
\end{lemma}

\begin{proof}
For $1\leq\ell_m\leq N$, let 
$$A_{\ell_m}=(Z_{\ell_1\dots\ell_m}-\mathbb{E}_{m}Z_{\ell_1\dots\ell_m})_{1\leq\ell_1,\dots,\ell_{m-1}\leq N}\in\mathbb{R}^{N^{m-1}}.$$
Conditional on $\prod_{j=1}^{m-1}\mathscr{X}_j$,
$A_1,\dots,A_N$ are independent due to assumption \ref{assumption_independent}. Also, $\|A_{\ell_m}\|_2\lesssim_\varepsilon N^{\frac{m-1}{2}}$ because of assumption \ref{assumption_bounded_potential} and $\mathbb{E}_{m}A_{\ell_m}=0$ owing to assumption \ref{assumption_schrodinger_system}.  Note that
$$\frac{1}{N^{m+1}}\left\|\sum_{\ell_m=1}^NA_{\ell_m}\right\|^2_2=\frac{1}{N^{m+1}}\sum_{1\leq\ell_1,\ell_{2},\dots,\ell_{m-1}\leq N} \left[\sum_{1\leq\ell_{m}\leq N} \left(Z_{\ell_1\dots\ell_m}-\mathbb{E}_{m}Z_{\ell_1\dots\ell_m}\right)
\right]^2=\frac{1}{2}K_{1,m+1},$$ using Lemma~\ref{lem:hoeffding_hilbert}, and we have for any $u>0$,
\begin{equation*}   \mathbb{P}\left(\left\|\sum_{\ell_m=1}^NA_{\ell_m}\right\|\gtrsim_\varepsilon u\right)\leq2\exp\left(-\frac{u^2}{ N\cdot N^{m-1}}\right),
\end{equation*}
\begin{comment}
i.e.,\begin{equation}
\mathbb{P}\left(\left\|\sum_{\ell_m=1}^NA_{\ell_m}\right\|^2\geq u^2\right)\lesssim\exp\left(-\frac{u^2}{ C N^{m}}\right)
\end{equation}
i.e.,\begin{equation}
\mathbb{P}\left(\left\|\frac{1}{N^{\frac{m+1}{2}}}\sum_{\ell_m=1}^NA_{\ell_m}\right\|^2\geq u^2\right)\lesssim\exp\left(-Nu^2/C\right)
\end{equation}
\end{comment}
i.e.,\begin{equation}
\mathbb{P}\left(\left\|\frac{1}{N^{\frac{m+1}{2}}}\sum_{\ell_m=1}^NA_{\ell_m}\right\|^2\gtrsim_\varepsilon \frac{t}{N}\right)\leq 2e^{-t}.
\end{equation}
Namely, for any $t>0$,
\begin{equation}
    \mathbb{P}\left(K_{1,m+1}\lesssim_\varepsilon\frac{t}{N}\right)\geq1-2e^{-t}.
\end{equation}
\end{proof}

\begin{lemma}[Concentration of $K_{1,r}$]\label{K1r}
Under Assumption (\ref{assumption_independent})-(\ref{assumption_schrodinger_system}), for any $t>0$, $r\in\{3,\dots,m\}$, with probability exceeding $1-2e^{-t}$,
\begin{equation}
    K_{1,r}\lesssim_\varepsilon\frac{t}{N}.
\end{equation}
\end{lemma}

\begin{proof}
For $1\leq\ell_{r-1}\leq N$, let 
$$B_{\ell_{r-1}}=(\mathbb{E}_{r,\dots,m}Z_{\ell_1\dots\ell_m}-\mathbb{E}_{r-1,\dots,m}Z_{\ell_1\dots\ell_m})_{1\leq\ell_1,\dots,\ell_{r-2}\leq N}\in\mathbb{R}^{N^{r-2}}.$$
Conditional on $\prod_{j=1}^{r-2}\mathscr{X}_j$, 
$B_1,\dots,B_N$ are independent on account of assumption \ref{assumption_independent}. Also, $\|B_{\ell_{r-1}}\|_2\lesssim_\varepsilon N^{\frac{r-2}{2}}$ as a result of assumption \ref{assumption_bounded_potential} and $\mathbb{E}_{r-1,\dots,m}B_{\ell_{r-1}}=0$ due to assumption \ref{assumption_schrodinger_system}. It is observed that
\begin{align*}
\frac{1}{N^{r}}\left\|\sum_{\ell_{r-1}=1}^NB_{\ell_{r-1}}\right\|^2_2
&= \frac{1}{N^{r}}\sum_{1\leq\ell_1,\ell_{2},\dots,\ell_{r-2}\leq N} \left[\sum_{1\leq\ell_{r-1}\leq N} \left(\mathbb{E}_{r,\dots,m}Z_{\ell_1\dots\ell_m}-\mathbb{E}_{r-1,\dots,m}Z_{\ell_1\dots\ell_m}\right)
\right]^2\\
&=\frac{1}{2^{m-r+2}}K_{1,r}.
\end{align*}
Via Lemma~\ref{lem:hoeffding_hilbert}, we obtain

\begin{equation*}   \mathbb{P}\left(\left\|\sum_{\ell_{r-1}=1}^NB_{\ell_{r-1}}\right\|\gtrsim_\varepsilon u\right)\leq2\exp\left(-\frac{u^2}{N\cdot N^{r-2}}\right),
\end{equation*}
i.e.,\begin{equation}
\mathbb{P}\left(\left\|\frac{1}{N^{\frac{r}{2}}}\sum_{\ell_{r-1}=1}^NB_{\ell_{r-1}}\right\|^2\gtrsim_\varepsilon \frac{t}{N}\right)\leq 2e^{-t}.
\end{equation}
Namely, for any $t>0$,
\begin{equation}
    \mathbb{P}\left(K_{1,r}\lesssim_\varepsilon\frac{t}{N}\right)\geq1-2e^{-t}.
\end{equation}

\end{proof}

\subsection{Concluding Lemma~\ref{gradnormp}}

By (\ref{sumK_j}), we know that 
$$\|\nabla\widehat{\Phi}_{\varepsilon}(\vf^*)\|^2= \sum_{j=1}^mK_j.$$
A combination of Lemma \ref{claim}, Lemma \ref{K1m+1} and Lemma \ref{K1r} gives Lemma~\ref{gradnormp}.

\section{Concentration of empirical potentials and joint optimal coupling density}\label{sec_concentration_of_empirical_potentials_and_joint_optimal_coupling_density}

We first derive a high probability bound for the empirical multimarginal distributions.

\begin{lemma}[Concentration of empirical multimarginal distributions]
\label{Ustat}
Let $\phi \in L^{\infty}(\otimes_{j=1}^m \nu_j)$ be such that $(\otimes_{j=1}^m \nu_j)(\phi) = 0$. Then, for all $t > 0$, we have with probability at least $1-2e^{-t}$, %over $\mathcal{X}_i$, $i=1,\dots,m$,
\[
|(\otimes_{j=1}^m \hat{\nu}^N_j)(\phi)| \leq \|\phi\|_\infty\sqrt{\frac{2t}{N}}.
\]
\end{lemma}

\begin{proof}
We only detail the proof of the inequality for one direction. The other tail follows analogously. For any $\lambda>0$, Chernoff's bound gives
\[
\mathbb{P}_{\cX}[(\otimes_{j=1}^m \hat{\nu}^N_j)(\phi) > t] \leq e^{-\lambda t}\mathbb{E}_{\mathcal{X}}\left[\exp\left\{\lambda(\otimes_{j=1}^m \hat{\nu}^N_j)(\phi)\right\}\right].
\]
Observe that $(\otimes_{j=1}^m \hat{\nu}^N_j)(\phi) $ could be expressed as
\[
(\otimes_{j=1}^m \hat{\nu}^N_j)(\phi) = \frac{1}{(N!)^{m-1}} \sum_{\sigma_2,\dots,\sigma_m\in\Sigma_N} \frac{1}{N}\sum_{k=1}^N \phi\left(X^{(1)}_k, X^{(2)}_{\sigma_2(k)},\dots,X^{(m)}_{\sigma_m(k)}\right),
\]
where $\Sigma_N$ is the set of permutations on $N$ elements. Combining this fact with Jensen's inequality gives the bound
\begin{align*}
\mathbb{P}_{\mathcal{X}}[(\otimes_{j=1}^m \hat{\nu}^N_j)(\phi) > t] &\leq e^{-\lambda t}\mathbb{E}_{\mathcal{X}}\left[\exp\left\{ \frac{\lambda}{(N!)^{m-1}} \sum_{\sigma_2,\dots,\sigma_m\in\Sigma_N} \frac{1}{N}\sum_{k=1}^N \phi(X^{(1)}_k, X^{(2)}_{\sigma_2(k)},\dots,X^{(m)}_{\sigma_m(k)})\right\}\right]\\
&\leq e^{-\lambda t}\mathbb{E}_{\mathcal{X}}\left[\frac{1}{(N!)^{m-1}} \sum_{\sigma_2,\dots,\sigma_m\in\Sigma_N} \exp\left\{\frac{\lambda}{N}\sum_{k=1}^N \phi(X^{(1)}_k, X^{(2)}_{\sigma_2(k)},\dots,X^{(m)}_{\sigma_m(k)})\right\}\right].
\end{align*}
Note that for any group of fixed permutations $(\sigma_2,\dots,\sigma_m)$, the joint law of $(X^{(1)}_k, X^{(2)}_{\sigma_2(k)},\dots,X^{(m)}_{\sigma_m(k)})_{k=1}^N$ is identical to that of $(\xi_1,\ldots,\xi_N)$ where $\xi_k \sim \otimes_{j=1}^m \nu_j$ are independent and identically distributed. Let $\Xi$ denote such an i.i.d. sample $(\xi_1,\ldots,\xi_N)$. Thus it holds that
\begin{align*}
&\mathbb{E}_{\mathcal{X}}\left[\frac{1}{(N!)^{m-1}} \sum_{\sigma_2,\dots,\sigma_m\in\Sigma_N} \exp\left\{\frac{\lambda}{N}\sum_{k=1}^N \phi(X^{(1)}_k, X^{(2)}_{\sigma_2(k)},\dots,X^{(m)}_{\sigma_m(k)})\right\}\right]\\
=& \frac{1}{(N!)^{m-1}} \sum_{\sigma_2,\dots,\sigma_m\in\Sigma_N} \mathbb{E}_{\mathcal{X}}\left[\exp\left\{\frac{\lambda}{N}\sum_{k=1}^N \phi(X^{(1)}_k, X^{(2)}_{\sigma_2(k)},\dots,X^{(m)}_{\sigma_m(k)})\right\}\right]\\
=&\mathbb{E}_{\Xi}\left[\exp\left\{\frac{\lambda}{N}\sum_{k=1}^N \phi(\xi_k)\right\}\right].
\end{align*}
Applying Hoeffding's Lemma and optimizing over $\lambda>0$ yields
\[
\mathbb{P}_{\mathcal{X}}[(\otimes_{j=1}^m \hat{\nu}^N_j)(\phi) > t]  \leq \exp\left\{-\frac{Nt^2}{2\|\phi\|_\infty^2}\right\}.
\]
\end{proof}

\noindent Now we provide the precise formulation of the bounded potential property as well as the optimization geometry in the empirical setting, for better clarity. The arguments are almost the same as the population version. Similarly as in the population version, we can define
$$\widehat{\cS}_L:=\Big\{\vf\in\prod_{j=1}^m L^{\infty}(\hat{\nu}^N_j) : 
\|\sum_{j=1}^mf_j\|_{L^{\infty}(\otimes_{k=1}^m\hat{\nu}^N_k)}\leq L,\,\,\hat{\nu}^N_i(f_i)=0\,\,\text{for}\,\,i\in[m-1]\Big\},$$
and compute the gradient 
$\nabla\hat{\Phi}_{\varepsilon}:\widehat{\mathscr{L}}_m\to\widehat{\mathscr{L}}_m'$ as
\begin{align*}
\left\langle \nabla \hat{\Phi}_{\varepsilon} (\vf), \vg \right\rangle_{\widehat{\mathscr{L}}_m} 
% &:=\sum_{j=1}^{m}\int g_j d\nu_j-\int\left(\sum_{j=1}^m g_j\right)\exp\left({\frac{\sum_{j=1}^mf_j-%c}{\varepsilon}}\right)d\left(\otimes_{k=1}^m\nu_k\right)\\
 =\int\Big[\Big(\sum_{j=1}^m g_j\Big)\Big(1-\exp\Big({\frac{\sum_{i=1}^mf_i-c}{\varepsilon}}\Big)\Big)\Big]d\left(\otimes_{k=1}^m\hat{\nu}^N_k\right),
\end{align*}
where $\widehat{\mathscr{L}}_m'$ denotes the dual space of $\widehat{\mathscr{L}}_m$. The optimization geometry w.r.t. $\widehat{\mathscr{L}}_m$ induces a norm of the gradient. For $\vg=(g_1,\dots, g_m)\in \widehat{\mathscr{L}}_m$, we define $\|\vg\|_{\widehat{\mathscr{L}}_m}:=( \sum_{j=1}^m\int g_j^2d\hat{\nu}^N_j )^{1/2}.$
The norm of empirical dual objective gradient could be similarly calculated as
\begin{equation}\label{empirical_gradient_dual_norm}
    \left\|\nabla\hat{\Phi}_{\varepsilon}(\vf)\right\|_{\widehat{\mathscr{L}}_m}^2=\sum_{j=1}^m\int\Big[\int1-\exp\Big({\frac{\sum_{i=1}^mf_i-c}{\varepsilon}}\Big)d\hat{\nu}^N_{-j}(x_{-j})\Big]^2d\hat{\nu}^N_j.
\end{equation}

\begin{proposition}[Bounded empirical dual potentials]\label{empirical_BDP}
The optimal empirical dual potential $\hat{\mvf}^*$ satisfies that 
\[
\max_{j \in [m]} \|\hat{f}^*_j\|_{\infty} \leq (3m^2-5m+3)\|c\|_{\infty} .
\]
where $\|\hat{f}^*_j\|_{\infty} := \|\hat{f}^*_j\|_{L^{\infty}(\nu_j)}$ and $\|c\|_{\infty} := \|c\|_{L^{\infty}(\otimes_{i=1}^m\nu_i)}$.
\end{proposition}

\begin{proof}
    The same argument as in the proof of Proposition \ref{BDP} implies that
\[
    \max_{j \in [m]} \|\hat{f}^*_j\|_{L^\infty(\hat{\nu}_j^N)} \leq 3\|c\|_{\infty} .
\]
At the same time, note that the canonical extension defines 
that for all $x_j\in\cX_j,\,j\in[m]$, 
\begin{align*}
\hat{f}_j^*(x_j) &= - \varepsilon \log \left\{ \int \exp\left(\frac{\sum_{i \neq j} \hat{f}_i^*(x_i)-c(x_1, \dots, x_m)}{\varepsilon}\right)d\hat{\nu}^N_{-j}(x_{-j}) \right\} \\
&\geq -\|c\|_\infty - \varepsilon \log \left\{ \int \exp\left(\frac{\sum_{i \neq j} \hat{f}_i^*(x_i)}{\varepsilon}\right)d\hat{\nu}^N_{-j}(x_{-j}) \right\}\\
&\geq -(3m-2) \|c\|_\infty.
\end{align*}
Also notice that 
\begin{align*}
\hat{f}_j^*(x_j) &= - \varepsilon \log \left\{ \int \exp\left(\frac{\sum_{i \neq j} \hat{f}_i^*(x_i)-c(x_1, \dots, x_m)}{\varepsilon}\right)d\hat{\nu}^N_{-j}(x_{-j}) \right\} \\
&\leq \|c\|_\infty - \varepsilon \log \left\{ \int \exp\left(\frac{\sum_{i \neq j} \hat{f}_i^*(x_i)}{\varepsilon}\right)d\hat{\nu}^N_{-j}(x_{-j}) \right\}\\
&\leq (3m^2-5m+3) \,\|c\|_\infty.
\end{align*}
\end{proof}

\begin{lemma}[Strong concavity of empirical dual objective]\label{sconc_empirical}
The empirical dual objective $\widehat{\Phi}_{\varepsilon}(\cdot)$ is $\beta$-concave w.r.t the norm $\|\cdot\|_{\widehat{\mathscr{L}}_m}$ on $\hat{\mathcal{S}}_{L}$ with $\beta=\frac{1}{\varepsilon}\exp\left(-\frac{L+\|c\|_{\infty}}{\varepsilon}\right)$, i.e., for all $\vf,\vg\in\hat{\mathcal{S}}_L$,
\begin{equation}\label{esconci}
\widehat{\Phi}_{\varepsilon}(\vf)-\widehat{\Phi}_{\varepsilon}(\vg)\geq\left\langle \nabla \widehat{\Phi}_{\varepsilon} (\vf),(\vf-\vg)\right\rangle_{\widehat{\mathscr{L}}_{m}}+\frac{\beta}{2}\|\vf-\vg\|_{\widehat{\mathscr{L}}_m}^2.
\end{equation}
\end{lemma}

\begin{proof}
    The proof is the same as the proof of Proposition \ref{sconc}.
\end{proof}

\noindent The concentration of empirical potentials and joint density will be derived next. First, recall that we use $\mvf^*\in[\mvf^*]$ to denote the optimal potential associated with $(\nu_1,\dots,\nu_m)$ satisfying $\nu_k(f^*_k) = 0$ for $k\in[m-1]$
and $\hat{\mvf}^*\in[\hat{\mvf}^*]$ to denote the optimal potential associated with $(\hat{\nu}^N_1,\dots,\hat{\nu}^N_m)$ satisfying $\hat{\nu}^N_k(f^*_k) = 0$ for $k\in[m-1]$. Next, we define
$$\bar{\vf}^*=(\bar{f}^*_1,\dots,\bar{f}^*_m):=\left(f^*_1-\hat{\nu}_1^N(f^*_1),\dots,f^*_{m-1}-\hat{\nu}_{m-1}^N(f^*_{m-1}),f^*_m+\sum_{k=1}^{m-1}\hat{\nu}_k^N(f^*_k)\right).$$
In view of the bounded dual potential proposition, we know that almost surely, $\bar{\vf}^*\in\hat{\mathcal{S}}_{(2m+1)\|c\|_\infty}.$ The following lemmas concerning $\bar{\vf}^*$ will be useful in the upcoming analysis.

\begin{proposition}\label{13}
    For $\mvx=(x_1,\dots,x_m)$, we have
    \begin{equation}\label{13eq}
        |\hat{p}_{\varepsilon}(\mvx)-p^*_{\varepsilon}(\mvx)|\lesssim_{m,\varepsilon}\sum_{j=1}^m|\bar{f}_j^*(x_j)-\hat{f}^*_j(x_j)|.
    \end{equation}
\end{proposition}

\begin{proof}By the Lipschitzness of $h(\cdot)=\exp(\cdot)$ on a bounded domain, we have
\begin{align*}
       |\hat{p}_{\varepsilon}(\mvx)-p^*_{\varepsilon}(\mvx)|&=\left|\exp\left(\frac{\sum_{j=1}^m\hat{f}^*_j(x_j)-c(\mvx)}{\varepsilon}\right)-\exp\left(\frac{\sum_{j=1}^mf^*_j(x_j)-c(\mvx)}{\varepsilon}\right)\right|\\
        &=\left|\exp\left(\frac{\sum_{j=1}^m\hat{f}^*_j(x_j)-c(\mvx)}{\varepsilon}\right)-\exp\left(\frac{\sum_{j=1}^m\bar{f}^*_j(x_j)-c(\mvx)}{\varepsilon}\right)\right|\\
        &\lesssim_{m,\varepsilon}\left|\sum_{j=1}^m\hat{f}^*_j(x_j)-\sum_{j=1}^m\bar{f}^*_j(x_j)\right|\\
        &\leq\sum_{j=1}^m\left|\hat{f}^*_j(x_j)-\bar{f}^*_j(x_j)\right|.
\end{align*}
\end{proof}

\begin{lemma}[Bound on empirical potential functions]\label{barf}
$$\|\hat{\vf}^*-\bar{\vf}^*\|^2_{\widehat{\mathscr{L}}_m}\leq\varepsilon^2\exp\left(\frac{(4m+4)\|c\|_\infty}{\varepsilon}\right)\|\nabla\widehat{\Phi}_{\varepsilon}(\vf^*)\|^2_{\widehat{\mathscr{L}}_m}.$$
\end{lemma}

\begin{proof}
Due to the strong concavity proved in Lemma \ref{sconc_empirical}, we have
$$\widehat{\Phi}_{\varepsilon}(\hat{\vf}^*)-\widehat{\Phi}_{\varepsilon}(\bar{\vf}^*)\geq\frac{1}{2\varepsilon}\exp\left({-\frac{(2m+2)\|c\|_\infty}{\varepsilon}}\right)\|\hat{\vf}^*-\bar{\vf}^*\|^2_{\widehat{\mathscr{L}}_m}.$$
Also, the Polyak-\L ojasiewicz inequality stated in Lemma \ref{PLineq} indicates that
$$\widehat{\Phi}_{\varepsilon}(\hat{\vf}^*)-\widehat{\Phi}_{\varepsilon}(\bar{\vf}^*)\leq\frac{\varepsilon\exp(\frac{(2m+2)\|c\|_\infty}{\varepsilon})}{2}\|\nabla\widehat{\Phi}_{\varepsilon}(\bar{\vf}^*)\|^2_{\widehat{\mathscr{L}}_m}=\frac{\varepsilon\exp(\frac{(2m+2)\|c\|_\infty}{\varepsilon})}{2}\|\nabla\widehat{\Phi}_{\varepsilon}(\vf^*)\|^2_{\widehat{\mathscr{L}}_m}.$$
The observation above implies the lemma.
\end{proof}

\begin{proposition}[Concentration of empirical joint density of optimal coupling]\label{diffp2}
    For $t>0$, with probability more than $1-2m(m-1)e^{-t}$, 
    $$\|p^*_{\varepsilon}(\mvx)-\hat{p}_{\varepsilon}(\mvx)\|^2_{L^2(\otimes_{j=1}^m\hat{\nu}_j^N)}\lesssim_{m,\varepsilon}\frac{t}{N}.$$
\end{proposition}
\begin{proof}
    By Proposition \ref{13}, it holds that
        $$|p^*_{\varepsilon}(\mvx)-\hat{p}_{\varepsilon}(\mvx)|\lesssim_\varepsilon\sum_{j=1}^m|\bar{f}_j^*(x_j)-\hat{f}^*_j(x_j)|.$$
Lemma \ref{gradnormp} and Lemma \ref{barf} yield that with probability exceeding $1-2m(m-1)e^{-t}$,
    $$\|\hat{\vf}^*-\bar{\vf}^*\|^2_{\widehat{\mathscr{L}}_m}\lesssim_\varepsilon\|\nabla\widehat{\Phi}_{ \varepsilon}(\vf^*)\|^2_{\widehat{\mathscr{L}}_m}\lesssim_{m,\varepsilon}\frac{t}{N}.$$
Thus, with probability at least $1-2m(m-1)e^{-t}$,
\begin{align*}
\|p^*_{\varepsilon}(\mvx)-\hat{p}_{\varepsilon}(\mvx)\|^2_{L^2(\otimes_{j=1}^m\hat{\nu}_j^N)}&\lesssim_{\varepsilon}  (\otimes_{j=1}^m\hat{\nu}_j^N)\left(\sum_{j=1}^m|\bar{f}_j^*(x_j)-\hat{f}^*_j(x_j)|^2\right)\\
&=\|\hat{\vf}^*-\bar{\vf}^*\|^2_{\widehat{\mathscr{L}}_m}\lesssim_{m,\varepsilon}\frac{t}{N}.
\end{align*}
\end{proof}

\noindent The lemmas above make us ready for the proof of theorem \ref{boundtest}.
\begin{proof}[Proof of Theorem \ref{boundtest}] 
%We only prove inequality ~\eqref{eqn:rate_bounded_test} since the coupling error bound (i.e., first inequality) can be derived in the same manner. 
Put $\mvx = (x_1, \dots, x_m)$. Note that
\begin{align}\label{boundtest_decomp}
     |(\pi_\varepsilon^*-\hat{\pi}^N_\varepsilon)(g)|
     =
     &\left|\int g\,p^*_{\varepsilon} d\otimes_{j=1}^m\nu_j-\int g\,\hat{p}_{\varepsilon} d\otimes_{j=1}^m\hat{\nu}^N_j\right|\nonumber\\
     \leq 
     & \left|\int g p^*_{\varepsilon} d\otimes_{j=1}^m\hat{\nu}^N_j-\int g \hat{p}_{\varepsilon} d\otimes_{j=1}^m\hat{\nu}^N_j\right|\nonumber\\
     & + \left|\int g\,  p^*_{\varepsilon} d\otimes_{j=1}^m\nu_j-\int g\, p^*_{\varepsilon} d\otimes_{j=1}^m\hat{\nu}^N_j\right|\nonumber\\
     \leq & \|g\|_\infty \|p^*_{\varepsilon}(\mvx)-\hat{p}_{\varepsilon}(\mvx)\|_{L^1(\otimes_{j=1}^m\hat{\nu}_j^N)} +\left|\left(\otimes_{j=1}^m\hat{\nu}^N_j-\otimes_{j=1}^m\nu_j\right)(g\, p^*_{\varepsilon})\right|\nonumber\\
    =: & (A)+(B).
\end{align}
By Proposition \ref{diffp2}, with probability at least $1-2m(m-1)e^{-t}$, the first term (A) above satisfies 
\begin{equation}\label{eqn:term_A}
    \|p^*_{\varepsilon}(\mvx)-\hat{p}_{\varepsilon}(\mvx)\|_{L^1(\otimes_{j=1}^m\hat{\nu}_j^N)}\leq\|p^*_{\varepsilon}(\mvx)-\hat{p}_{\varepsilon}(\mvx)\|_{L^2(\otimes_{j=1}^m\hat{\nu}_j^N)}\lesssim_{m,\varepsilon}\sqrt{\frac{t}{N}}.
\end{equation}
The bound for the second term (B) follows from Lemma \ref{Ustat} and the boundedness of the potentials from Proposition \ref{BDP}. Indeed, with probability exceeding $1-2e^{-t}$, 
$$\left|\left(\otimes_{j=1}^m\hat{\nu}^N_j-\otimes_{j=1}^m\nu_j\right)(g\,p^*_{\varepsilon})\right|\lesssim_{\varepsilon}\sqrt{\frac{2t}{N}}\|g\|_\infty.$$ %\lesssim\left|\left(\otimes_{j=1}^m\hat{\nu}^N_j-\otimes_{j=1}^m\nu_j\right)(p^*_{\varepsilon})\|h\|_\infty\right|
\end{proof}

\begin{proof}[Proof of Lemma \ref{diff_empirical_population_f_i_bounded_by_other_potentials}]
It is sufficient to show that 
\begin{align*}
\| [\mvf^*] - [\hat{\mvf}^*] \|_{\widetilde{\cC}^1}
\leq
\|\bar{\mvf}^* - \hat{\mvf}^*\|_{\cC^1} = \cO_\bP \left( \sqrt{\frac{d \log N}{N}} \right).
\end{align*}
For that to hold, it is enough to establish
\begin{align*}
\sup_{\mvx} \left\| \bar{\mvf}^* (\mvx)  - \hat{\mvf}^* (\mvx) \right\|
= \cO_\bP \left( \sqrt{\frac{d \log N}{N}} \right),\quad \text{and} \quad
\sup_{\mvx} \left\| \nabla \bar{\mvf}^* (\mvx)  - \nabla \hat{\mvf}^* (\mvx) \right\|
= \cO_\bP \left( \sqrt{\frac{d \log N}{N}} \right).
\end{align*} 
In this proof, to shorten the notation, given a function 
\[
F : \underbrace{\mathbb{R}^d \otimes \cdots \otimes \mathbb{R}^d}_{m\text{ times}} \to \mathbb{R},
\]
and $X_1^{(j)},\dots,X_N^{(j)}\overset{i.i.d.}{\sim} \nu_j$
we define $F_{-i}( x_i, \mvX_{\mvl_{(-i)}})$ as 
\begin{align*}
    F\left(X_{\ell^{(1)}}^{(1)},\dots,X_{\ell^{(i-1)}}^{(i-1)},
        x_i,
        X_{\ell^{(i+1)}}^{(i+1)}, \dots,
        X_{\ell^{(m)}}^{(m)} \right),
\end{align*}
where we denote $(X_{\ell^{(1)}}^{(1)},\dots,X_{\ell^{(i-1)}}^{(i-1)},
        x_i,
        X_{\ell^{(i+1)}}^{(i+1)}, \dots,
        X_{\ell^{(m)}}^{(m)})$ as
$( x_i, \mvX_{\mvl_{(-i)}})$ for $\mvl_{(-i)} = (\ell^{(1)},\dots,\ell^{(i-1)},\ell^{(i+1)}, \dots,\ell^{(m)})$.

\noindent\underline{\textbf{Step 1.}} We first prove the first asymptotic equality. It is sufficient to show
\begin{align*}
\sup_{x_i\in[-1,1]^d}|\bar{f}_i^*(x_i) - \hat{f}_i^*(x_i)| = \cO_\bP \left( \sqrt{\frac{d \log N}{N}} \right).
\end{align*}
Define $\tilde{f}^*_i(x_i) : = -\varepsilon\log \left(\frac{1}{N^{m-1}}\sum_{\mvl_{(-i)}\in [N]^{m-1}} e^{\frac{\left(\oplus_{j\neq i} \bar{f}_j^* -c\right)_{-i} \left( x_i, \mvX_{\mvl_{-i}}\right) }{\varepsilon}} \right).$ On the one hand,
    \begin{align*}
        |\bar{f}_i^*(x_i) &- \tilde{f}_i^*(x_i)| \\
        &= \varepsilon \left|\log \left(\frac{1}{N^{m-1}}\sum_{\mvl_{(-i)}\in [N]^{m-1}} e^{\frac{\left(\oplus_{j\neq i} \bar{f}_j^* -c\right)_{-i} \left(x_i, \mvX_{\mvl_{-i}}\right) }{\varepsilon}} \right) 
        - \log\int e^{\frac{\oplus_{j\neq i} \bar{f}_j^*(x_j) -c(x_i,x_{-i})  }{\varepsilon}} d\nu_{-i}(\mvx)  \right|\\
        &\lesssim \left|\frac{1}{N^{m-1}}\sum_{\mvl_{(-i)}\in [N]^{m-1}} e^{\frac{\left(\oplus_{j\neq i} \bar{f}_j^* -c\right)_{-i} \left(x_i, \mvX_{\mvl_{-i}}\right) }{\varepsilon}} - \int e^{\frac{\oplus_{j\neq i} \bar{f}_j^*(x_j) -c(x_i,x_{-i})  }{\varepsilon}} d\nu_{-i}(\mvx) \right|.
    \end{align*}
By Lemma \ref{Lipschitz_plus_light_tail_union_bound}, we know that
\begin{align*}
    \sup_{x_i} |\bar{f}_i^*(x_i) &- \tilde{f}_i^*(x_i)| = \cO_\bP \left(\sqrt{\frac{d \log N}{N}}\right).
\end{align*}
On the other hand,
\begin{align*}
    |&\tilde{f}^*_i(x_i) - \hat{f}^*_i(x_i)|\\
   &= \varepsilon \left|\log \left(\frac{1}{N^{m-1}}\sum_{\mvl_{(-i)}\in [N]^{m-1}} e^{\frac{\left(\oplus_{j\neq i} \bar{f}_j^* -c\right)_{-i} \left(x_i, \mvX_{\mvl_{-i}}\right) }{\varepsilon}} \right) 
    - \log \left(\frac{1}{N^{m-1}}\sum_{\mvl_{(-i)}\in [N]^{m-1}}  e^{\frac{\left(\oplus_{j\neq i} \hat{f}_j^* -c\right)_{-i} \left(x_i, \mvX_{\mvl_{-i}}\right) }{\varepsilon}} \right) \right|\\
    &\lesssim \frac{1}{N^{m-1}}\sum_{\mvl_{(-i)}\in [N]^{m-1}} 
    \left|  e^{\frac{\left(\oplus_{j\neq i} \bar{f}_j^* -c\right)_{-i} \left( x_i, \mvX_{\mvl_{-i}}\right) }{\varepsilon}} - e^{\frac{\left(\oplus_{j\neq i} \hat{f}_j^* -c\right)_{-i} \left(x_i, \mvX_{\mvl_{-i}}\right) }{\varepsilon}}  \right|\\
    & \lesssim \sum_{\substack{j\neq i\\j\in[m]}} \frac{1}{N} \sum_{\ell=1}^N |\hat{f}^N_j-\bar{f}^*_j|(X_\ell^{(j)}).
\end{align*}
Based on Lemma \ref{gradnormp} and Lemma \ref{barf}, we know that $$\sup_{x_i}\sum_{j\neq i}\| \bar{f}^*_j -\hat{f}^*_j\|_{L^1(\hat{\nu}^N_j)}\lesssim\|\nabla\widehat{\Phi}_{\varepsilon}(\vf^*)\|_{\widehat{\mathscr{L}}_m} = \cO_\bP( N^{-1/2}).$$

\noindent\underline{\textbf{Step 2.}} Now, we prove the second asymptotic equality. Based on the topology we have, it is equivalent to show
\begin{align*}
\sup_{x_i\in[-1,1]^d} \max_{1\leq k\leq d} \left| \frac{\partial}{\partial_{x_{i,k}}} \left( \bar{f}_i^*(x_i) - \hat{f}_i^*(x_i) \right) \right|= \cO_\bP \left( \sqrt{\frac{d \log N}{N}} \right).
\end{align*}

\noindent Taking derivative, and we have via multimarginal Schr\"{o}dinger system that
\begin{align*}
\frac{\partial}{\partial_{x_{i,k}}}  \bar{f}_i^*(x_i) %&= \nabla_{x_i} \varepsilon \log \int e^{\frac{\sum_{j\neq i} \bar{f}^*_j(x_j) - c(\mvx)}{\varepsilon}} d\nu_{-i}(\mvx)\\
%& = -\frac{\int \nabla_{x_i} c(\mvx) e^{\frac{\sum_{j\neq i} \bar{f}^*_j(x_j) - c(\mvx)}{\varepsilon}} d\nu_{-i}(\mvx)}{\int e^{\frac{\sum_{j\neq i} \bar{f}^*_j(x_j) - c(\mvx)}{\varepsilon}} d\nu_{-i}(\mvx)}\\
& = \int \frac{\partial}{\partial_{x_{i,k}}}  c(\mvx) p^*_\varepsilon(\mvx) d\nu_{-i}(\mvx),\\
\frac{\partial}{\partial_{x_{i,k}}}  \hat{f}_i^*(x_i) %\nabla_{x_i} \varepsilon \log \int e^{\frac{\sum_{j\neq i} \hat{f}^*_j(x_j) - c(\mvx)}{\varepsilon}} d\hat{\nu}^N_{-i}(\mvx)\\
%& = -\frac{\int \nabla_{x_i} c(\mvx) e^{\frac{\sum_{j\neq i} \hat{f}^*_j(x_j) - c(\mvx)}{\varepsilon}} d\hat{\nu}^N_{-i}(\mvx)}{\int e^{\frac{\sum_{j\neq i} \hat{f}^*_j(x_j) - c(\mvx)}{\varepsilon}} d\hat{\nu}^N_{-i}(\mvx)}\\
& = \int\frac{\partial}{\partial_{x_{i,k}}}  c(\mvx) \hat{p}_\varepsilon(\mvx) d\hat{\nu}^N_{-i}(\mvx).
\end{align*}  

\noindent Define $S(x_i) := \int \frac{\partial}{\partial_{x_{i,k}}}  c(\mvx) p^*_\varepsilon(\mvx) d\hat{\nu}^N_{-i}(\mvx)$ and triangle inequality gives

\begin{align*}
    \left|\frac{\partial}{\partial_{x_{i,k}}} \hat{f}_i^*(x_i) - \frac{\partial}{\partial_{x_{i,k}}}  \bar{f}_i^*(x_i)\right|
    \leq 
    \left| \frac{\partial}{\partial_{x_{i,k}}} \hat{f}_i^*(x_i) - S(x_i)\right|
    + \left| S(x_i) - \frac{\partial}{\partial_{x_{i,k}}}  \bar{f}_i^*(x_i) \right|.\\
    %& = \left| \int \frac{\partial}{\partial_{x_{i,k}}}  p^*_\varepsilon(\mvx) d\hat{\nu}^N_{-i}(\mvx)- \int \frac{\partial}{\partial_{x_{i,k}}}  c(\mvx) \hat{p}_\varepsilon(\mvx) d\hat{\nu}^N_{-i}(\mvx) \right| \\
    %& + \left|\int \frac{\partial}{\partial_{x_{i,k}}}  c(\mvx) p^*_\varepsilon(\mvx) d\nu_{-i}(\mvx) - \int \frac{\partial}{\partial_{x_{i,k}}}  c(\mvx) p^*_\varepsilon(\mvx) d\hat{\nu}^N_{-i}(\mvx) \right|. \\
\end{align*}
By Lemma \ref{Lipschitz_plus_light_tail_union_bound}, 
\begin{align*}
\max_{1\leq k\leq d}\sup_{x_i\in[-1,1]^d}\left|S(x_i)-\frac{\partial}{\partial_{x_{i,k}}}  \bar{f}_i^*(x_i)\right| = \cO_\bP \left(\sqrt{\frac{d \log N}{N}}\right).
\end{align*}
The other term could be bounded as follows. 
\begin{align*}
    \left| \frac{\partial}{\partial_{x_{i,k}}}  \hat{f}_i^*(x_i) - S(x_i)\right| &= \left| \int \frac{\partial}{\partial_{x_{i,k}}}  c(\mvx) p^*_\varepsilon(\mvx) d\hat{\nu}^N_{-i}(\mvx)- \int \frac{\partial}{\partial_{x_{i,k}}}  c(\mvx) \hat{p}_\varepsilon(\mvx) d\hat{\nu}^N_{-i}(\mvx) \right| \\
    &\lesssim  \int  \left| p^*_\varepsilon(\mvx) - \hat{p}_\varepsilon(\mvx)  \right| d\hat{\nu}^N_{-i}(\mvx) \\
    &\lesssim_{m,\varepsilon}
    |\bar f_i^*(x_i)-\widehat f_i^*(x_i)|
    +\sum_{j\ne i}
      \|\bar f_j^*-\widehat f_j^*\|_
      {L^1(\widehat\nu_j^N)}.\\
    %&= \frac{1}{N^{m-1}}\sum_{\mvl_{(-i)}\in [N]^{m-1}} \left|  e^{\frac{\left(\oplus_{j\neq i} \bar{f}_j^* -c\right)_{-i} \left(x_i,\mvX_{\mvl_{-i}}\right) }{\varepsilon}} - e^{\frac{\left(\oplus_{j\neq i} \hat{f}_j^* -c\right)_{-i} \left(x_i,\mvX_{\mvl_{-i}}\right) }{\varepsilon}} \right|\\ & \lesssim \sum_{\substack{j\neq i\\j\in[m]}} \frac{1}{N} \sum_{\ell=1}^N |\hat{f}^N_j-\bar{f}^*_j|(X_\ell^{(j)}).
\end{align*}
The first term on the right hand side is bounded in Step 1. For the second term on the right hand side, based on Lemma \ref{gradnormp} and Lemma \ref{barf}, we know that 
\begin{align*}
\sup_{x_i}\sum_{j\ne i}
\|\bar f_j^*-\hat f_j^*\|_{L^1(\widehat\nu_j^N)}
&\le
\sqrt{m-1}
\Bigl(
\sum_{j\ne i}
\|\bar f_j^*-\hat f_j^*\|_{L^2(\widehat\nu_j^N)}^2
\bigr)^{1/2} \\
&\lesssim_{m,\varepsilon}
\|\nabla\widehat{\Phi}_{\varepsilon}(\vf^*)\|_{\widehat{\mathscr{L}}_m} 
=\cO_{\mathbb P}(N^{-1/2}).
\end{align*}
%$$\sup_{x_i}\sum_{j\neq i}\| \bar{f}^*_j -\hat{f}^*_j\|_{L^1(\hat{\nu}^N_j)}\lesssim\|\nabla\widehat{\Phi}_{\varepsilon}(\vf^*)\|^2_{\widehat{\mathscr{L}}_m} = \cO_\bP( N^{-1/2}).$$

\end{proof}

\begin{lemma}[Uniform bound on empirical joint density of optimal coupling]
\label{main_lemma_bound_on_norm_Aij-Dij}
For a given cost function $c\in\cC^2$,
    \begin{equation*}
        \sup_{x_i\in[-1,1]^d}\int \left| (\hat{p}_\varepsilon - p^*_\varepsilon)(\mvx) \right|^2 d\hat{\nu}_{-i}^N =\cO_\bP\left(\frac{d\log N}{N}\right).
    \end{equation*}
\end{lemma}

\begin{proof}[Proof of Lemma \ref{main_lemma_bound_on_norm_Aij-Dij}]
By Proposition \ref{13}, we have
\begin{equation*}
        |\hat{p}_{\varepsilon}(\mvx)-p^*_{\varepsilon}(\mvx)|^2\lesssim_\varepsilon m \sum_{j=1}^m|\bar{f}_j^*(x_j)-\hat{f}^*_j(x_j)|^2.
\end{equation*}
Thus, 
\begin{align*}
    \sup_{x_i\in[-1,1]^d}\int \left| (\hat{p}_\varepsilon - p^*_\varepsilon)(\mvx) \right|^2 d\hat{\nu}_{-i}^N \lesssim_{m,\varepsilon} 
    \sup_{x_i\in[-1,1]^d} |\bar{f}_i^*(x_i)-\hat{f}^*_i(x_i)|^2 + \sum_{j\neq i} \frac{1}{N} \sum_{\ell^{(j)}=1}^N \left|\bar{f}_j^*(X^{(j)}_{\ell^{(j)}})-\hat{f}^*_j(X^{(j)}_{\ell^{(j)}}) \right|^2.
\end{align*}
The first term on the right hand side is dealt with in Lemma \ref{diff_empirical_population_f_i_bounded_by_other_potentials}.
The second term on the right hand side is bounded by Lemma \ref{gradnormp} and Lemma \ref{barf}.
\end{proof}

\begin{proof}[Proof of Theorem \ref{supremp_coupling}]
In this proof, we shall denote $C > 0$ as a generic constant depending only on $m,\varepsilon,\beta$ and $d$, whose value may vary from line to line. From (\ref{boundtest_decomp}) in the proof of Theorem \ref{boundtest}, we know that for every $g\in\widetilde{\mathcal{H}}$, $|(\pi_\varepsilon^*-\hat{\pi}^N_\varepsilon)(g)|$ can be bounded as following
\begin{align*}
|(\pi_\varepsilon^*-\hat{\pi}^N_\varepsilon)(g)|\leq
\|g\|_\infty \|p^*_{\varepsilon}(\mvx)-\hat{p}_{\varepsilon}(\mvx)\|_{L^1(\otimes_{j=1}^m\hat{\nu}_j^N)} +\left|\left(\otimes_{j=1}^m\hat{\nu}^N_j-\otimes_{j=1}^m\nu_j\right)(g\, p^*_{\varepsilon})\right|
\end{align*}
As a result, 
\begin{align*}
    \mathbb{E}\sup_{g\in\widetilde{\mathcal{H}}}|(\pi_\varepsilon^*-\hat{\pi}^N_\varepsilon)(g)|
    \leq & \mathbb{E}\sup_{g\in\widetilde{\mathcal{H}}}\left[\|g\|_\infty \|p^*_{\varepsilon}(\mvx)-\hat{p}_{\varepsilon}(\mvx)\|_{L^1(\otimes_{j=1}^m\hat{\nu}_j^N)}\right] \nonumber\\ & +\mathbb{E}\sup_{g\in\widetilde{\mathcal{H}}}\left|\left(\otimes_{j=1}^m\hat{\nu}^N_j-\otimes_{j=1}^m\nu_j\right)(gp^*_\varepsilon)\right|\nonumber\\
    =: & (\text{I}) + (\text{II}).
\end{align*}
Note that we have proved that for any $t > 0$,~\eqref{eqn:term_A} holds with probability at least $1-2m(m-1)e^{-t}$. 
%by Proposition \ref{diffp2}, for $t>0$, , 
%$$\|p^*_{\varepsilon}(\mvx)-\hat{p}_{\varepsilon}(\mvx)\|_{L^1(\otimes_{j=1}^m\hat{\nu}_j^N)}\leq \|p^*_{\varepsilon}(\mvx)-\hat{p}_{\varepsilon}(\mvx)\|_{L^2(\otimes_{j=1}^m\hat{\nu}_j^N)}\lesssim_{m,\varepsilon}\sqrt{\frac{t}{N}}.$$
Integrating this tail probability bound, we have the following expectation bound for term $(\text{I})$
\begin{equation}\label{nonpA'}
\mathbb{E}\sup_{g\in\widetilde{\mathcal{H}}}\left[\|g\|_\infty \|p^*_{\varepsilon}(\mvx)-\hat{p}_{\varepsilon}(\mvx)\|_{L^1(\otimes_{j=1}^m\hat{\nu}_j^N)}\right]\lesssim_{m,\varepsilon} L/\sqrt{N}.
\end{equation}
Next, we are going to control term $(\text{II})$. By the product telescoping identity,
\[
\begin{aligned}
&\biggl(
\bigotimes_{i=1}^m \widehat{\nu}_i^N
-
\bigotimes_{i=1}^m \nu_i
\biggr)(g p_\varepsilon^*)
\\
&\quad =
\sum_{j=1}^m
\int
\Bigl( 
\widehat{\nu}_j^N-\nu_j
\Bigr)
\Bigl(
p_\varepsilon^*(\cdot,x_{-j})g(\cdot,x_{-j})
\Bigr)
\,d\biggl[
\Bigl(\bigotimes_{k<j}\widehat{\nu}_k^N\Bigr)
\otimes
\Bigl(\bigotimes_{k>j}\nu_k\Bigr)
\biggr](x_{-j}).
\end{aligned}
\]
For every fixed $x_{-j}$, restriction to the $j$-th coordinate block gives
\[
\|g(\cdot,x_{-j})\|_\beta
\le \|g\|_\beta
\le L,
\]
so $g(\cdot,x_{-j})\in \mathcal H([-1,1]^d;\beta,L)$. Consequently,
\[
\begin{aligned}
&\sup_{g\in\widetilde{\mathcal H}}
\Biggl|
\biggl(
\bigotimes_{i=1}^m \widehat{\nu}_i^N
-
\bigotimes_{i=1}^m \nu_i
\biggl)(g p_\varepsilon^*)
\Biggl|
\\
&\qquad\le
\sum_{j=1}^m
\int
\sup_{h\in \mathcal H([-1,1]^d;\beta,L)}
\left|
( 
\widehat{\nu}_j^N-\nu_j
)
\left(p_\varepsilon^*(\cdot,x_{-j})h\right)
\right| d\biggl[
\Bigl(\bigotimes_{k<j}\widehat{\nu}_k^N\Bigr)
\otimes
\Bigl(\bigotimes_{k>j}\nu_k\Bigr)
\biggr](x_{-j}).
\end{aligned}
\]
For fixed $x_{-j}$, the function class
$
\left\{
p_\varepsilon^*(\cdot,x_{-j})h:
h\in \mathcal H([-1,1]^d;\beta,L)
\right\}
$
has envelope $L\|p_\varepsilon^*\|_\infty$. Moreover,
\[
\begin{aligned}
&\log N\left(
\left\{
p_\varepsilon^*(\cdot,x_{-j})h:
h\in \mathcal H([-1,1]^d;\beta,L)
\right\},
\eta,\|\cdot\|_\infty
\right)
\\
&\qquad\le
\log N\left(
\mathcal H([-1,1]^d;\beta,L),
\frac{\eta}{\|p_\varepsilon^*\|_\infty},
\|\cdot\|_\infty
\right)\lesssim_{\beta,d}
\left(
\frac{L}{\eta}
\right)^{d/\beta}.
\end{aligned}
\]
The same chaining argument
as in the proof of Theorem~5.5, equations (\ref{incre})--(\ref{case3bound}) yields
\[
\sup_{x_{-j}}
\mathbb E_j
\sup_{h\in \mathcal H([-1,1]^d;\beta,L)}
\left|
(\widehat{\nu}_j^N-\nu_j)
\left(p_\varepsilon^*(\cdot,x_{-j})h\right)
\right|
\lesssim
\begin{cases}
LN^{-1/2}, & d<2\beta,\\
LN^{-1/2}\log(eN), & d=2\beta,\\
LN^{-\beta/d}, & d>2\beta,
\end{cases}
\]
where $\mathbb E_j$ denotes expectation over the $j$-th marginal sample. Taking expectations
and summing over $j\in[m]$, we obtain
\[
(\mathrm{II})\lesssim
\begin{cases}
LN^{-1/2}, & d<2\beta,\\
LN^{-1/2}\log N, & d=2\beta,\\
LN^{-\beta/d}, & d>2\beta.
\end{cases}
\]
Here the implicit constants may depend on
$m,\varepsilon,\|c\|_\infty,\beta,d$, but not on $N$ or $L$.
Combining this estimate with \eqref{nonpA'} completes the proof of \eqref{eqn:holder_bound_coupling}. The proof of $W_p$ bound for $p=1,2$ follows exactly the same as that in the proof of Theorem \ref{supremp}.

\end{proof}

\begin{proof}[Proof of Theorem \ref{supremp}]
In this proof, we shall denote $C > 0$ as a generic constant depending only on $m,\varepsilon,\beta$ and $d$, whose value may vary from line to line. From (\ref{boundtest_decomp}) in the proof of Theorem \ref{boundtest}, we know that for every $h\in\mathcal{H}$, $|(\bar{\nu}_\varepsilon-\hat{\nu}^N_\varepsilon)(h)|$ can be bounded as following
\begin{align}
|(\bar{\nu}_\varepsilon-\hat{\nu}^N_\varepsilon)(h)| 
%&=\left|\int (h\circ T_\alpha)\,p^*_{\varepsilon} d\otimes_{j=1}^m\nu_j-\int (h\circ T_\alpha)\,\hat{p}_{\varepsilon} d\otimes_{j=1}^m\hat{\nu}^N_j\right| \nonumber\\
&\leq  \|h\|_\infty \|p^*_{\varepsilon}(\mvx)-\hat{p}_{\varepsilon}(\mvx)\|_{L^1(\otimes_{j=1}^m\hat{\nu}_j^N)} +\left|\left(\otimes_{j=1}^m\hat{\nu}^N_j-\otimes_{j=1}^m\nu_j\right)(p^*_{\varepsilon}(h\circ T_\alpha))\right|.\nonumber
\end{align}
As a result, 
\begin{align}\label{controlAB}
    \mathbb{E}\sup_{h\in\mathcal{H}}|(\bar{\nu}_\varepsilon-\hat{\nu}^N_\varepsilon)(h)| 
    \leq & \mathbb{E}\sup_{h\in\mathcal{H}}\left[\|h\|_\infty \|p^*_{\varepsilon}(\mvx)-\hat{p}_{\varepsilon}(\mvx)\|_{L^1(\otimes_{j=1}^m\hat{\nu}_j^N)}\right] \nonumber\\ & +\mathbb{E}\sup_{h\in\mathcal{H}}\left|\left(\otimes_{j=1}^m\hat{\nu}^N_j-\otimes_{j=1}^m\nu_j\right)(p^*_{\varepsilon}(h\circ T_\alpha))\right|\nonumber\\
    =: & (A)+(B).
\end{align}
Note that we have proved that for any $t > 0$,~\eqref{eqn:term_A} holds with probability at least $1-2m(m-1)e^{-t}$. 
%by Proposition \ref{diffp2}, for $t>0$, , 
%$$\|p^*_{\varepsilon}(\mvx)-\hat{p}_{\varepsilon}(\mvx)\|_{L^1(\otimes_{j=1}^m\hat{\nu}_j^N)}\leq \|p^*_{\varepsilon}(\mvx)-\hat{p}_{\varepsilon}(\mvx)\|_{L^2(\otimes_{j=1}^m\hat{\nu}_j^N)}\lesssim_{m,\varepsilon}\sqrt{\frac{t}{N}}.$$
Integrating this tail probability bound, we have the following expectation bound for term $(A)$
\begin{equation}\label{nonpA}
\mathbb{E}\sup_{h\in\mathcal{H}}\left[\|h\|_\infty \|p^*_{\varepsilon}(\mvx)-\hat{p}_{\varepsilon}(\mvx)\|_{L^1(\otimes_{j=1}^m\hat{\nu}_j^N)}\right]\lesssim_{m,\varepsilon} L/\sqrt{N}.
\end{equation}
Next, we are going to control term $(B)$. Let $X_h:=\sqrt{N}\left(\otimes_{j=1}^m\hat{\nu}^N_j-\otimes_{j=1}^m\nu_j\right)(p^*_{\varepsilon}(h\circ T_\alpha))$ be a mean-zero process indexed by $h \in \cH$. By Lemma \ref{Ustat}, we know that with probability more than $1-2e^{-t}$, 
\begin{equation}\label{incre}
\left|X_h-X_{\tilde{h}}\right|\lesssim_{\varepsilon}\|h-\tilde{h}\|_\infty \sqrt{t}.
\end{equation}
Thus, $X_h$ satisfies the sub-Gaussian condition in Lemma \ref{dudley}. For any given $h\in\mathcal{H}$, it is obvious  from (\ref{incre}) that $\mathbb{E}|X_h|\lesssim_{m,\varepsilon} L$. From Theorem 2.7.1 in \cite{vanderVaartWellner1996}, we know
\begin{equation}\label{coveringn}
    \log N(\mathcal{H}, \eta, \|\cdot\|_\infty) \lesssim_{\beta,d} \Big({\frac{L} {\eta}}\Big)^{d / \beta}.
\end{equation}

\noindent \underline{\textbf{Case 1: $d<2\beta$}}. By (\ref{coveringn}), it is clear that $\int_{0}^{2L}\sqrt{\log N\left(\mathcal{H}, \varepsilon,\|\cdot\|_\infty\right)}d\varepsilon<\infty$ under this setting. Starting from the increment condition (\ref{incre}), one can obtain, via the usual Dudley's bound (see, for example, Lemma \ref{dudley} with $\delta=0$), 
\begin{align}\label{0case1}
    \mathbb{E}\sup_{h\in\mathcal{H}} |\left(\otimes_{j=1}^m\hat{\nu}^N_j-\otimes_{j=1}^m\nu_j\right)(p^*_{\varepsilon}(h\circ T_\alpha))|&\lesssim_{m,\varepsilon}\ N^{-1/2}\int_{0}^{2L}\sqrt{\log N\left(\mathcal{H}, \varepsilon,\|\cdot\|_\infty\right)}d\varepsilon \nonumber\\
    &\lesssim_{\beta,d} L/\sqrt{N}.
\end{align}
Combining (\ref{controlAB}), (\ref{nonpA}) and (\ref{0case1}), we see that there is some constant $C = C(m,\varepsilon,\beta,d) > 0$ such that
\begin{equation}
    \mathbb{E}\sup_{h\in\mathcal{H}}|(\bar{\nu}_\varepsilon-\hat{\nu}^N_\varepsilon)(h)|\leq C L N^{-1/2}.
\end{equation}

\noindent \underline{\textbf{Case 2: $d>2\beta$}}. 
We have $D=:\sup_{h,\tilde{h}\in\mathcal{H}} \| h - \tilde{h}\|_\infty \leq 2L$. For any $\delta\in[0,D]$, notice that
\begin{equation}\label{nonp1}
\mathbb{E} \Big[\sup_{\{\gamma,\gamma'\in \mathcal{H} : \; \|\gamma - \gamma'\|_\infty\leq\delta\}} (X_\gamma - X_{\gamma'}) \Big] \lesssim_\varepsilon\sqrt{N}\delta,
\end{equation}
and
\begin{equation}\label{nonp2}
\int_{\delta/4}^D\sqrt{\log N(\mathcal{H},\varepsilon, \|\cdot\|_\infty)}d\varepsilon\lesssim_{\beta,d} L^{\frac{d}{2\beta}} \int_{\delta/4}^{2L}\varepsilon^{-\frac{d}{2\beta}} \, d\varepsilon\lesssim_{\beta,d} L^{\frac{d}{2\beta}} \delta^{-\frac{d-2\beta}{2\beta}}.
\end{equation}
By virtue of (\ref{nonp1}) and (\ref{nonp2}), applying Lemma \ref{dudley}, for any $\delta\in[0,D]$, it holds true that  
\begin{equation*}
\mathbb{E}\sup_{h\in\mathcal{H}}\left|\sqrt{N}\left(\otimes_{j=1}^m\hat{\nu}^N_j-\otimes_{j=1}^m\nu_j\right)(p^*_{\varepsilon}(h\circ T_\alpha))\right| \leq C [L+\sqrt{N}\delta + L^{\frac{d}{2\beta}} \delta^{-\frac{d-2\beta}{2\beta}}].
\end{equation*}
Optimize over $\delta\in[0, D]$, and we know that for $\delta\asymp LN^{-\beta/d}$, we have
\begin{equation}\label{nonpB}
\mathbb{E}\sup_{h\in\mathcal{H}}\left|\left(\otimes_{j=1}^m\hat{\nu}^N_j-\otimes_{j=1}^m\nu_j\right)(p^*_{\varepsilon}(h\circ T_\alpha))\right| \leq C L N^{-\beta/d},
\end{equation}
for some constant $C> 0$. Combining (\ref{controlAB}), (\ref{nonpA}) and (\ref{nonpB}), the result under Case 2 is proved.

\noindent \underline{\textbf{Case 3: $d=2\beta$}}. The argument under this setting is similar to that for the previous case. Inequalities $(\ref{nonp1})$ and $\mathbb{E}|X_h|\lesssim_{m,\varepsilon} L$
still hold. Additionally, observe that for $d=2\beta$,
\begin{equation}\label{case3dud}
\int_{\delta/4}^D\sqrt{\log N(\mathcal{H},\varepsilon, \|\cdot\|_\infty)}d\varepsilon\lesssim_{\beta,d} L \int_{\delta/4}^{2L}\varepsilon^{-1}d\varepsilon= L\log(8L) + L\log(1/\delta).
\end{equation}
Due to (\ref{nonp1}), (\ref{case3dud}), and Lemma \ref{dudley}, we find that for some constant $C>0$,
\begin{equation*}
\mathbb{E}\sup_{h\in\mathcal{H}}\left|\sqrt{N}\left(\otimes_{j=1}^m\hat{\nu}^N_j-\otimes_{j=1}^m\nu_j\right)(p^*_{\varepsilon}(h\circ T_\alpha))\right| \leq C [ L\log(8eL) + \sqrt{N}\delta + L \log(1/ \delta)].
\end{equation*}
Taking $\delta \asymp L N^{-1/2} \log{N}$, the above bound becomes 
\begin{equation}\label{case3bound}
\mathbb{E}\sup_{h\in\mathcal{H}}\left|\left(\otimes_{j=1}^m\hat{\nu}^N_j-\otimes_{j=1}^m\nu_j\right)(p^*_{\varepsilon}(h\circ T_\alpha))\right| \leq C L N^{-1/2} \log{N}.
\end{equation}
The combination of (\ref{controlAB}), (\ref{nonpA}) and (\ref{case3bound}) leads to the conclusion in~\eqref{eqn:holder_bound}.

\begin{comment}
    For any function class $\mathcal{H}$, we have 
\begin{equation}
         \mathbb{E}\sup_{h\in\mathcal{H}} |(\bar{\nu}_\varepsilon-\hat{\nu}^N_\varepsilon)(h)|\lesssim_{m,\varepsilon}\inf_{\delta>0}\left\{\delta+N^{-1/2}\int_{\delta/4}^D\sqrt{\log N(\mathcal{H},\varepsilon, \|\cdot\|_\infty)}d\varepsilon\right\}.
\end{equation}
\end{comment}

\noindent Finally, we prove~\eqref{eqn:wasserstein_bound}. For the $W_1$ bound, invoking the well-known Kantorovich duality~\citep{villani2021topics}, 

\begin{equation*}
W_1(\bar{\nu}_\varepsilon,\hat{\nu}^N_\varepsilon)
=
\sup\left\{ 
\Bigl|
\int fd\bar{\nu}_\varepsilon - \int fd\hat{\nu}^N_\varepsilon
\Bigr|
: f\in\mathcal{H}([-1,1]^d; 1,1) \right\},
\end{equation*}
where we recall that $\mathcal{H}([-1,1]^d; 1,1)$ is the class of Lipschitz continuous functions with constant one. This observation concludes the proof combined with the covering bound $\log N(\mathcal{H}([-1,1]^d; 1,1), \eta, \|\cdot\|_\infty) \lesssim_{d} \eta^{-d}$ from Example 5.10 in~\cite{Wainwright_2019}. As for $W_2$ bound in (\ref{eqn:wasserstein_bound}), by arguments in Lemma 3 of \cite{ChizatRoussillonLegerVialardPeyer20},

\begin{align*}
W^2_2(\bar{\nu}_\varepsilon,\hat{\nu}^N_\varepsilon)&\leq
\sup_\mu |W^2_2(\bar{\nu}_\varepsilon,\mu)-W^2_2(\hat{\nu}^N_\varepsilon,\mu)|\nonumber\\
&\leq 
2\sup \left\{ 
\Bigl|
\int fd\bar{\nu}_\varepsilon - \int fd\hat{\nu}^N_\varepsilon
\Bigr|
: f\in\mathcal{H}([-1,1]^d; 1,1) \cap \text{Conv}([-1,1]^d) \right\}
+
\left| \int \|\cdot\|^2 d(\bar{\nu}_\varepsilon- \hat{\nu}^N_\varepsilon) \right|,
\end{align*}
with $\text{Conv}([-1,1]^d)$ the collection of convex functions $f : [-1,1]^d \to \bR$. From~\cite{Bronshtein_convex-covering}, we have $\log N(\mathcal{H}([-1,1]^d; 1,1) \cap \text{Conv}([-1,1]^d), \eta, \|\cdot\|_\infty) \lesssim_{d}  \eta^{-d/2}$ for the first term above. The second term of the preceding display has already been well-controlled by Theorem \ref{boundtest}. Finally, ~\eqref{eqn:wasserstein_bound} follows from~\eqref{eqn:holder_bound} with $\beta = p$ for $p = 1$ and $p = 2$.

\end{proof}

\section{Sample complexity and weak limit of the cost functional}\label{sec_sample_complexity_proof}

\noindent The proof of sample complexity (Theorem~\ref{thm:rate_cost}) and that of weak limit (Theorem \ref{CLT_for_cost_functional}) of the cost functional are similar. We therefore report them together in this section.

\begin{proof}[Proof of Theorem~\ref{thm:rate_cost}]
By the strong duality between~\eqref{primal2} and~\eqref{dual2}, we have
\[
\mathbb{E}\left[S_{\varepsilon}(\nu_1,\dots,\nu_m) - S_{\varepsilon}(\hat{\nu}^N_1,\dots,\hat{\nu}^N_m)\right]^2 = \mathbb{E}\left[\Phi_{ \varepsilon}(\vf^*) - \widehat{\Phi}_{\varepsilon}(\hat{\vf}^*)\right]^2.
\]
Decomposing
$\Phi_{ \varepsilon}(\vf^*) - \widehat{\Phi}_{\varepsilon}(\hat{\vf}^*)
=\Phi_{ \varepsilon}(\vf^*) - \widehat{\Phi}_{\varepsilon}(\vf^*)+ \widehat{\Phi}_{\varepsilon}(\vf^*) - \widehat{\Phi}_{\varepsilon}(\hat{\vf}^*),$
we may bound 
\begin{equation}\label{decom}
\left|\Phi_{ \varepsilon}(\vf^*) - \widehat{\Phi}_{\varepsilon}(\hat{\vf}^*)\right|^2
\leq2\left|\Phi_{ \varepsilon}(\vf^*) - \widehat{\Phi}_{\varepsilon}(\vf^*)\right|^2
+ 2\left|\widehat{\Phi}_{ \varepsilon}(\vf^*) - \widehat{\Phi}_{\varepsilon}(\hat{\vf}^*)\right|^2=:(\text{I})+(\text{II}).
\end{equation}
\underline{\bf Term (I)}. Plugging into the definition of population and empirical dual objective functions in~\eqref{dual1} and~\eqref{dual2}, we have
\begin{equation}\label{term1,0}
\begin{gathered}
\mathbb{E}\Big[\Phi_{ \varepsilon}(\vf^*) - \widehat{\Phi}_{\varepsilon}(\vf^*)\Big]^2
= \mathbb{E}\Big[\sum_{j=1}^m(v_j-\hat{v}^N_j)(f_j^*)+\varepsilon(\otimes_{k=1}^m\hat{v}^N_k-\otimes_{k=1}^mv_k)\exp\Big(\frac{\sum_{k=1}^m f^*_k-c}{\varepsilon}\Big)\Big]^2
\\
\leq2\mathbb{E}\Big[\sum_{j=1}^m(v_j-\hat{v}^N_j)(f_j^*)\Big]^2+2\varepsilon^2\mathbb{E}\Big[(\otimes_{k=1}^m\hat{v}^N_k-\otimes_{k=1}^mv_k)\exp\Big(\frac{\sum_{k=1}^m f^*_k-c}{\varepsilon}\Big)\Big]^2.
\end{gathered}
\end{equation}
For the first term on the RHS of~\eqref{term1,0}, we use the sample independence to obtain
\begin{align}\label{term1,1}
\nonumber
&\mathbb{E}\Big[\sum_{j=1}^m(\nu_j-\hat{\nu}^N_j)(f_j^*)\Big]^2
=\mathbb{E}\Big[\sum_{j=1}^m\frac{1}{N}\sum_{k=1}^N\Big(f^*_j(X^{(j)}_k)-\nu_j(f^*_j)\Big)\Big]^2\\
&=\frac{1}{N^2}\mathbb{E}\Big[\sum_{1\leq j, j'\leq m}\sum_{1\leq k, k'\leq N}\left(f^*_j(X^{(j)}_k)-\nu_j(f^*_j)\right)\left(f^*_{j'}(X^{(j')}_{k'})-\nu_{j'}(f^*_{j'})\right)\Big]\nonumber\\
%&=\frac{1}{N^2}\mathbb{E}\left[\sum_{j=1}^m\sum_{k=1}^N\left(f^*_j(X^{(j)}_k)-\nu_j(f^*_j)\right)^2\right]=\frac{1}{N}\sum_{j=1}^m\mathbb{E}\left[\left(f^*_j(X_1^{(j)})-\nu_j(f^*_j)\right)^2\right]\nonumber\\
&=\frac{1}{N}\sum_{j=1}^m\mathbb{E}\Big[\left(f^*_j(X_1^{(j)})-\nu_j(f^*_j)\right)^2\Big]=\frac{1}{N}\sum_{j=1}^m\text{Var}_{\nu_j}(f^*_j)\leq\frac{m\|c\|_{\infty}^2}{N},
\end{align}
where the last inequality follows from the pointwise boundedness of the dual potentials in Proposition \ref{BDP}. For the second term on the RHS of~\eqref{term1,0}, we use the marginal feasibility constrain (\ref{multiconst_equiv}) for the multimarginal Schr\"{o}dinger system to deduce that
\begin{align}\label{cross1}
&\mathbb{E}\Big[(\otimes_{k=1}^m\hat{\nu}^N_k-\otimes_{k=1}^m\nu_k)\exp\Big(\frac{\sum_{k=1}^m f^*_k-c}{\varepsilon}\Big)\Big]^2
=\mathbb{E}\Big[ \left(\otimes_{k=1}^m\hat{\nu}^N_k\right)\Big(1-\exp\Big(\frac{\sum_{k=1}^m f^*_k-c}{\varepsilon}\Big)\Big)\Big]^2\nonumber\\
&=\frac{1}{N^{2m}}\mathbb{E}\Big\{ \sum_{1\leq k^{(1)},\dots,k^{(m)}\leq N} \Big[ 1-\exp\Big({\frac{\sum_{i=1}^mf^*_i(X^{(i)}_{k^{(i)}})-c(X^{(1)}_{k^{(1)}},\dots,X^{(m)}_{k^{(m)}})}{\varepsilon}}\Big) \Big] \Big\}^2 \nonumber\\
&=\frac{1}{N^{2m}}\sum_{\substack{1\leq k^{(1)},\dots,k^{(m)}\leq N \\ 1\leq k'^{(1)},\dots,k'^{(m)}\leq N}}  \mathbb{E}\left[\left(1-p^*_{\varepsilon}(X^{(1)}_{k^{(1)}},\dots, X^{(m)}_{k^{(m)}})\right)\left(1-p^*_{\varepsilon}(X^{(1)}_{k'^{(1)}},\dots, X^{(m)}_{k'^{(m)}})\right)\right], \nonumber \\
%&=\frac{1}{N^{2m}}\sum_{\substack{1\leq k^{(1)},\dots,k^{(m)}\leq N \\ 1\leq k'^{(1)},\dots,k'^{(m)}\leq N}}\text{Cov}\left[p^*_{\alpha, \varepsilon}(X^{(1)}_{k^{(1)}},\dots, X^{(m)}_{k^{(m)}}), \,p^*_{\alpha, \varepsilon}(X^{(1)}_{k'^{(1)}},\dots, X^{(m)}_{k'^{(m)}})\right],
\end{align}
where $\mathbb{E}[(1-p^*_{\varepsilon}(X^{(1)}_{k^{(1)}},\dots, X^{(m)}_{k^{(m)}}))(1-p^*_{\varepsilon}(X^{(1)}_{k'^{(1)}},\dots, X^{(m)}_{k'^{(m)}}))]=0$,
\begin{comment}
    \begin{align*}
%&\text{Cov}\left[p^*_{\alpha, \varepsilon}(X^{(1)}_{k^{(1)}},\dots, X^{(m)}_{k^{(m)}}), \,p^*_{\alpha, \varepsilon}(X^{(1)}_{k'^{(1)}},\dots, X^{(m)}_{k'^{(m)}})\right] \\
&\mathbb{E}\left[\left(1-p^*_{\alpha, \varepsilon}(X^{(1)}_{k^{(1)}},\dots, X^{(m)}_{k^{(m)}})\right)\left(1-p^*_{\varepsilon}(X^{(1)}_{k'^{(1)}},\dots, X^{(m)}_{k'^{(m)}})\right)\right]=0,
\end{align*}
\end{comment}
if one of the following is true:
\begin{enumeratei}
	\item  $k^{(j)}\neq k'^{(j)}$, for every $j\in[m]$, which leverages the independence;
	\item $k^{(j_0)}=k'^{(j_0)}$  for some $j_0\in[m]$ and $k^{(j)}\neq k'^{(j)}$, for every $j\in[m]\backslash\{j_0\}$, due to the marginal feasibility condition (\ref{multiconst_equiv}) of the multimarginal Schr\"{o}dinger system.
\end{enumeratei}
Note that there are $(N(N-1))^m$ many terms in Case (i) and $mN(N(N-1))^{(m-1)}$ many terms in Case (ii). Thus we have
\begin{align}\label{term1,2}
(\ref{cross1})&\leq \frac{N^{2m}-N^m(N-1)^m-mN^m(N-1)^{m-1}}{N^{2m}}\exp\left(\frac{(4m+4)\|c\|_{\infty}}{\varepsilon}\right)\nonumber\\
%&=\exp\left(\frac{4\|c\|_{\infty}}{\varepsilon}\right)\left(\frac{N^{2m}-N^{m}(N-1+m)(N-1)^{m-1}}{N^{2m}}\right)\nonumber\\
&\leq\exp\left(\frac{(4m+4)\|c\|_{\infty}}{\varepsilon}\right)\left(1-(1-\frac{1}{N})^{m-1}\right) \leq\exp\left(\frac{(4m+4)\|c\|_{\infty}}{\varepsilon}\right)\frac{m}{N}.
%&\leq\exp\left(\frac{4\|c\|_{\infty}}{\varepsilon}\right)\left(\frac{m-1}{N}-\frac{(m-1)(m-2)}{2N^2}\right)\leq\exp\left(\frac{4\|c\|_{\infty}}{\varepsilon}\right)\frac{m}{N}.
\end{align}	
Combining (\ref{term1,0}), (\ref{term1,1}) and (\ref{term1,2}), we establish
\begin{equation}\label{term1}
\bE(\text{I}) \leq\frac{4m\|c\|_{\infty}^2}{N}+\frac{4m\varepsilon^2}{N}\exp\left(\frac{(4m+4)\|c\|_{\infty}}{\varepsilon}\right).
\end{equation}
\underline{\bf Term (II)}. Since $\widehat{\Phi}_{ \varepsilon}(\cdot)$ is strongly concave (cf. Lemma~\ref{sconc_empirical}), it follows from the Polyak-{\L}ojasiewicz (PL) inequality in Lemma \ref{PLineq} that
%\begin{equation*}
%\left|\widehat{\Phi}_{ \varepsilon}(\vf^*) - \widehat{\Phi}_{\varepsilon}( \hat{\vf}^*)\right| \leq\frac{\varepsilon\exp{\left(\frac{2\|c\|_{\infty}}{\varepsilon}\right)}}{2}\|\nabla\widehat{\Phi}_{\varepsilon}(\vf^*)\|_{\widehat{\mathscr{L}}_m}^2.\end{equation*}
%Hence, we get
\begin{equation}\label{PL0}
\mathbb{E}\left[\widehat{\Phi}_{ \varepsilon}(\vf^*) - \widehat{\Phi}_{ \varepsilon}( \hat\vf^*)\right]^2 \leq\frac{\varepsilon^2}{4} \exp{\left(\frac{(4m+4)\|c\|_{\infty}}{\varepsilon}\right)} \mathbb{E}\|\nabla\widehat{\Phi}_{ \varepsilon}(\vf^*)\|_{\widehat{\mathscr{L}}_m}^4.
\end{equation}
Using (\ref{empirical_gradient_dual_norm}), the empirical dual version of Lemma~\ref{norm}, we find that
\begin{align*}
&\mathbb{E}\|\nabla\widehat{\Phi}_{ \varepsilon}(\vf^*)\|_{\widehat{\mathscr{L}}_m}^2= \mathbb{E}\sum_{j=1}^m\int \Big[\int (1-p^*_{\varepsilon}) d\otimes_{i\neq j}\hat{\nu}^N_i\Big]^ 2d\hat{\nu}^N_j \\
%&=\sum_{j=1}^m\frac{1}{N}\sum_{1\leq k^{(j)}\leq N}\mathbb{E}\left[\left(\otimes_{i\neq j}\hat{\nu}^N_i\right)\left(1-p^*_{\varepsilon}(x_1,\dots,x_{j-1},X^{(j)}_{k^{(j)}},x_{j+1},\dots,x_m)\right)\right]^2\nonumber\\
&=\frac{1}{N^{2m-1}}\sum_{\substack{1\leq j\leq m\\ 1\leq k^{(j)}\leq N}}\mathbb{E}\Big[\sum_{1\leq \ell^{(1)},\dots,\ell^{(j-1)},\ell^{(j+1)},\dots, \ell^{(m)}\leq N}\left(1-p^*_{\varepsilon}(X^{(1)}_{\ell^{(1)}},\dots,X^{(j-1)}_{\ell^{(j-1)}},X^{(j)}_{k^{(j)}},X^{(j+1)}_{\ell^{(j+1)}},\dots,X^{(m)}_{\ell^{(m)}})\right)\Big]^2\nonumber\\
\end{align*}
	\begin{align}\label{cross2}
	=\frac{1}{N^{2m-1}} \sum_{\substack{
			1 \leq j \leq m \\
			1 \leq k^{(j)} \leq N \\
			1 \leq \ell^{(1)}, \dots, \ell^{(j-1)}, \ell^{(j+1)}, \dots, \ell^{(m)} \leq N \\
			1 \leq \ell'^{(1)}, \dots, \ell'^{(j-1)}, \ell'^{(j+1)}, \dots, \ell'^{(m)} \leq N
	}} &\mathbb{E} 
	\left[ 
	\left( 1 - p^*_{\varepsilon}(X^{(1)}_{\ell^{(1)}}, \dots, X^{(j-1)}_{\ell^{(j-1)}}, X^{(j)}_{k^{(j)}}, X^{(j+1)}_{\ell^{(j+1)}}, \dots, X^{(m)}_{\ell^{(m)}}) 
	\right) 
	\right.\nonumber\\
	&\quad\times 
	\left. 
	\left( 1 - p^*_{\varepsilon}(X^{(1)}_{\ell'^{(1)}}, \dots, X^{(j-1)}_{\ell'^{(j-1)}}, X^{(j)}_{k^{(j)}}, X^{(j+1)}_{\ell'^{(j+1)}}, \dots, X^{(m)}_{\ell'^{(m)}}) 
	\right)
	\right],
\end{align}
\begin{comment}
\begin{align*}
       =\frac{1}{N^{2m-1}} \sum_{\substack{
			1 \leq j \leq m \\
			1 \leq k^{(j)} \leq N \\
			1 \leq \ell^{(1)}, \dots, \ell^{(j-1)}, \ell^{(j+1)}, \dots, \ell^{(m)} \leq N \\
			1 \leq \ell'^{(1)}, \dots, \ell'^{(j-1)}, \ell'^{(j+1)}, \dots, \ell'^{(m)} \leq N
	}} &\text{Cov} 
	\left[ 
	\left( 1 - p^*_{\varepsilon}(X^{(1)}_{\ell^{(1)}}, \dots, X^{(j-1)}_{\ell^{(j-1)}}, X^{(j)}_{k^{(j)}}, X^{(j+1)}_{\ell^{(j+1)}}, \dots, X^{(m)}_{\ell^{(m)}}) 
	\right),
	\right.\nonumber\\
	&\quad
	\left. 
	\left( 1 - p^*_{\varepsilon}(X^{(1)}_{\ell'^{(1)}}, \dots, X^{(j-1)}_{\ell'^{(j-1)}}, X^{(j)}_{k^{(j)}}, X^{(j+1)}_{\ell'^{(j+1)}}, \dots, X^{(m)}_{\ell'^{(m)}}) 
	\right)
	\right]
\end{align*}
\end{comment}
where the summand in the last expression equals to
\begin{align*}
	&\text{Cov}\left[p^*_{\varepsilon}(X^{(1)}_{\ell^{(1)}}, \dots, X^{(j-1)}_{\ell^{(j-1)}}, X^{(j)}_{k^{(j)}}, X^{(j+1)}_{\ell^{(j+1)}}, \dots, X^{(m)}_{\ell^{(m)}}) , \,p^*_{\varepsilon}(X^{(1)}_{\ell'^{(1)}}, \dots, X^{(j-1)}_{\ell'^{(j-1)}}, X^{(j)}_{k^{(j)}}, X^{(j+1)}_{\ell'^{(j+1)}}, \cdots, X^{(m)}_{\ell'^{(m)}}) \right]\\
	&=0, \qquad \text{whenever}
\end{align*}
\begin{enumeratei}
    \setcounter{enumi}{2}
	\item  $\ell^{(i)}\neq \ell'^{(i)}$, for every $i\neq j$, once again due to the marginal feasibility constraints (\ref{multiconst_equiv}) of the multimarginal Schr\"{o}dinger system.
\end{enumeratei}
Note that there are $mN\left(N(N-1)\right)^{(m-1)}$ many terms in case (iii), so we can bound %\xc{at least?}\pl{I think case 3 has exactly these many terms, although there are maybe more terms such that the covariance is zero} 
\begin{align}\label{PL1}
(\ref{cross2}) & \leq\frac{ m N^{2m-1} - m N^m(N-1)^{m-1}}{N^{2m-1}}\exp\Big(\frac{(2m+2)\|c\|_{\infty}}{\varepsilon}\Big)\nonumber\\
&=m\exp\Big(\frac{(2m+2)\|c\|_{\infty}}{\varepsilon}\Big)\Big(1-(1-\frac{1}{N})^{(m-1)}\Big)\leq\exp\Big(\frac{(2m+2)\|c\|_{\infty}}{\varepsilon}\Big)\frac{m^2}{N}.
\end{align}
At the same time, since the dual potentials are bounded as Proposition \ref{BDP}, we have
\begin{equation}\label{PL2}
\|\nabla\widehat{\Phi}_{ \varepsilon}(\vf^*)\|_{\widehat{\mathscr{L}}_m}^2 = \sum_{j=1}^m\int \Big(\int (1-p^*_{\varepsilon}) d\hat{\nu}^N_{-j}\Big)^ 2d\hat{\nu}^N_j 
\leq 
m\exp\Big(\frac{(4m+4)\|c\|_{\infty}}{\varepsilon}\Big).
\end{equation}
Combining (\ref{PL0}), (\ref{cross2}), (\ref{PL1}) and (\ref{PL2}), we get
\begin{equation}\label{PL}
	\mathbb{E}\left[\widehat{\Phi}_{ \varepsilon}(\vf^*) - \widehat{\Phi}_{\varepsilon}( \hat{\vf}^*)\right]^2 \leq\frac{m^3\varepsilon^2}{4N} \exp{\left(\frac{(10m+10)\|c\|_{\infty}}{\varepsilon}\right)}.
\end{equation}
Finally, putting all (\ref{decom}), (\ref{term1}) and (\ref{PL}) pieces together, we obtain the desired bound
\begin{equation}
\mathbb{E}\left|\Phi_{ \varepsilon}(\vf^*) - \widehat{\Phi}_{\varepsilon}(\hat{\vf}^*)\right|^2\leq \frac{C}{N},
\end{equation}
where $C= C({m,\varepsilon,\|c\|_\infty}) > 0$ is a constant depending only on $m,\varepsilon,\|c\|_\infty$.
\end{proof}

\begin{proof}[Proof of Theorem \ref{CLT_for_cost_functional}]
Decomposing
$$\widehat{\Phi}_{ \varepsilon}(\hat{\vf}^*) - \Phi_{\varepsilon}(\vf^*) 
=\widehat{\Phi}_{\varepsilon}(\vf^*) - \Phi_{\varepsilon}(\vf^*)
+\widehat{\Phi}_{ \varepsilon}(\hat{\vf}^*) - \widehat{\Phi}_{ \varepsilon}(\vf^*):=(\text{I})+(\text{II}).$$
\noindent\underline{\bf Term (I)}
To shorten the notation, we write 
\begin{align*}
    U_N &: =\widehat{\Phi}_{\varepsilon}(\vf^*) - \Phi_{\varepsilon}(\vf^*)\\ 
    & = \sum_{j=1}^m\frac{1}{N}\sum_{k=1}^N \left( f^*_j(X^{(j)}_k) - \int f^*_jd\nu_j \right)  -  \frac{\varepsilon}{N^m}  \sum_{1\leq k^{(1)},\dots,k^{(m)}\leq N} \Big( p^*_{\varepsilon}(X^{(1)}_{k^{(1)}},\dots,X^{(m)}_{k^{(m)}}) - 1\Big).
\end{align*}

\noindent Now, we define 
\begin{equation*}
    U_{N,\,\text{proj}} := \sum_{j=1}^m\frac{1}{N}\sum_{k=1}^N \left( f^*_j(X^{(k)}_j) - \int f^*_jd\nu_j \right).
\end{equation*}

\noindent It is easily observed that  $\text{Var}(U_{N,\,\text{proj}})
= \frac{1}{N}
\sum_{j=1}^m\text{Var}_{X_j\sim\nu_j}(f^*_j(X_j))$ and hence one easily obtains that
\begin{equation*}
    \sqrt{N}\, U_{N,\,\text{proj}}\overset{w}{\longrightarrow} \cN\left(0,\sum_{j=1}^m\text{Var}_{X_j\sim\nu_j}(f^*_j(X_j))\right).
\end{equation*}

\noindent Denote $R_N = U_N-U_{N,\,\text{proj}} = \frac{-\varepsilon}{N^m}  \sum_{k^{(1)},\dots,k^{(m)}} \Big( p^*_{\varepsilon}(X^{(1)}_{k^{(1)}},\dots,X^{(m)}_{k^{(m)}}) - 1\Big)$ and Cauchy-Schwarz inequality gives
\begin{align}\label{Var_Cauchy}
   \left| \text{Var}(U_{N}) - (\text{Var}(U_{N,\,\text{proj}})+\text{Var}(R_N)) \right| \leq 2\sqrt{\text{Var}(U_{N,\,\text{proj}}) \text{Var}(R_N)}.\nonumber
\end{align}

\noindent The marginal feasibility constrain (\ref{multiconst_equiv}) for the multimarginal Schr\"{o}dinger system implies that 
\begin{align*}
    \text{Var}(R_{N}) & =\bE R^2_{N} = \frac{\varepsilon^2}{N^{2m}} \bE \left[ \sum_{1\leq k^{(1)},\dots,k^{(m)}\leq N} \Big( p^*_{\varepsilon}(X^{(1)}_{k^{(1)}},\dots,X^{(m)}_{k^{(m)}}) - 1\Big) \right]^2\\
    &=\frac{\varepsilon^2}{N^{2m}}  \sum_{\substack{1\leq k^{(1)},\dots,k^{(m)}\leq N \\ 1\leq k'^{(1)},\dots,k'^{(m)}\leq N}}  \mathbb{E}\left[\left(1-p^*_{\varepsilon}(X^{(1)}_{k^{(1)}},\dots, X^{(m)}_{k^{(m)}})\right)\left(1-p^*_{\varepsilon}(X^{(1)}_{k'^{(1)}},\dots, X^{(m)}_{k'^{(m)}})\right)\right], \nonumber\\
    &= \frac{\varepsilon^2}{N^{2m}} \sum_{r=0}^m  \sum_{|\vS|=r} \bE \left( 1-p^*_{\varepsilon}(\vX_{\vS},\vX_{\vS^c}) \right) \left( 1-p^*_{\varepsilon} (\vX_{\vS},\vX'_{\vS^c}) \right).
\end{align*}
Here, to lighten the notation, we use $\vX_{\vS}$ to indicate the same elements between $(X^{(1)}_{k^{(1)}},\dots, X^{(m)}_{k^{(m)}})$ and $(X^{(1)}_{k'^{(1)}},\dots, X^{(m)}_{k'^{(m)}})$. $\vX_{\vS^c}$(resp. $\vX'_{\vS^c}$) denotes the remaining elements in 
$(X^{(1)}_{k^{(1)}},\dots, X^{(m)}_{k^{(m)}})$ (resp.$(X^{(1)}_{k'^{(1)}},\dots, X^{(m)}_{k'^{(m)}})$ ). Note that 
$$\bE \left( 1-p^*_{\varepsilon}(\vX_{\vS},\vX_{\vS^c}) \right) \left( 1-p^*_{\varepsilon} (\vX_{\vS},\vX'_{\vS^c}) \right)=0,$$
if one of the following is true:
\begin{enumeratei}
\item $|\vS|=0$, namely, $k^{(j)}\neq k'^{(j)}$, for every $j\in[m]$, which leverages the independence;
\item $|\vS|=1$, namely, $k^{(j_0)}=k'^{(j_0)}$  for some $j_0\in[m]$ and $k^{(j)}\neq k'^{(j)}$, for every $j\in[m]\backslash\{j_0\}$, due to the marginal feasibility condition (\ref{multiconst_equiv}) of the multimarginal Schr\"{o}dinger system.
\end{enumeratei}
As a consequence,
\begin{align}
    \text{Var}(R_N) & = \frac{\varepsilon^2}{N^{2m}} \sum_{r=2}^m  \sum_{|\vS|=r} \bE \left( 1-p^*_{\varepsilon}(\vX_{\vS},\vX_{\vS^c}) \right) \left( 1-p^*_{\varepsilon} (\vX_{\vS},\vX'_{\vS^c}) \right)\nonumber\\
    &\leq \frac{\varepsilon^2}{N^{2m}} \sum_{r=2}^m \binom{m}{r} N^r(N(N-1))^{(m-r)} e^{\frac{(2m+2)\|c\|_\infty}{\varepsilon}}\lesssim_\varepsilon \sum_{r=2}^m\frac{1}{N^r}\lesssim \frac{1}{N^2}.
\end{align}
So we have $\frac{\text{Var}(U_N)}{\text{Var}(U_{N,\,\text{proj}})}\to 1$ as $N\to\infty$  in combination with (\ref{Var_Cauchy}) and thus 
$$\frac{U_N}{\sqrt{\text{Var}(U_N)}}-\frac{U_{N,\,\text{proj}}}{\sqrt{\text{Var}(U_{N,\,\text{proj}}})}\overset{\bP}{\longrightarrow} 0.$$
Slutsky's theorem yields
\begin{equation}\label{ConvNormalI}
    \sqrt{N}\, U_{N}\overset{w}{\longrightarrow} \cN\left(0,\sum_{j=1}^m\text{Var}_{X_j\sim\nu_j}(f^*_j(X_j))\right).
\end{equation}

\noindent\underline{\bf Term (II)}
Since $\widehat{\Phi}_{ \varepsilon}(\cdot)$ is strongly concave (cf. Lemma~\ref{sconc_empirical}), it follows from the Polyak-{\L}ojasiewicz (PL) inequality in Lemma \ref{PLineq} that
\begin{equation}
    0\leq \widehat{\Phi}_{ \varepsilon}(\hat{\vf}^*) - \widehat{\Phi}_{\varepsilon}(\vf^*)  
    \leq 
    \frac{\varepsilon}{2}\exp\left(\frac{(2m+2)\|c\|_\infty}{\varepsilon}\right)
    \|\nabla\widehat{\Phi}_{\varepsilon}(\vf^*)\|_{\widehat{\mathscr{L}}_m}^2.
\end{equation}

\noindent Using (\ref{empirical_gradient_dual_norm}), the empirical dual version of Lemma~\ref{norm}, we find that
\begin{align*}
&\mathbb{E}\|\nabla\widehat{\Phi}_{\varepsilon}(\vf^*)\|_{\widehat{\mathscr{L}}_m}^2= \mathbb{E}\sum_{j=1}^m\int \Big[\int (1-p^*_{\varepsilon}) d\otimes_{i\neq j}\hat{\nu}^N_i\Big]^ 2d\hat{\nu}^N_j \\
&=\frac{1}{N^{2m-1}}\sum_{\substack{1\leq j\leq m\\ 1\leq k^{(j)}\leq N}}\mathbb{E}\Big[\sum_{1\leq \ell^{(1)},\dots,\ell^{(j-1)},\ell^{(j+1)},\dots, \ell^{(m)}\leq N}\left(1-p^*_{\varepsilon}(X^{(1)}_{\ell^{(1)}},\dots,X^{(j-1)}_{\ell^{(j-1)}},X^{(j)}_{k^{(j)}},X^{(j+1)}_{\ell^{(j+1)}},\dots,X^{(m)}_{\ell^{(m)}})\right)\Big]^2\nonumber\\
\end{align*}
	\begin{align*}
	=\frac{1}{N^{2m-1}} \sum_{\substack{
			1 \leq j \leq m \\
			1 \leq k^{(j)} \leq N \\
			1 \leq \ell^{(1)}, \dots, \ell^{(j-1)}, \ell^{(j+1)}, \dots, \ell^{(m)} \leq N \\
			1 \leq \ell'^{(1)}, \dots, \ell'^{(j-1)}, \ell'^{(j+1)}, \dots, \ell'^{(m)} \leq N
	}} &\mathbb{E} 
	\left[ 
	\left( 1 - p^*_{\varepsilon}(X^{(1)}_{\ell^{(1)}}, \dots, X^{(j-1)}_{\ell^{(j-1)}}, X^{(j)}_{k^{(j)}}, X^{(j+1)}_{\ell^{(j+1)}}, \dots, X^{(m)}_{\ell^{(m)}}) 
	\right) 
	\right.\nonumber\\
	&\quad\times 
	\left. 
	\left( 1 - p^*_{\varepsilon}(X^{(1)}_{\ell'^{(1)}}, \dots, X^{(j-1)}_{\ell'^{(j-1)}}, X^{(j)}_{k^{(j)}}, X^{(j+1)}_{\ell'^{(j+1)}}, \dots, X^{(m)}_{\ell'^{(m)}}) 
	\right)
	\right],\nonumber\\
\end{align*}
    \begin{align*}
        =\frac{1}{N^{2m-1}}\sum_{r=1}^m \sum_{|\vS|=r} \bE (1-p^*_{\varepsilon}(\vX_{\vS},\vX_{\vS^c})) (1-p^*_{\varepsilon}(\vX_{\vS},\vX'_{\vS^c})).
\end{align*}
where the summand in the last expression equals to 0 whenever 
\begin{enumeratei}
    \setcounter{enumi}{2}
	\item  $|\vS|=1$, namely, $\ell^{(j)}\neq \ell'^{(j)}$, for every $j\in[m]$, once again due to the marginal feasibility constraints (\ref{multiconst_equiv}) of the multimarginal Schr\"{o}dinger system.
\end{enumeratei}
So we can bound  
\begin{align}
\bE \left[\widehat{\Phi}_{ \varepsilon}(\hat{\vf}^*) - \widehat{\Phi}_{ \varepsilon}(\vf^*)\right] 
&\leq \frac{\varepsilon}{2N^{2m-1}}\exp\left(\frac{(2m+2)\|c\|_\infty}{\varepsilon}\right) \sum_{r=2}^m \sum_{|\vS|=r} \bE (1-p^*_{\varepsilon}(\vX_{\vS},\vX_{\vS^c})) (1-p^*_{\varepsilon}(\vX_{\vS},\vX'_{\vS^c}))\nonumber\\
& \leq \frac{\varepsilon}{2N^{2m-1}}\sum_{r=2}^m
r \binom{m}{r} N^r(N(N-1))^{(m-r)} e^{\frac{(4m+4)\|c\|_\infty}{\varepsilon}}\nonumber\lesssim_\varepsilon \frac{1}{N}.\nonumber
\end{align}
Thus, 
\begin{equation}\label{Convprob0II}
\sqrt{N}\left[\widehat{\Phi}_{ \varepsilon}(\hat{\vf}^*) - \widehat{\Phi}_{ \varepsilon}(\vf^*)\right]\overset{\bP}{\longrightarrow}0.
\end{equation}
Combining (\ref{ConvNormalI}), (\ref{Convprob0II}) and Slutsky's Lemma, the proof is complete.
\end{proof}

\section{Weak limit of expectation under entropic optimal transport coupling}\label{Sec_Weak_limit_of_expectation_under_entropic_optimal_transport_coupling}
This section is primarily concerned with proving Theorem \ref{CLT_test_function_very_complicated_computation}. For better readability, we describe the main structure of the argument here, while the more technical components are provided in the referenced sections.

\begin{proof}[Proof of Theorem \ref{CLT_test_function_very_complicated_computation}]
In the remainder of the proof, to simplify the notation, we assume, without loss of generality, $\cX_j = [-1,1]^d$ for $j\in[m]$. Also, we write  $\bC(\mvx):= \exp\left({-\frac{c(\mvx)}{\varepsilon}}\right)$. At the same time, due to the specifics of multimarginal Schr\"{o}dinger system, we use the quotient space
\[
\widetilde{\cC}^1:= \cC^1/ \sim,
\]
with $\cC^1 = \prod_{j=1}^m \cC^1(\mathcal{X}_j) $.
Now we consider the map
\[
\widehat{\mvT} : \widetilde{\cC}^1 \to \widetilde{\cC}^1 
\]
defined for $j \in [m]$ and $x_j \in [-1,1]^d$ as
\begin{align*}
\widehat{T}_j(\mvphi)(x_j) :=& \quad \varepsilon\log \left( \int_{\mathcal{X}_{-j}} e^{ \frac{\sum_{i=1}^m \phi_i(x_i) - c(x_j, \, x_{-j})}{\varepsilon}} \, d\hat{\nu}^N_{-j}(x_{-j}) \right).
\end{align*}
The corresponding empirical Schrödinger potential $[\hat{\mvf}^*]= (\hat{f}^*_1, \dots, \hat{f}^*_m) \in \widetilde{\cC}^1$ associated with $(\hat{\nu}^N_1,\dots,\hat{\nu}^N_m)$ solves the empirical Schrödinger system
$$\widehat{\mvT}([\hat{\mvf}^*]) = 0.$$
Similarly, the map
\[
\mvT : \widetilde{\cC}^1 \to \widetilde{\cC}^1 
\]
is defined for $j \in [m]$ and $x_j \in [-1,1]^d$ as
\begin{align*}
T_j(\mvphi)(x_j) :=& \quad \varepsilon\log \left( \int_{\mathcal{X}_{-j}} e^{ \frac{\sum_{i=1}^m \phi_i(x_i) - c(x_j, \, x_{-j})}{\varepsilon}} \, d\nu_{-j}(x_{-j}) \right).
\end{align*}
The corresponding population Schrödinger potential $[\mvf^*]= (f^*_1, \dots, f^*_m) \in \widetilde{\cC}^1$ associated with $(\nu_1,\dots,\nu_m)$ solves the population Schrödinger system
$$\mvT([\mvf^*]) = 0.$$

\noindent \underline{\textbf{Step 1: Linearization of the empirical EMOT system: preliminary step}}. 
\begin{lemma}\label{Linearization_of_the_empirical_EMOT_system} Under standard assumptions,
\begin{align*}
\left \|\widehat{\mvGamma} \left( [\mvf^*] - [\hat{\mvf}^*] \right) - \varepsilon \fL \right\|_{\cC^1} 
=
\cO_\bP \Bigl( \frac{d \log N}{N} \bigr),
\end{align*}
where 
\begin{align}
 \fL := \begin{pmatrix}
   B_{-1}(\hat{\nu}^N_{-1} - \nu_{-1} ) \\
    \vdots\\
    B_{-m}(\hat{\nu}^N_{-m} - \nu_{-m} ) 
\end{pmatrix}
\end{align}
with
\begin{align}
    B_{-i}(\hat{\nu}_{-i}^N - \nu_{-i}) := \int p_\varepsilon^*(x_i, x_{-i}) d (\hat{\nu}_{-i}^N - \nu_{-i})(x_{-i}).
\end{align}
\end{lemma}
\begin{lemma}\label{order_of_centered_linerized_B_i}
Under standard assumptions,
\begin{align}
    \| B_{-i}( \hat{\nu}^N_{-i} - \nu_{-i})\|_{\cC^1} = \cO_\bP \left(\sqrt{\frac{d\log N}{N}} \right).
\end{align}
\end{lemma}

The proof of Lemma \ref{Linearization_of_the_empirical_EMOT_system} and Lemma \ref{order_of_centered_linerized_B_i} could be found in Appendix \ref{appendix_lineratization_empirical_EMOT}.

\noindent \underline{\textbf{Step 2: Constructing auxiliary operators to approximate empirical operators}}.

\noindent Define for $i\neq j$, $D_{ij} : \cC^1([-1,1]^d)\to \cC^1([-1,1]^d)$ as
\begin{align*}
    (D_{ij} f_j) (x_i) 
    =
    \int f_j(x_j) p^*_\varepsilon(x_i, x_{-i})d \hat{\nu}^N_{-i}(x_{-i}).
\end{align*}
We could consider $\widetilde{\mvGamma}: \cC^1 \to \cC^1 $ 
\begin{align}
    \begin{pmatrix}
    f_1\\ \vdots\\f_m
    \end{pmatrix}
        \mapsto 
    \begin{pmatrix}
    f_1\\ \vdots\\f_m
    \end{pmatrix}
    +
    \begin{pmatrix}
    \sum_{k\neq 1}D_{1k}f_k\\ 
    \vdots\\ 
    \sum_{k\neq m}D_{mk}f_k 
        \end{pmatrix} 
    : = (\text{I} + \bD) \begin{pmatrix}
    f_1\\ \vdots\\f_m
    \end{pmatrix}.
\end{align}

\noindent \underline{\textbf{Step 3: Bridging auxiliary operators and population operators}}.

\noindent The Fr\'{e}chet derivative $\widehat{\mvGamma} : \widetilde{\cC}^1 \to \widetilde{\cC}^1 $ of $\widehat{\mvT}$ could be computed as
\begin{align}
    \begin{pmatrix}
    f_1\\ \vdots\\f_m
    \end{pmatrix}
        \mapsto 
    \begin{pmatrix}
    f_1\\ \vdots\\f_m
    \end{pmatrix}
    +
    \begin{pmatrix}
    \sum_{k\neq 1}A_{1k}f_k \\ 
    \vdots\\ 
    \sum_{k\neq m}A_{mk}f_k 
        \end{pmatrix} 
    : = ( \text{I} +\bA) \begin{pmatrix}
    f_1\\ \vdots\\f_m
    \end{pmatrix},
\end{align}
where $A_{ij} : \cC^1([-1,1]^d)\to \cC^1([-1,1]^d)$ is defined for $i\neq j$ as
\begin{align*}
    (A_{ij} f_j) (x_i) 
    =
    \int f_j(x_j) \hat{p}_\varepsilon(x_i, x_{-i})d \hat{\nu}^N_{-i}(x_{-i}).
\end{align*}
It is easily observed that $\text{I} + \bA$ is the Fr\'{e}chet derivative $\widehat{\mvGamma} : \widetilde{\cC}^1 \to \widetilde{\cC}^1 $ of $\widehat{\mvT}$. 
Define for $i\neq j$, $H_{ij} : \cC^1([-1,1]^d)\to \cC^1([-1,1]^d)$ as
\begin{align*}
    (H_{ij} f_j) (x_i) 
    =
    \int f_j(x_j) p^*_\varepsilon(x_i, x_{-i})d \nu_{-i}(x_{-i}).
\end{align*}
We could consider $\mvGamma: \widetilde{\cC}^1 \to \widetilde{\cC}^1 $ 
\begin{align}
    \begin{pmatrix}
    f_1\\ \vdots\\f_m
    \end{pmatrix}
        \mapsto 
    \begin{pmatrix}
    f_1\\ \vdots\\f_m
    \end{pmatrix}
    +
    \begin{pmatrix}
    \sum_{k\neq 1}H_{1k}f_k \\ 
    \vdots\\ 
    \sum_{k\neq m}H_{mk}f_k 
        \end{pmatrix} 
    : = (\text{I} + \bH) \begin{pmatrix}
    f_1\\ \vdots\\f_m
    \end{pmatrix}.
\end{align}
It is also easily observed that $\text{I} + \bH$ is the Fr\'{e}chet derivative $\mvGamma : \widetilde{\cC}^1 \to \widetilde{\cC}^1 $ of $\mvT$.

\noindent \underline{\textbf{Step 4: Operator invertibility and operator norm bounds}}. 

\noindent As demonstrated by Lemma 3.2 in \cite{CalierChizatLaborde_DisplacementSmoothnessofEOT}, one could leverage Fredholm Alternative Theorem to show that $\mvGamma(\cdot)$ is an invertible linear self-map of $\widetilde{\cC}^1$ satisfying 
\begin{equation*}
    \left\|\mvGamma^{-1}\right\|\leq C,
\end{equation*}
for some constant $C$ depending on the marginals $\nu_1,\dots,\nu_m$ and the cost function $c$. Now we consider the operator norm bound of $\|\bA-\bH\|$.

\begin{lemma}\label{bound_on_norm_A-H}
For any $i,j\in[m]$ with $i\neq j$, we have that
\begin{equation}
        \| \bA - \bH\|_{\widetilde{\cC}^1\to\widetilde{\cC}^1}  = \cO_\bP \left( (m-1)\, \sqrt{\frac{d \log N}{N}} + (m-1)^2\,N^{-1/(2\vee d)} (\log N)^{\mathbf 1_{\{d=2\}}} \right).
\end{equation}
\end{lemma}

\noindent \underline{\textbf{Step 5: Linearization of the empirical EMOT system: final step}}.

\noindent This section is to conclude the linearization of the empirical EMOT system and thus establish the building block of the weak limit proof, given in the following lemma.
\begin{lemma}\label{Linearization_potential}
Under standard assumptions, we have
\begin{align*}
    \left \| \left( [\mvf^*] - [\hat{\mvf}^*] \right) -\varepsilon \mvGamma^{-1}\fL \right\|_{\cC^1} = o_\bP \left( N^{-1/2} \right).
\end{align*}
\end{lemma}

\begin{comment}
\pl{1. Control the $C^1$ norm of $\hat{\mvf} -\mvf$ is doable. See the $b_n-b^*$ bound in RS.}

\pl{2. This proof could be simplified; maybe we do not need to consider the events with growing probability any more. The idea is to start with the invertibility of $\mvGamma$, whose invertibility is always guaranteed. We don't delve into any invertiblility until we transfer the linearization term from $\hat{\mvGamma}$ to $\mvGamma$.}
\end{comment}

\noindent \underline{\textbf{Step 6: Decomposition of $\sqrt{N} \left(\hat{\pi}^N_\varepsilon (g) - \pi^*_\varepsilon (g) \right)$}}.
\begin{equation}\label{decomposition_four_part_CLT_coupling}
    \sqrt{N} \left( \hat{\pi}^N_\varepsilon(g) - \pi_\varepsilon^* (g) \right)= \sqrt{N} \int g\,(d\hat{\pi}^N_\varepsilon-d\pi^*_\varepsilon) := A + B + \sum_{j=1}^m C_j + D.
\end{equation}
Here,
\begin{equation*}
    A := \sqrt{N}\int g (\mvx) e^{-\frac{c(\mvx)}{\varepsilon}}(e^{\frac{\sum_{i=1}^m \hat{f}^*_i(x_i)}{\varepsilon}} - e^{\frac{\sum_{i=1}^m f^*_i(x_i)}{\varepsilon}})d(\otimes_{k=1}^m\nu_k)(\mvx).
\end{equation*}

\begin{equation*}
    B := \sqrt{N}\int g (\mvx) e^{-\frac{c(\mvx)}{\varepsilon}}(e^{\frac{\sum_{i=1}^m \hat{f}^*_i(x_i)}{\varepsilon}} - e^{\frac{\sum_{i=1}^m f^*_i(x_i)}{\varepsilon}})d\left(\otimes_{k=1}^m\hat{\nu}^N_k - \otimes_{k=1}^m\nu_k\right)(\mvx).
\end{equation*}
To lighten the notation, for $\mvr=(r_1,\dots,r_m)\in\{0,1\}^m$, we write $\tilde{\nu}_{r_j}=\nu_j$ if $r_j=0$ and $\tilde{\nu}_{r_j}=\hat{\nu}^N_j-\nu_j$ if $r_j=1$. $\sharp\mvr=\sum_{j=1}^m r_j$. We define $\tilde{\mvnu}_{\mvr}:=\otimes \tilde{\nu}_{r_j}$. For example, $ \tilde{\mvnu}_{\mvr}=(\hat{\nu}_1^N-\nu_1)\otimes\nu_{2}\otimes\dots\otimes\nu_m$ if $\mvr=(1,0,\dots,0)$,  $ \tilde{\mvnu}_{\mvr}=(\hat{\nu}_1^N-\nu_1)\otimes\nu_{2}\otimes(\hat{\nu}_3^N-\nu_3)\otimes\nu_4\dots\otimes\nu_m$ if $\mvr=(1,0,1,0\dots,0)$.
\begin{equation*}
    C_j := \sqrt{N}\int g (\mvx) e^{-\frac{c(\mvx)}{\varepsilon}}e^{\frac{\sum_{i=1}^m f^*_i(x_i)}{\varepsilon}} d(\nu_1\otimes\dots\nu_{j-1}\otimes (\hat{\nu}_j^N-\nu_j)\otimes\nu_{j+1}\otimes\dots\otimes\nu_m)(\mvx).
\end{equation*}
\begin{equation*}
    D := \sum _{\sharp\mvr=2}^m \sum_{ \tilde{\mvnu}_{\mvr}} \sqrt{N}\int  g  (\mvx) e^{-\frac{c(\mvx)}{\varepsilon}}e^{\frac{\sum_{i=1}^m f^*_i(x_i)}{\varepsilon}} d \tilde{\mvnu}_{\mvr}(\mvx).
\end{equation*}
Lemma \ref{exponential_potential_infinity_norm_order} gives that 
$$A = \frac{\sqrt{N}}{\varepsilon}\int g (\mvx) e^{\frac{\sum_{k=1}^m f^*_k(x_k)-c(\mvx)}{\varepsilon}}\left(\sum_{i=1}^m \hat{f}^*_i(x_i)-f^*_i(x_i)\right)d(\otimes_{k=1}^m\nu_k)(\mvx) + o_\bP(1).$$
Using (\ref{Linearization_potential}), we have that $A = \bar{A} + o_\bP(1)$ with
$$\bar{A} = -\sqrt{N} \bigint \bigoplus\left(
      \mvGamma^{-1}
  \begin{pmatrix}
        \displaystyle \int p^*_\varepsilon(x_1, x_{-1}) d (\hat{\nu}_{-1}^{N} - \nu_{-1})\\ 
        \vdots\\
        \displaystyle\int p^*_\varepsilon(x_m, x_{-m}) d (\hat{\nu}_{-m}^{N} - \nu_{-m})
    \end{pmatrix}  \right) g\, p^*_{\varepsilon} \,\,\, d(\otimes_{k=1}^m\nu_k).$$
Lemma \ref{vanishing_term_B_proved_via_chaining} indicates that $B=o_\bP(1).$ Fubini's theorem gives that 
\begin{equation*}
    C_j = \sqrt{N} (\hat{\nu}_j^N-\nu_j)(g^{(j)})
\end{equation*}
where $g^{(j)}:\cX_j\to\bR$ is defined as $(\otimes_{k\neq j}\nu_k)(g\,p^*_{\varepsilon}).$

\noindent In advantage of Lemma B in Section 6.3.2 in \cite{Robert_Serfling_Approximation09}, we get that $D = o_\bP(1)$.

\noindent \underline{\textbf{Step 7: Manipulating the multi-sample V-statistics}}. Finally, we realize that the term of interest $(\ref{decomposition_four_part_CLT_coupling})$ is actually a centered V-statistics plus an extra asymptotically negligible term. At the same time, it is noteworthy that term $\bar{A}$ is a V-statistics of order $m$ and degree $(m-1,\dots,m-1)$ with a non-symmetric kernel and $C_j$ is a U-statistics of order 1 for $j\in[m]$. To derive the weak limit of this type, we need to first adjust and symmetrize the kernel and then use U-statistics to approximate V-statistics as Lemma \ref{V_statistics_U_statistics_difference}. Finally we apply the CLT for multi-sample U-statistics, thanks to Theorem 4.5.1 in \cite{Korolyuk_Borovskich_U_Statistics}, and then the conclusion follows. First define
\begin{align*}
    &\Psi(y^{(j)}_{\alpha}, \alpha = i_{j,1},\dots,i_{j,m-1}; j=1,\dots,m) \\
    &: =  - \bigint 
    \bigoplus\left(\mvGamma^{-1}\begin{pmatrix}
         p^*_{\varepsilon}(x_1,y^{(2)}_{i_{2,1}},\dots,y^{(m)}_{i_{m,1}})-1\\ 
         p^*_{\varepsilon}(y^{(1)}_{i_{1,1}},x_2,\dots,y^{(m)}_{i_{m,2}})-1\\ 
        \vdots\\
        p^*_{\varepsilon}(y^{(1)}_{i_{1,m-2}},y^{(2)}_{i_{2,m-2}},\dots,x_{m-1},y^{(m)}_{i_{m,m-1}})-1\\ 
        p^*_{\varepsilon}(y^{(1)}_{i_{1,m-1}},y^{(2)}_{i_{2,m-1}},\dots,y^{(m-1)}_{i_{m-1,m-1}},x_m)-1
    \end{pmatrix}  \right) g (\mvx) p^*_{\varepsilon} (\mvx)d(\otimes_{k=1}^m\nu_k) (\mvx).
\end{align*}
To symmetrize the kernel above, we further define
\begin{align*}
    &\quad \Psi^{\text{sym}}(y^{(j)}_{\alpha}, \alpha = i_{j,1},\dots,i_{j,m-1}; j=1,\dots,m) \\
    &= \frac{1}{\left((m-1)!\right)^m} \sum_{\substack{ \pi_j\in S_{m-1}\\1\leq j\leq m}}
    \Psi(y^{(j)}_{\beta}, \beta = i_{j,\pi_j(1)},\dots,i_{j,\pi_j(m-1)}; j=1,\dots,m),
\end{align*}
where $S_{m-1}$ denotes all the permutation on $\{1,\dots,m-1\}$. Finally, we define
\begin{align*}
    &\quad \cK(y^{(j)}_{\beta}, \beta = i_{j,1},\dots,i_{j,m-1}; j=1,\dots,m)\\
    &= \Psi^{\text{sym}}(y^{(j)}_{\beta}, \beta = i_{j,1},\dots,i_{j,m-1}; j=1,\dots,m) + \frac{1}{m-1}\sum_{j=1}^m\sum_{t=1}^{m-1} s_j(y_{i_{j,t}}^{(j)})
\end{align*}
with $s_j(\cdot)= (\nu_{-j})(g\,p^*_{\varepsilon})(\cdot) - (\otimes_{k=1}^m\nu_k)(g\,p^*_{\varepsilon})$. One could verify that
\begin{align*}
   \frac{1}{\sqrt{N}} (\bar{A} + \sum_{j=1}^m C_j) = \frac{1}{N^{m(m-1)}}\sum_\flat \cK(X^{(j)}_{\beta}, \beta = i_{j,1},\dots,i_{j,m-1}; j=1,\dots,m):=V_N, %+ \frac{1}{m-1}\sum_{j=1}^m\sum_{t=1}^{m-1} \bE s_j(X_t^{(j)}),
\end{align*}
where $\sum_\flat $ denotes the summation over all indices $1\leq i_{j,k}\leq N,\, 1\leq k\leq m-1,\, j=1,\dots, m.$ It is noteworthy that $V_N$ is centered here due to the multimarginal Schr\"{o}dinger system and the fact that the precomposition of a centered random
variable with a bounded linear operator is again centered. Thus, let us define
\begin{align*}
    U_N :=  \frac{1}{\tbinom{N}{m-1}^m }
        \sum_{\substack{1\leq i_{j,1}<\dots<i_{j,m-1}\leq N\\\, j=1,\dots,m}}
         \cK(X^{(j)}_{\beta}, \beta = i_{j,1},\dots,i_{j,m-1}; j=1,\dots,m).
\end{align*}
By Theorem 4.5.1 in \cite{Korolyuk_Borovskich_U_Statistics}, we get that
%Theorem 5.6.2 in \cite{Korolyuk_Borovskich_U_Statistics}) 
\begin{align*}
    \sqrt{N} U_N\overset{w}{\to} \cN(0,\bar{\sigma}_\varepsilon^2(g)), 
\end{align*}
where 
\begin{align*}
     \bar{\sigma}_\varepsilon^2(g)=(m-1)^2\sum_{j=1}^m \text{Var}(\phi^{(j)}(X_1^{(j)})),
\end{align*}
with 
\begin{align*}
    \phi^{(j^*)}(x) = \bE \left[  \cK(X^{(j)}_{\beta},  \beta = 1,\dots,m-1, j\in[m]\backslash\{j^*\}; x,X^{(j^*)}_{2},\dots,X^{(j^*)}_{m-1})  \right], 
\end{align*}
for $m$ independent sequences $X_1^{(j)},\dots, X_{m-1}^{(j)}\overset{\text{i.i.d.}}{\sim}\nu_j$ , $j\in[m]$ of independent random variables. Finally, Slutsky's Lemma and Lemma \ref{V_statistics_U_statistics_difference} concludes the proof.

\begin{comment}
\begin{remark}
multiple ways to verify the covariance:
\begin{enumerate}
    \item 
\begin{align*}
    &\sigma_\varepsilon^2(h) \\
    &= \lim_{N\to\infty}\frac{\left[\tbinom{N}{m-1} \right]^m}{N} \sum_{\substack{0\leq d_j\leq m-1 \\ 1\leq j\leq m}}\prod_{j=1}^m \tbinom{m-1}{d_j}\tbinom{N-m+1}{m-1-d_j} \xi_{d_1,\dots,d_m}
\end{align*}
with
\begin{align*}
    \xi_{d_1,\dots,d_m} :=  \text{Var} &\left(      \bE \left[  \cK(X^{(j)}_{\beta}, \beta = i_{j,1},\dots,i_{j,m-1}; j=1,\dots,m) |X^{(j)}_{j1},\dots,X^{(j)}_{jd_j}, 1\leq j\leq m  \right]      \right).
\end{align*}
\item U-statistics Hoeffding decomposition variance check.
\item specific choice of m (e.g. m=3).
\end{enumerate}
\end{remark}
\end{comment}

\end{proof}

\begin{proof}[Detailed verification of Remark \ref{remark_non_degenrate_variance}. ]
Define for $j\in[m]$
\[
b_j(x_j)=\int g(x_j,x_{-j})p_\varepsilon^*(x_j,x_{-j})
\,d\nu_{-j}(x_{-j})-\pi_\varepsilon^*(g).
\]
Also define the tuple of functions
$q^{j,y}=(q_1^{j,y},\ldots,q_m^{j,y})$ by
\begin{equation*}
q_i^{j,y}(x_i)
:=\begin{cases}
0, & i=j,\\[5pt]
\displaystyle
\int p_\varepsilon^*(x_1,\ldots,x_m)
\prod_{k\notin\{i,j\}}d \nu_k(x_k)
\biggr|_{x_j=y}-1,
& i\ne j.
\end{cases}
\end{equation*}
The multimarginal Schr\"{o}dinger system (\ref{multiconst_equiv}) yields that
$\int q_i^{j,y} (x_i) d \nu_i(x_i) =0$  for all $i\in[m]$ and $\int q_i^{j,y}(x_i)\,d\nu_j(y) = 0$.
For notational simplicity, we define the operator $\bigoplus : \widetilde{\cC}^1\to  \cC^1(\cX)$ by $\bigoplus [u] (\mvx)= \sum_{j=1}^m u_j(x_j)$.
We claim that 
 \begin{equation}
\label{eq:nd_variance_short_projection}
(m-1)\bar\phi^{(j)}(y)
=
b_j(y)
-
\pi_\varepsilon^*\!\left(
g\,\bigoplus(\mvGamma^{-1}[q^{j,y}])\right).
\end{equation}
Actually, this follows from the computation of the first-order projections
$\bar\phi^{(j)}(y)
=\mathbb E[\cK\mid X_1^{(j)}=y]$
of the kernel $\cK$
defining the variance in Theorem \ref{CLT_test_function_very_complicated_computation}.
The second marginal fluctuation term $\bE \bigl[ \frac1{m-1}\sum_{k=1}^m\sum_{t=1}^{m-1}b_k(X_t^{(k)})\mid X_1^{(j)}=y \bigr]$ equals $b_j(y)/(m-1)$ due to independence. 
As for the first term, let $R^{\boldsymbol{\tau}}$ be the input tuple to $\mvGamma^{-1}$
in the summand of $\Psi$ corresponding to
$\boldsymbol{\tau}\in S_{m-1}^{m}$.
For $i\ne j$, denote by $r_{ij}$ the position of the $j$-th marginal
observation in row $i$.
Independence and the multimarginal Schr\"{o}dinger system \eqref{multiconst_equiv} give,
for $\nu_j$-almost every $y$,
\[
\mathbb E\!\left[
R_i^{\boldsymbol{\tau}}(x_i)
\,\middle|\,X_1^{(j)}=y
\right]
=
\begin{cases}
0, & i=j,\\
\mathbf 1_{\{\tau_j(r_{ij})=1\}}q_i^{j,y}(x_i),
& i\ne j.
\end{cases}
\]
Since
$|\{\tau_j\in S_{m-1}:\tau_j(r_{ij})=1\}| = (m-2)!$
and $q_j^{j,y}=0$, averaging yields
\[
\mathbb E\,
\biggl[
\frac{1}{((m-1)!)^m}
\sum_{\boldsymbol{\tau}\in S_{m-1}^{m}}
R^{\boldsymbol{\tau}}
\, \biggm|\,X_1^{(j)}=y
\biggr]
=
\frac{q^{j,y}}{m-1}.
\]
Passing this conditional expectation through the bounded linear
maps in $\Psi$, we know that the first term contributes the $-\frac{1}{m-1}\pi_\varepsilon^*\!\left(
g\,\bigoplus\!\left(\mvGamma^{-1}[q^{j,y}]\right) \right)$ term.
Consequently, the asymptotic variance term in the theorem equals
\begin{equation}
\label{eq:nd-variance}
\begin{aligned}
\bar\sigma_\varepsilon^2(g)
&=(m-1)^2\sum_{j=1}^m
\operatorname{Var}_{\nu_j}(\bar\phi^{(j)})\\
&=\sum_{j=1}^m\int
\left[
b_j(y)-\pi_\varepsilon^*\!\left(
g\,\bigoplus\!\left(\mvGamma^{-1}[q^{j,y}]\right)
\right)
\right]^2\,d\nu_j(y).
\end{aligned}
\end{equation}
Suppose first that $\bar\sigma_\varepsilon^2(g)=0$. Then for $\nu_j$-almost every $y$, 
$$b_j(y)
=\pi_\varepsilon^*\!\left(
g\,\bigoplus\!\left(\mvGamma^{-1}[q^{j,y}]\right)
\right).$$
Let $u=(u_1,\ldots,u_m)$ be any $\cC^1$ tuple such that
$\nu_j(u_j)=0$ for every $j\in[m]$. The $i$-th component of
$\sum_j\int q^{j,y}u_j(y)\,d\nu_j(y)$ is
\[
\sum_{j\ne i}\int
u_j(x_j)p_\varepsilon^*(x_i,x_{-i})\,d\nu_{-i}(x_{-i}).
\]
Then in the quotient space $\widetilde{\cC}^1$, with $[\cdot]$ denoting the equivalence class, 
\begin{equation}
\label{eq:nd-tuple-identity}
\biggl[\sum_{j=1}^m\int q^{j,y}u_j(y)\,d\nu_j(y)\biggr]
=
\mvGamma[u]-[u].
\end{equation}
Hence, for the bounded linear operator $[v]\mapsto \pi_\varepsilon^*\,
\bigl(
g\,\bigoplus\!\left(\mvGamma^{-1}[v]
\bigr)
\right)$ on the quotient space $\widetilde{C}^1$, we obtain
\[
\begin{aligned}
\sum_{j=1}^m \nu_j(b_j u_j)
&=
\sum_{j=1}^m
\int
u_j(y)\,
\pi_\varepsilon^*\,
\bigl(
g\,\bigoplus\!\left(\mvGamma^{-1}[q^{j,y}]
\bigr)
\right)
\,d\nu_j(y)
\\
&=
\pi_\varepsilon^*\,
\biggl(
g\,\bigoplus\,
\Bigl(
\mvGamma^{-1}
\Bigl[
\sum_{j=1}^m
\int
q^{j,y}u_j(y)\,d\nu_j(y)
\Bigr]
\Bigr)
\biggr)
\\
&=
\pi_\varepsilon^*\!
\left(
g\,\bigoplus\,
\bigl(
\mvGamma^{-1}\bigl(\mvGamma[u]-[u]\bigr)
\bigr)
\right)\\
%=\pi_\varepsilon^*\!\left(g\,\bigoplus [u]\right)-\pi_\varepsilon^*\!\left(g\,\bigoplus(\mvGamma^{-1}[u])\right)\\
&=\sum_{j=1}^m\nu_j(b_ju_j)
-\pi_\varepsilon^*\!\left(g\,\bigoplus(\mvGamma^{-1}[u])\right).
\end{aligned}
\]
This implies that $\pi_\varepsilon^*\!\left(g\,\bigoplus(\mvGamma^{-1}[u])\right) = 0.$
Since \(q^{j,y}\) is a componentwise centered \(\cC^1\) tuple, take \(u=q^{j,y}\) and then equation \eqref{eq:nd_variance_short_projection} yields $b_j(y)=0$ for $\nu_j$-a.e. $y$.

\noindent Conversely, suppose that $b_j=0$, $\nu_j$-almost surely, for $j\in[m]$. Choose any representative $u$ of
$\mvGamma^{-1}[q^{j,y}]$.
The quotient identity $\mvGamma[u]=[q^{j,y}]$ means that there are
constants $c_i$, with $\sum_i c_i=0$, such that
\begin{equation*}
u_i(x_i)+\sum_{k\ne i}
\int u_k(x_k)p_\varepsilon^*(x_i,x_{-i})
\,d\nu_{-i}(x_{-i})
=q_i^{j,y}(x_i)+c_i.
\end{equation*}
Integrating the equation above against $\nu_i$, the multimarginal Schr\"{o}dinger system (\ref{multiconst_equiv}) yields that
\[
\sum_{k=1}^m\nu_k(u_k)
= c_i,
\]
and therefore
$$m \sum_{k=1}^m\nu_k(u_k)
=
\sum_{i=1}^m c_i 
=
0.$$
Consequently,
\begin{align*}
&\pi_\varepsilon^*\!\left(
g\,\bigoplus\!\left(\mvGamma^{-1}[q^{j,y}]\right)
\right)\\
&\quad=
\sum_{i=1}^m\int u_i(x_i)
\left[\int g(x_i,x_{-i})p_\varepsilon^*(x_i,x_{-i})
\,d\nu_{-i}(x_{-i})\right]d\nu_i(x_i)\\
&\quad=
\sum_{i=1}^m\nu_i\!\left(
u_i\bigl(b_i+\pi_\varepsilon^*(g)\bigr)
\right)
=\pi_\varepsilon^*(g)\sum_{i=1}^m\nu_i(u_i)
=0.
\end{align*}
Then equation \eqref{eq:nd_variance_short_projection} gives that $\bar\sigma_\varepsilon^2(g)=0$.
\end{proof}

\section{Technical details for Linearization of the empirical EMOT system}\label{appendix_lineratization_empirical_EMOT}
The purpose of this section is to present the detailed proofs of Lemma \ref{Linearization_of_the_empirical_EMOT_system} and Lemma \ref{order_of_centered_linerized_B_i}. We also include several technical results that will recur in other sections. The proofs of the two lemmas are given first, followed by a detailed discussion of the supporting arguments.

\begin{proof}[Proof of Lemma \ref{Linearization_of_the_empirical_EMOT_system}]
The proof is carried out in two steps, with all technical details collected in Appendix \ref{appendix_lineratization_empirical_EMOT}. The first step, conducted in 
subsection \ref{appendix_step_1}, 
is to establish Lemma \ref{Step_1_Linearization}, implying that
\begin{align*}
        \|\widehat{\mvT} ([\mvf^*]) -\widehat{\mvT} ([\hat{\mvf}^*]) -\widehat{\mvGamma}([\mvf^*] - [\hat{\mvf}^*]) \|_{\cC^1} = \cO_\bP \left( \| [\mvf^*] - [\hat{\mvf}^*] \|_{\cC^1}^2 \right) = \cO_\bP \left(\frac{d \log N}{N}\right),
    \end{align*}
where the last equality is implied by Lemma \ref{diff_empirical_population_f_i_bounded_by_other_potentials}.
The second step, which is detailed in subsection \ref{appendix_step_2}, starts from noting that 
\begin{align*}
    \widehat{\mvT} ([\hat{\mvf}^*]) = \mvT ([\mvf^*]) = 0. 
\end{align*}
So we have
\begin{align}\label{equation_Linearization_step_1}
    \|\widehat{\mvT} ([\mvf^*]) -\mvT ([\mvf^*]) -\widehat{\mvGamma} ( [\mvf^*] - [\hat{\mvf}^*] ) \|_{\cC^1} =  \cO_\bP \left(\frac{d \log N}{N}\right).
\end{align}
Lemma \ref{Step_2_Linearization}, as the main lemma in subsection \ref{appendix_step_2}
establishes that 
\begin{align}\label{equation_Linearization_step_2}
    \|\widehat{\mvT} ([\mvf^*]) -\mvT ([\mvf^*]) - \varepsilon\fL\|_{\cC^1} = \cO_\bP \left( \|[\mvf^*] - [\hat{\mvf}^*]\|_{\cC^1}^2 \right)=\cO_\bP \left( \frac{d \log N}{N} \right).
\end{align}
The proof is complete combining (\ref{equation_Linearization_step_1}) and (\ref{equation_Linearization_step_2}) together.
\end{proof}

\begin{proof}[Proof of Lemma \ref{order_of_centered_linerized_B_i}]
It is enough to show that, for all $i\in[m]$,
\begin{align*}
        \|B_{-i}( \hat{\nu}^N_{-i} - \nu_{-i})\|_\infty = \sup_{x_i} |B_{-i}( \hat{\nu}^N_{-i} - \nu_{-i})|
        = \cO_\bP \left(\sqrt{\frac{d\log N}{N}} \right),
\end{align*}
as well as
\begin{align*}
    \| \nabla_{x_i} B_{-i}( \hat{\nu}^N_{-i} - \nu_{-i})\|_\infty = \sup_{x_i}\max_{1\leq k\leq d} \left| \frac{\partial}{\partial x_{i,k}} B_{-i}( \hat{\nu}^N_{-i} - \nu_{-i}) \right|
    =\cO_\bP \left(\sqrt{\frac{d \log N}{N}} \right).
\end{align*}
These two asymptotic equality comes as a corollary of Lemma \ref{Lipschitz_plus_light_tail_union_bound}.
\end{proof}

\noindent The subsequent discussion provides additional details.

\subsection*{\texorpdfstring{\underline{Step 1. Building Fréchet differentiability}}{Step 1. Building Fréchet differentiability}}
\label{appendix_step_1}
We first briefly recapitulate the settings along with notations from Section \ref{Sec_Weak_limit_of_expectation_under_entropic_optimal_transport_coupling} as follows.
Let $m\ge2$ and $\cX_j\subset[-1,1]^{d}$ be compact sets.
For each $j$, let $\nu_j$ be a Borel probability on $\cX_j$,
and $\widehat{\nu}_j^N=\frac1N\sum_{k=1}^N\delta_{X_k^{(j)}}$
the empirical measure of $N$ i.i.d.\ samples from $\nu_j$.
Write the empirical product over all coordinates except $i$ as
$\widehat{\nu}_{-i}^{\,N}:=\bigotimes_{j\ne i}\widehat{\nu}_j^N$.
Let $c\in \cC^2(\cX)$ and set $\bC=e^{-c/\varepsilon}$.
Since each $\cX_j$ is compact, the derivatives $\partial_{x_{i,\ell}}c$ are continuous and bounded:
\[
  L_i := \|\nabla_{x_i} c\|_\infty 
  = \max_{1\leq k\leq d} \sup_{\mvx\in\cX}  |\frac{\partial}{\partial_{x_{i,k}}}c(\mvx)|
  < \infty, \qquad L:=\sum_{1\le i\le m}L_i < \infty.
\]
\begin{comment}
      = \left \|\begin{pmatrix}
      \partial_{x_{i,1}}c
      \\ \vdots\\
      \partial_{x_{i,d}}c
  \end{pmatrix} \right \|_\infty 
\end{comment}
We consider the Banach space
\[
\cC^1:=  \prod_{j=1}^m \cC^1(\mathcal{X}_j). 
\]
%\begin{align*}
  %\|u\|_{\cC^1} := \max_{1\le j\le m}\Big(\|u_j\|_\infty,\ \|\nabla u_j\|_\infty\Big).
%\end{align*}  
For $\mvphi=(\phi_1,\dots,\phi_m)\in\cC^1$ and $i\in [m]$, define
\begin{equation}\label{eq:F1C1}
  \widehat{T}_i(\mvphi)(x_i)
  := \phi_i(x_i)
     + \varepsilon\log\!\int \exp\!\Big( \frac{\sum_{j\ne i}\phi_j(x_j)}{\varepsilon}\Big)\,
     \bC(x_1,\dots,x_m)\,d\widehat{\nu}_{-i}^{\,N}(x_{-i}),
\end{equation}
and set $\widehat{T}(\mvphi):=(\widehat{T}_1(\mvphi),\dots,\widehat{T}_m(\mvphi))$. All the upcoming integrals and differentiations make sense due to the regularity we have.

\begin{proposition}[First and second Fr\'{e}chet derivative]\label{prop_frechet_12_derivative}
\quad
\begin{enumeratea}
    \item $\widehat{T}: \cC^1 \to \cC^1$ satisfies that
    \begin{equation*}
      \nabla_{x_i}\widehat{T}_i(\mvphi)(x_i)
      = \nabla_{x_i} \phi_i(x_i)
        + \varepsilon\,\mathbb{E}_{\pi^{(i,N)}_{\mvphi,x_i}}\!\big[\nabla_{x_i}\log \bC(x_1,\dots,x_m)\big]
      = \nabla_{x_i} \phi_i(x_i) -\mathbb{E}_{\pi^{(i,N)}_{\mvphi,x_i}}\!\big[\nabla_{x_i} c(x_1,\dots,x_m)\big],
    \end{equation*}
    where
    \[
     \pi^{(i,N)}_{\mvphi,x_i}(dx_{-i})
     := \frac{\exp \left( \frac{\sum_{j\ne i}\phi_j(x_j)}{\varepsilon}\right) \bC(x_1,\dots,x_m)}
    {\int \exp\left ( \frac{\sum_{j\ne i}\phi_j(x_j)}{\varepsilon}\right) \bC(x_1,\dots,x_m)\,d\widehat{\nu}_{-i}^{\,N}(x_{-i})}
     \,d\widehat{\nu}_{-i}^{\,N}(x_{-i}).
    \]
Thus,
\[
  \|\nabla_{x_i} \widehat{T}_i(\mvphi)\|_\infty \le \|\nabla_{x_i} \phi_i\|_\infty + L_i.
\]
\item For $\mvphi\in\cC^1$ and $\mvh=(h_1,\dots,h_m)\in\cC^1$,
\begin{equation}\label{First_order_Frechet_derivative_equation}
  D\widehat{T}_i(\mvphi)[\mvh](x_i)
  = h_i(x_i) + \sum_{j\ne i}\mathbb{E}_{\pi^{(i,N)}_{\mvphi,x_i}}\!\big[h_j(X_j)\big].
\end{equation}
Hence,
$$\|D\widehat{T}(\mvphi)\|_{\cC^1\to\cC^1}\le \max\{m,\,1+2(m-1)L/\varepsilon\}.$$

\item  For $(\mvh^{(1)},\mvh^{(2)})\in\cC^1\times\cC^1$,
\begin{align}\label{Second_order_Frechet_derivative_equation}
  D^2\widehat{T}_i(\mvphi)[(\mvh^{(1)},\mvh^{(2)})](x_i)
  =\frac{1}{\varepsilon} \sum_{j\ne i}\sum_{k\ne i}
     \mathrm{Cov}_{\pi^{(i,N)}_{\mvphi,x_i}}\!
       \Big(h^{(1)}_j(X_j),h^{(2)}_k(X_k)\Big).
\end{align}
Consequently,
\[
\|D^2\widehat{T}(\mvphi)\|_{(\cC^1,\,\cC^1)\to\cC^1}
  \le \max\left\{\frac{2(m-1)^2}{\varepsilon},\,\frac{6(m-1)^2L}{\varepsilon^2}\right\}.
\]
\end{enumeratea}
\end{proposition}

\begin{proof}
(a) Direct computation and boundedness of $\nabla_{x_i}c$ gives the stated result.

(b) Linearization of the log-sum-exp operator \eqref{eq:F1C1} yields
the first equality, see \cite{chizat2023doublyregularizedentropicwasserstein} for details on log-sum-exp operator.  For the gradient, use the
score identity (\ref{score_differentiation_formula}) with
$S^{(i,N)}_{\mvphi,x_i}
  := \nabla_{x_i}\log \bC(x_1,\dots,x_m)
   - \mathbb{E}_{\pi^{(i,N)}_{\mvphi,x_i}}[\nabla_{x_i}\log \bC]$:
\[
  \nabla_{x_i} D\widehat{T}_i(\mvphi)[\mvh]
  = \nabla_{x_i} h_i
    + \sum_{j\ne i}\mathbb{E}_{\pi^{(i,N)}_{\mvphi,x_i}}\!\big[h_j(X_j)\,S^{(i, N)}_{\mvphi,x_i}(X_{-i})\big].
\]
Thus,
\[
  \|D\widehat{T}_i(\mvphi)[\mvh]\|_{\cC^1}
  \le \max\Big(
      m\,\|\mvh\|_\infty,\,
      \|\nabla h_i\|_\infty + 2(m-1)\frac{L_i}{\varepsilon}\,\|\mvh\|_\infty
      \Big),
\]
because of $\|S^{(i, N)}_{\mvphi,x_i}\|_\infty\le2L_i/\varepsilon$. Namely,
$$\|D\widehat{T}_i(\mvphi)\| \leq \max\{m,\,1+2(m-1)L/\varepsilon\}.$$
The claimed bound follows in virtue of Lemma \ref{matrix_like_operator_norm_bound_0}.

(c) Linearization of equation (\ref{First_order_Frechet_derivative_equation}) gives (\ref{Second_order_Frechet_derivative_equation}), see \cite{chizat2023doublyregularizedentropicwasserstein} for properties on log-sum-exp operator. Note that by Lemma \ref{Differentiation_of_expectation_against_exponential_tilted_probability},
\begin{align*}
\nabla_{x_i} &\mathrm{Cov}_{\pi^{(i,N)}_{\mvphi,x_i}}\!\Big(h^{(1)}_j(X_j),h^{(2)}_k(X_k)\Big) \\
= & \,\, \bE_{\pi^{(i,N)}_{\mvphi,x_i}}[h^{(1)}_j(X_j)h^{(2)}_k(X_k)S^{(i, N)}_{\mvphi,x_i}]\\
& - \,\,\bE_{\pi^{(i,N)}_{\mvphi,x_i}}[h^{(1)}_j(X_j)] \bE_{\pi^{(i,N)}_{\mvphi,x_i}}[h^{(2)}_k(X_k)S^{(i, N)}_{\mvphi,x_i}]
-\bE_{\pi^{(i,N)}_{\mvphi,x_i}}[h^{(2)}_k(X_k)] \bE_{\pi^{(i,N)}_{\mvphi,x_i}}[h^{(1)}_j(X_j)S^{(i, N)}_{\mvphi,x_i}].
\end{align*}
For $(\mvh^{(1)},\mvh^{(2)})\in\cC^1\times\cC^1$, since $\|S^{(i)}_{\mvphi,x_i}\|_\infty\le2L_i/\varepsilon$,
the following bound is obtained by summing over $j\neq i, k\neq i$.
\begin{align*}
  \|\nabla_{x_i} D^2\widehat{T}_i(\mvphi)[(\mvh^{(1)},\mvh^{(2)})]\|_\infty
  &\le \frac{6(m-1)^2L_i}{\varepsilon^2}\,
     \|\mvh^{(1)}\|_\infty\|\mvh^{(2)}\|_\infty.
\end{align*}
As a result,
\[
\|D^2\widehat{T}(\mvphi)\|_{(\cC^1,\,\cC^1)\to\cC^1}
= \max_{1\leq i\leq m}\|D^2\widehat{T}_i(\mvphi)\|_{(\cC^1,\,\cC^1)\to\cC^1}
  \le \max\left\{\frac{2(m-1)^2}{\varepsilon},\,\frac{6(m-1)^2\,L}{\varepsilon^2}\right\}.
\]
\end{proof}

\begin{lemma}\label{Step_1_Linearization}
For $\mvphi\in\cC^1$ and $\mvh\in\cC^1$, we define $\widehat{R}(\mvh) := (\widehat{R}_1(\mvh),\dots\widehat{R}_m(\mvh))$, where 
\[
  \widehat{R}_i(\mvh)
  := \widehat{T}_i(\mvphi + \mvh) - \widehat{T}_i(\mvphi) - D\widehat{T}_i(\mvphi)[\mvh]
  = \int_0^1 (1-t)\,D^2\widehat{T}_i(\mvphi +t \mvh)[\mvh,\mvh]\,dt.
\]
Then we have

\begin{equation}\label{eq:R-bound-C1}
  \|\widehat{R}(\mvh)\|_{\cC^1}  = \max_{1\leq i \leq m}\|\widehat{R}_i(\mvh)\|_{\cC^1} 
  \le
  \max \left\{ \frac{(m-1)^2}{\varepsilon},\, \frac{3(m-1)^2L}{\varepsilon^2} \right\}\,\|\mvh\|_{\cC^1}^2.
\end{equation}

\end{lemma}

\begin{proof}
Taylor’s theorem in Banach spaces (see, cf. \cite{Gerald_Teschl_Taylor_Banach_20}) yields the integral remainder formula.
Realizing that
$\|D^2\widehat{T}_i(\mvphi +t \mvh)[\mvh,\mvh]\|_{\cC^1}\le   \max\{2(m-1)^2/\varepsilon,\,6(m-1)^2\,L/\varepsilon^2\} \|\mvh\|_{\cC^1}^2$,
%then integrate $\int_0^1(1-t)\,dt=1/2$.
we have
\begin{equation*}
  \|\widehat{R}_i(\mvh)\|_{\cC^1} 
  \le
  \max \left\{ \frac{(m-1)^2}{\varepsilon},\, \frac{3(m-1)^2L}{\varepsilon^2} \right\}\,\|\mvh\|_{\cC^1}^2.
\end{equation*}
This completes the proof.
\end{proof}

\begin{comment}
\begin{lemma}[Differentiation of expectation against exponential tilted probability]\label{Differentiation_of_expectation_against_exponential_tilted_probability}
Given $m$ probability $\nu_j\in\cP(\Omega_j)$, $j\in[m]$ and a function $s: \prod_{j=1}^m\Omega_j\to \bR$, define 
\begin{align*}
        \frac{d\pi_{i}}{d\nu_{-i}} := \frac{c(x_i,x_{-i})}{Z(x_i)},
\end{align*}
where $Z(x_i) = \int c(x_i,x_{-i}) d\nu_{-i}(x_{-i})$. For $\psi:\Omega_{-i}\to\bR$, we have
\begin{align*}
    \nabla_{x_i} \bE_{\pi_i} [\psi] = \bE_{\pi_i} [\psi S^{(i)}],
\end{align*}
where $S^{(i)} = \nabla_{x_i} \log c(x_i,x_{-i}) - \bE_{\pi_i} \nabla_{x_i} \log c(x_i,x_{-i}).$
\end{lemma}
\end{comment}

\begin{lemma}[Differentiation under an exponentially tilted law]
\label{Differentiation_of_expectation_against_exponential_tilted_probability}
Let $\mu_j$ be arbitrary probability with compact support $\cX_j$ for $j\in [m]$. 
%Fix $i\in[m]$ and write $\cX_{-i}:=\prod_{j\neq i}\cX_j$, $\mu_{-i}:=\bigotimes_{j\neq i}\mu_j$.
For $k:\cX_i\times\cX_{-i}\to(0,\infty) \in \cC^2$, each $x_i\in\cX_i$, define
\[
Z(x_i):=\int_{\cX_{-i}} k(x_i,x_{-i})\,d\mu_{-i}(x_{-i}).
\]
Define the probability $\pi_i(\,\cdot\,|x_i)$ on $\cX_{-i}$ via
\[
\frac{d\pi_i(\,\cdot\,|x_i)}{d\mu_{-i}}(x_{-i})
=\frac{k(x_i,x_{-i})}{Z(x_i)}.
\]
For a bounded measurable function $\psi:\cX_{-i}\to\mathbb{R}$, $x_i\mapsto \mathbb{E}_{\pi_i(\cdot|x_i)}[\psi]$ is differentiable and
\begin{align}\label{score_differentiation_formula}
\nabla_{x_i}\,\mathbb{E}_{\pi_i(\cdot|x_i)}[\psi]
\;=\;
\mathbb{E}_{\pi_i(\cdot|x_i)}\big[\psi\,S^{(i)}(x_i,\cdot)\big],
\end{align}
with $S^{(i)}(x_i,x_{-i})
:=\nabla_{x_i}\log k(x_i,x_{-i})
-\mathbb{E}_{\pi_i(\cdot|x_i)}\!\big[\nabla_{x_i}\log k(x_i,\cdot)\big].$ Equivalently,
\[
\nabla_{x_i}\,\mathbb{E}_{\pi_i(\cdot|x_i)}[\psi]
\;=\;\mathrm{Cov}_{\pi_i(\cdot|x_i)}\!\big(\psi,\ \nabla_{x_i}\log k(x_i,\cdot)\big).
\]
\end{lemma}

\begin{proof}
We sometimes in the proof write $\pi_i(\cdot|x_i)$ as $\pi_i$ for simplicity. By definition, we have $\mathbb{E}_{\pi_i(\cdot|x_i)}[\psi]=\frac{N(x_i)}{Z(x_i)}$, with
$$N(x_i):=\int_{\cX_{-i}}\psi(x_{-i})\,k(x_i,x_{-i})\,d\mu_{-i}(x_{-i}).$$ Dominated convergence theorem yields
\[
\nabla_{x_i}N(x_i)=\int \psi(x_{-i})\,\nabla_{x_i}k(x_i,x_{-i})\,d\mu_{-i}(x_{-i}), 
\qquad
\nabla_{x_i}Z(x_i)=\int \nabla_{x_i}k(x_i,x_{-i})\,d\mu_{-i}(x_{-i}).
\]
Differentiate the quotient $N(x_i)/Z(x_i)$:
\[
\nabla_{x_i}\,\mathbb{E}_{\pi_i}[\psi]
=\frac{\nabla_{x_i}N(x_i)}{Z(x_i)}-\frac{N(x_i)}{Z(x_i)^2}\,\nabla_{x_i}Z(x_i)
=\frac{1}{Z(x_i)}\!\int\!\psi\,\nabla_{x_i}k\,d\mu_{-i}
-\Big(\mathbb{E}_{\pi_i}[\psi]\Big)\frac{1}{Z(x_i)}\!\int\!\nabla_{x_i}k\,d\mu_{-i}.
\]
We could rearrange and obtain
\[
\frac{1}{Z(x_i)}\int \psi\,\nabla_{x_i}k\,d\mu_{-i}
=\frac{1}{Z(x_i)}\int \psi\,k\,\nabla_{x_i}\log k\,d\mu_{-i}
=\mathbb{E}_{\pi_i}\big[\psi\,\nabla_{x_i}\log k\big],
\]
and similarly
\[
\frac{1}{Z(x_i)}\int \nabla_{x_i}k\,d\mu_{-i}
=\mathbb{E}_{\pi_i}\big[\nabla_{x_i}\log k\big].
\]
Hence
\[
\nabla_{x_i}\,\mathbb{E}_{\pi_i}[\psi]
=\mathbb{E}_{\pi_i}\big[\psi\,\nabla_{x_i}\log k\big]
-\mathbb{E}_{\pi_i}[\psi]\,\mathbb{E}_{\pi_i}\big[\nabla_{x_i}\log k\big]
=\mathrm{Cov}_{\pi_i}\!\big(\psi,\nabla_{x_i}\log k\big).
\]
Writing $S^{(i)}:=\nabla_{x_i}\log k-\mathbb{E}_{\pi_i}[\nabla_{x_i}\log k]$ gives the stated form 
$\nabla_{x_i}\mathbb{E}_{\pi_i}[\psi]=\mathbb{E}_{\pi_i}[\psi S^{(i)}]$.

%Finally, note $\mathbb{E}_{\pi_i}[S^{(i)}]=0$ by construction, so the identity indeed expresses a covariance (score-function) formula.
\end{proof}

\begin{remark}
If $k(x_i,x_{-i})=\exp\{s(x_i,x_{-i})\}$ with a differentiable function $s$, then $\nabla_{x_i}\log k=\nabla_{x_i}s$ and the identity becomes
\[
\nabla_{x_i}\,\mathbb{E}_{\pi_i}[\psi]
=\mathrm{Cov}_{\pi_i}\!\big(\psi,\nabla_{x_i}s\big)
=\mathbb{E}_{\pi_i}\big[\psi\,\big(\nabla_{x_i}s-\mathbb{E}_{\pi_i}[\nabla_{x_i}s]\big)\big].
\]
\end{remark}

\subsection*{\texorpdfstring{\underline{Step 2. Linearization of log difference}}{Step 2. Linearization of log difference}}\label{appendix_step_2} 

For given continuously differentiable cost function $c:\cX\to\bR$, population potentials $ f^*_j : \cX_j \to \bR,\,j\in[m]$, we write for $i\in[m]$,
\begin{align*}
  h_{x_i}(x_{-i}) \;:=\; \exp\!\big( (\oplus_{j\neq i} f^*_j(x_j)-c(x_{i},x_{-i}))/\varepsilon\big),
\end{align*}
and we have
\begin{align*}
    \nabla_{x_i} h_{x_i}(x_{-i}) \;=\; -\frac{1}{\varepsilon}\,h_{x_i}(x_{-i})\,\nabla_{x_i} c(x_i,x_{-i}).
\end{align*}
We know the nondegeneracy condition through Proposition \ref{BDP} that for some constant $b_0>0$,
\[
  V(x_i)\;:=\;\nu_{-i}h_{x_i} \;=\; \int h_{x_i}(x_{-i})\,d\nu_{-i}(x_{-i}) \ge\ b_0>0
  \quad \text{for all } x_i\in\cX_i.
\]
\begin{lemma}\label{Step_2_Linearization}
Define the linearized remainder
\[
  \Delta_N(x_i)
  \;:=\; \log (\hat{\nu}^N_{-i} h_{x_i}) \;-\; \log  (\nu_{-i} h_{x_i}) \;-\; \frac{(\hat{\nu}^N_{-i} - \nu_{-i})h_{x_i}}{\nu_{-i} h_{x_i}}.
\]
Then under standard assumptions,
\[
\sup_{x_i} |\Delta_N(x_i)|  \;=\; \cO_\bP\!\Big(\frac{d\log N}{N}\Big) \;=\; o_\bP(N^{-1/2}).
\]
%Then, for each fixed $x_i$,
%\[
 % \nabla_{x_i}\Delta_N(x_i) \;=\; \cO_\bP(N^{-1}) \;=\; o_\bP(N^{-1/2}).
%\]
If, in addition, $c\in \cC^2$, then
\[
  \sup_{x_i}\|\nabla_{x_i}\Delta_N(x_i)\| = \sup_{x_i} \max_{1\leq k\leq d}\left| \frac{\partial}{\partial_{x_{i,k}}}\Delta_N(x_i) \right|
  \;=\; \cO_\bP\!\Big(\frac{d\log N}{N}\Big) \;=\; o_\bP(N^{-1/2}).
\]
\end{lemma}

\begin{proof}
Set
\[
  W(x_i):=(\hat{\nu}_{-i}^N - \nu_{-i})h_{x_i},\qquad V(x_i):=\nu_{-i}h_{x_i},\qquad u(x_i):=\frac{W(x_i)}{V(x_i)}.
\]
Taylor expansion of $\log(1+u)$ with integral remainder leads to
\[
  \Delta_N(x_i) \;=\; -\tfrac12\,u(x_i)^2 \;+\; R\big(u(x_i)),
\]
where
\[
  R(u) \;=\; \int_0^u \frac{(u-t)^2}{(1+t)^3}\,dt.
\]
It is easily checked that 
\begin{align*}
     \qquad R(u)=\cO(u^3),\quad R'(u)=\cO(u^2)\quad \text{for}\ |u|\le \tfrac12 .
\end{align*}
Differentiation yields
\[
  \nabla_{x_i}\Delta_N(x_i)
  \;=\; -\,u(x_i)\,\nabla_{x_i} u(x_i) \;+\; \nabla_{x_i}u(x_i) R'(u(x_i)),
  %\nabla_{x_i}\!\big(u(x_i)^2 R(u(x_i))\big),
\]
where
\[
  \nabla_{x_i} u(x_i)
  \;=\; \frac{V(x_i)\,\nabla_{x_i} W(x_i) - W(x_i)\,\nabla_{x_i} V(x_i)}{V^2(x_i)}
  \;=\; \frac{(\hat{\nu}^N_{-i}-\nu_{-i})\nabla_{x_i} h_{x_i}}{V(x_i)}
        \;-\; \frac{W(x_i)}{V(x_i)}\cdot \frac{\nu_{-i}\nabla_{x_i} h_{x_i}}{V(x_i)}.
\]
Differentiation under the integrals is justified by the regularity of 
$h_{x_i}$. Hence, for $h_{x_i}$ and $\nabla_{x_i} h_{x_i}$ bounded as a consequence of $c\in\cC^2$,
\[
  W(x_i)=(\hat{\nu}^N_{-i}-\nu_{-i})h_{x_i}=\cO_\bP(N^{-1/2}),\qquad
  (\hat{\nu}^N_{-i}-\nu_{-i})\nabla_{x_i} h_{x_i}=\cO_\bP(N^{-1/2}).
\]
So $u(x_i)= \cO_\bP(N^{-1/2})$ and $\nabla_{x_i} u(x_i) = \cO_\bP(N^{-1/2})$ (pointwise in $x_i$). Therefore
\[
  u(x_i)\,\nabla_{x_i} u(x_i)=\cO_\bP(N^{-1}).
\]
On the event $\sup_{x_i} |u(x_i)|\le \tfrac12$ (which holds with probability~$\to1$),
\begin{align*}
\nabla_{x_i}(R(u(x_i))
  \; =\; R'(u(x_i))\,\nabla_{x_i} u(x_i) =\cO_\bP\!\big(|u(x_i)|^2 \nabla_{x_i} u(x_i)\big)
  \;=\; o_\bP\!\big(u(x_i)\,\nabla_{x_i} u(x_i)\big).
\end{align*}
Hence $\nabla_{x_i}\Delta_N(x_i)=
-u(x_i)\,\nabla_{x_i} u(x_i)
+o_\bP(u(x_i)\nabla_{x_i} u(x_i))=\cO_\bP(N^{-1})$. If $c\in \cC^2$ in $x_i$ with bounded $\nabla_{x_i}^2 c$, then Lemma \ref{Lipschitz_plus_light_tail_union_bound} gives
\[
  \sup_{x_i}|(\hat{\nu}^N_{-i}-\nu_{-i})h_{x_i}|,
  \quad
  \sup_{x_i} \|(\hat{\nu}^N_{-i}-\nu_{-i})\nabla_{x_i} h_{x_i}\|
  \;=\; \cO_\bP\!\Big(\sqrt{d \log N/N} \Big).
\]
Thus $\sup_{x_i} |u(x_i)|,\sup_{x_i}\|\nabla_{x_i} u(x_i)\|=\cO_\bP\!\big(\sqrt{d \log N/N}\big)$, thereby
\[
  \sup_{x_i}\|\nabla_{x_i}\Delta_N(x_i)\|
  \;\le\; \sup_{x_i} |u(x_i)|\,\sup_{x_i} \|\nabla_{x_i} u(x_i)\|(1+o_\bP(1))
  \;=\; \cO_\bP\!\Big(d \log N/N\Big).
\]
\end{proof}

\noindent The following lemma is presented in this section for its immediate use in the linearization of the EMOT system. However, it is of broad relevance and will be applied in subsequent developments in deriving weak limits.

\begin{lemma}\label{Lipschitz_plus_light_tail_union_bound}
Suppose that $\cX_j\subset\bR^d$ is compact for $j\in[m]$ and $K = (K_1,\dots,K_r): \prod_{j=1}^m \cX_j \to \bR^r$ satisfies for some $i^*\in[m]$,
$$\sup_{y_j}\|K(y_1,\dots,y_{(i^*-1)}, x, y_{(i^*+1)} ,\dots,y_{m})-K(y_1,\dots,y_{(i^*-1)}, \tilde{x}, y_{(i^*+1)} ,\dots,y_{m})\|_{\infty} \leq L \|x-\tilde{x}\|_2,$$
and 
$$\|K\|_\infty := \sup_{\mvy\in  \prod_{j=1}^m \cX_j }\|K(y_1,\dots,y_m)\|_{\infty} := \sup_{\mvy\in  \prod_{j=1}^m \cX_j } \max_{1\leq s\leq r} |K_s(y_1,\dots,y_m)| < \infty.$$
Define
$$F_N(x) = \frac{1}{N^{m-1}} \sum_{
\substack{1\leq\ell^{(i)}\leq N\\1\leq i\leq m\\ i\neq i^*}} \overline{K}_{-i^*} (X_{\ell^{(1)}}^{(1)},\dots,X_{\ell^{(i^*-1)}}^{(i^*-1)}, x, X_{\ell^{(i^*+1)}}^{(i^*+1)} ,\dots,X_{\ell^{(m)}}^{(m)}),
$$
with $\overline{K}_{-i^*}(\cdot)$ the centered version of $K(\cdot)$,
\begin{align*}
\overline{K}_{-i^*} (Z_1,&\dots,Z_{(i^*-1)}, x, Z_{(i^*+1)} ,\dots,Z_{m}) \\
&:= K(Z_1,\dots,Z_{(i^*-1)}, x, Z_{(i^*+1)} ,\dots,Z_{m}) - \bE K(Z_1,\dots,Z_{(i^*-1)}, x, Z_{(i^*+1)} ,\dots,Z_{m}).
\end{align*}
Then we have
\begin{align*}
    \bE \sup_{x\in\cX_{i^*}} \|F_N(x)\|_{\infty} \lesssim \|K\|_\infty \sqrt{\frac{d\log (L^2N)}{N}}.
\end{align*}

\begin{comment}
Furthermore for any $\delta>0$, with probability more than $1-\delta$,
\begin{align*}
    \sup_{x\in[-1,1]^d} \|F_N(x)\|_2 \lesssim \frac{\|K\|_\infty}{\sqrt{N}} \left(\sqrt{d\log N} + \sqrt{\log\frac{1}{\delta}}\right).
\end{align*}
\end{comment}

\end{lemma}

\begin{proof}
For any fixed $x$, bounded difference inequality gives that for all $u>0$,
\begin{align*}
    \bP\left( \|F_N(x)\|_{\infty} \geq u \right)\leq 2\exp \left(-\frac{C_1 N u^2}{\|K\|^2_\infty} \right) 
\end{align*}
for some universal constant $C_1>0$ depending on $m,r,\|c\|_\infty$.  For any $u>0$, pick a $\frac{u}{4L}$-net $\cN_{\frac{u}{4L}}$ of $\cX_{i^*}$ and then we have via the Lipschitzness,
\begin{align*}
        \sup_{x\in\cX_{i^*}} \|F_N(x)\|_{\infty} \leq \max_{x\in\cN_{\frac{u}{4L}}} \|F_N(x)\|_{\infty} + \frac{u}{2}.
\end{align*}
 Note that $|\cN_{\frac{u}{4L}}|\leq (\frac{C_2 L}{u})^d$ for some universal constant $C_2>0$ depending on the diameter of $\cX_{i^*}$. As a consequence, union bound provides
\begin{align*}
    \bP\left(\sup_{x\in \cX_{i^*}} \|F_N(x)\|_{\infty} \geq u\right) \leq \bP\left(\max_{x\in\cN_{\frac{u}{4L}}} \|F_N(x)\|_{\infty} \geq u/2\right) \leq \Big(\frac{C_2 L}{u}\Big)^{d}\exp \left(-\frac{C_1 N u^2}{\|K\|^2_\infty} \right).
\end{align*}

\noindent Let us define
\[
v_0\ :=\ \max\!\Big\{1,\ \frac{d}{2}\,\log \left(\frac{L^2 C_1 C_2^2 N}{\|K\|^2_\infty} \right)\Big\},
\qquad
u_0\ :=\ \|K\|_\infty \sqrt{\frac{v_0}{C_1N}}.
\]

\noindent Write $S:= \sup_{x\in\cX_{i^*}} \|F_N(x)\|_{\infty} $ and then split the tail integral at $u_0$:
\[
\mathbb{E}S
\ =\ \int_0^\infty \mathbb{P}(S>u)\,du
\ \le\ \underbrace{\int_0^{u_0}1\,du}_{=:I_1}
\ +\ 
\underbrace{\int_{u_0}^\infty \Big(\frac{C_2 L}{u}\Big)^{d}e^{-C_1 N u^2/\|K\|_\infty^2}\,du}_{=:I_2}.
\]

\noindent Clearly $I_1=u_0$. For $I_2$, change of variables $v = C_1 N u^2/\|K\|^2_\infty$ gives
\begin{align*}
I_2
&= \frac{C_2^d L^{d}}{2}\,\frac{(C_1 N)^{\frac{d}{2}-\frac12}}{\|K\|^{d-1}_\infty}
\int_{v_0}^{\infty} v^{-\frac{d+1}{2}} e^{-v}\,dv\\
&\leq \frac{C_2^d L^{d}}{2}\,\frac{(C_1N)^{\frac{d}{2}-\frac12}}{\|K\|^{d-1}_\infty}
\,v_0^{-\frac{d+1}{2}}\,e^{-v_0}.
\end{align*}
%where we used $\int_{v_0}^{\infty} v^{-(d+1)/2}e^{-v}\,dv \le v_0^{-(d+1)/2}e^{-v_0}$. 
By the definition of $v_0$,
\[
v_0^{-(d+1)/2}\ \le\ 1, \quad\quad e^{-v_0}\ \le\  \left(\frac{\|K\|_\infty}{C_2 L}\right)^d \Big(\frac{ 1 }{C_1 N}\Big)^{d/2}.
\]
Substituting into the bound for $I_2$ yields
\[
I_2\ \le\ \frac{1}{2}\,(C_1 N)^{-1/2} \|K\|_\infty.
\]
Finally, using $v_0\le \frac{d}{2}\log\left(\frac{L^2 C_1 C_2^2 N}{\|K\|^2_\infty}\right)$ when that term exceeds $1$ (otherwise $v_0=1$ and the displayed upper bound still holds), we obtain
\[
\mathbb{E}S\ 
\ \le\
\frac{ \|K\|_\infty}{\sqrt{C_1}}\,
\sqrt{\frac{d\,\log \left(\frac{L^2 C_1 C_2^2 N}{\|K\|^2_\infty} \right)}{2\,N}} + \frac{ \|K\|_\infty}{2\sqrt{C_1}} \frac{1}{\sqrt{N}}.
\]

\begin{comment}
For the proof of high-probability inequality, we use bounded difference inequality. For any $\delta>0$, with probability more than $1-\delta$,
\begin{align*}
    \sup_{x\in[-1,1]^d} \|F_N(x)\|_2 \leq \bE\sup_{x\in[-1,1]^d} \|F_N(x)\|_2 + C \| K \|_\infty \sqrt{\frac{\log\frac{1}{\delta}}{N}}.
\end{align*}
Combining with the expectation bound obtained above, the proof is complete.
\end{comment}

\end{proof}

\section{Technical details for operator norms estimates}\label{appendix_technical_details_for_operator_norms_estimates}
This section collects several technical lemmas establishing bounds on the operator norms of auxiliary operators that appear throughout the paper. We focus on the operator norm bound of $\|\bA-\bH\|$ through $\|\bD-\bH\|$ and $\|\bD-\bA\|$. As in Appendix \ref{Sec_Weak_limit_of_expectation_under_entropic_optimal_transport_coupling} and Appendix \ref{appendix_lineratization_empirical_EMOT}, we follow the notation convention that
$\cC^1:=  \prod_{j=1}^m \cC^1(\mathcal{X}_j)$.

\begin{proof}[proof of Lemma \ref{bound_on_norm_A-H}]
First note that $\bA$ and $\bH$ could be considered as operators on $\widetilde{\cC}^1$ due to multimarginal Schr\"{o}dinger systems. Thus,
\begin{align*}
    \| \bA - \bH\|_{\widetilde{\cC}^1\to\widetilde{\cC}^1} \leq \| \bA - \bH\|_{\cC^1\to\cC^1} \leq \| \bA - \bD\|_{\cC^1\to\cC^1} + \| \bD - \bH\|_{\cC^1\to\cC^1}.
\end{align*}
The bounds for the right-hand-side terms in the inequality are obtained in Lemma \ref{bound_on_norm_A-D} and Lemma \ref{bound_on_norm_D-H}, respectively.   
\end{proof}

\begin{lemma}\label{bound_on_norm_A-D}
For any $i,j\in[m]$ with $i\neq j$, we have that
\begin{equation}
        \| \bD - \bA\|_{\cC^1\to\cC^1}  = \cO_\bP \left( (m-1)\, \sqrt{\frac{d \log N}{N}}\right).
\end{equation}
\end{lemma}

\begin{comment}
\begin{lemma}\label{bound_on_norm_Aij-Dij}
For any $i,j\in[m]$ with $i\neq j$, we have that
\begin{equation}
        \| A_{ij} - D_{ij}\|  = \cO_\bP \left(\sqrt{\frac{d \log N}{N}}\right).
\end{equation}
\end{lemma}
\end{comment}

\begin{comment}
    \begin{align*}
         \| &\nabla_{x_i} (A_{ij} - D_{ij})f_j \|_{\infty}
         =
         \sup_{x_i\in[-1,1]^d} \left\| \nabla_{x_i} \int f_j(x_j) (\hat{p}_\varepsilon - p^*_\varepsilon)(\mvx) d\hat{\nu}_{-i}^N (\mvx) \right\|\\
         &\lesssim
         \|f_j\|_\infty \sup_{x_i\in[-1,1]^d}\int \left| (\hat{p}_\varepsilon - p^*_\varepsilon)(\mvx) \right| d\hat{\nu}_{-i}^N (\mvx)\\
         & = \frac{\|f_j\|_\infty}{N^{m-1}} \sup_{x_i\in[-1,1]^d}  \sum_{\ell^{(1)},\dots,\ell^{(i-1)},\ell^{(i+1)},\dots,\ell^{(m)}} \left| (\hat{p}_\varepsilon - p^*_\varepsilon)(X_{\ell^{(1)}}^{(1)}, \dots, X_{\ell^{(i-1)}}^{(i-1)}, x_i, X_{\ell^{(i+1)}}^{(i+1)}, \dots, X_{\ell^{(m)}}^{(m)}) \right|.  \\
    \end{align*}
\end{comment}

\begin{lemma}\label{bound_on_norm_D-H}
For any $i,j\in[m]$ with $i\neq j$, we have that
\begin{equation}
        \| \bD - \bH \|_{\cC^1\to\cC^1} = \cO_\bP \left( (m-1)^2\,N^{-1/(2\vee d)} (\log N)^{\mathbf 1_{\{d=2\}}} \right).
\end{equation}

\end{lemma}

\begin{proof}[Proof of Lemma \ref{bound_on_norm_A-D}]
By Lemma \ref{matrix_like_operator_norm_bound}, it suffices to show that
\begin{equation}
        \| A_{ij} - D_{ij}\|  = \cO_\bP \left(\sqrt{\frac{d \log N}{N}}\right).
\end{equation}
\begin{comment}
Recall that for $T:\cC^1([-1,1]^d)\to \cC^1([-1,1]^d)$, the operator norm is defined as
\[
  \|T\|_{\cC^1\to \cC^1}
  := \sup_{\|f\|_{\cC^1}\le 1}\Big(\|Tf\|_\infty + \|\nabla (Tf)\|_\infty\Big).
\]
\end{comment}
\begin{comment}
Note that 
\begin{align*}
\| A_{ij} - D_{ij} \|_{\cC^1\to \cC^1} = \sup_{\|f_j\|_{\cC^1}\leq 1}\Big(\|(A_{ij} - D_{ij}) f_j\|_\infty + \|\nabla_{x_i} (A_{ij} - D_{ij})(f_j)\|_\infty\Big)\\
\leq \sup_{\|f_j\|_{\infty}\leq 1}\Big(\|(A_{ij} - D_{ij}) f_j\|_\infty + \|\nabla_{x_i} (A_{ij} - D_{ij})(f_j)\|_\infty\Big).
\end{align*}
\end{comment}
We have that
    \begin{align*}
         \| &(A_{ij} - D_{ij})f_j \|_\infty^2
         =
         \sup_{x_i\in[-1,1]^d} \left|\int f_j(x_j) (\hat{p}_\varepsilon - p^*_\varepsilon)(\mvx) d\hat{\nu}_{-i}^N (\mvx) \right|^2\\
         &\leq
         \|f_j\|_\infty^2 \sup_{x_i\in[-1,1]^d}\int \left| (\hat{p}_\varepsilon - p^*_\varepsilon)(\mvx) \right|^2 d\hat{\nu}_{-i}^N(\mvx) \\
         & = \frac{\|f_j\|^2_\infty}{N^{m-1}} \sup_{x_i\in[-1,1]^d}  \sum_{\ell^{(1)},\dots,\ell^{(i-1)},\ell^{(i+1)},\dots,\ell^{(m)}} \left| (\hat{p}_\varepsilon - p^*_\varepsilon)(X_{\ell^{(1)}}^{(1)}, \dots, X_{\ell^{(i-1)}}^{(i-1)}, x_i, X_{\ell^{(i+1)}}^{(i+1)}, \dots, X_{\ell^{(m)}}^{(m)}) \right|^2.  \\
    \end{align*}
This term is bounded by Lemma \ref{main_lemma_bound_on_norm_Aij-Dij}. Using Proposition \ref{empirical_BDP}, we have
\begin{align*}
         \| &\nabla_{x_i} (A_{ij} - D_{ij})f_j \|_{\infty}
         =
         \max_{1\leq k \leq d} \sup_{x_i\in[-1,1]^d} \left| \frac{\partial}{\partial_{x_{i,k}}} \int f_j(x_j) (\hat{p}_\varepsilon - p^*_\varepsilon)(\mvx) d\hat{\nu}_{-i}^N (\mvx) \right|\\
         & = \frac{1}{\varepsilon} 
         \max_{1\leq k \leq d} \sup_{x_i\in[-1,1]^d} \left |\int  f_j \left[ \hat{p}_\varepsilon(\mvx) \left( \frac{\partial}{\partial_{x_{i,k}}} [\hat{f}^*_i] - \frac{\partial}{\partial_{x_{i,k}}} c \right) - p^*_\varepsilon(\mvx)\left(\frac{\partial}{\partial_{x_{i,k}}} [f^*_i] - \frac{\partial}{\partial_{x_{i,k}}}  c \right)\right]  d\hat{\nu}_{-i}^N (\mvx) \right | \\
         &\lesssim
         \|f_j\|_\infty \max_{1\leq k \leq d}\sup_{x_i\in[-1,1]^d}\int \left| (\hat{p}_\varepsilon - p^*_\varepsilon)(\mvx) \right| d\hat{\nu}_{-i}^N (\mvx)\\
         &+  \|f_j\|_\infty \max_{1\leq k \leq d} \sup_{x_i\in[-1,1]^d} \left |\int  \hat{p}_\varepsilon(\mvx) \frac{\partial}{\partial_{x_{i,k}}} [\hat{f}^*_i]  - p^*_\varepsilon(\mvx)  \frac{\partial}{\partial_{x_{i,k}}} [f^*_i]  d\hat{\nu}_{-i}^N (\mvx) \right |\\
         &=
         \|f_j\|_\infty  \max_{1\leq k \leq d} \sup_{x_i\in[-1,1]^d}\int \left| (\hat{p}_\varepsilon - p^*_\varepsilon)(\mvx) \right| d\hat{\nu}_{-i}^N (\mvx) \\
         &+  \|f_j\|_\infty  \max_{1\leq k \leq d}\sup_{x_i\in[-1,1]^d} \left |\int  \hat{p}_\varepsilon(\mvx) (\frac{\partial}{\partial_{x_{i,k}}} [\hat{f}^*_i] - \frac{\partial}{\partial_{x_{i,k}}} [f^*_i]) +\frac{\partial}{\partial_{x_{i,k}}} [f^*_i] (\hat{p}_\varepsilon(\mvx)- p^*_\varepsilon(\mvx))  d\hat{\nu}_{-i}^N (\mvx) \right |\\
         &\lesssim
         \|f_j\|_\infty  \max_{1\leq k \leq d}\sup_{x_i\in[-1,1]^d}\int \left| (\hat{p}_\varepsilon - p^*_\varepsilon)(\mvx) \right| d\hat{\nu}_{-i}^N (\mvx)  +  \|f_j\|_\infty  \max_{1\leq k \leq d}\sup_{x_i\in[-1,1]^d} \left | \frac{\partial}{\partial_{x_{i,k}}} [\hat{f}^*_i] - \frac{\partial}{\partial_{x_{i,k}}} [f^*_i] \right|. \\
    \end{align*}
The proof is complete due to Lemma \ref{main_lemma_bound_on_norm_Aij-Dij} and Lemma \ref{diff_empirical_population_f_i_bounded_by_other_potentials}.
\end{proof}

\begin{proof}[Proof of Lemma \ref{bound_on_norm_D-H}]
\begin{comment}
    \begin{align*}
    \| D_{ij} - H_{ij} \|
    =&
    \sup_{f_j\in \cC^1([-1,1]^d)}\quad
    \sup_{x_i\in[-1,1]^d} \left|\int f_j(x_j)  p^*_\varepsilon(\mvx) d(\hat{\nu}_{-i}^N - \nu_{-i})(\mvx) \right|\\
    &+ \sup_{f_j\in \cC^1([-1,1]^d)}\quad
    \sup_{x_i\in[-1,1]^d} \left|\int f_j(x_j)  \nabla_{x_i} p^*_\varepsilon(\mvx) d(\hat{\nu}_{-i}^N - \nu_{-i})(\mvx) \right|
    \\
    =&
    \sup_{x_i\in[-1,1]^d}\quad
    \sup_{f_j\in \cC^1([-1,1]^d)}
     \left|\int f_j(x_j)  p^*_\varepsilon(\mvx) d(\hat{\nu}_{-i}^N - \nu_{-i})(\mvx) \right|\\
    &+  \sup_{x_i\in[-1,1]^d}\quad
    \sup_{f_j\in \cC^1([-1,1]^d)}
    \left|\int f_j(x_j)  \nabla_{x_i} p^*_\varepsilon(\mvx) d(\hat{\nu}_{-i}^N - \nu_{-i})(\mvx) \right|.
    \\
\end{align*}
\end{comment}
By Lemma \ref{matrix_like_operator_norm_bound}, it suffices to show that
\begin{equation}
        \| D_{ij} - H_{ij} \| = \cO_\bP \left( (m-1) N^{-1/(2\vee d)} (\log N)^{\mathbf 1_{\{d=2\}}} \right).
\end{equation}

\noindent Since $c\in\cC^2$, direct differentiation of the Schr\"{o}dinger system under the integral and Proposition \ref{BDP} gives  that $\|p_\varepsilon^*\|_{\cC^2} \lesssim 1$. Denote 
\(\frac {\partial} {\partial_{x_{i,0}}} p_\varepsilon^*:=p_\varepsilon^*\) for notational convenience. Then, for
\(0\leq k\leq d\),
\[
\frac {\partial} {\partial_{x_{i,k}}}
\bigl((D_{ij}-H_{ij})f_j\bigr)(x_i)
=
\int
f_j(x_j)\frac {\partial} {\partial_{x_{i,k}}}p_\varepsilon^*(x_i,x_{-i})
\,d(\hat{\nu}_{-i}^N - \nu_{-i})(x_{-i}),
\]
where \(k=0\) denotes the function itself. Moreover, by the product
rule, for $\nabla_{x_{-i}} : = (\nabla_{x_{\ell}})_{\ell\neq i}$
\[
\max_{0\leq k\leq d}
\sup_{x_i\in[-1,1]^d}
\Bigl\|
\nabla_{x_{-i}}
\bigl(
f_j(x_j)\frac {\partial} {\partial_{x_{i,k}}}p_\varepsilon^*(x_i,\cdot)
\bigr)
\Bigr\|_\infty
\lesssim
\|f_j\|_{\cC^1}.
\]
Therefore, Kantorovich--Rubinstein duality
\cite{villani2021topics} gives
\begin{align*}
\|(D_{ij}-H_{ij})f_j\|_{\cC^1}
&\lesssim
\max_{0\leq k\leq d}
\sup_{x_i\in[-1,1]^d}
\left|
\int
f_j(x_j)\frac {\partial} {\partial_{x_{i,k}}}p_\varepsilon^*(x_i,x_{-i})
\,d(\hat{\nu}_{-i}^N - \nu_{-i})(x_{-i})
\right| \\
&\lesssim
\|f_j\|_{\cC^1}
W_1(\hat{\nu}_{-i}^N,\nu_{-i}).
\end{align*}
Taking the supremum over \(\|f_j\|_{\cC^1}\leq1\), we obtain
\[
\|D_{ij}-H_{ij}\|
\lesssim
W_1(\hat{\nu}_{-i}^N,\nu_{-i}).
\]
Taking the product of optimal couplings between
\(\hat{\nu}_k^N\) and \(\nu_k\), \(k\neq i\), gives
\[
W_1(\hat{\nu}_{-i}^N,\nu_{-i})
\le
\sum_{k\neq i}W_1(\hat{\nu}_k^N,\nu_k).
\]
Standard empirical Wasserstein bounds yield that \cite{Lei_empiricalWasserstein,villani2021topics,WeedBach},
\[
W_1(\hat{\nu}_k^N,\nu_k)
=
\cO_\bP \left( N^{-1/(2\vee d)} (\log N)^{\mathbf 1_{\{d=2\}}} \right),
\qquad k\in[m]
\]
and the proof is complete.

\begin{comment}
\begin{align*}
    \sup_{f_j\in \cC^1([-1,1]^d)} &\quad  \sup_{x_i\in[-1,1]^d}
     \left\|\int f_j(x_j)   \nabla_{x_i} p^*_\varepsilon(\mvx) d(\hat{\nu}_{-i}^N - \nu_{-i})(\mvx) \right\|\\
     &=\sup_{x_i\in[-1,1]^d}\quad
    \sup_{f_j\in \cC^1([-1,1]^d)}
     \left\|\int f_j(x_j)   \nabla_{x_i} p^*_\varepsilon(\mvx) d(\hat{\nu}_{-i}^N - \nu_{-i})(\mvx) \right\|\\
     &\lesssim \|f_j\|_{\cC^1} \sup_{f\in \cC^1([-1,1]^{d(m-1)})}
     \left|\int f(\mvx) d(\hat{\nu}_{-i}^N - \nu_{-i})(\mvx) \right| = \|f_j\|_{\cC^1} \cO_\bP(N^{-\frac{1}{(m-1)d}}).
\end{align*}
\end{comment}
\end{proof}

\begin{proof}[Proof of Lemma \ref{Linearization_potential}]
Lemma \ref{Linearization_of_the_empirical_EMOT_system} gives
\begin{align}\label{linerization_equation_1}
\left \|\widehat{\mvGamma} \left( [\mvf^*] - [\hat{\mvf}^*] \right) - \varepsilon \fL \right\|_{\cC^1} = \cO_\bP \left( \frac{d\log N}{N} \right).
\end{align}
Lemma \ref{order_of_centered_linerized_B_i} and Lemma \ref{bound_on_norm_A-D} gives that 
\begin{align}\label{linerization_equation_2}
\left \|\widetilde{\mvGamma} \left( [\mvf^*] - [\hat{\mvf}^*] \right) - 
\widehat{\mvGamma} \left( [\mvf^*] - [\hat{\mvf}^*] \right) \right\|_{\cC^1} = o_\bP \left( N^{-1/2} \right).
\end{align}
Lemma \ref{order_of_centered_linerized_B_i} and Lemma \ref{bound_on_norm_D-H} gives that on some event $E^{(2)}_N$ with $\bP(E^{(2)}_N)\to1$ as $N\to\infty$, 
\begin{align}\label{linerization_equation_3}
\left \|\mvGamma \left( [\mvf^*] - [\hat{\mvf}^*] \right) - \widetilde{\mvGamma} \left( [\mvf^*] - [\hat{\mvf}^*] \right) \right\|_{\cC^1} = o_\bP \left( N^{-1/2} \right).
\end{align}
As a consequence of (\ref{linerization_equation_1}), (\ref{linerization_equation_2}) and (\ref{linerization_equation_3}),
\begin{align*}
    \left \|\mvGamma \left( [\mvf^*] - [\hat{\mvf}^*] \right) - \varepsilon \fL \right\|_{\cC^1} = o_\bP \left( N^{-1/2} \right).
\end{align*}
As mentioned in Step 4, $\mvGamma$ is invertible and $\|\mvGamma^{-1}\| \leq C$ for some constant. Thus,
\begin{align*}
    \left \| \left( [\mvf^*] - [\hat{\mvf}^*] \right) - \varepsilon\mvGamma^{-1}\fL \right\|_{\cC^1} = o_\bP \left( N^{-1/2} \right).
\end{align*}
\end{proof}

\noindent Some useful properties on linear operator theory are listed below.
\begin{lemma}\label{matrix_like_operator_norm_bound_0}
    Given an operator $S : \cC^1  \to \cC^1 $ defined as  $$S(\mvf):= \begin{pmatrix}
        L_1 \mvf\\
        \vdots\\
        L_m \mvf\\
    \end{pmatrix}$$
with $L_i: \cC^1 \to \cC^1([-1,1]^d) $ bounded operator defined for $i\in[m]$ with $\|L_i\|\leq M$, we have 
$$\|S\| \leq M.$$
\end{lemma}

\begin{proof}
    We see that
    \begin{equation*}
        \|Sf\|_{\cC^1} = \max_{1\leq i\leq m} \| (Sf)_i \|_{\cC^1} 
        \leq \max_{1\leq i\leq m} \|L_{i}\| \|\mvf\|_{\cC^1} \leq M \|\mvf\|_{\cC^1}.
    \end{equation*}
\end{proof}

\begin{lemma}\label{matrix_like_operator_norm_bound}
    Given an operator $S : \cC^1  \to \cC^1 $ defined as  $$S\begin{pmatrix}
        f_1\\
        \vdots\\
        f_m
    \end{pmatrix} := \begin{pmatrix}
        \sum_{j\neq 1} S_{1j}f_j\\
        \vdots\\
         \sum_{j\neq m} S_{mj}f_j\\
    \end{pmatrix}$$
with $S_{ij}: \cC^1([-1,1]^d) \to \cC^1([-1,1]^d) $ bounded operator defined for $i\neq j$ with $\|S_{ij}\|\leq K$, we have 
$$\|S\| \leq (m-1)K.$$
\end{lemma}

\begin{proof}
    We see that
    \begin{equation*}
        \|Sf\|_{\cC^1} = \max_{1\leq i\leq m} \| (Sf)_i \|_{\cC^1} \leq \max_{1\leq i\leq m} \sum_{j\neq i} \|S_{ij}\| \|f_j\|_{\cC^1} \leq (m-1) K\|f\|_{\cC^1}.
    \end{equation*}
\end{proof}

\section{Omitted proofs in the main paper}\label{omitted_proofs}
\begin{proof}[Proof of Proposition \ref{LipcontiMSB}]
Suppose $\pi^*_\varepsilon$ (resp. $\tilde{\pi}^*_\varepsilon$) is the optimal coupling of the $m$-marginal Schr\"{o}dinger system $S_{\varepsilon}(\nu_1,\dots,\nu_m)$ (resp. $S_{\varepsilon}(\tilde{\nu}_1,\dots,\tilde{\nu}_m)$) and  
$\gamma\in\Pi(\pi^*_\varepsilon,\tilde{\pi}^*_\varepsilon)$ is the optimal coupling attaining $W_1(\pi^*_\varepsilon,\tilde{\pi}^*_\varepsilon)$. 
Define $\pi_0\in\Pi\left({T_{\alpha}}_{\sharp}\pi^*_\varepsilon,{T_{\alpha}}_{\sharp}\tilde{\pi}^*_\varepsilon\right)$ via
$$\pi_0=(T_\alpha, T_\alpha)_{\sharp}\gamma,$$
namely, $\pi_0(A\times B)=\gamma(T_\alpha^{-1}(A)\times T_\alpha^{-1}(B))$ for $A,\,B$ measurable. Notice that $T_\alpha(\mvx)$ is continuous, and we have %Stone-Weierstrass theorem for product space
\begin{align}\label{Lip}
&W_1(\bar{\nu}_\varepsilon, \tilde{\nu}_{\varepsilon})=W_1({T_{\alpha}}_{\sharp}\pi^*_\varepsilon,{T_{\alpha}}_{\sharp}\tilde{\pi}^*_\varepsilon)\leq\int\|\mvx-\mvy\|d\pi_0(\mvx,\mvy)\nonumber\\
&=\int\|T_\alpha(\mvx)-T_\alpha(\mvy)\|d\gamma(\mvx,\mvy)\leq\int\|\mvx-\mvy\|d\gamma(\mvx,\mvy)=W_1(\pi_\varepsilon^*,\tilde{\pi}_\varepsilon),\nonumber\\
\end{align}
using the Lipschiz property of $T_\alpha(\mvx)$. Moreover, Theorem 3.3 in \cite{BayraktarEcksteinZhang_stability} implies that under this setting, 
\begin{equation}\label{shadowstability}
W_1(\pi_\varepsilon^*,\tilde{\pi}_\varepsilon)\leq \sqrt{m}W_2(\bm{\nu},\bm{\tilde{\nu}})+C W_2(\bm{\nu},\bm{\tilde{\nu}})^{1/2}.
\end{equation}
Here, $C > 0$ is a constant depending on $\varepsilon$, and  the second moment of $\nu_j$ and $\tilde{\nu}_j$ for $j \in [m]$. Combining (\ref{Lip}) and (\ref{shadowstability}), the proof is complete.
\end{proof}

\begin{proof}[Proof of Proposition \ref{BDP}]
The multimarginal Schr\"odinger system (\ref{multiconst_equiv}) yields
\[
\begin{aligned}
&e^{-\|c\|_\infty /\varepsilon}
\int_{\mathcal X_{-j}}
\exp\!\left(
\frac{\sum_{i\ne j}f_i^*(x_i)}{\varepsilon}
\right)\,d\nu_{-j}(x_{-j})\\
&\le
\int_{\mathcal X_{-j}}
\exp\!\left(
\frac{\sum_{i\ne j}f_i^*(x_i)-c(x_j,x_{-j})}
     {\varepsilon}
\right)\,d\nu_{-j}(x_{-j})\\
&=e^{-f_j^*(x_j)/\varepsilon}
\le e^{\|c\|_\infty/\varepsilon}
\int_{\mathcal X_{-j}}
\exp\!\left(
\frac{\sum_{i\ne j}f_i^*(x_i)}{\varepsilon}
\right)\,d\nu_{-j}(x_{-j})
\end{aligned}
\]
for $\nu_j$-almost every $x_j$. Namely,
\[
-
\|c\|_\infty
\le 
f_j^*(x_j)
+
\varepsilon\log \int_{\mathcal X_{-j}}
\exp\!\left(
\frac{\sum_{i\ne j}f_i^*(x_i)}{\varepsilon}
\right)\,d\nu_{-j}(x_{-j})
\le 
\|c\|_\infty.
\]
As a result, 
\begin{equation}
\label{eq:prop22-centered-bound}
\|f_j^*-\int f_j^*\,d\nu_j\|_{L^\infty(\nu_j)}\le 2\|c\|_\infty.
\end{equation}
Nonnegativity of relative entropy therefore gives
\[
S_\varepsilon(\nu_1,\ldots,\nu_m)\ge-\|c\|_\infty.
\]
Conversely, the product measure $\otimes_{k=1}^m\nu_k$ is an admissible coupling, so
\[
S_\varepsilon(\nu_1,\ldots,\nu_m)
\le \int c\,d (\otimes_{k=1}^m\nu_k)+\varepsilon\operatorname{KL}(\otimes_{k=1}^m\nu_k\|\otimes_{k=1}^m\nu_k)
=\int c\,d(\otimes_{k=1}^m\nu_k)\le \|c\|_\infty.
\]
Also note that the multimarginal Schr\"odinger system (\ref{multiconst_equiv}) and the normalization of the optimal potential gives that
\begin{equation}
\label{eq:prop22-mean-identity}
\begin{aligned}
S_\varepsilon(\nu_1,\ldots,\nu_m)=\sum_{j=1}^m \int f_j^*\,d\nu_j = \int f_m^*\,d\nu_m.
\end{aligned}
\end{equation}
Finally, 
\[
\begin{aligned}
\biggl|\sum_{j=1}^m f_j^*(x_j)\biggr|
&\le
\biggl|\sum_{j=1}^m \int f_j^*\,d\nu_j\biggr|
+\sum_{j=1}^m \biggl|f_j^*(x_j)-\int f_j^*\,d\nu_j \biggr|\le (2m+1) \|c\|_\infty.
\end{aligned}
\]
This completes the proof.
\end{proof}

\begin{proof}[Proof of Lemma \ref{norm}]
By the duality of operator norm, we can write
\begin{equation}
\|\nabla\Phi_{\varepsilon}(\vf)\|_{\mathscr{L}_{m}} = \sup\left\{\left\langle\nabla\Phi_{\varepsilon}(\vf),\vg\right\rangle_{\mathscr{L}_{m}}, \|\vg\|_{\mathscr{L}_{m}} \leq 1\right\}.
\end{equation}
Note that
\begin{align*}
&\langle\nabla\Phi_{\varepsilon}(\vf),\vg\rangle_{\mathscr{L}_{m}}\\
&=\sum_{j=1}^m\int\left[g_j\left(1-\exp\left({\frac{\sum_{i=1}^mf_i-c}{\varepsilon}}\right)\right)\right]d\left(\otimes_{k=1}^m\nu_k\right) \\
&\leq\sum_{j=1}^m\left(\int g_j^2d\nu_j\right)^{1/2} \left\{\int\left[\int1-\exp\left({\frac{\sum_{i=1}^mf_i-c}{\varepsilon}}\right)d\nu_{-j}(x_{-j})\right]^2d\nu_j\right\}^{1/2}\\
&\leq \left(\sum_{j=1}^m\int g_j^2d\nu_j\right)^{1/2} \left\{\sum_{j=1}^m\int\left[\int1-\exp\left({\frac{\sum_{i=1}^mf_i-c}{\varepsilon}}\right)d\nu_{-j}(x_{-j})\right]^2d\nu_j\right\}^{1/2},\\
\end{align*}
where the last two inequalities both follow from the Cauchy-Schwarz inequality, with the equality attained if  $g_j=\int\left(1-\exp\left({\frac{\sum_{i=1}^mf_i-c}{\varepsilon}}\right)\right)d\nu_{-j}(x_{-j})$,
for $j\in[m]$. This completes the proof.
\end{proof}

\begin{proof}[Proof of Proposition \ref{sconc}]
For $\vf=(f_1,\dots, f_m),\,\vg=(g_1,\dots,g_m)\in\mathcal{S}_L$, $t\in[0,1]$, we define
\begin{align*}
&h(t) := \Phi_{\varepsilon}((1-t)\vf+t\vg)\\
&=(1-t)\sum_{j=1}^m\int f_jd \nu_j+t\sum_{j=1}^m\int g_jd\nu_j-\varepsilon\int\exp\left(\frac{\sum_{j=1}^m\left((1-t)f_j+tg_j\right)-c}{\varepsilon}\right)d\left(\otimes_{j=1}^m\nu_j\right)+\varepsilon.
\end{align*}
Taking derivatives, we get
\begin{align*}
h'(t) &= \sum_{j=1}^m\int (g_j-f_j)d\nu_j-\int \left(\sum_{j=1}^m(g_j-f_j)\right)\exp\left(\frac{\sum_{j=1}^m\left((1-t)f_j+tg_j\right)-c}{\varepsilon}\right)d\left(\otimes_{j=1}^m\nu_j\right)\\
&=\left\langle\nabla\Phi_{\varepsilon}\left((1-t)\vf+t\vg\right), \vg-\vf\right\rangle_{\mathscr{L}_m},\\
h^{''}(t)&=-\frac{1}{\varepsilon}\int \left(\sum_{j=1}^m(g_j-f_j)\right)^2\exp\left(\frac{\sum_{j=1}^m\left((1-t)f_j+tg_j\right)-c}{\varepsilon}\right)d\left(\otimes_{j=1}^m\nu_j\right).
\end{align*}
The strong concavity (\ref{sconci}) could be rewritten as
\begin{equation}
	h(0)-h(1)\geq -h'(0)+\frac{\beta}{2}\|\vf-\vg\|_{\mathscr{L}_m}^2.
\end{equation}
It suffices to show that $h^{''}(t)\leq-\beta\|\vf-\vg\|_{\mathscr{L}_m}^2,$ for all $t\in[0,1]$, namely
\begin{align}\label{SL}
&\frac{1}{\varepsilon}\int\left(\sum_{j=1}^m(g_j-f_j)\right)^2\exp\left(\frac{\sum_{j=1}^m\left((1-t)f_j+tg_j\right)-c}{\varepsilon}\right)d\left(\otimes_{j=1}^m\nu_j\right)\geq\beta\sum_{j=1}^m\int \left(g_j-f_j\right)^2d\nu_j\nonumber\\
&=\beta\int\sum_{j=1}^m\left(g_j-f_j\right)^2d(\otimes_{j=1}\nu_k)=\beta\int\left(\sum_{j=1}^m(g_j-f_j)\right)^2d(\otimes_{j=1}\nu_k),
\end{align}
where we make use of the fact that $\vf,\vg\in\mathcal{S}_L$ to derive the last equality.
The fact that $\vf,\vg\in\mathcal{S}_ L$ indicates that $\beta=\frac{1}{\varepsilon}\exp\left(-\frac{L+\|c\|_\infty}{\varepsilon}\right)$ qualifies to make (\ref{SL}) hold true.
\end{proof}

\section{Technical lemmas}\label{appendix_technical_lemmas}

\begin{lemma}[Tightness of couplings]\label{tight}
Let $\mathcal{Y}=\Pi_{j=1}^m\mathcal{Y}_j$. Assume that $\mathscr{{T}}_j\subset\mathcal{P}(\mathcal{Y}_j)$ is tight for $j\in[m]$. Then the set $\Pi(\mathscr{T}_1,\dots,\mathscr{T}_m) := \{ \gamma \in \mathcal{P}(\mathcal{Y}) \mid {e_j}_{\sharp}\gamma\in\mathscr{T}_j \}$
is tight. Particularly, we have the tightness of $\Pi(\nu_1,\dots,\nu_m) = \{ \gamma \in \mathcal{P}(\mathcal{Y}) \mid {e_j}_{\sharp}\gamma=\nu_j \}$.
\end{lemma}

\begin{proof}[Proof of Lemma \ref{tight}]
Let $\delta > 0$. By the tightness of $\mathscr{T}_j$ we can find a compact set $K_j \subseteq \mathcal{Y}_j$ such that $\mu_j(\mathcal{Y}_j \setminus K_j) < \frac{\delta}{m}$, for any $\mu_j\in\mathscr{T}_j$.

Let $K := K_1 \times \dots \times K_m$ and let $\gamma \in \Pi(\mathscr{T}_1,\dots,\mathscr{T}_m)$. Due to the fact  ${e_j}_{\sharp}\gamma\in\mathscr{T}_j$, and the fact that
$$ \mathcal{Y} \setminus K \subseteq \left(\left((\mathcal{Y}_1 \setminus K_1) \times \prod_{k=2}^m \mathcal{Y}_k\right) \bigcup \left(\mathcal{Y}_1 \times (\mathcal{Y}_2 \setminus K_2) \times \prod_{k=3}^m \mathcal{Y}_k\right) \bigcup \dots \bigcup \left(\prod_{k=1}^{m-1} \mathcal{Y}_k \times (\mathcal{Y}_m \setminus K_m)\right)\right), $$
one has $\gamma(\mathcal{Y} \setminus K) \leq \delta$.
\end{proof}

\begin{lemma}[Polyak-\L ojasiewicz inequality]\label{PLineq}
Let \( S \subset \mathbb{H} \) be a convex subset of a Hilbert space \( \mathbb{H} \) and \( f : \mathbb{H} \to \mathbb{R} \) be a \(\beta\)-strongly convex function on \( S \). Then, we have for all \( x \in S \),
\[
	f(x) - \inf_{y \in \mathbb{H}} f(y) \leq \frac{1}{2\beta} \|\nabla f(x)\|_{\mathbb{H}}^2.
\]
\end{lemma}
\begin{proof}
	See \cite{KarimiNutiniSchmidt_PL}.
\end{proof}

\begin{lemma}[Hoeffding's inequality in Hilbert space]\label{lem:hoeffding_hilbert}
    Let $X_1, \dots, X_n$ be independent mean-zero random variables taking values in a Hilbert space $(\bH, \|\cdot\|_{\bH})$. If $\|X_i\|_{\bH} \leq C$ for some constant $C > 0$, then for every $t > 0$, we have
    \begin{equation}
        \bP \left( \left\| \sum_{i=1}^n X_i \right\|_{\bH} \geq t \right) \leq 2 \exp\left( -{\frac{t^2} {8 n C^2}} \right).
    \end{equation}
\end{lemma}

\begin{proof}
    See Lemma 17 in~\cite{rigollet2022samplecomplexityentropicoptimal}.
\end{proof}

\begin{lemma}[Dudley's entropy integral bound]\label{dudley}
Let $\{X_t,\, t \in T\}$ be a zero-mean process satisfying the sub-Gaussian condition with respect to distance $\rho_X$, i.e., for any $t,\tilde{t}\in T$, $\mathbb{P}(|X_t-X_{\tilde{t}}|\geq \tau)\leq C_1\exp\left({-\frac{C_2\tau^2}{\rho^2_X(t,\tilde{t})}}\right)$ for some universal constant $C_1, C_2>0$. Then for any $\delta \in [0,D]$, with $D=\sup_{t,\tilde{t}\in T}\rho_X(t,\tilde{t})$ denoting the diameter of $T$ under $\rho_X$, we have, for some universal constants $C_3,C_4>0$,
\[
\mathbb{E}\left[\sup_{t,\tilde{t}\in T} (X_t - X_{\tilde{t}})\right] \leq C_3\mathbb{E}\left[\sup_{\substack{\gamma,\gamma'\in T\\\rho_X(\gamma,\gamma')\leq\delta}} (X_\gamma - X_{\gamma'})\right] + C_4\int_{\delta/4}^D\sqrt{\log N(T,\varepsilon, \rho_X)}d\varepsilon.
\]
\end{lemma}
\begin{proof}
    See Theorem 5.22 in \cite{Wainwright_2019}.
\end{proof}

Once the order of the difference of the empirical and population potential is settled as in Lemma \ref{diff_empirical_population_f_i_bounded_by_other_potentials},
the following lemma is a straightforward extension of Lemma B.1 in \cite{AlbertoGonzalez-SanzJean-MichelLoubesJonathanNiles-Weed2024weaklimit} from two marginal case to non-smooth multimarginal case. The proof is just Taylor expansion of exponential function as suggested by \cite{AlbertoGonzalez-SanzJean-MichelLoubesJonathanNiles-Weed2024weaklimit}. 
\begin{lemma}\label{exponential_potential_infinity_norm_order}
    Let \(\Omega \subset \mathbb{R}^{d}\) be a compact set, \(\nu_1,\dots,\nu_m \in \cP(\Omega)\), and let \(\hat{\nu}^N_1,\dots,\hat{\nu}^N_m \) be their empirical measures. Then
    \begin{align}\label{bound_taylor_expansion_littleop_order}
    \left\|{e^{ \frac{\sum_{i=1}^m \hat{f}^*_i(x_i)}{\varepsilon}} - e^{ \frac{\sum_{i=1}^m f^*_i(x_i)}{\varepsilon}} - e^{ \frac{\sum_{i=1}^m f^*_i(x_i)}{\varepsilon}}\left( \frac{ \sum_{i=1}^m \hat{f}^*_i(x_i)-f^*_i(x_i)}{\varepsilon} \right)} \right\|_{\cC^1}
        &= \cO_{\bP}\left( \frac{\log N}{N} \right), %= o_\bP(N^{-1/2} ),
    \end{align}
    \begin{align}\label{bound_taylor_expansion_bigO_order}
 \left\| {e^{ \frac{\sum_{i=1}^m \hat{f}^*_i(x_i)}{\varepsilon}} - e^{\frac{\sum_{i=1}^m f^*_i(x_i)}{\varepsilon}}} \right\|_{\cC^1}
        &= \mathcal{O}_{\bP}\left(  \sqrt{\frac{\log N}{N}}\right). 
    \end{align}
\end{lemma}

\begin{proof}
Since \(c\in C^2\), Proposition \ref{BDP}, Proposition \ref{empirical_BDP}, and the Schrödinger equations imply that the population and empirical potentials are uniformly bounded in \(C^1\). Thus we have, via second-order Taylor expansion
\begin{align*}
 &\left \|{ e^{ \frac{\sum_{i=1}^m \hat{f}^*_i(x_i)}{\varepsilon}} - e^{ \frac{\sum_{i=1}^m f^*_i(x_i)}{\varepsilon} } -e^{ \frac{\sum_{i=1}^m f^*_i(x_i)}{\varepsilon}}\left( \frac{\sum_{i=1}^m \hat{f}^*_i(x_i) - \bar{f}^*_i(x_i)}{\varepsilon}\right)}\right\|_{\cC^1}\\
 &= \left\|{e^{ \frac{\sum_{i=1}^m \hat{f}^*_i(x_i)}{\varepsilon}} - e^{ \frac{\sum_{i=1}^m \bar{f}^*_i(x_i)}{\varepsilon} } -e^{ \frac{\sum_{i=1}^m f^*_i(x_i)}{\varepsilon}}\left( \frac{\sum_{i=1}^m \hat{f}^*_i(x_i) - 
 \bar{f}^*_i(x_i)}{\varepsilon}\right)}\right\|_{\cC^1}\\
 &= \cO_{\bP}\left(\left\|\sum_{i=1}^m \hat{f}^*_i(x_i)-\bar{f}^*_i(x_i)\right\|^2_{\cC^1}\right).
\end{align*}
By virtue of Lemma \ref{diff_empirical_population_f_i_bounded_by_other_potentials}, we have  
\begin{align}\label{empirical_popularion_dual_difference_order}
\left\|\sum_{i=1}^m \hat{f}^*_i(x_i)-\bar{f}^*_i(x_i)\right\|_{\cC^1}
= \cO_\bP \bigl( \sqrt{\log N/N} \bigr),
\end{align}
and then (\ref{bound_taylor_expansion_littleop_order}) follows. To prove (\ref{bound_taylor_expansion_bigO_order}), apply the reverse triangle inequality to (\ref{bound_taylor_expansion_littleop_order}):
\begin{align}\label{reverse_triangle_inequality}
\left\|e^{ \frac{\sum_{i=1}^m \hat{f}^*_i(x_i)}{\varepsilon}} - e^{ \frac{\sum_{i=1}^m f^*_i(x_i)}{\varepsilon}} - e^{ \frac{\sum_{i=1}^m f^*_i(x_i)}{\varepsilon}}\left(\frac{\sum_{i=1}^m \hat{f}^*_i(x_i) - \bar{f}^*_i(x_i)}{\varepsilon}\right)\right\|_{\cC^1} 
\end{align}
\begin{align*}
\geq \left\|e^{ \frac{\sum_{i=1}^m \hat{f}^*_i(x_i)}{\varepsilon}} - e^{ \frac{\sum_{i=1}^m f^*_i(x_i)}{\varepsilon}}\right\|_{\cC^1} - \left\|e^{ \frac{\sum_{i=1}^m f^*_i(x_i)}{\varepsilon}}\left(\frac{\sum_{i=1}^m \hat{f}^*_i(x_i) - \bar{f}^*_i(x_i)}{\varepsilon}\right)\right\|_{\cC^1}.
\end{align*}
Apply (\ref{empirical_popularion_dual_difference_order}) , (\ref{reverse_triangle_inequality}) and Proposition \ref{BDP} to get:
\[
\left\|e^{ \frac{\sum_{i=1}^m \hat{f}^*_i(x_i)}{\varepsilon}} - e^{ \frac{\sum_{i=1}^m f^*_i(x_i)}{\varepsilon}}\right\|_{\cC^1}
    \lesssim_\varepsilon \left\|\left(\sum_{i=1}^m \hat{f}^*_i(x_i)-\bar{f}^*_i(x_i)\right)\right\|_{\cC^1} + o_{\bP}\left( N^{-1/2}  \right) 
    = 
    \cO_\bP \bigl( \sqrt{\log N/N} \bigr).
\]
\end{proof}

\begin{lemma}[Vanishing term B]\label{vanishing_term_B_proved_via_chaining} 
\begin{equation*}
    B = \sqrt{N}\int g (\mvx) e^{-\frac{c(\mvx)}{\varepsilon}}(e^{\frac{\sum_{i=1}^m \hat{f}^*_i(x_i)}{\varepsilon}} - e^{\frac{\sum_{i=1}^m f^*_i(x_i)}{\varepsilon}}) d\left(\otimes_{k=1}^m\hat{\nu}^N_k - \otimes_{k=1}^m\nu_k\right)(\mvx) = o_\bP(1).
\end{equation*}
\end{lemma}

\begin{proof}
Set $\cH := \{ f g e^{-c/\varepsilon} , f\in\cF\}$ with $\cF = \{ f\in\cC^1([-1,1]^{md}), \|f\|_{\cC^1}\leq 1  \}$. For any $f,f'\in\cF$, note that $\|fg e^{-c/\varepsilon}  -f'g e^{-c/\varepsilon} \|_{\infty}\leq \|ge^{-c/\varepsilon}\|_{\infty}\,\|f-f'\|_{\infty}$, hence any $(\varepsilon/\|ge^{-c/\varepsilon}\|_\infty)$-net $\{f_1,\ldots,f_K\}$ of $\cF$ gives an
$\varepsilon$-net $\{f_1ge^{-c/\varepsilon},\ldots,f_Kge^{-c/\varepsilon}\}$ of $\cH$. Therefore $\sup_{h\in\cH}\|h\|_{\infty}
\leq \|ge^{-c/\varepsilon}\|_\infty \sup_{f\in\mathcal F}\|f\|_{\infty} < \infty$ and
\[
\cN(\varepsilon,\cH,\|\cdot\|_{\infty})
\le \cN(\varepsilon/\|ge^{-c/\varepsilon}\|_\infty,\cF,\|\cdot\|_{\infty}).
\]
The remaining arguments follows the chaining argument Lemma \ref{dudley}, same as that in the proof of Theorem \ref{supremp} in Section \ref{sec_concentration_of_empirical_potentials_and_joint_optimal_coupling_density} by virtue of the sub-Gaussianity from Lemma \ref{Ustat}, now with dimension $md$. Namely, we have, up to some constant depending on $m,\|g\|_\infty,\varepsilon,d$,
\begin{align*}
| B |
=
\cO_\bP
\biggl(
N^{\,1/2-1/(2\vee md)}
(\log N)^{\mathbf1_{\{md=2\}}} \Bigl\|e^{\frac{\sum_{i=1}^m \hat{f}^*_i(x_i)}{\varepsilon}} - e^{\frac{\sum_{i=1}^m f^*_i(x_i)}{\varepsilon}}\Bigr\|_{\cC^1}
\biggr). 
\end{align*}
Combining with Lemma \ref{exponential_potential_infinity_norm_order} and the proof is complete.
\end{proof}

\begin{lemma}\label{V_statistics_U_statistics_difference}
    Suppose $\Phi(x_{\alpha}^{(j)};\alpha= i_{j1},\dots,i_{jm_j}, j=1,\dots,m)$ is a bounded kernel of order $m$ and degree $(m_1,\dots,m_m)$ with $m_k>1$ for some $k\in[m]$, symmetric inside each group. Then given $m$ independent series $\{X^{(j)}_{1},\dots,X^{(j)}_{N}\}, j=1,\dots,m$ of independent random variables,  for 
    \begin{align*}
        V_N := \frac{1}{N^{\sum_{k=1}^m m_k}}\sum_{\substack{1\leq i_{jk}\leq N\\1\leq k\leq m_j\\j=1,\dots,m}}\Phi(X_{\alpha}^{(j)};\alpha= i_{j1},\dots,i_{jm_j}, j=1,\dots,m),
    \end{align*}
and     
    \begin{align*}
        U_N :=
        \frac{1}{\prod_{k=1}^m \tbinom{N}{m_k}} 
        \sum_{\substack{1\leq i_{j1}<\dots<i_{jm_j}\leq N\\\, j=1,\dots,m}}
        \Phi(X_{\alpha}^{(j)};\alpha= i_{j1},\dots,i_{jm_j}, j=1,\dots,m),
    \end{align*}
we have
\begin{align*}
    |V_N-U_N| \lesssim \frac{\|\Phi\|_\infty}{N}.
\end{align*}
\end{lemma}

\begin{proof}
Due to the symmetry of the kernel $\Phi$ inside each group, we could write
\begin{align*}
    U_N =  \frac{1}{\prod_{k=1}^m (N)_{m_k}} 
        \sum_{\substack{1\leq i_{j1},\dots,i_{jm_j}\leq N\\i_{jp}\neq i_{jq},\, 1\leq p<q\leq m_j \\ j=1,\dots,m}}
        \Phi(X_{\alpha}^{(j)};\alpha= i_{j1},\dots,i_{jm_j}, j=1,\dots,m),
\end{align*}
where $(N)_{m_k}= N(N-1)\dots(N-m_k+1)$ denotes the falling factorial notation. At the same time, observe that
\begin{align*}
    V_N-U_N 
    & = \left(\frac{1}{N^{\sum_{k=1}^m m_k}}-\frac{1}{\prod_{k=1}^m (N)_{m_k}} \right)
   \sum_{\substack{1\leq i_{j1},\dots,i_{jm_j}\leq N\\i_{jp}\neq i_{jq},\, 1\leq p<q\leq m_j \\ j=1,\dots,m}}
    \Phi(X_{\alpha}^{(j)};\alpha= i_{j1},\dots,i_{jm_j}, j=1,\dots,m)\\
    & +  \frac{1}{N^{\sum_{k=1}^m m_k}}\sum_{*}
    \Phi(X_{\alpha}^{(j)};\alpha= i_{j1},\dots,i_{jm_j}, j=1,\dots,m) \\
    &: = A+B, 
\end{align*}
where $\sum_{*}$ denotes the summation over all index $(i_{j\alpha};\alpha = i_{j_1},\dots,i_{j_{m_j}},j=1,\dots,m)$, $1\leq i_{j1},\dots,i_{jm_j}\leq N,  j=1,\dots,m$ with $i_{jp}= i_{jq}$ for some $1\leq p<q\leq m_j$.
We have, by the boundedness of the kernel, say, $|\Phi(\cdot)|\leq C$, for term A, 
\begin{align*}
   |A| 
   &\leq C  \prod_{k=1}^m (N)_{m_k} \left|\frac{1}{N^{\sum_{k=1}^m m_k}} -\frac{1}{\prod_{k=1}^m (N)_{m_k}} \right|\\
   &=   \frac{C}{N^{-m+\sum_{k=1}^m m_k}}\left|\prod_{i=1}^m\prod_{j=1}^{m_i-1}(N-j)-N^{-m+\sum_{k=1}^m m_k}\right| \lesssim \frac{1}{N}.
\end{align*}
Here the last inequality follows from the fact that the highest degree term of the polynomial $\prod_{i=1}^m\prod_{j=1}^{m_i-1}(N-j)$ equals $N^{-m+\sum_{k=1}^m m_k}$ and the assumption that $m_k>1$ for some $k\in[m]$.
Similarly, for term B, noticing that there are $ N^{\sum_{k=1}^m m_k} - \prod_{k=1}^m (N)_{m_k}$ many terms in the summation $\sum_*$,
\begin{align*}
    |B| 
    &\leq \frac{C}{N^{\sum_{k=1}^mm_k}}\left|  N^{\sum_{k=1}^m m_k} - \prod_{k=1}^m (N)_{m_k} \right|\\
    &= \frac{C}{N^{-m+\sum_{k=1}^mm_k}} \left| N^{-m+\sum_{k=1}^m m_k} - \prod_{i=1}^m\prod_{j=1}^{m_i-1}(N-j)\right|\lesssim \frac{1}{N},
\end{align*}
following the same reasoning as above. 
\end{proof}

\section{Other Auxiliary tools}\label{Section_Hadamard_differentiability_Intro}
We briefly recapitulate useful tools related to bootstrap consistency in this section, in particular, Hadamard differentiability and functional delta method. For two normed spaces $\fD$ and $\fC$, a map $\phi:\Theta\subset\fD\to\fC$ is called Hadamard directionally differentiable at $\theta\in\Theta$ if there exists a map $\phi'_\theta: \fF_\Theta(\theta)\to\fC$ such that
\begin{align}\label{definition_Hadamard_differentiability}
    \lim_{t\to0^+}\frac{\phi(\theta_t)  - \phi(\theta)} {t} = \phi'_\theta(h),
\end{align}
for any sequence $(\theta_t)_{t>0}\subset\Theta$ with $\frac{\theta_t-\theta}{t}\to h$ as $t\to0^+$.
Here, $\fF_\Theta(\theta)$ is the tangent cone to $\Theta$ at $\theta$ defined as
\begin{align*}
    \fF_\Theta(\theta) = \{ h\in\fD, \quad h= \lim_{\substack{ \theta_t\to\theta\in\Theta  \\t\to0^+ } } \frac{\theta_t-\theta}{t}  \}.
\end{align*}
The derivative $\phi'_\theta$ can be shown to be continuous and positively homogeneous but not necessarily linear. If (\ref{definition_Hadamard_differentiability}) holds for $h\in\fD_0$ for some subset $\fD_0\subset\fF_\Theta(\theta)$, then we say $\phi$ is Hadamard directionally differentiable at $\theta\in\Theta$ tangentially to $\fD_0$. In such case, the derivative $\phi'_\theta$ is only valid on $\fD_0$. Finally, if the derivative $\phi'_\theta$ is linear, we call $\phi$ to be Hadamard differentiable at $\theta$ (tangentially to $\fD_0$ if $\phi'_\theta$ is defined only on $\fD_0$).

\begin{lemma}[Chain rule for Hadamard differentiability]
    If $\phi : \Theta_\phi\subset\fD \to\fC$ is Hadamard differentiable at $\theta\in\Theta_\phi$ tangentially to $\fD_0$ and $\psi: \Theta_\psi\subset\fC\to\fL$ is Hadamard differentiable at $\phi(\theta)$ tangentially to $\phi'_\theta(\fD_0)$, then $\psi\circ\phi : \Theta_\phi\to \fL$ is Hadamard differentiable at $\theta$ tangentially to $\fD_0$ with derivative $\psi'_{\phi(\theta)}\circ \phi'_\theta$.
\end{lemma}

\begin{lemma}[Functional delta method]
Let $\fD$ and $\fC$ be two normed spaces, and a map $\phi:\Theta\subset\fD\to\fC$ that is Hadamard directionally differentiable at $\theta\in\Theta$ tangentially to some set $\fD_0\subset\fF_\Theta(\theta)$. Let $T_n:\Omega\to\Theta$ be maps such that $r_n(T_n-\theta)\overset{w}{\longrightarrow}T$
for some $r_n\to\infty$ and Borel measurable map $T:\Omega\to\fD$ with values in a separable subset of $\fD_0$. Then
\begin{itemize}
    \item $r_n( \phi(T_n) - \phi(\theta) )\overset{w}{\longrightarrow} \phi'_\theta(T)$;
    \item If $\Theta$ is also convex and $\fD_0 = \fF_\Theta(\theta)$, then $r_n \left( \phi(T_n) - \phi(\theta)  - \phi'_\theta( T_n - \theta ) \right)\to 0$ in outer probability.
\end{itemize}
For more details on outer probability, see \cite{vanderVaartWellner1996}. The proofs of the previous two theorems could be found in \cite{vanderVaartWellner1996,Romisch_Werner_Delta_Method_Intro}.

\end{lemma}

\begin{proof}[Proof of Theorem \ref{bootstrap_cost} and Theorem \ref{bootstrap_coupling}]
The proof of either Theorem \ref{bootstrap_cost} or Theorem \ref{bootstrap_coupling} is the same,  basically applying Theorem 23.9 in \cite{Van_der_Vaart_Asymptotic_statistics00}. Let us first verify the Hadamard differentiability. Recall that the following map $\cH$ is Hadamard differentiable by Theorem 6 in \cite{ZivGoldfeldKengoKatoGabrielRiouxRitwikSadhu}. It is assumed that $c\in\cC^\infty$ in \cite{ZivGoldfeldKengoKatoGabrielRiouxRitwikSadhu} for simplicity and essentially one only needs $c\in\cC^{2k}$ with $k\geq d/2+1$ by tracking their proof.
\begin{align*}
    \cH: &\prod_{j=1}^m\cP(\cX_j) \quad \to \quad \prod_{j=1}^m \cC^k(\cX_j)\\
    &\mvnu : = (\nu_1,\dots,\nu_m) \mapsto \cH(\mvnu)=(f^*_1,\dots,f^*_m)
\end{align*}
with $(f^*_1,\dots,f^*_m)$ denotes the Schr\"{o}dinger potential associated to $\mvnu : = (\nu_1,\dots,\nu_m)$.
As a consequence, 
$$   S_\varepsilon(\mvnu) = \sum_{j=1}^m \nu_j(f^*_j)$$ 
and 
$$F(\mvnu) := \int g(\mvx) \exp\left(\frac{\sum_{j=1}^mf^*_j(x_j) - c(\mvx)}{\varepsilon}\right) d\otimes_{j=1}^m\nu_j(x_j)$$
are both Hadamard differentiable. This fact comes obvious once we realize 
$$ S_\varepsilon(\mvnu) = H_1 (\mvnu, \cH(\mvnu))$$
with
\begin{align*}
    H_1 : \prod_{j=1}^m\cP(\cX_j) \times \prod_{j=1}^m \cC^k(\cX_j) &\to \bR\\
    (\mvnu,\mvf) &\mapsto H_1 (\mvnu,\mvf)= \sum_{j=1}^m \int f_j(x_j) d\nu_j(x_j)
\end{align*}
being Hadamard differentiable and 
$$ F(\mvnu) = H_2 \left((\text{id},\cH)(\mvnu) \right)$$ with
\begin{align*}
    H_2 : \prod_{j=1}^m\cP(\cX_j) \times \prod_{j=1}^m \cC^k(\cX_j) &\to \bR\\
    (\mvnu,\mvf) &\mapsto H_2 (\mvnu,\mvf)=  \int g(\mvx)\exp\left(\frac{\sum_{j=1}^mf_j(x_j) - c(\mvx)}{\varepsilon}\right) d(\otimes_{k=1}^m\nu_k)(\mvx)
\end{align*}
Hadamard differentiable. Define $\cF_{0,j} := \{ f\in\cC^k(\cX_j), \|f\|_{\cC^k} \leq R_k\}$ for some constant $R_k>0$ depending on $c,\varepsilon,d$ and we know that the optimal potential $f^*_j\in\cF_{0,j}$. For $f_j\in \cF_{0,j}$, to shorten the notation, we denote $\sum_{j=1}^mf_j(x_j)$ as $f_1\oplus f_2\oplus\dots\oplus f_m$. After that, define $$\cF^* := \left\{f_1\oplus f_2\oplus\dots\oplus f_m; f_j\in \cF_{0,j}, j\in[m] \right\},$$
we need to check that as maps into $\ell^\infty(\cF^*)$, the sequence
\[
\sqrt{N} \left( \otimes_{j=1}^m\hat{\nu}_j^B  - \otimes_{j=1}^m\hat{\nu}^N_j \right)
\]
is asymptotically measurable and converges conditionally in distribution to
\[
\left(\sum_{j=1}^mG^{(j)}_{\nu_j}(f_j) \right)_{\oplus_{j=1}^m f_j \in \cF^*}
\]
given $(X^{(1)}_1,\dots,X^{(m)}_1), (X^{(1)}_2,\dots,X^{(m)}_2), \dots, (X^{(1)}_N,\dots,X^{(m)}_N)$. Here, $G^{(j)}_{\nu_j}$ are independent and the weak limits of $\sqrt{N} \left(\hat{\nu}_j^N - \nu_j\right)$, $j\in[m]$. Since $\cF_{0,j}$ is Donsker w.r.t.\ $\nu_j$ (cf. \cite{vanderVaartWellner1996}), each of
\[
\sqrt{N}\bigl(\hat{\nu}_j^B - \hat{\nu}^N_j\bigr)
\]
is asymptotically measurable and converges conditionally in distribution for $j\in [m]$
(cf.\ Chapter~3.6 in \cite{vanderVaartWellner1996}).
By Lemma~1.4.4 and Example~1.4.6 in \cite{vanderVaartWellner1996}, as maps into
$\ell^\infty(\cF_{0,1})\times\dots\times \ell^\infty(\cF_{0,m})$, the sequence
\[
\left( \sqrt{N}(\hat{\nu}_1^B - \hat{\nu}_1^N),
\dots, \sqrt{N}(\hat{\nu}_m^B - \hat{\nu}_m^N) \right)
\]
is asymptotically measurable and converges conditionally in distribution to
$(G^{(1)}_{\nu_1},\dots,G^{(m)}_{\nu_m})$. Since the map
\[
\ell^\infty(\cF_{0,1})\times\dots\times \ell^\infty(\cF_{0,m}) \ni (q_1,\dots,q_m)
\mapsto \left(\sum_{j=1}^m q_j(f_j) \right)_{\oplus_{j=1}^m f_j \in \cF^*}
\in \ell^\infty(\cF^*)
\]
is continuous, we see that, as maps into $\ell^\infty(\cF^*)$,
\[
\sqrt{N} \left( \otimes_{j=1}^m\hat{\nu}_j^B  - \otimes_{j=1}^m\hat{\nu}^N_j \right)
\]
is asymptotically measurable and converges conditionally in distribution to
\[
\left(\sum_{j=1}^mG^{(j)}_{\nu_j}(f_j) \right)_{\oplus_{j=1}^m f_j \in \cF^*}
\]
as desired. The rest follows from Theorem~23.9 in \cite{Van_der_Vaart_Asymptotic_statistics00}.
\end{proof}

\section{ Bootstrap verification} \label{sec: bootstrap_verification}
For each specified triple of marginal distributions and sample
size $n\in\{50,100,250,500,1000\}$, we generate $R=5000$ independent
datasets, each consisting of $n$ independent observations from each
of the three marginal distributions. The three marginal samples
are independent. For dataset $r$, let
$\hat\theta_n^{(r)}$ denote the empirical plug-in estimate of a scalar target
$\theta^*$, obtained by solving the empirical entropic transport
problem. Independently within each marginal sample, we draw
$B=4999$ bootstrap resamples of size $n$ with replacement.
We refit the transport problem for each resample, obtaining
$\hat\theta_n^{(r,b)}$, $b=1,\ldots,B$. In the experiment, we consider 
$$\theta^* \in \Bigl\{ S_\varepsilon, \int (\frac{x_1+x_2+x_3}{3})^2 d \pi_\varepsilon^*, F(t): =\int \mathbf 1_{ \{c(x,y,z)\leq t\} } d \pi_\varepsilon^* 
\Bigr\}$$
where
$
S_\varepsilon
=
\inf_{\pi\in\Pi(\mu_1,\mu_2,\mu_3)}
\left\{
\int c\,d\pi
+
\varepsilon\operatorname{KL}
\!\left(\pi\,\middle\|\,\mu_1\otimes\mu_2\otimes\mu_3\right)
\right\}
$
is the cost functional and
$\pi_\varepsilon^*$ denotes the corresponding optimal coupling. $\int (\frac{x_1+x_2+x_3}{3})^2 d \pi_\varepsilon^*$ and $F(t)$ 
assess bootstrap validity for the Schr\"odinger coupling integrated against specific test functions. 
$F(t)$ is called colocalization curve \cite{Alberto_Sanz_Gonzalez_Shayan_EOT_non_smooth_CLT23}. 

We consider both regularization parameters
$\varepsilon\in\{0.05,0.2\}$ and both squared and unsquared cost functions
\[
c_1(x_1,x_2,x_3)
=
\frac14\sum_{1\le j<k\le3}|x_j-x_k|,
\qquad
c_2(x_1,x_2,x_3)
=
\frac18\sum_{1\le j<k\le3}(x_j-x_k)^2.
\]
For every specified marginal triple, we include all combinations
\[
(c,\varepsilon,n)
\in
\{c_1,c_2\}
\times\{0.05,0.2\}
\times\{50,100,250,500,1000\}.
\]
Thus, there are $20$ conditions per marginal triple.
Define the centered bootstrap differences
$
D_n^{(r,b)}
=
|\hat\theta_n^{(r,b)}-\hat\theta_n^{(r)}|,
$
and let $Q_{r,u}^{B}$ denote their empirical $u$-quantile.
%computed using linear interpolation.
At nominal confidence level $\gamma$, the basic bootstrap
interval is
$
I_{n,\gamma}^{(r)}
=
\left[
\hat\theta_n^{(r)}-Q_{r,\gamma}^{B},
\;
\hat\theta_n^{(r)}+
Q_{r,\gamma}^{B}
\right].
$
Its empirical coverage is
$
\widehat C_n(\gamma)
=
\frac{1}{R}\sum_{r=1}^{R}
\mathbf{1}
\left\{
\theta^*\in I_{n,\gamma}^{(r)}
\right\},
$
where we approximate \(\theta^*\) by solving the population entropic transport problem using a refined quadrature discretization of the known marginal distributions, independently of the simulated datasets, since closed-form population values are unavailable. 
 
Figure~\ref{fig:bootstrap-calibration} summarizes bootstrap
calibration for the four groups of marginal configurations considered below.
For each statistic and each combination of marginal triples,
cost, regularization parameter, and sample size, we report
$
100\max_{\gamma\in\Gamma}
\bigl|\widehat C_n(\gamma)-\gamma\bigr|$, 
with 
$
\Gamma=\{0.05,0.10,\ldots,0.95\}.
$
Each entry therefore represents
the largest absolute calibration error over the specified
confidence levels, expressed in percentage points.

The corresponding calibration curves are presented in
Figures~\ref{fig:calibration-beta-2-8},
\ref{fig:calibration-beta-2-6},
\ref{fig:calibration-beta-2-2-5-5-9-9}, and
\ref{fig:calibration-beta-2-4}. Each curve plots empirical coverage against
nominal confidence for a fixed sample size. The diagonal $y=x$
indicates exact calibration.

\begin{figure}[htbp]
    \centering
    \includegraphics[width=\textwidth]{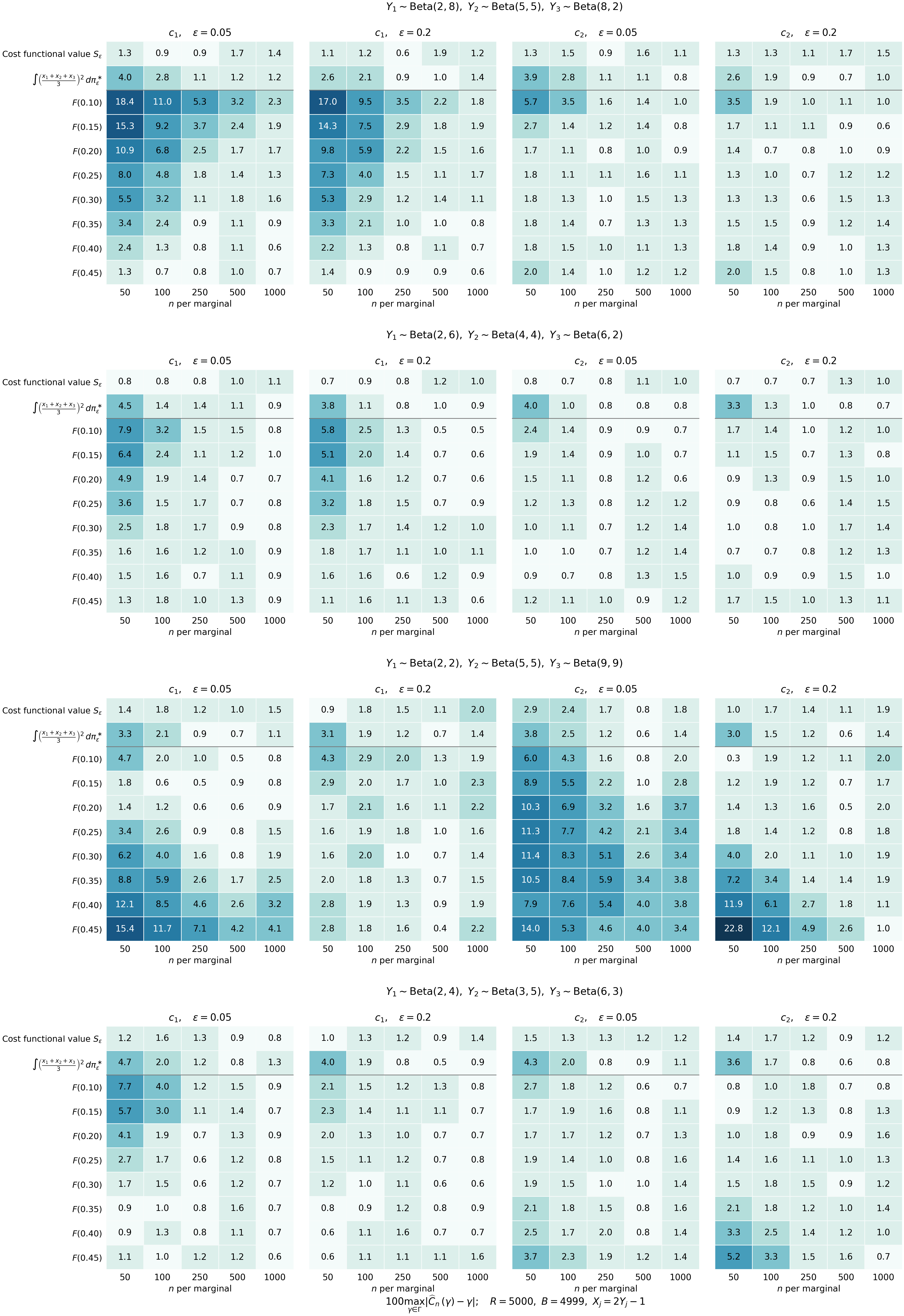}
    \caption{
    The input observations are $X_j=2Y_j-1$, with samples drawn
independently from the three marginal distributions. 
Each entry is
$100\max_{\gamma\in\Gamma}
|\widehat C_n(\gamma)-\gamma|$,
where $\Gamma=\{0.05,0.10,\ldots,0.95\}$ and
$\widehat C_n(\gamma)$ denotes empirical coverage at nominal
confidence level $\gamma$.
Entries are expressed in percentage points and rounded
to one decimal place.
}
    \label{fig:bootstrap-calibration}
\end{figure}

% ============================================================
% Beta(2,8), Beta(5,5), Beta(8,2)
% ============================================================
\CalibrationFour
  {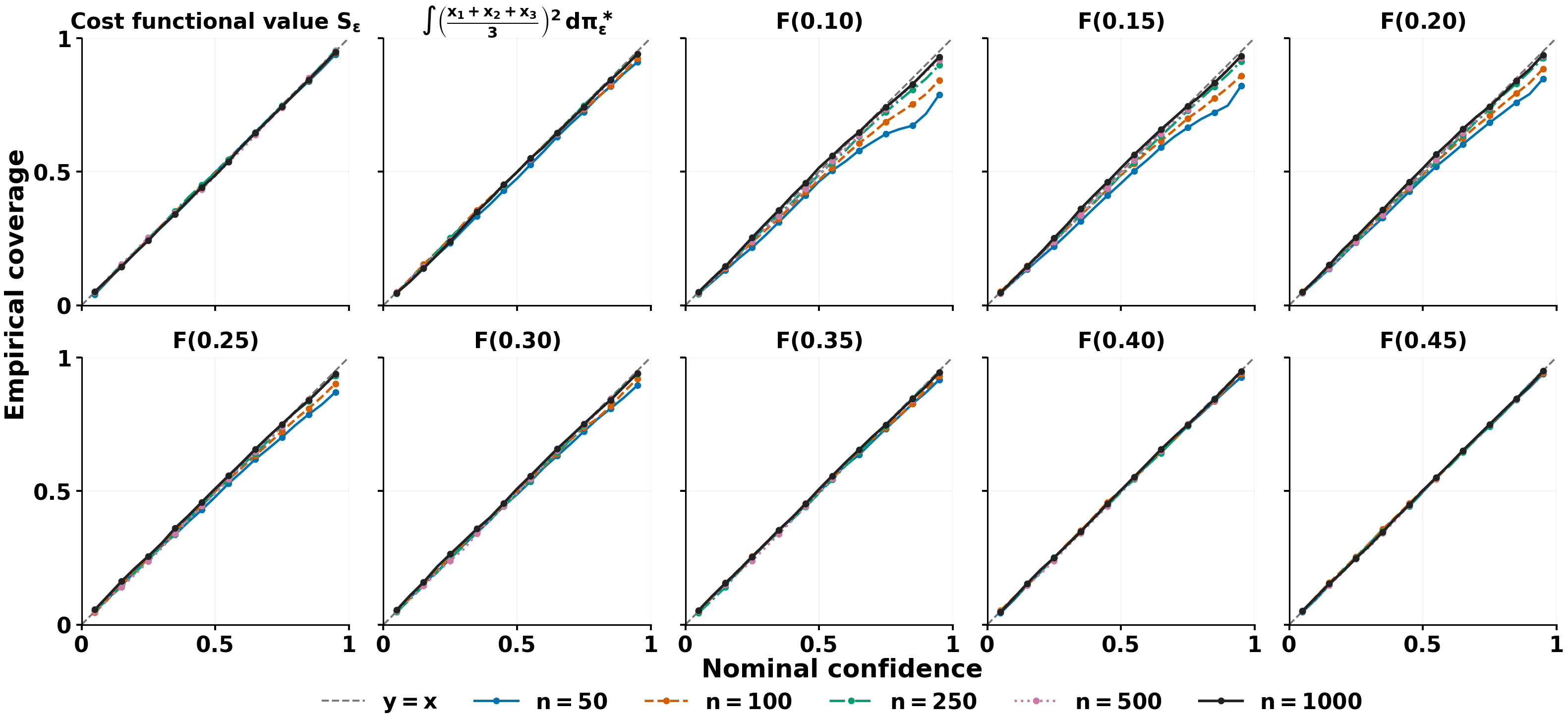}
  {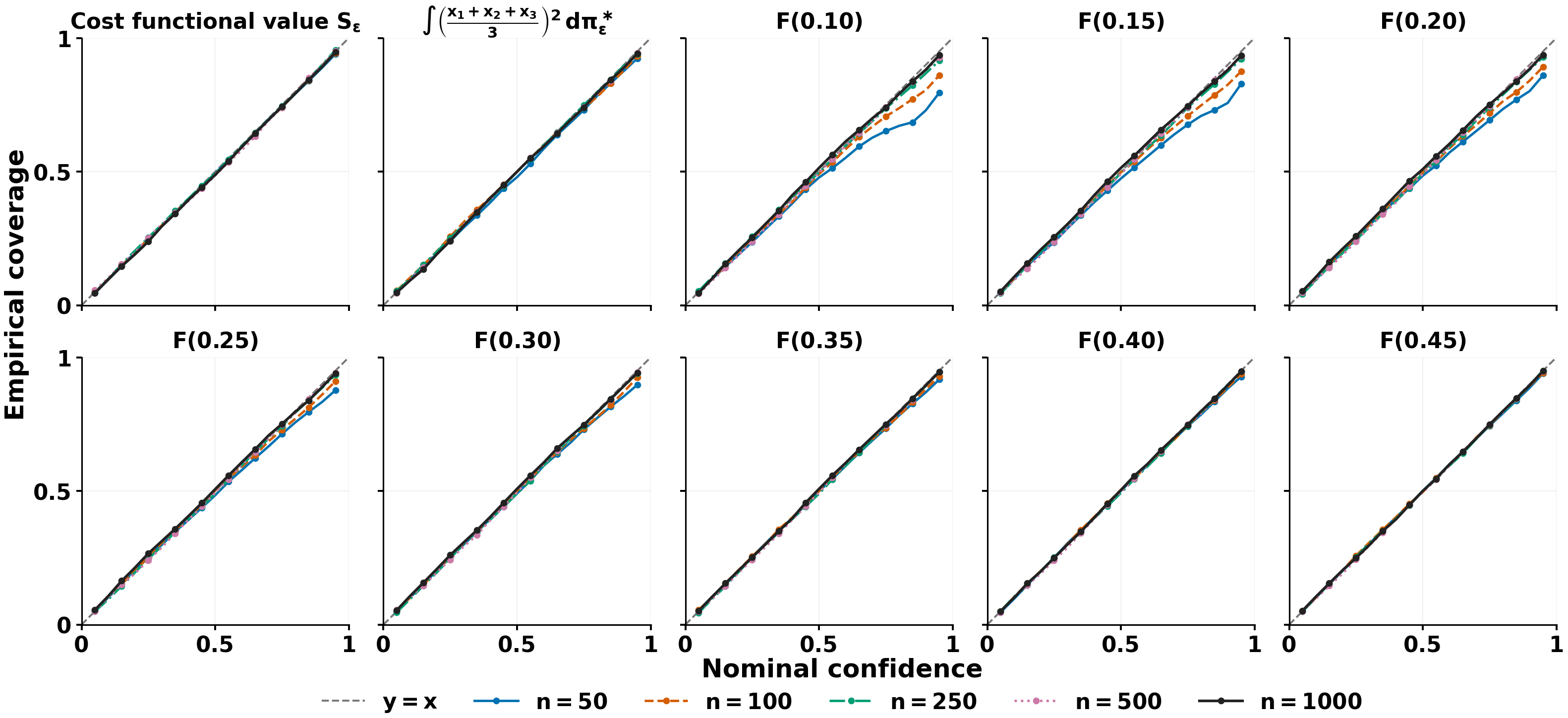}
  {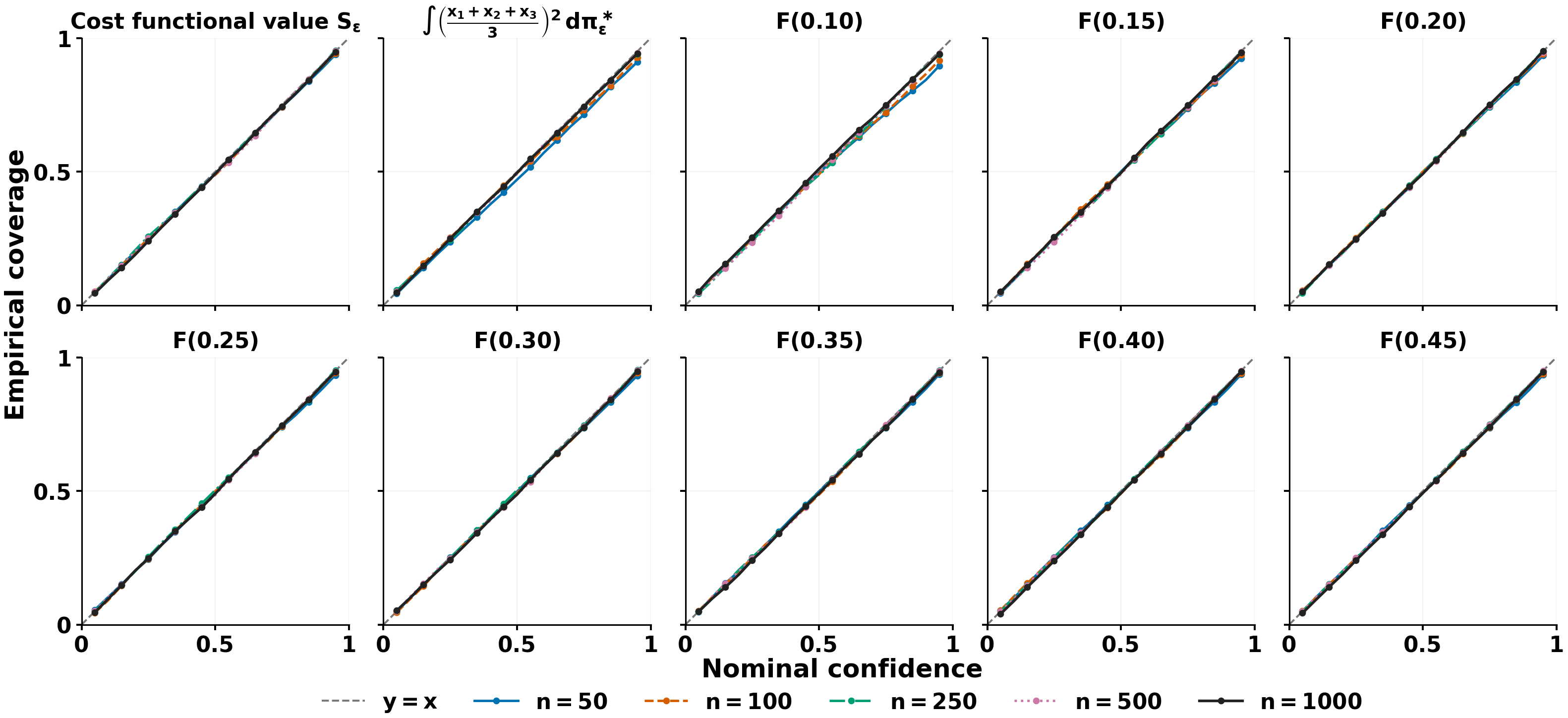}
  {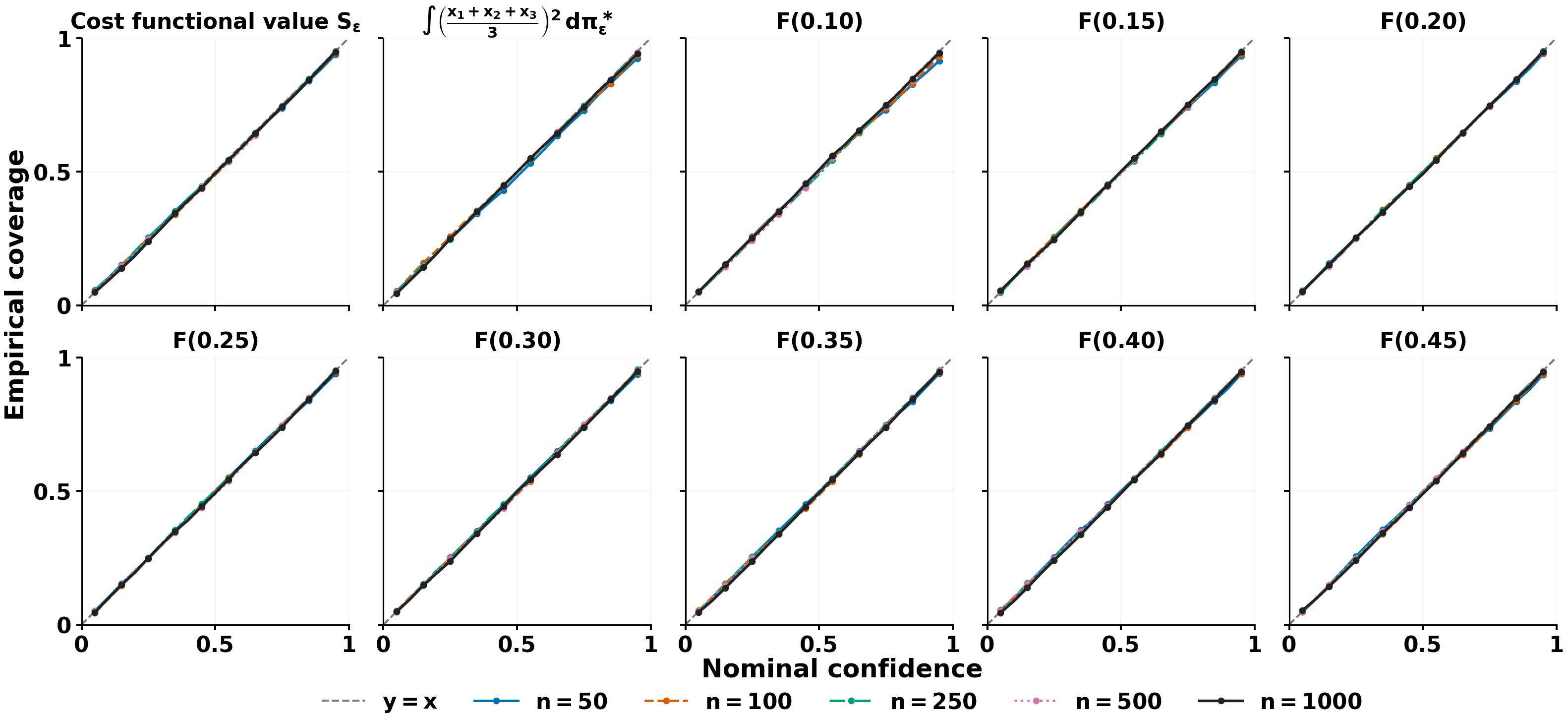}
  {\MarginalA}
  {fig:calibration-beta-2-8}

% ============================================================
% Beta(2,6), Beta(4,4), Beta(6,2)
% ============================================================

\CalibrationFour
  {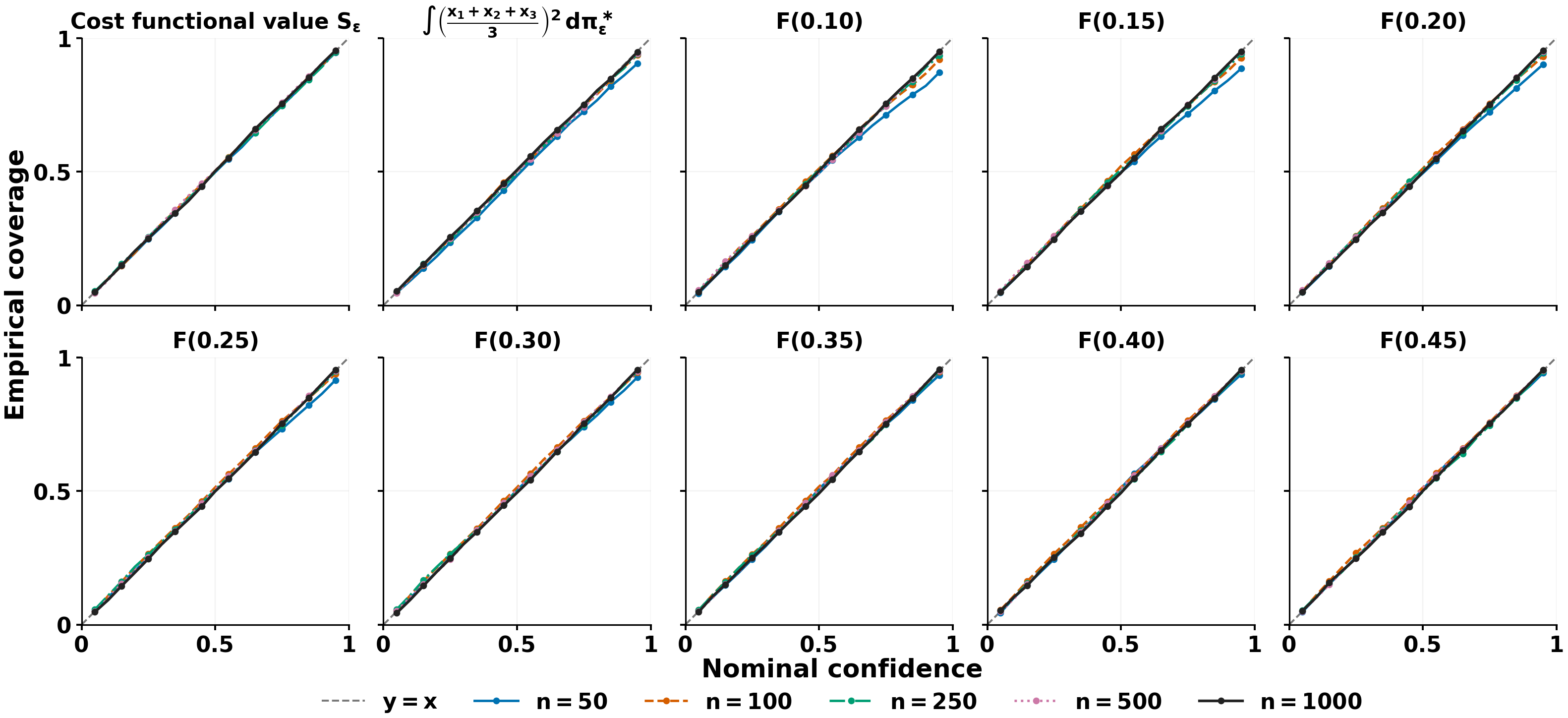}
  {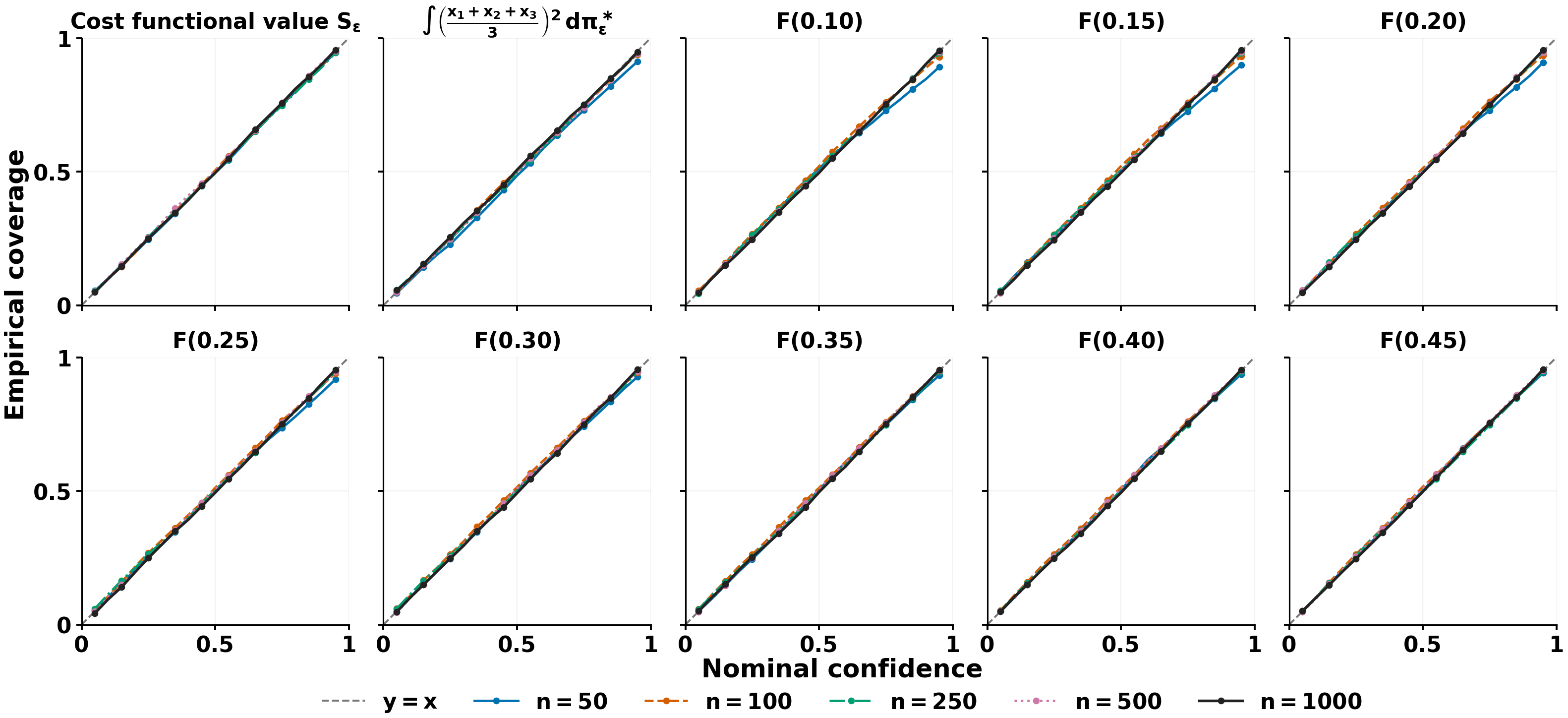}
  {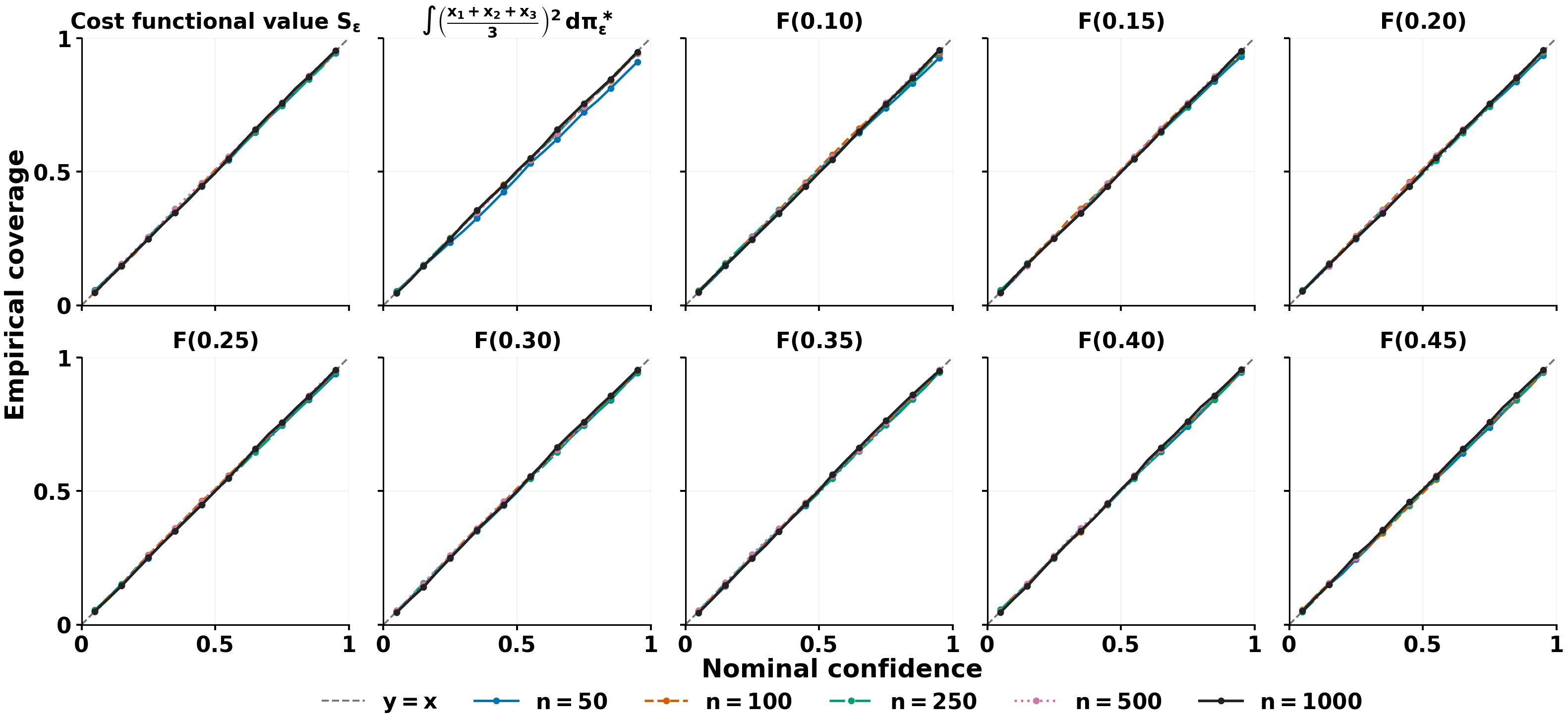}
  {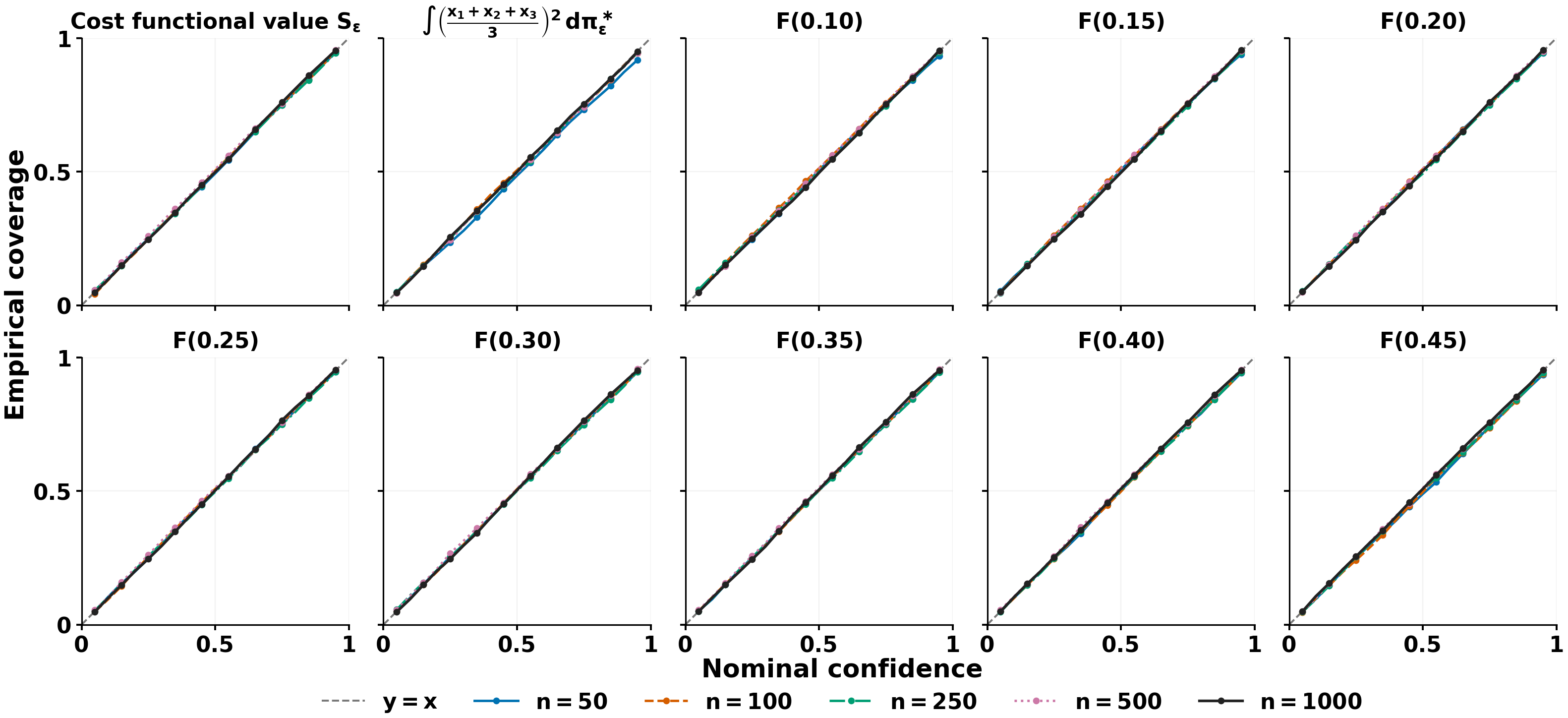}
  {\MarginalC}
  {fig:calibration-beta-2-6}

% ============================================================
% Beta(2,2), Beta(5,5), Beta(9,9)
% ============================================================

\CalibrationFour
  {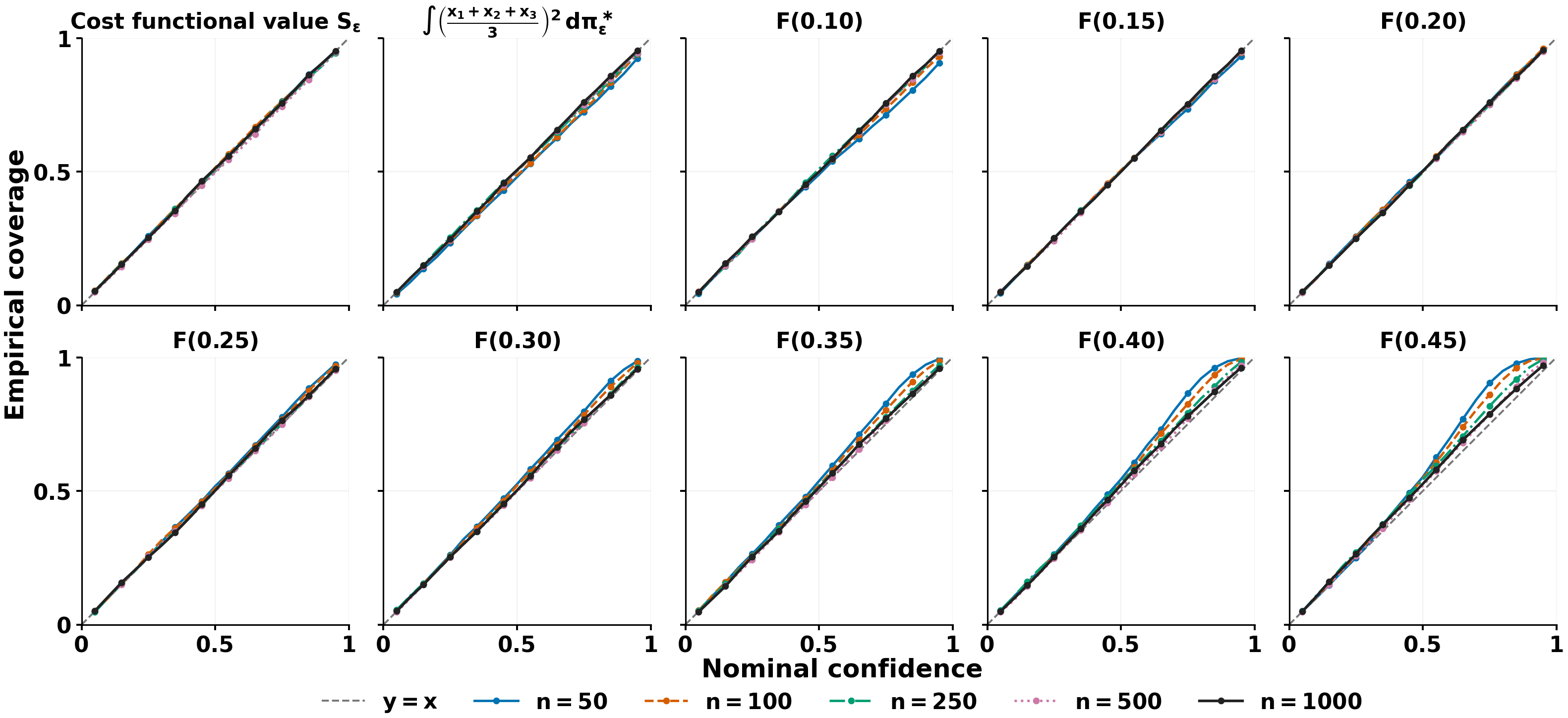}
  {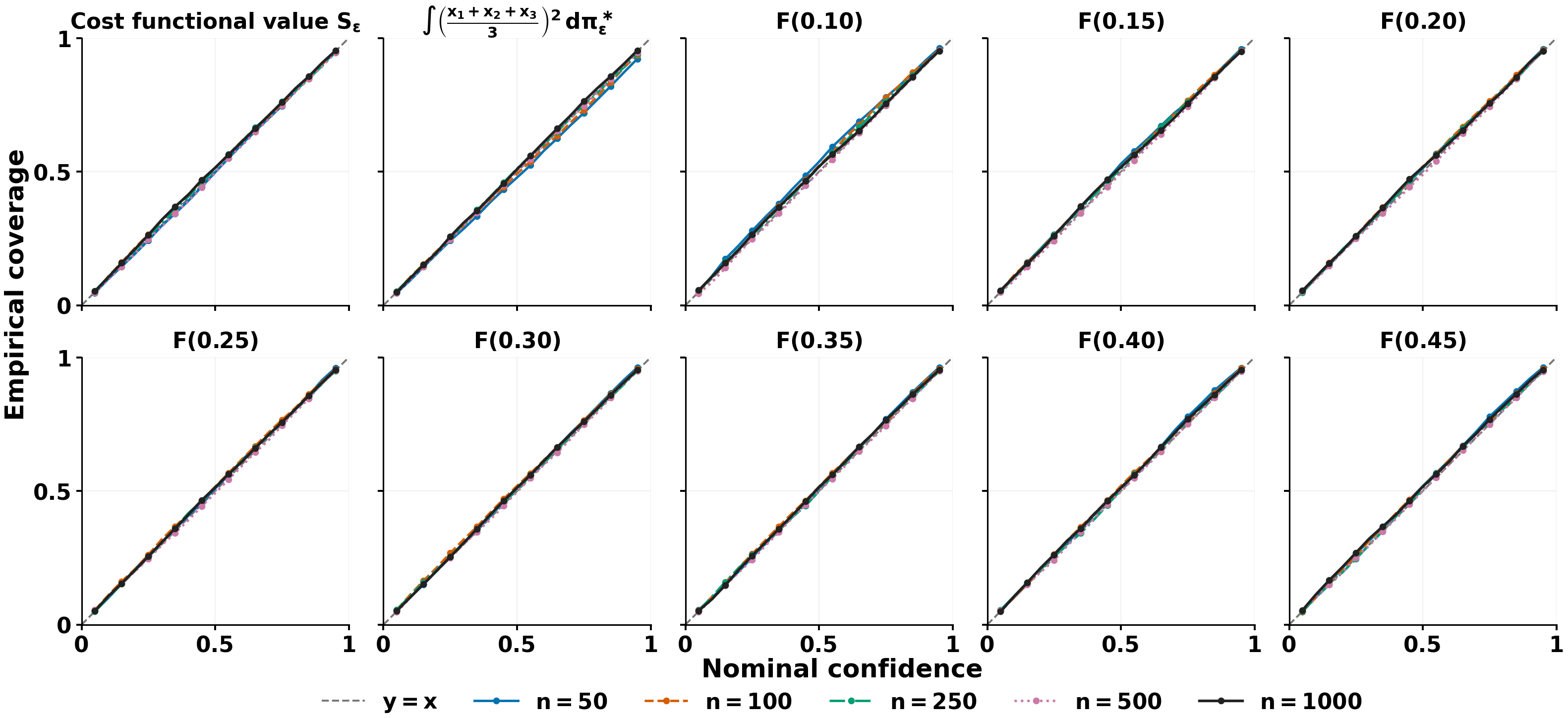}
  {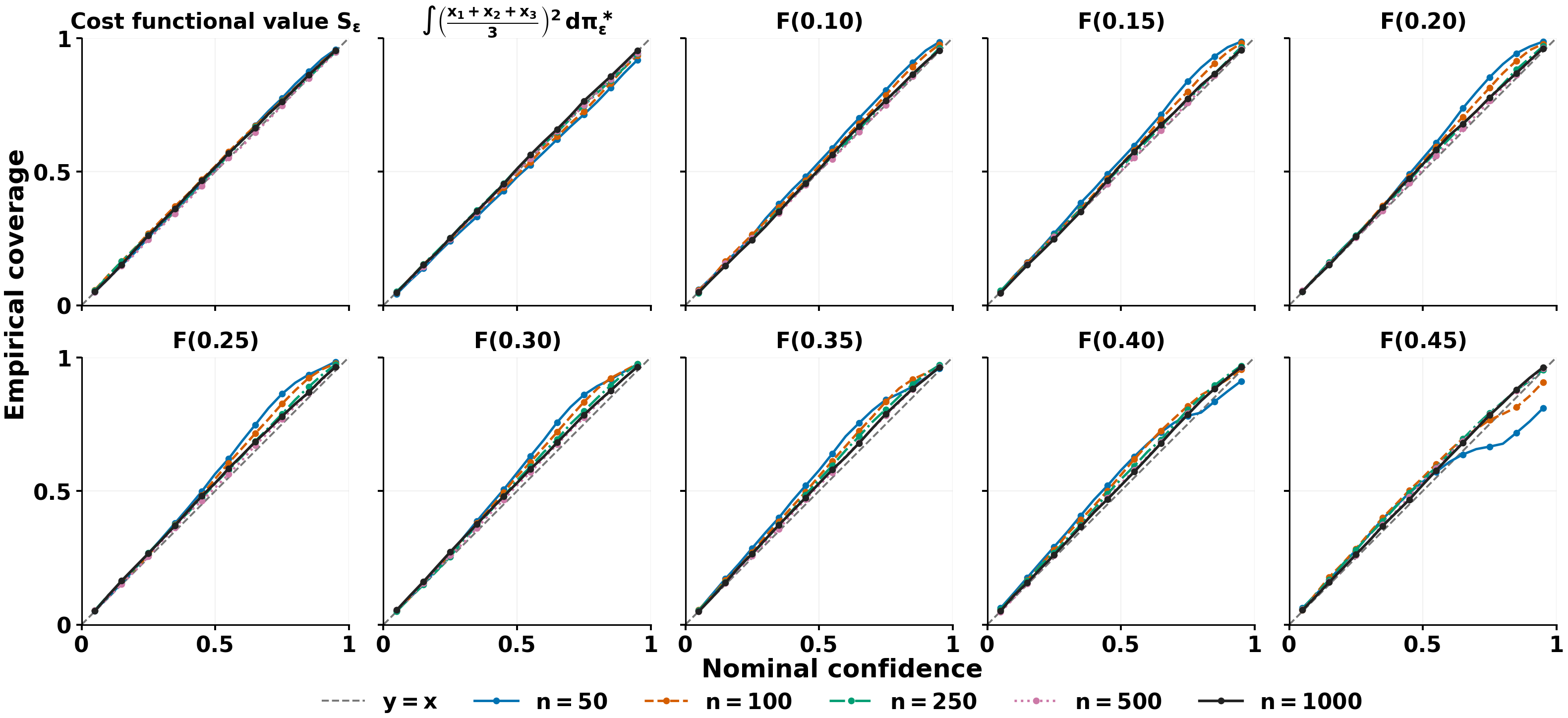}
  {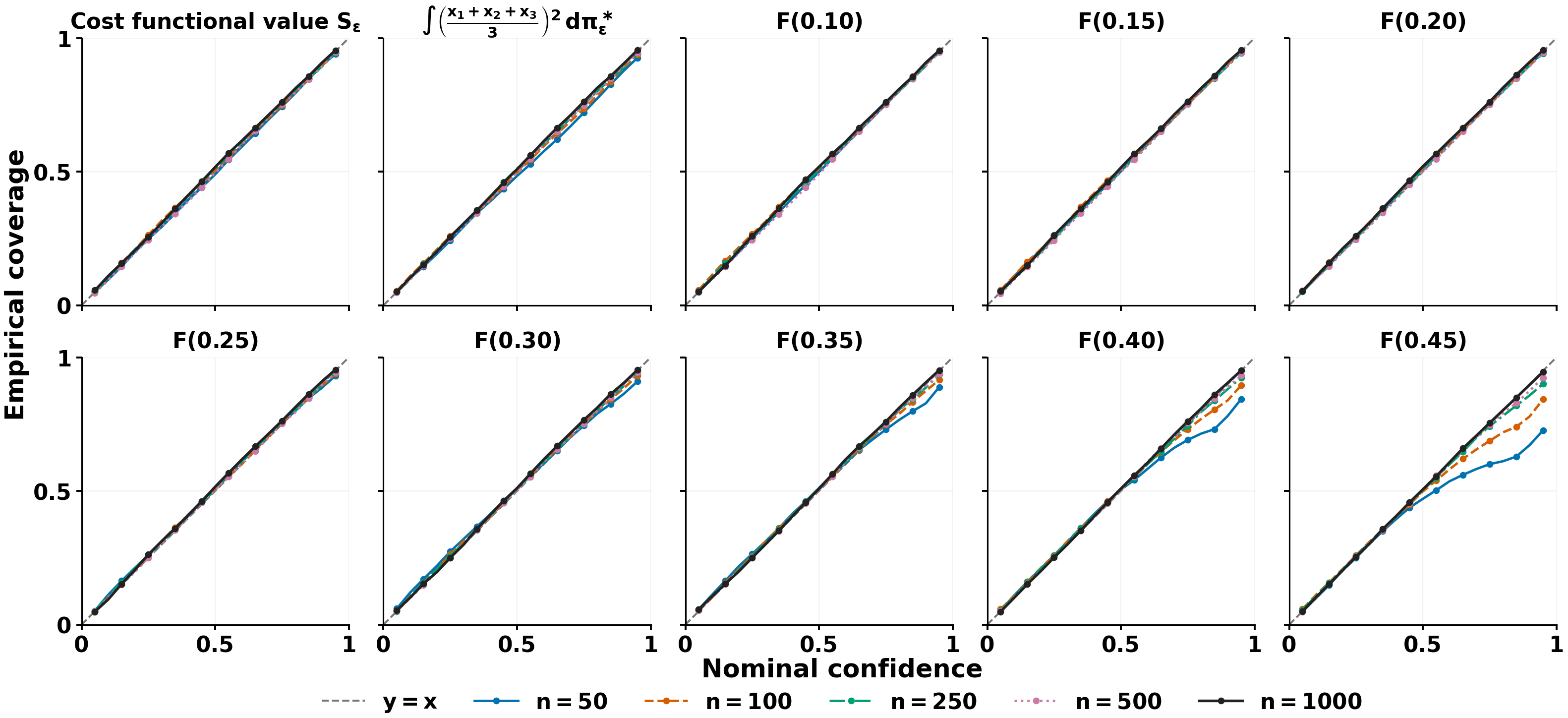}
  {\MarginalD}
  {fig:calibration-beta-2-2-5-5-9-9}
% ============================================================
% Beta(2,4), Beta(3,5), Beta(6,3)
% ============================================================

\CalibrationFour
  {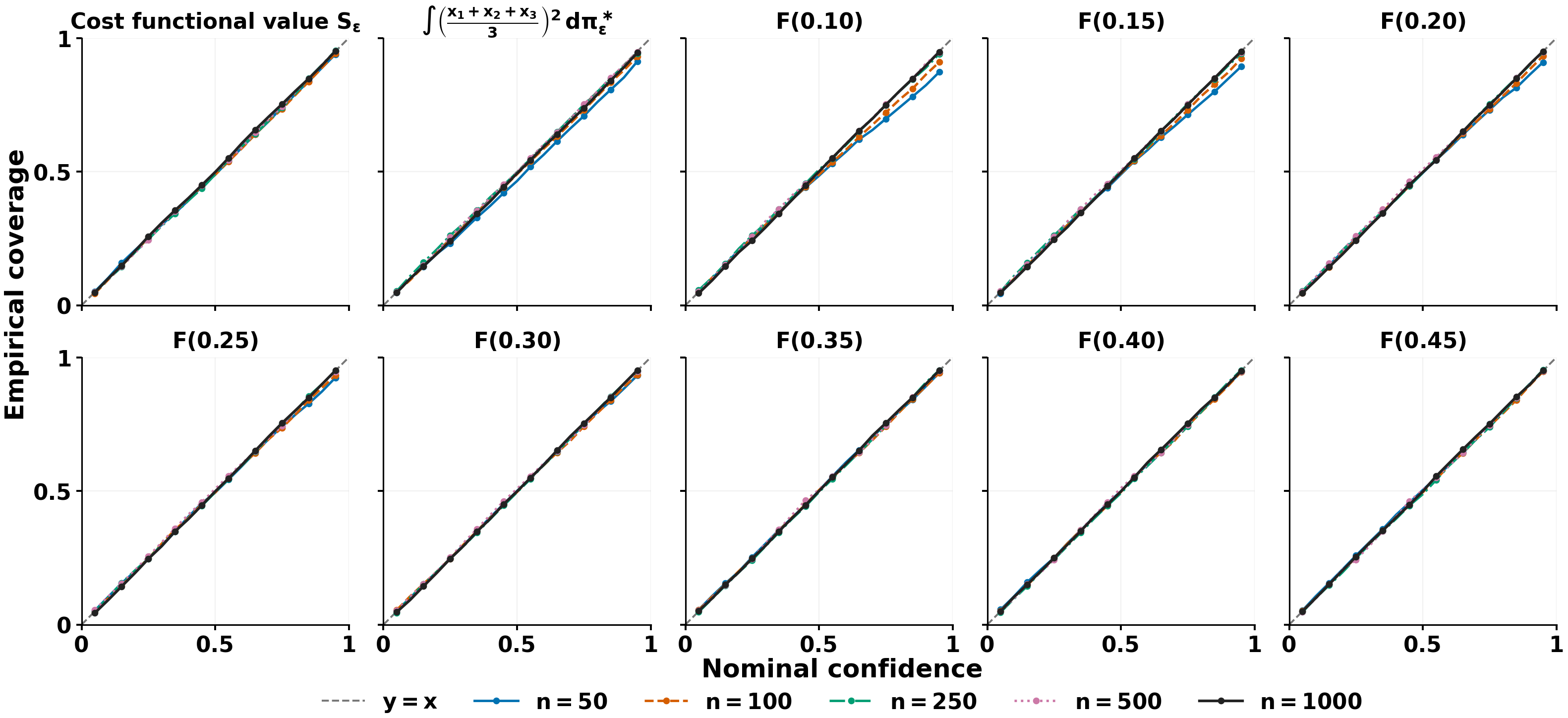}
  {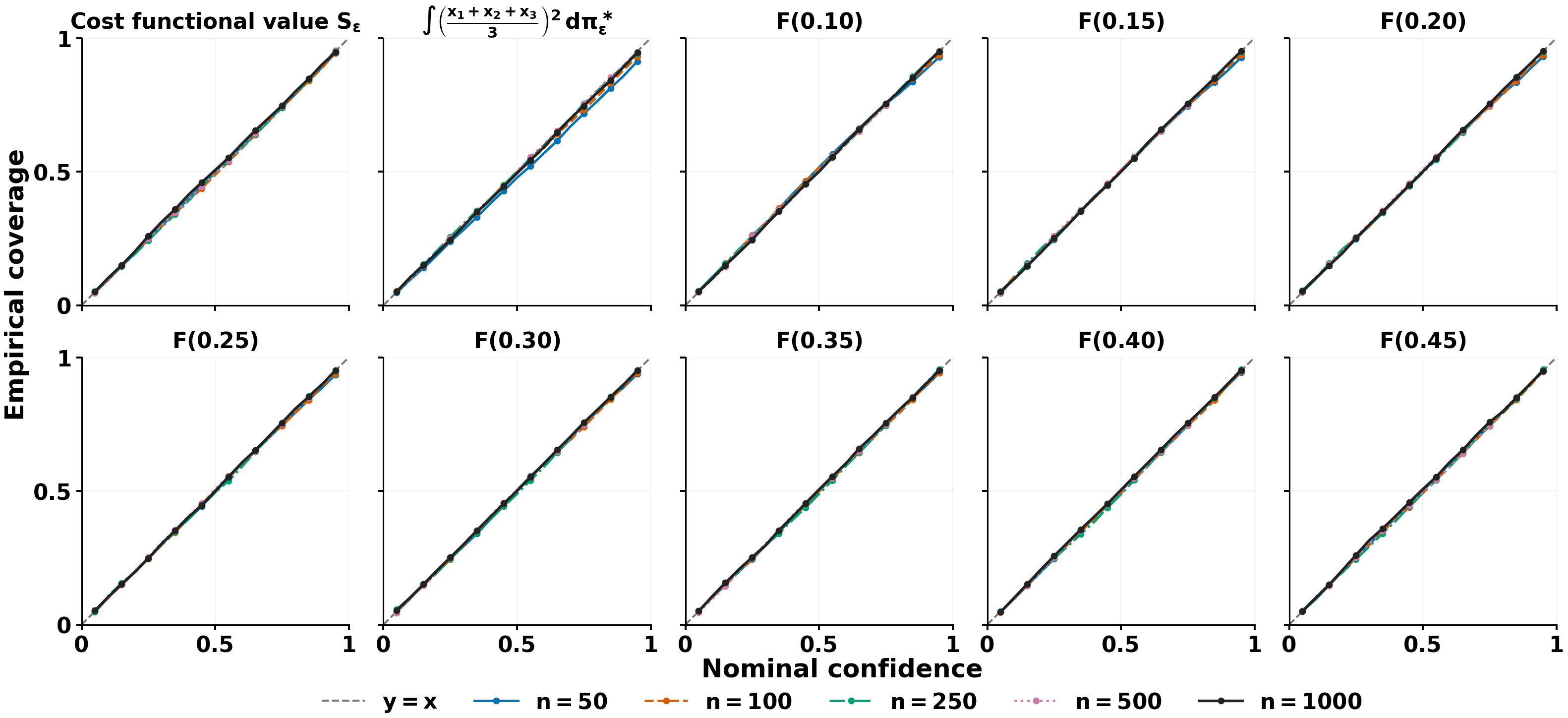}
  {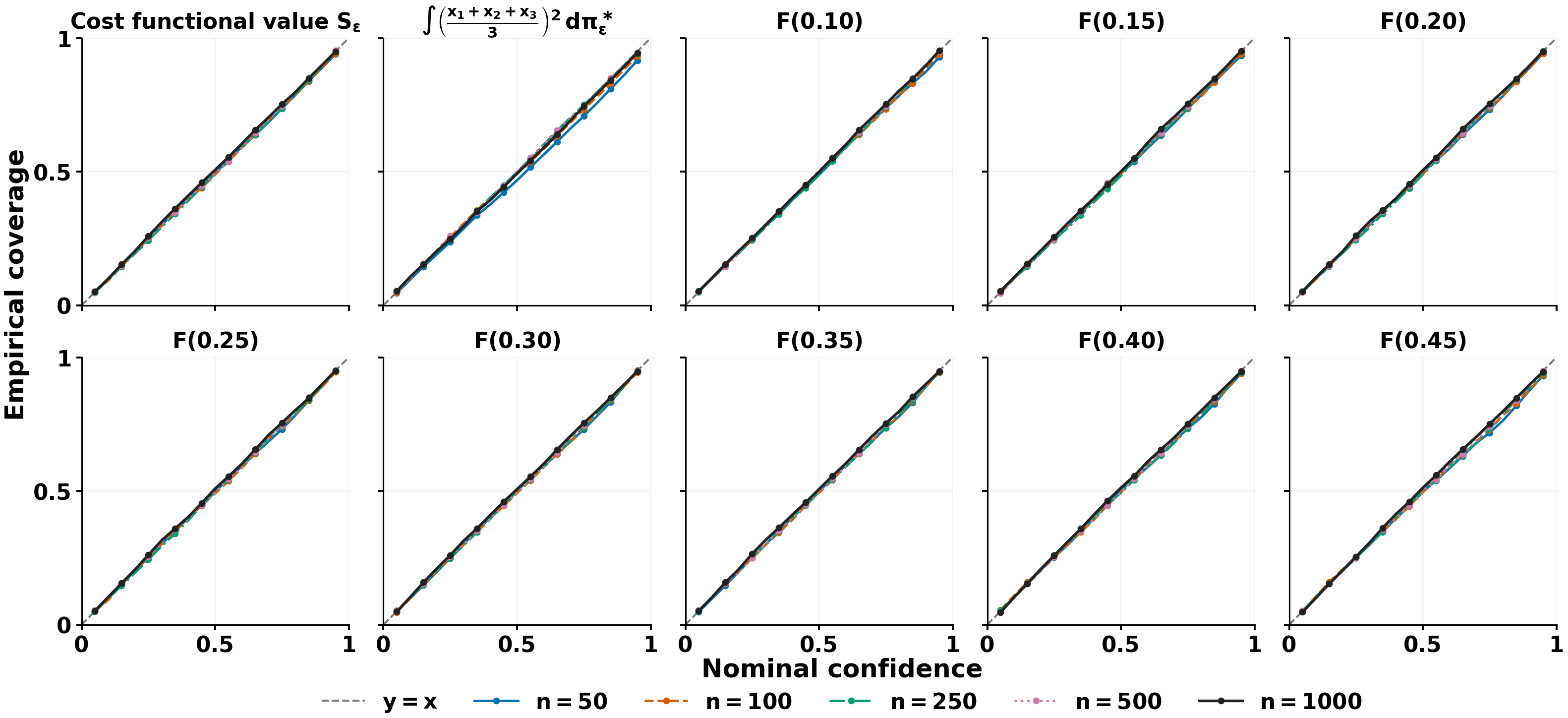}
  {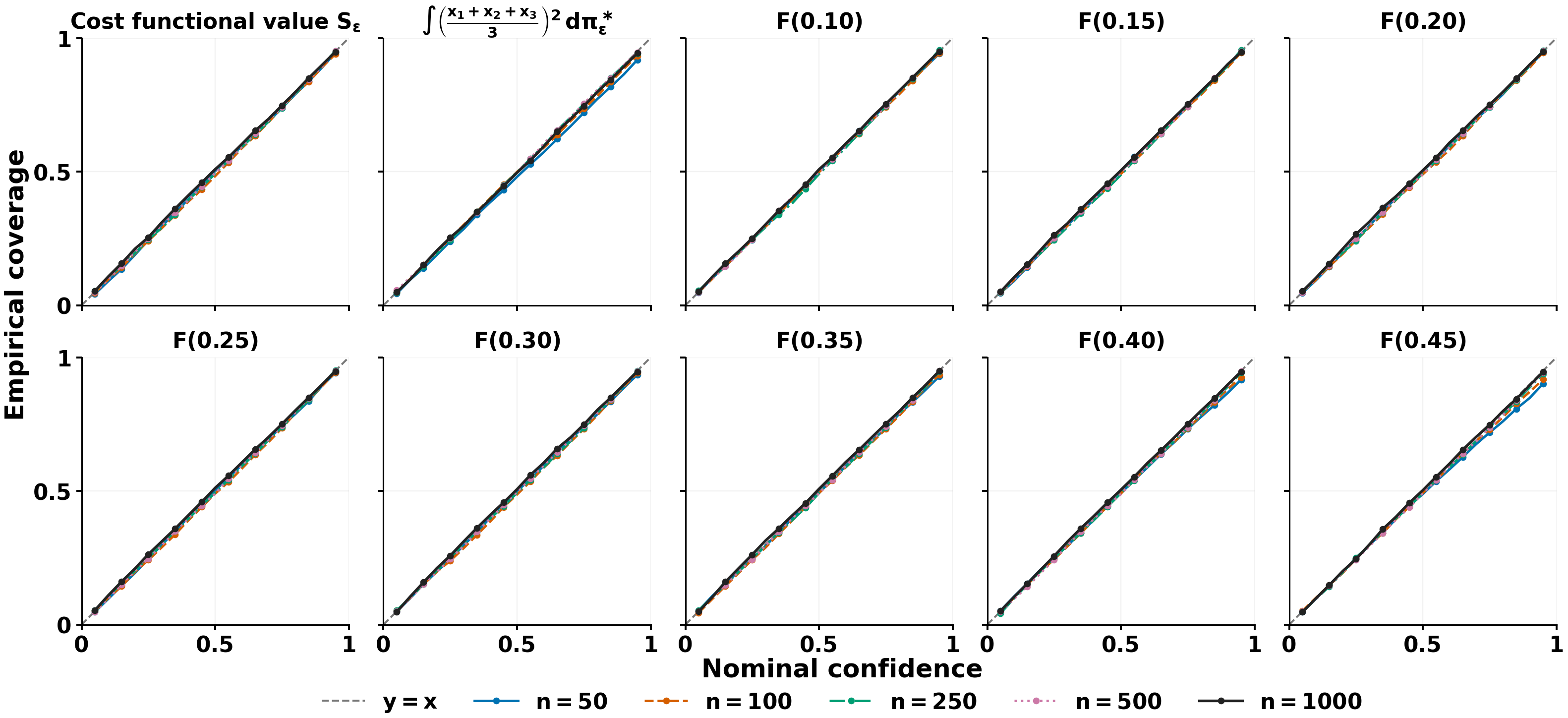}
  {\MarginalE}
  {fig:calibration-beta-2-4}

\end{document}